\documentclass[11pt,reqno]{amsart}

\usepackage{amsmath}
\usepackage{amsfonts}
\usepackage{amssymb}
\usepackage{amsthm}
\usepackage{mathrsfs}

\usepackage{etoolbox}
\makeatletter
\let\ams@starttoc\@starttoc
\makeatother
\usepackage[parfill]{parskip}
\makeatletter
\let\@starttoc\ams@starttoc
\patchcmd{\@starttoc}{\makeatletter}{\makeatletter\parskip\z@}{}{}
\makeatother

\usepackage{enumerate}
\usepackage{geometry}
\usepackage{hyperref}

\newcommand{\RR}{\mathbf{R}}

\newcommand{\NN}{\mathbf{N}}

\newcommand{\cA}{\mathcal{A}}
\newcommand{\cB}{\mathcal{B}}

\newcommand{\cE}{\mathcal{E}}
\newcommand{\cF}{\mathcal{F}}
\newcommand{\cG}{\mathcal{G}}
\newcommand{\cH}{\mathcal{H}}
\newcommand{\cI}{\mathcal{I}}

\newcommand{\cL}{\mathcal{L}}
\newcommand{\cM}{\mathcal{M}}

\newcommand{\cQ}{\mathcal{Q}}
\newcommand{\cR}{\mathcal{R}}
\newcommand{\cS}{\mathcal{S}}

\newcommand{\cU}{\mathcal{U}}

\newcommand{\sB}{\mathscr{B}}

\newcommand{\sE}{\mathscr{E}}

\newcommand{\sL}{\mathscr{L}}
\newcommand{\sM}{\mathscr{M}}

\DeclareMathOperator{\lspan}{span}

\DeclareMathOperator{\ind}{ind}
\DeclareMathOperator{\nul}{nul}

\DeclareMathOperator{\support}{spt}
\DeclareMathOperator{\supp}{spt}

\DeclareMathOperator{\area}{area}

\DeclareMathOperator{\graph}{graph}
\DeclareMathOperator{\dist}{dist}

\DeclareMathOperator{\loc}{loc}

\DeclareMathOperator{\divg}{div}

\DeclareMathOperator{\genus}{genus}

\DeclareMathOperator{\riem}{Rm}
\DeclareMathOperator{\ricc}{Ric}
\DeclareMathOperator{\inj}{inj}

\DeclareMathOperator{\sff}{\mathrm{I\!I}}

\newcommand{\restr}{\mathbin{\vrule height 1.6ex depth 0pt width
0.13ex\vrule height 0.13ex depth 0pt width 1.3ex}}
\newcommand{\eps}{\varepsilon}

\newcommand{\weaklyto}{\rightharpoonup}

\newcommand{\energyunit}{h_0}
\newcommand{\expansioncoeff}{A_0}
\newcommand{\ve}[1]{\mathbf{#1}}

\theoremstyle{plain} \newtheorem{defi}{Definition}
\theoremstyle{plain} \newtheorem{rema}[defi]{Remark}
\theoremstyle{plain} \newtheorem{theo}[defi]{Theorem}
\theoremstyle{plain} \newtheorem{prop}[defi]{Proposition}
\theoremstyle{plain} \newtheorem{coro}[defi]{Corollary}
\theoremstyle{plain} \newtheorem{lemm}[defi]{Lemma}
\theoremstyle{plain} \newtheorem{exam}[defi]{Example}
\theoremstyle{plain} \newtheorem{conj}[defi]{Conjecture}
\theoremstyle{plain} \newtheorem*{theo*}{Theorem}
\theoremstyle{plain} 
\theoremstyle{plain} 
\theoremstyle{definition} \newtheorem*{clai}{Claim}

\numberwithin{defi}{section} 
\numberwithin{equation}{section} 

\allowdisplaybreaks

\title[Allen--Cahn on $3$-manifolds]{Minimal surfaces and the Allen--Cahn equation on 3-manifolds: index, multiplicity, and curvature estimates}
\author{Otis Chodosh}
\address{Department of Mathematics\\Princeton
University\\Princeton, NJ 08544}
\address{School of Mathematics\\Institute for Advanced Study\\Princeton, NJ 08540}
\email{ochodosh@math.princeton.edu}
\author{Christos Mantoulidis}
\address{Department of Mathematics\\ Massachusetts Institute of Technology \\Cambridge, MA 02139}
\email{c.mantoulidis@mit.edu}
\date{\today}
\begin{document}

\begin{abstract}
The Allen--Cahn equation is a semilinear PDE which is deeply linked to the theory of minimal hypersurfaces via a singular limit. We prove \emph{curvature estimates} and strong \emph{sheet separation estimates} for stable solutions (building on recent work of Wang--Wei \cite{WangWei}) of the Allen--Cahn equation on a $3$-manifold. Using these, we are able to show for generic metrics on a $3$-manifold, minimal surfaces arising from Allen--Cahn solutions with bounded energy and bounded Morse index are two-sided and occur with multiplicity one and the expected Morse index. This confirms, in the Allen--Cahn setting, a strong form of the \emph{multiplicity one conjecture} and the \emph{index lower bound conjecture} of Marques--Neves \cite{marques:ICM, neves:ICM} in $3$-dimensions regarding min-max constructions of minimal surfaces. 

Allen--Cahn min-max constructions were recently carried out by Guaraco \cite{Guaraco} and Gaspar--Guaraco \cite{GasparGuaraco}. Our resolution of the multiplicity one and the index lower bound conjectures shows that these constructions can be applied to give a new proof of \emph{Yau's conjecture on infinitely many minimal surfaces} in a $3$-manifold with a generic metric (recently proven by Irie--Marques--Neves \cite{IrieMarquesNeves}) with \emph{new} geometric conclusions. Namely, we prove that a $3$-manifold with a generic metric contains, for every $p = 1, 2, 3, \ldots$, a two-sided embedded minimal surface with Morse index $p$ and area $\sim p^{\frac 13}$, as conjectured by Marques--Neves. 
\end{abstract}

\maketitle

\vspace{-1cm}

\tableofcontents

\section{Introduction}

Minimal surfaces---critical points of the area functional with respect to local deformations---are fundamental objects in Riemannian geometry due to their intrinsic interest and richness, as well as deep and surprising applications to the study of other geometric problems. Because many manifolds do not contain \emph{any} area-minimizing hypersurfaces, one is quickly led to the study of surfaces that are only critical points of the area functional. Such surfaces are naturally constructed by min-max (i.e., mountain-pass) type methods. To this end, Almgren and Pitts \cite{Pitts} have developed a far-reaching theory of existence and regularity (cf.\ \cite{SchoenSimon}) of min-max (unstable) minimal hypersurfaces. In particular, their work implies that any closed Riemannian manifold $(M^{n},g)$ contains at least one minimal hypersurface $\Sigma^{n-1}$ (in sufficiently high dimensions, $\Sigma$ may have a thin singular set). This result motivates a well-known question of Yau: ``do all $3$-manifolds contain infinitely many immersed minimal surfaces?'' \cite{Yau:problems}.

Recently, there have been several amazing applications of Almgren--Pitts theory to geometric problems, including the proof of the Willmore conjecture by Marques--Neves \cite{MarquesNeves:Willmore} and the resolution of Yau's conjecture for generic metrics in dimensions 3 through 7 by Irie--Marques--Neves \cite{IrieMarquesNeves}. In spite of this, certain basic questions concerning the Almgren--Pitts construction remain unresolved: including whether or not the limiting minimal surfaces can arise with multiplicity (for a generic metric) as well as whether or not one-sided minimal surfaces can arise as limits of an ``oriented'' min-max sequence (see, however, \cite{KMN:catenoid,MarquesNeves:multiplicity}).
\footnote{
Added in proof: There has been dramatic progress in Almgren--Pitts theory since we first posted this article. In particular, we note that A. Song \cite{Song:full-yau} has proved the full Yau conjecture in dimensions 3 through 7, and X. Zhou \cite{Zhou:multiplicity-one} proved the multiplicity one conjecture in the Almgren--Pitts setting, also in dimensions 3 through 7.
}

Guaraco \cite{Guaraco} has proposed an alternative to Almgren--Pitts theory, later extended by Gaspar--Guaraco \cite{GasparGuaraco}, which is based on study of a semilinear PDE known as the Allen--Cahn equation
\begin{equation} \label{eq:ac.pde}
\varepsilon^{2} \Delta_{g} u = W'(u) 
\end{equation}
and its singular limit as $\eps\searrow 0$. There is a well known expectation that, in $\eps\searrow 0$ limit, solutions to \eqref{eq:ac.pde} produce minimal surfaces whose regularity reflects the solutions' variational properties. In particular:
\begin{enumerate}
	\item It is known that the Allen--Cahn functional $\Gamma$-converges to the perimeter functional \cite{Modica,Sternberg}, so minimizing solutions to \eqref{eq:ac.pde} converge as $\eps \searrow 0$ to minimizing hypersurfaces (and are thus regular away from a codimension $7$ singular set). 
	\item Under weaker assumptions on the sequence of solutions, one obtains different results. In general, solutions to \eqref{eq:ac.pde} on a Riemannian manifold $(M^n, g)$ have a naturally associated $(n-1)$-varifold obtained by ``smearing out'' their level sets of $u$, weighted by the gradient,
	\[ V[u](\varphi) \triangleq \energyunit^{-1} \int \varphi(x, T_x \{ u = u(x) \}) \, \varepsilon |\nabla u(x)|^2 \,  d\mu_g(x), \; \varphi \in C^0_c(\operatorname{Gr}_{n-1}(M)). \]
	Here, $h_{0} > 0$ is a constant that is canonically associated with $W$ (see Section \ref{subsec:heteroclinic.solution}). A deep result of Hutchinson--Tonegawa \cite[Theorem 1]{HutchinsonTonegawa00} ensures that $V$ limits to a varifold with a.e.\ integer density as $\eps\searrow0$. If, in addition, one assumes that the solutions are stable, Tonegawa--Wickramasekera \cite{TonegawaWickramasekera12} have shown that the limiting varifold is stable and satisfies the conditions of Wickramasekera's deep regularity theory \cite{Wickramasekera14}; thus the limiting varifold is a smooth stable minimal hypersurface (outside of a codimension $7$ singular set). In two dimensions, this was shown by Tonegawa \cite{Tonegawa05}.
\end{enumerate}

Guaraco's approach has certain advantages when compared with Almgren--Pitts theory:
\begin{enumerate}
	\item A key difficulty in the work of Almgren--Pitts is a lack of a Palais--Smale condition, which is usually fundamental in mountain pass constructions. On the other hand, the Allen--Cahn equation does satisfy the usual Palais--Smale condition for each $\eps>0$ (see \cite[Proposition 4.4]{Guaraco}), so this aspect of the theory is much simpler. 
	
	We note, however, that the bulk of the regularity theory in Guaraco's work is applied \emph{after} taking the limit $\eps\searrow 0$ and thus relies on the deep works of Wickaramsekera \cite{Wickramasekera14} and Tonegawa--Wickramasekera \cite{TonegawaWickramasekera12}. This places a more serious burden on regularity theory than Almgren--Pitts.
	
	\item In Almgren--Pitts theory, there is no ``canonical'' approximation of the limiting min-max surface by nearby elements of a sweepout. On the other hand, Allen--Cahn provides a canonical approximation built out of the function $u$ (which satisfies a PDE). It is thus natural to suspect that this might be useful when studying the geometric properties of the limiting surface. 
	
	For example, Hiesmayr \cite{Hiesmayr} and Gaspar \cite{Gaspar} have shown that index upper bounds for Allen--Cahn solutions directly pass to the limiting surface (we note that the Almgren--Pitts version of this result has been proven by Marques--Neves \cite{MarquesNeves:multiplicity}). Moreover, the second-named author has recently shown \cite{Mantoulidis} that $1$-parameter Allen--Cahn min-max on a surface produces a smooth immersed curve with at most one point of self-intersection; in general, Almgren--Pitts on a surface will only produce a geodesic net (cf.\ \cite{Aiex:ellipsoids}).
\end{enumerate}

Our main contributions in this work are as follows:
\begin{enumerate}
	\item We show (see Theorem \ref{theo:curv.est} below) that the individual level sets of stable solutions to the Allen--Cahn equation on a $3$-manifold with energy bounds satisfy a priori curvature estimates (similar to stable minimal surfaces). Using this, we are can avoid the regularity theory of Wickramasekera and Tonegawa--Wickramasekera entirely, making the whole theory considerably more self-contained. 
	\item More fundamentally, our curvature estimates (and strong sheet separation estimates, which we will discuss below) allow us to study geometric properties of the limiting minimal surface using the ``canonical'' PDE approximations that exist \emph{prior} to taking the $\eps \searrow 0$ limit. In particular, we will prove the multiplicity one conjecture of Marques--Neves \cite{MarquesNeves:multiplicity} in the Allen--Cahn setting (see Theorem \ref{theo:mult.intro-version} below) for min-max sequences on $3$-manifolds. In fact, we prove a strengthened version of the conjecture by ruling out (generically) stable components and one-sided surfaces. 
\end{enumerate}

As an application of our multiplicity one results we are able to give a new proof of Yau's conjecture on infinitely many minimal surfaces in a $3$-manifold, when the metric is bumpy (see Corollary \ref{coro:yau-intro} below). This has been recently proven using Almgren--Pitts theory\footnote{We note that after the first version of this work was posted, Gaspar--Guaraco \cite{GasparGuaraco:weyl} gave a new proof of  Yau's conjecture for generic metrics (in the spirit of Irie--Marques--Neves \cite{IrieMarquesNeves}) by proving a Weyl law for their Allen--Cahn $p$-widths.} by Irie--Marques--Neves \cite{IrieMarquesNeves}, for a slightly different class of metrics; their proof works in $(M^{n},g)$ for $3\leq n\leq 7$ and proves, in addition, that the minimal surfaces are dense. Our proof establishes several new geometric properties of the surfaces; in particular, we show that they are two-sided and that their area and Morse index behaves as one would expect, based on the theory of $p$-widths \cite{Gromov:waist,Guth:minimax,MarquesNeves:posRic,GasparGuaraco}. 

We wish to emphasize two things:
\begin{enumerate}
	\item Our results work at the level of sequences of critical points of the Allen--Cahn energy functional with uniform energy and Morse index bounds. At no point do we use any min-max characterization of the limiting surface; min-max is merely used as a tool to construct nontrivial sequences of critical points with energy and index bounds.
	\item Our results highlight the philosophy that the solutions to Allen--Cahn provide a ``canonical'' approximation of the min-max surfaces. 
\end{enumerate}

\subsection{Notation}

In all that follows, $(M^n, g)$ is a smooth Riemannian manifold.

\begin{defi}
	A function $W \in C^{\infty}(\RR)$ is a \emph{double-well potential} if:
	\begin{enumerate}
		\item $W$ is non-negative and vanishes precisely at $\pm 1$;
		\item $W$ satisfies $W'(0) = 0$, $t W'(t) < 0$ for $|t| \in (0,1)$, and $W''(0) \not = 0$;
		\item $W''(\pm 1) =2 $;
		\item $W(t) = W(-t)$.
	\end{enumerate}
\end{defi}
The standard double-well potential is $W(t) = \frac 1 4 (1-t^{2})^{2}$, in which case \eqref{eq:ac.pde}  becomes $\varepsilon^{2} \Delta_{g} u = u^{3}-u$. 

The Allen--Cahn equation, \eqref{eq:ac.pde}, is the Euler--Lagrange equation for the energy functional
\[
E_{\eps}[u] = \int_{M} \left( \frac \eps 2 |\nabla u|^{2} + \frac{W(u)}{\eps} \right) \, d\mu_{g}.
\]
Depending on what we wish to emphasize, we will go back and forth between saying that a function $u$ is a solution of \eqref{eq:ac.pde} on $M$ (or in a domain $U \subset M$) or a critical point of $E_\eps$ (resp. of $E_\eps \restr U$). The second variation of $E_{\eps}$ is easily computed (for $\zeta,\psi \in C^{\infty}_{c}(M)$) to be
\begin{equation}\label{eq:second.var.AC}
\delta^{2}E_{\eps}[u]\{\zeta,\psi\}  =\int_{M} \left( \eps \langle\nabla \zeta,\nabla \psi \rangle + \frac{W''(u)}{\eps} \zeta\psi\right) \, d\mu_{g}.
\end{equation}
We are thus led to the notion of stability and Morse index (with respect to Dirichlet eigenvalues). 
\begin{defi} \label{def:ac.stable} \label{def:ac.morse.index}
	For $(M^{n},g)$ a complete Riemannian manifold and $U \subset M \setminus \partial M$ open, we say that a critical point of $E_{\eps} \restr U$ is \emph{stable} on $U$ if $\delta^{2}E_{\eps}[u]\{\zeta,\zeta\} \geq 0$ for all $\zeta \in C^{\infty}_{c}(U)$. More generally, we say $u$ has Morse index $k$, denoted $\ind(u) = k$, if 
	\[
	\max \{ \dim V : \delta^{2}E_{\eps}[u]\{\zeta,\zeta\} < 0 \text{ for all } \zeta \in V\setminus\{0\}\} = k,
	\]
	where the maximum is taken over all subspaces $V \subset C^{\infty}_{c}(U)$. Sometimes we will write $\ind(u;U)=k$ to emphasize the underlying set. Note that $\ind(u; U) = 0$ if and only if $u$ is stable on $U$.
\end{defi}

When $u$ is a solution of \eqref{eq:ac.pde} and $\nabla u(x) \neq 0$, we will write:
\begin{enumerate}
	\item $\nu(x) = \tfrac{\nabla u(x)}{|\nabla u(x)|}$ for the unit normal of the level set of $u$ through $x$;
	\item $\sff(x)$ for the second fundamental form of the level set of $u$ through $x$;
	\item $\cA(x)$ for the ``Allen--Cahn'' or ``enhanced'' second fundamental form of the level set:
	\[ \cA = \frac{\nabla^2 u - \nabla^2 u(\cdot, \nu) \otimes \nu^\flat}{|\nabla u|} \left( = \nabla \left( \frac{\nabla u}{|\nabla u |} \right)(x) \right). \]
\end{enumerate}
One may check that
\[
|\cA(x)|^{2} = |\sff(x)|^{2} + |\nabla_{T} \log|\nabla u(x)||^{2},
\]
where $\nabla_{T}$ represents the gradient in the directions orthogonal to $\nabla u$; in other words, $|\cA|$ strictly dominates the second fundamental form of the level sets.

Finally, we will often use Fermi coordinates centered on a hypersurface. To avoid confusion about which hypersurface the coordinates are associated to, we will define a function
\[ Z_\Sigma(y,z) \triangleq \exp_y(z \nu_\Sigma(y)), \; y \in \Sigma, \; z \in \RR, \]
where $\nu_\Sigma$ will denote a distinguished normal vector to $\Sigma$. In this paper, $\nu_\Sigma$ is generally taken to be the upward pointing unit normal. Note that the pullback of the metric $g$ along $Z_{\Sigma}$ has the form $g_{z} + dz^{2}$, which is the setting that most of our analysis will take place below.

\subsection{Main results}

\subsubsection{Curvature estimates for stable solutions of \eqref{eq:ac.pde} on $3$-manifolds} We start this section by discussing the concept of stability applied to minimal surfaces, since that guides some aspects of our work in the Allen--Cahn setting.

We recall that a two-sided minimal surface $\Sigma^{2} \subset (M^{3},g)$ with normal vector $\nu$ is said to be \emph{stable} if it satisfies
\begin{equation}\label{eq:stable.min.surf}
\int_{\Sigma} \left( |\nabla_\Sigma \zeta|^{2} - (|\sff_\Sigma|^{2} + \ricc_g(\nu,\nu))\zeta^{2}\right) d\mu_g \geq 0
\end{equation}
for $\zeta \in C^{\infty}_{c}(\Sigma)$. Here, we briefly recall the well-known curvature estimates of Schoen \cite{Sch83} for stable minimal surfaces. If $\Sigma^{2}\subset (M^{3},g)$ is a complete, two-sided stable minimal surface, then the second fundamental form of $\Sigma$, $\sff_{\Sigma}$, satisfies
\begin{equation}\label{eq:curv.est.Schoen}
|\sff_{\Sigma}|(x) d(x,\partial\Sigma) \leq C = C(M,g).
\end{equation}
Observe that \eqref{eq:curv.est.Schoen} readily implies a stable Bernstein theorem: ``a complete two-sided stable minimal surfaces $\Sigma$ in $\RR^{3}$ without boundary must be a flat plane.'' On the other hand, the stable Bernstein theorem (proven in \cite{Fischer-Colbrie-Schoen,doCarmoPeng,Pogorelov}) implies \eqref{eq:curv.est.Schoen} by a well known blow-up argument: if \eqref{eq:curv.est.Schoen} failed for a sequence of stable minimal surfaces $\Sigma_{j}$, then by choosing a point of (nearly) maximal curvature and rescaling appropriately (cf.\ \cite{White:PCMI}), we can produce $\tilde\Sigma_{j}$ a sequence of minimal surfaces in manifolds $(M_{j}^{3},g_{j})$ that are converging on compact sets to $\RR^{3}$ with the flat metric, and so that $d_{g_{j}}(0,\partial\Sigma_{j}) \to \infty$, $|\sff_{\Sigma_{j}}|$ uniformly bounded on compact sets, and $|\sff_{\Sigma_{j}}|(0) = 1$. The second fundamental form bounds yield local $C^{2}$ bounds for the surfaces $\Sigma_{j}$, which may then be upgraded to $C^{k}$ bounds for all $k$. Thus, passing to a subsequence, the surfaces $\Sigma_{j}$ converge smoothly to a complete stable minimal surface $\Sigma_{\infty}$ without boundary in $\RR^{3}$. Because the convergence occurs in $C^{2}$, the we see that $|\sff_{\Sigma_{\infty}}|(0) = 1$, so $\Sigma_{\infty}$ is non-flat. This contradicts the stable Bernstein theorem. 

As such, before discussing curvature estimates for stable solution to Allen--Cahn, we must discuss the stable Bernstein theorem for complete solutions on $\RR^{3}$. In general, it is not known if there are stable solutions to Allen--Cahn $\Delta u = W'(u)$ on $\RR^{3}$ with non-flat level sets. However, under the additional assumption of quadratic energy growth, i.e.,
\[
(E_{1} \restr B_R(0))[u] \leq \Lambda R^{2},
\]
then it follows from the work of Ambrosio--Cabre \cite{AmbrosioCabre00} (see also \cite{FarinaMariValdinoci13}) that $u$ has flat level sets. We note that the corresponding stable Bernstein theorem on $\RR^{2}$ is known to hold without any energy growth assumption; see the works of Ghoussoub--Gui \cite{GhoussoubGui98} and Ambrosio--Cabre \cite{AmbrosioCabre00}.

As such, one may expect that the blow-up argument described above may be used to prove curvature estimates. However, there is a fundamental difficulty present in the Allen--Cahn setting: if $u_{i}$ are stable solutions of \eqref{eq:ac.pde} on $(M^{3},g)$, then if their curvature (we will make this precise below) is diverging, then if we rescale by a factor $\lambda_{i}\to\infty$ in a blow-up argument this changes $\eps_{i}$ to $\lambda_{i}\eps_{i}$. If $\lambda_{i}\eps_{i}$ converges to a non-zero constant, then standard elliptic regularity implies the rescaled functions limit smoothly to an entire stable solution of Allen--Cahn on $\RR^{3}$. The smooth convergence guarantees that this solution will have non-flat level sets. If the original functions $u_{i}$ had uniformly bounded energy, we can show that the limit has quadratic area growth, which contradicts the aforementioned Bernstein theorem. However, if $\lambda_{i}\eps_{i}$ still converges to zero, we must argue differently. In this case, we have a sequence of solutions to Allen--Cahn whose level sets are uniformly bounded in a $C^{2}$-sense. This can be used to show that the level sets converge to a plane (possibly with multiplicity) in the $C^{1,\alpha}$-sense. If the level sets behaved precisely like minimal surfaces, we could upgrade this $C^{1,\alpha}$-convergence using elliptic regularity, to conclude that the limit was not flat. However, in this situation, the level sets themselves do not satisfy a good PDE, so this becomes a significant obstacle.

Recently, a fundamental step in understanding this issue has been undertaken by Wang--Wei \cite{WangWei}. They have developed a technique for gaining geometric control of solutions to Allen--Cahn whose level sets are converging with Lipschitz bounds. Using this (and the $2$-dimensional stable Bernstein theorem) they have proven curvature estimates for individual level sets of stable solutions on two-dimensional surfaces. Moreover, they have shown that if one cannot upgrade $C^{2}$ bounds to $C^{2,\alpha}$ convergence, then by appropriately rescaling the height functions of the nodal sets, one obtains a nontrivial solution to the a system of PDE's known as the Toda system (see \cite[Remark 14.1]{WangWei}). Finally, their proof of curvature estimates in $2$-dimensions points to the crucial observation that it is necessary to use stability to upgrade the regularity of the convergence of the level sets. 

This brings us to our first main result here, which is an extension of the Wang--Wei curvature estimates to $3$ dimensions. Our $3$-dimensional curvature estimates can be roughly stated as follows (see Theorem \ref{theo:curvature.estimate} for a slightly more refined statement and the proof)
\begin{theo}\label{theo:curv.est}
	For a complete Riemannian metric on $\overline{B_{2}}(0) \subset \RR^{3}$ and a stable solution $u$ to \eqref{eq:ac.pde} with $E_{\eps}(u) \leq E_{0}$, the enhanced second fundamental form of $u$ satisfies
	\[
	\sup_{B_{1}(0) \cap \{|u| < 1-\beta\}} |\cA|(x)\leq C = C(g,E_{0},W,\beta)
	\]
	as long as $\eps >0$ is sufficiently small. 
\end{theo}
We emphasize that Wang--Wei's $2$-dimensional estimates \cite[Theorem 3.7]{WangWei} do not require the energy bound (see also \cite[Theorem 4.13]{Mantoulidis} for the Riemannian modifications of this result). Note that we cannot expect to prove estimates with a constant that tends to $0$ as $\eps\searrow 0$ (which was the case in \cite{WangWei}) since---unlike geodesics---minimal surfaces do not necessarily have vanishing second fundamental form. 

We note that due to our curvature estimates, it is not hard to see that stable (and more generally, uniformly bounded index) solutions to the Allen--Cahn equation (with uniformly bounded energy) in a $3$-manifold limit to a $C^{1,\alpha}$ surface that has vanishing (weak) mean curvature. Standard arguments thus show that the surface is smooth. Thus, our estimates show that it is possible to completely avoid the regularity results of Wickramasekera and Wickramasekera--Tonegawa \cite{Wickramasekera14,TonegawaWickramasekera12} in the setting of Allen--Cahn min-max on a $3$-manifold (cf.\ \cite{Guaraco}).

\begin{rema}
	We briefly remark on the possibility of extending curvature estimates to higher dimensions:
	\begin{enumerate}
		\item For $n \geq 8$, curvature estimates fail for stable (and even minimizing) solutions to the Allen--Cahn equation. See: \cite{PacardWei:stable,LiuWangWei}.
		\item For $4\leq n \leq 7$, the Allen--Cahn stable Bernstein result is not known (even with an energy growth condition). 
	\end{enumerate}
	Even if the stable Bernstein theorem were to be  established in dimensions $4\leq n \leq 7$, we note that our proof currently uses the dimension restriction $n=3$ in one other place: we use a logarithmic cutoff function in the proof of our sheet separation estimates (Propositions \ref{prop:bootstrapped.stable.estimates} and \ref{prop:ultimate.stable.estimates}). \footnote{Added in proof: Wang--Wei have recently found \cite{WangWei2} the appropriate higher dimensional replacement for the log-cutoff argument used here. We note that the stable Bernstein problem for Allen--Cahn remains open in dimensions $4\leq n\leq 7$.} 
	
	On the other hand, we remark that the curvature estimate for minimizing solutions can be proven using the  ``multiplicity one'' nature of minimizers \cite[Theorem 2]{HutchinsonTonegawa00}, together with \cite[Section 15]{WangWei} (or Remark \ref{rema:major.goal}).
	
	We note that the case of complete minimizers is closely related to the well known ``De Giorgi conjecture.'' See  \cite{GhoussoubGui98,AmbrosioCabre00,Savin:DGconj,delPinoKowalczykWei:DG-counterexample,Wang:Allard}. 
\end{rema}

\subsubsection{Strong sheet separation estimates for stable solutions} A key ingredient in the proof of our curvature estimates is showing that distinct sheets of the nodal set of a stable solution to the Allen--Cahn equation remain sufficiently far apart. This aspect was already present in the work of Wang--Wei. For our applications to the case of uniformly bounded Morse index (and thus min-max theory), we must go beyond the sheet separation estimates proven in \cite{WangWei}. We prove in Proposition \ref{prop:ultimate.stable.estimates} that distinct sheets of nodal sets of a stable solution to the Allen--Cahn equation must be separated by a sufficiently large distance so that the location of the nodal sets becomes ``mean curvature dominated.'' 

In particular, as a consequence of these estimates, we show in Theorem \ref{theo:bounded.index} that if a sequence of stable solutions to the Allen--Cahn equation converge with multiplicity to a closed two-sided minimal surface $\Sigma$, then there is a positive Jacobi field along $\Sigma$ (which implies that $\Sigma$ is stable). It is interesting to compare this to the examples constructed by del Pino--Kowalczyk--Wei--Yang of minimal surfaces in $3$-manifolds with positive Ricci curvature that are the limit with multiplicity of solutions to the Allen--Cahn equation \cite{delPinoKowalczykWeiYang:interface}. Note that such a minimal surface cannot admit a positive Jacobi field, so the point here is that the Allen--Cahn solutions are not stable. (In fact, our Theorem \ref{theo:bounded.index} implies that they have diverging Morse index.) Note that the separation $D$ between the sheets of the examples constructed in \cite{delPinoKowalczykWeiYang:interface} satisfy, as $\eps \searrow 0$,
\[
D \sim \sqrt{2} \eps |\log\eps|  - \frac{1}{\sqrt{2}}\eps  \log | \log \eps| ,
\]
while we prove in Proposition \ref{prop:ultimate.stable.estimates} that stability implies that the separation satisfies
\[
D - \left( \sqrt{2} \eps |\log \eps| - \frac{1}{\sqrt{2}} \eps \log|\log \eps| \right) \to -\infty.
\]
We emphasize that the improved separation estimates here are not contained in the work of Wang--Wei \cite{WangWei} and are fundamental for the subsequent applications of our results.

\subsubsection{The multiplicity one-conjecture for limits of the Allen--Cahn equation in $3$-manifolds} 

In their recent work \cite{MarquesNeves:multiplicity}, Marques--Neves make the following conjecture: 
\begin{conj}[Multiplicity one conjecture]
	For generic metrics on $(M^{n},g)$, $3\leq n\leq 7$, two-sided unstable components of closed minimal hypersurfaces obtained by min-max methods must have multiplicity one. 
\end{conj}

In \cite{MarquesNeves:multiplicity}, Marques--Neves confirm this in the case of a one parameter Almgren--Pitts sweepout. The one parameter case had been previously considered for metrics of positive Ricci curvature by Marques--Neves \cite{MarquesNeves:rigidity.min.max} and subsequently by Zhou \cite{Zhou:posRic}. See also \cite[Corollary E]{Guaraco} and \cite[Theorem 1]{GasparGuaraco} for results comparing the Allen--Cahn setting to Almgren--Pitts setting which establish multiplicity one for hypersurfaces obtained by a one parameter Allen--Cahn min-max method in certain settings. We also note that Ketover--Liokumovich--Song \cite{Song,KetoverLiokumovich,KetoverLiokumovichSong} have proven multiplicity (and index) estimates for one parameter families in the Simon--Smith \cite{SimonSmith} variant of Almgren--Pitts in $3$-manifolds.\footnote{Added in proof: As noted before, the full multiplicity one conjecture for Almgren--Pitts (in dimensions $3$ through $7$) has now been proven by X. Zhou \cite{Zhou:multiplicity-one}.}

We recall the following standard definition:
\begin{defi} \label{def:bumpy.metric}
	We say that a metric $g$ on a Riemannian manifold $M^{n}$ is \emph{bumpy} if there is no immersed closed minimal hypersurface $\Sigma^{n-1}$ with a non-trivial Jacobi field. 
\end{defi}

By work of White \cite{White:bumpy.old,White:bumpy.new}, bumpy metrics are generic in the sense of Baire category. Here, ``generic'' will always mean in the Baire category sense. 

We are able to prove a strong version of the multiplicity one conjecture (when $n=3$) for minimal surfaces obtained by Allen--Cahn min-max methods with an \emph{arbitrary} number of parameters. Such a method was set up by Gaspar--Guaraco \cite{GasparGuaraco}. 

Indeed, we prove that for \emph{any} metric $g$ on a closed $3$-manifold, the unstable components of such a surface are multiplicity one. Moreover, for a generic metric, we show that \emph{each} component of the surface occurs with multiplicity one (not just the unstable components). Finally, we are able to show for generic metrics on a $n$-manifold, $3\leq n\leq 7$, the minimal surfaces constructed by Allen--Cahn min-max methods are two-sided. For a one-parameter Almgren--Pitts sweepoints in a $n$-manifold $3\leq n\leq 7$ with positive Ricci curvature, this was proven by Ketover--Marques--Neves \cite{KMN:catenoid}. More precisely, our main results here are as follows (see Theorem \ref{theo:bounded.index} and Corollary \ref{coro:mult.one.conj} for the full statements).

\begin{theo}[Multiplicity and two-sidedness of minimal surfaces constructed via Allen--Cahn min-max]\label{theo:mult.intro-version}
	Let $\Sigma^{2}\subset (M^{3},g)$ denote a smooth embedded minimal surface constructed as the $\eps\searrow 0$ limit of solutions to the Allen--Cahn equation on a $3$-manifold with uniformly bounded index and energy. If $\Sigma$ occurs with multiplicity or is one-sided, then it carries a positive Jacobi field (on its two-sided double cover, in the second case). 
	
	Note that positive Jacobi fields do not occur when $g$ is bumpy or when $g$ has positive Ricci curvature. Thus, in either of these cases,  each component of $\Sigma$ is two-sided and occurs with multiplicity one.
\end{theo}

\begin{rema}
	We re-emphasize that our theorem applies generally to sequences of Allen--Cahn solutions with uniformly bounded energy and Morse index. Thus, unlike the proofs in the Almgren--Pitts setting, we do not need to make use of any min-max characterization of the limiting surface to rule out multiplicity.
\end{rema}

Our proof here is modeled on the study of bounded index minimal hypersurfaces in a Riemannian manifold. Indeed, Sharp has shown that minimal hypersurfaces in $(M^{n},g)$ for $3\leq n\leq 7$ with uniformly bounded area and index are smoothly compact away from finitely many points where the index can concentrate \cite{Sharp} (see also White's proof \cite{White:curvature} of the Choi--Schoen compactness theorem \cite{ChoiSchoen}). A crucial point there is to prove that higher multiplicity of the limiting  surface produces a positive Jacobi field (even across the points of index concentration (where the convergence of the hypersurfaces need not occur smoothly). This can be handled via an elegant argument of White, based on the construction of a local foliation by minimal surfaces to use as a barrier for the limiting surfaces (cf.\ \cite{White:compactness.new}). 

In the minimal surface setting, the existence of the foliation is a simple consequence of the implicit function theorem. However, in the Allen--Cahn setting, the singular limit $\eps\searrow 0$ limit complicates this argument. Instead, we construct barriers by a more involved fixed point method in Theorem \ref{theo:dirichlet.data.construction}. Once that theorem is proven, we show how the barriers can be used to bound the Jacobi fields along the points of index concentration in the process of the proof of Theorem \ref{theo:bounded.index} by carrying out a new sliding plane type argument for the Allen--Cahn equation on Riemannian manifolds. Our proof of Theorem \ref{theo:dirichlet.data.construction} is modeled on the work of Pacard \cite{Pacard12} (with appropriate extension to the case of Dirichlet boundary conditions), but there is a significant technical obstruction here: we do not know that the level sets of the solution Allen--Cahn converge smoothly, but only in $C^{2,\alpha}$. To apply the fixed point argument, we need some control on higher derivatives. By an observation of Wang--Wei \cite[Lemma 8.1]{WangWei}, we control one higher derivative of the level sets, but only by a constant that is $O(\eps^{-1})$ (see \eqref{eq:dirichlet.data.sigma.c3alpha}). This complicates the proof of Theorem \ref{theo:dirichlet.data.construction}. 

\subsubsection{Index lower bounds} 

Lower semicontinuity of the Morse index along the singular limit $\eps\searrow 0$ of a sequence of solutions to the Allen--Cahn equation is proven by Hiesmayr \cite{Hiesmayr} (for two-sided surfaces) and Gaspar \cite{Gaspar} without assuming two-sidedness (see also \cite{Le:2ndvar}). On the other hand, upper semicontinuity of the index does not hold in general (cf.\ Example \ref{exam:upper.semi.fails.index}). Here, we establish upper semicontinuity of the index, in all dimensions, under the a priori assumption that the limiting surface is multiplicity one.\footnote{We note that Marques--Neves had previously announced the analogous index uppper-semicontinuity result for multiplicity one Almgren--Pitts limits and that their proof \cite{MarquesNeves:uper-semi-index} appeared shortly after the first version of this paper.} In particular we prove (see Theorem \ref{theo:index.lower.bounds} for the full statement)
\begin{theo}[Upper semicontinuity of the index in the multiplicity one case] \label{theo:index.semicontinuity}
	Suppose that a smooth embedded minimal hypersurface $\Sigma^{n-1}\subset (M^{n},g)$ is the multiplicity one limit as $\eps\searrow 0$ of a sequence of solutions $u$ to the Allen--Cahn equation. Then for $\eps>0$ sufficiently small,
	\[
	\nul(\Sigma) + \ind(\Sigma) \geq \nul(u) + \ind(u). 
	\]
\end{theo}
To prove this upper semicontinuity, we need to delve deeper into the equation that controls the level sets of $u$ and obtain a more accurate approximation. What was done for Theorem \ref{theo:curv.est}---while well suited to understanding the phenomenon of multiplicity---does not suffice for Theorem \ref{theo:index.semicontinuity}.

\subsubsection{Applications related to Yau's conjecture on infinitely many minimal surfaces}

A well known conjecture of Yau posits that any closed $3$-manifold admits infinitely many immersed minimal surfaces \cite{Yau:problems}. By considering the $p$-widths introduced by Gromov \cite{Gromov:waist} (see also \cite{Guth:minimax}), Marques--Neves proved \cite{MarquesNeves:posRic} that a closed Riemannian manifold $(M^{n},g)$ (for $3\leq n\leq 7$) with positive Ricci curvature admits infinitely many minimal surfaces. Moreover, by an ingenious application of the Weyl law for the $p$-widths proven by Liokumovich--Marques--Neves \cite{LMN:Weyl}, Irie--Marques--Neves \cite{IrieMarquesNeves} (see also the recent work of Gaspar--Guaraco \cite{GasparGuaraco:weyl} that appeared after the first version of this paper was posted) have recently shown that the set of metrics on a closed Riemannian manifold $(M^{n},g)$ (with $3\leq n \leq 7$) with the property that the set of minimal surfaces is dense in the manifold is generic (see also \cite{MarquesNevesSong}).

We note that the arguments in each of \cite{MarquesNeves:posRic,IrieMarquesNeves,GasparGuaraco:weyl} to prove the existence of infinitely many minimal surfaces are \emph{necessarily} \emph{indirect}, as they do not rule out the $p$-widths being achieved with higher multiplicity.  Having overcome this obstacle, we may give a ``direct'' proof (for $n=3$) of Yau's conjecture for bumpy metrics\footnote{We note that \cite{IrieMarquesNeves,GasparGuaraco:weyl} prove Yau's conjecture for a different (also generic) set of metrics.} with some new geometric conclusions (see Corollaries \ref{coro:mult.one.conj}, \ref{coro:Yau.conj} for proofs).

\begin{coro}[Yau's conjecture for bumpy metrics and geometric properties of the minimal surfaces]
	\label{coro:yau-intro}
	Let $(M^{3},g)$ denote a closed $3$-manifold with a bumpy metric. Then, there is $C=C(M,g,W)>0$ and a smooth embedded minimal surfaces $\Sigma_{p}$ for each positive integer $p>0$ so that
	\begin{itemize}
		\item each component of $\Sigma_{p}$ is two-sided,
		\item the area of $\Sigma_{p}$ satisfies $C^{-1} p^{\frac 1 3}\leq \area_{g}(\Sigma_{p}) \leq C p^{\frac 1 3}$, 
		\item the index of $\Sigma_{p}$ is satisfies $\ind(\Sigma_{p}) = p$, 
		and
		\item the genus of $\Sigma_{p}$ satisfies $\genus(\Sigma_p) \geq \frac p 6 - C p^{\frac 1 3}$.
	\end{itemize}
	In particular, thanks to the index estimate, all of the $\Sigma_{p}$ are geometrically distinct. 
\end{coro}

We emphasize that each of the bullet points in the preceding corollary do not follow from the work of Irie--Marques--Neves \cite{IrieMarquesNeves}. Some of these properties were conjectured by Marques and Neves in \cite[p.\ 24]{marques:ICM}, \cite[p.\ 17]{neves:ICM},  \cite[Conjecture 6.2]{MarquesNeves:spaceOfCycles}. In particular, they conjectured that a generic Riemannian manifold contains an embedded two-sided minimal surface of each positive Morse index.

\begin{rema}[Yau's conjecture for $3$-manifolds with positive Ricci curvature]
	We note that because the multiplicity-one property also holds even for non-bumpy metrics of positive Ricci curvature, we may also give a ``direct'' proof of Yau's conjecture for a $3$-manifold with positive Ricci curvature (this was proven by Marques--Neves \cite{MarquesNeves:posRic} in dimensions $3\leq n\leq 7$ using Almgren--Pitts theory). We obtain, exactly as in Corollary \ref{coro:Yau.conj}, the new conclusions that the surfaces $\Sigma_{p}$ are two-sided, have $\area(\Sigma_{p})\sim p^{\frac 13}$, $\ind(\Sigma_{p}) \leq p$ and $\nul(\Sigma_{p})+\ind(\Sigma_{p})\geq p$. 
	
	Moreover, approximating the metric by a sequence of bumpy metrics and passing to the limit (the limit occurs smoothly and with multiplicity one due to the positivity of the Ricci curvature, cf.\ \cite{Sharp}), we find that there is a sequence $\Sigma_{p}'$ (we do not know if this is the same sequence as $\Sigma_{p}$) with these properties and additionally satisfies the genus bound (note that $\Sigma_{p}$ is connected by Frankel's theorem) for possibly a larger constant $C$
	\[
	\genus(\Sigma_{p}') \geq \frac{p}{6} - Cp^{\frac 13}.
	\]
	It is interesting to observe that when $(M^{3},g)$ is the round $3$-sphere, combining our bound $\ind(\Sigma_{p}') \leq p$ with work of Savo \cite{Savo} implies that 
	\[
	\genus(\Sigma_{p}') \leq 2 p - 8
	\]
	as long as $p$ is sufficiently large to guarantee that $\genus(\Sigma_{p}') \geq 1$. Similar conclusions can be derived in certain other $3$-manifolds embedded in Euclidean spaces by \cite{AmbrozioCarlottoSharp:index.genus}.
	
	There has been significant activity concerning the index of the minimal surfaces constructed in \cite{MarquesNeves:posRic}, but before the present work, all that was known was that: for a bumpy metric of positive Ricci curvature, there are closed embedded minimal surfaces of arbitrarily large Morse index \cite{LiZhou,CKM,Carlotto:arb-large}, albeit without information on their area.
\end{rema}

\begin{rema}[Connected components in Corollary {\ref{coro:yau-intro}}]
	Unless $(M, g)$ has the Frankel property (e.g., when it has positive Ricci curvature), the minimal surfaces $\Sigma_p$ obtained in Corollary \ref{coro:yau-intro} may be disconnected. In this case, every connected component $\Sigma_p'$ of $\Sigma_p$ must satisfy:
	\begin{itemize}
		\item $\Sigma_p'$ is two-sided and has $\area_g(\Sigma_p')\leq C p^{\frac 1 3}$,
	\end{itemize}
	and, by a counting argument, there will exist at least one component $\Sigma_p'$ of $\Sigma_p$ such that
	\begin{itemize}
		\item $\genus(\Sigma_p') \geq C^{-1} \ind(\Sigma_p') \geq C^{-1} p^{\frac 2 3}$.
	\end{itemize}
	See Corollary \ref{coro:Yau.conj.components}.
	
	It is not clear that the component $\Sigma_{p}'$ will have unbounded area. In a follow up paper \cite{ChodoshMantoulidis:unbounded-area} we prove the following dichotomy; either
	\begin{enumerate}
		\item $(M,g)$ contains a sequence of connected closed embedded stable minimal surfaces with unbounded area, or
		\item some connected component $\Sigma_{p}''$ of the surfaces $\Sigma_{p}$ obtained in Corollary \ref{coro:yau-intro} has $\area_{g}(\Sigma_{p}'')\geq Cp^{\frac 13}$. 
	\end{enumerate}
	We note that by \cite{CKM,Carlotto:arb-large}, when $(M^{3},g)$ is a bumpy metric with positive scalar curvature the prior condition cannot hold, so the latter alternative holds and, moreover, $\ind(\Sigma_{p}'')\to\infty$. It would be interesting to determine if one can find a connected component $\Sigma_{p}''$ with arbitrarily large area and $\ind(\Sigma_{p}'')\geq c p$ for some $c\in (0,1)$.
\end{rema}

\subsection{One-dimensional heteroclinic solution, $\mathbb{H}$} \label{subsec:heteroclinic.solution}

Recall that the one-dimensional Allen-Cahn equation with $\varepsilon=1$ is $u'' = W'(u)$, for a function $u = u(t)$ of one variable. It's not hard to see that this ODE admits a unique bounded solution with the properties
\[ u(0) = 0, \; \lim_{t \to -\infty} u(t) = -1, \; \lim_{t \to \infty} u(t) = 1. \]
We call this the one-dimensional heteroclinic solution, and denote it as $\mathbb{H} : \RR \to (-1, 1)$. It's also standard to see that the heteroclinic solution satisfies:
\begin{align} 
\mathbb{H}(\pm t)  & = \pm 1 \mp \expansioncoeff \exp(-\sqrt{2} t) + O(\exp(-2\sqrt{2} t)), \label{eq:heteroclinic.expansion.i} \\
\mathbb{H}'(\pm t) & = \sqrt{2} \expansioncoeff \exp(-\sqrt{2} t) + O(\exp(-2\sqrt{2} t)), \label{eq:heteroclinic.expansion.ii} \\
\mathbb{H}''(\pm t) & = - 2 \expansioncoeff \exp(-\sqrt{2} t) + O(\exp(-2\sqrt{2} t)), \label{eq:heteroclinic.expansion.iii}
\end{align}
as $t \to \infty$, for some fixed $A_0 > 0$ that depends on $W$. Moreover,
\[ \int_{-\infty}^\infty (\mathbb{H}'(t))^2 \, dt = \energyunit, \]
where $\energyunit > 0$ also depends on $W$; it is explicitly given by
\[ \energyunit = \int_{-1}^1 \sqrt{2W(t)} \, dt. \]
Finally, we also define
\begin{equation} \label{eq:heteroclinic.eps}
\mathbb{H}_\eps(t) \triangleq \mathbb{H}(\eps^{-1} t), \; t \in \RR,
\end{equation}
which is clearly a solution of $\eps^2 \mathbb{H}_\eps'' = W'(\mathbb{H}_\eps)$.

\subsection{Organization of the paper}

In Section \ref{sec:jacobi.toda.reduction} we make precise \emph{the dependence of the regularity} of the nodal set $\{ u = 0 \}$ of bounded energy and bounded curvature solutions of \eqref{eq:ac.pde} \emph{on the distance} between its different sheets. The dependence is essentially modeled by a Toda system; see, e.g., \eqref{eq:jacobi.toda} and Remark \ref{rema:major.goal}. Restricting to $n=3$ dimensions, in Section \ref{sec:stable.solutions} we use the stability of Allen--Cahn solutions to bootstrap the distance estimates from Section \ref{sec:jacobi.toda.reduction} until they become sharp. In Section \ref{sec:bounded.index} we study solutions of \eqref{eq:ac.pde} with bounded energy and Morse index in $n=3$ dimensions. We use our strong sheet separation estimates from Section \ref{sec:stable.solutions} to construct, in the presence of multiplicity, positive Jacobi fields on the limiting minimal surface away from finitely many points. Then, a ``sliding plane'' argument (modulo a barrier construction deferred to Section \ref{sec:dirichlet.data}) allows us to extend the Jacobi field to the entire limiting surface.

In Section \ref{sec:multiplicity.one} we return to the arbitrary dimensional setting and prove the Morse index is lower semicontinuous for smooth multiplicity one limits. 
In Section \ref{sec:applications} we apply all our tools to prove a strong form of Marques' and Neves' multiplicity one conjecture, and Yau's conjecture for generic metrics.
In Section \ref{sec:dirichlet.data} we construct curved sliding plane barriers for \eqref{eq:ac.pde} that resemble multiplicity-one heteroclinic solutions with prescribed Dirichlet data centered on nondegenerate minimal submanifolds-with-boundary $\Sigma^{n-1} \subset (M^n, g)$, $n \geq 3$.

In Appendix \ref{app:mean.curvature.graphs}, we recall several expressions related to the mean curvature and second fundamental form of graphical hypersurfaces in a Riemannian manifold. In Appendix \ref{app:WW-results} we recall several auxiliary results from \cite{WangWei}. In Appendix \ref{app:proof.lem.comp.improved}, we prove Lemma \ref{lemm:h.phi.comparison.improved} relating regularity of the ``centering'' functions $h_{\ell}$ to that of the function $\phi$ with improved error estimates. In Appendix \ref{app:proof.stab.inproved}, we derive the Toda-system stability inequality with improved error estimates \eqref{eq:toda.stability.estimate.sharper}. In Appendix \ref{app:interpolation.lemma} we recall an interpolation inequality for H\"older norms.

\subsection{Acknowledgments}

O.C. was supported in part by the Oswald Veblen fund and NSF Grant no.\ 1638352. He would like to thank Simon Brendle and Michael Eichmair for their continued support and encouragement, as well as Costante Bellettini, Guido De Philippis, Daniel Ketover, and Neshan Wickramasekera for their interest and for enjoyable discussions. C.M. would like to thank Rick Schoen, Rafe Mazzeo, and Yevgeniy Liokumovich for helpful conversations on topics addressed by this paper. Both authors would like to thank Fernando Cod\'a Marques and Andr\'e Neves very much for their interest and encouragement. They are also grateful to Davi Maximo for pointing out a mistake in the original version of Corollary \ref{coro:yau-intro}. This work originated during the authors' visit to the Erwin Schr\"odinger International Institute for Mathematics and Physics (ESI) during the ``Advances in General Relativity Workshop'' during the summer of 2017, which they would like to acknowledge for its support. Finally, the authors would like to thank the referee for their careful reading of the manuscript and many helpful suggestions.

\section{From phase transitions to Jacobi-Toda systems}


\label{sec:jacobi.toda.reduction}

\subsection{Approximation by superimposed heteroclinics} \label{subsec:jacobi.toda.setup} \label{subsec:approximate.solutions}

In this section we follow Wang-Wei's  \cite{WangWei} investigation of local properties of solutions to the Allen-Cahn equation,
\[ \eps^2 \Delta_g u = W'(u), \]
whose nodal set $\{ u = 0 \}$ can be (locally) decomposed as a union of graphs over a fixed hypersurface (to be denoted $\Sigma$), whose height functions (to be denoted $f_1, \ldots, f_Q$) are bounded in $C^2$ and small in $C^1$. The ultimate goal is to deduce, in a quantitative sense, that the height functions approximately satisfy a Jacobi-Toda system.

The reason we rework the setup is twofold:
\begin{enumerate}
	\item First, most of the analysis in \cite{WangWei} was performed in $\RR^n$, while here we include the details necessary to handle the Riemannian setting (cf.\ \cite[Section 16]{WangWei}).
	\item Secondly (and more fundamentally), we combine the argument from \cite{WangWei} with a further bootstrap argument based on improved error estimates. This allows us to prove much sharper separation estimates than were obtained in \cite{WangWei}. Indeed, we will show that the behavior of the transition layers is dominated by mean curvature, rather than interaction between the layers. This will be crucial for our subsequent applications in Section \ref{sec:bounded.index}. 
\end{enumerate}

Let's set things up. Suppose that $D^{n-1}$ is an $(n-1)$-dimensional disk, over which we take a topological cylinder $\Omega \triangleq D \times [-1,1]$, whose coordinates we label $X = (y, z) \in D \times [-1,1]$. Consider a smooth metric $g$ on $\Omega$, which we assume to be in Fermi coordinate form with respect to $\Sigma$; in $(y, z)$ coordinates:
\[ g = g_z + dz^2. \]
For convenience, we denote $\Sigma \triangleq D \times \{0\} \subset \Omega$.
Let us require that
\begin{equation} \label{eq:sheets.sff.bound}
	\sum_{\ell = 0}^3 |\nabla_\Sigma^\ell \sff_\Sigma| \leq \eta.
\end{equation}
We additionally assume that $\Sigma$ is covered by $C^{4}$-coordinate charts so that the induced metric on $\Sigma$, $g_{0}$ is $C^{3}$-close to the Euclidean metric in the charts, i.e.,
\begin{equation}\label{eq:sheets.metric.bound}
\sum_{\ell=0}^{3} |\partial^{(\ell)}_{y} ((g_{0})_{ij} - \delta_{ij})| \leq \eta.
\end{equation}

We make no assumptions on the mean curvature of $\Sigma$ beyond what follows automatically from \eqref{eq:sheets.sff.bound}. Notice that, as a consequence of \eqref{eq:sheets.sff.bound}-\eqref{eq:sheets.metric.bound}, Fermi coordinates with respect to $\Sigma$ are a $C^4$ diffeomorphism.

In all that follows, we denote for $y_0 \in \Sigma \setminus \partial \Sigma$ and $0 < r < \dist_{g_0}(y_0, \partial \Sigma)$,
\[ B^{n-1}_r(y_0) \triangleq \{ y \in \Sigma : \dist_{g_0}(y, y_0) < r \}, \]
where $\dist_{g_0}$ is the intrinsic distance on $\Sigma$. We assume, without loss of generality, that $\Sigma = \overline{B}^{n-1}_2(0)$. 

\begin{rema}\label{rema:scale-vs-WW}
We have chosen to work at the original scale, rather than rescaling by $\eps$ as in \cite{WangWei}. This does not affect our subsequent analysis, but certain expressions will change by appropriate multiples of $\eps$.
\end{rema}

Let $u : \Omega \to (-1,1)$ be a critical point of $E_\varepsilon \restr \Omega$, with
\begin{align}
	\varepsilon & \leq \varepsilon_0, \label{eq:sheets.eps.bound} \\
	(E_{\varepsilon} \restr \Omega)[u_i] & \leq E_0, \label{eq:sheets.energy.bound} \\
	\varepsilon |\nabla u| & \geq c_0^{-1} > 0 \text{ on } \Omega \cap \{ |u| \leq 1-\beta \}, \label{eq:sheets.lower.density.bound} \\
	|\cA| & \leq c_0 \text{ on } \Omega \cap \{ |u| \leq 1-\beta \}; \label{eq:sheets.enhanced.sff.bound} \\
\intertext{By \eqref{eq:sheets.lower.density.bound}, \eqref{eq:sheets.enhanced.sff.bound}, and elliptic regularity, we automatically also get}
	\eps |\nabla \cA| + \eps^2 |\nabla^2 \cA| & \leq c_0 \text{ on } \Omega \cap \{ |u| \leq 1-\beta \} \label{eq:sheets.enhanced.sff.grad.bound}
\end{align}
for a possibly larger $c_0 > 0$. See \cite[Lemma 8.1]{WangWei}. With regard to the nodal set of $u$, we require
\begin{align}
	\{ u = 0 \} \cap \Omega & = \bigcup_{\ell=1}^Q \Gamma_\ell, \label{eq:sheets.graph.decomposition} \\
\intertext{where $\Gamma_\ell = \graph_\Sigma f_\ell$ denote normal graphs over $\Sigma$ ordered so that $f_{1} < f_{2}< \dots < f_{Q}$, and the graphing functions $f_\ell : \Sigma \to \RR$ are assumed to satisfy}
	|f_\ell| + |\nabla_\Sigma f_\ell| & \leq \eta, \label{eq:sheets.graph.apriori.C1.bounds} \\
\intertext{and (this alternatively follows automatically from \eqref{eq:sheets.sff.bound} and \eqref{eq:sheets.enhanced.sff.bound})}
	|\nabla_\Sigma^2 f_\ell| & \leq c_0. \label{eq:sheets.graph.apriori.C2.bounds}
\end{align}
Finally, after possibly sending $z\mapsto -z$, we can assume that for $z \approx -1$, $u(y,z) \approx -1$. The constants that appear above are to be considered independent of $\eps \leq \eps_{0}$ and fixed so that
\begin{equation} \label{eq:sheets.constants}
	c_0 \gg 1, \; 0 < \varepsilon_0, \beta, \eta \ll 1, \; Q \in \{1, 2, \ldots\}.
\end{equation}

Denote, for $\ell \in \{1, \ldots, Q\}$, $y_0 \in \Sigma$, $r > 0$:
\begin{enumerate}
	\item $\Pi : \Omega \to \Sigma$ to be the closest point projection onto $\Sigma$ with respect to $g$.
	\item $C_r(y_0) \triangleq \{ X \in \Omega : \Pi(X) \in B_r^{n-1}(y_0) \}$.
	\item $\Gamma_\ell(r) \triangleq \Gamma_\ell \cap C_r(0)$.
	\item $Z_{\Gamma_\ell} : \Gamma_\ell(3/2) \times [-1,1] \to \Omega$ to be the normal exponential map with respect to $\Gamma_\ell$. 
	\item $\Pi_\ell : \Omega \to \Gamma_\ell$ to be the closest point projection onto $\Gamma_\ell$ with respect to $g$.
	\item $d_\ell : \Omega \to \RR$ to be the signed distance from $\Gamma_\ell$ (with respect to $g$), which is positive above it and negative below it.
	\item $D_\ell \triangleq \min \{ |d_{\ell-1}|, |d_{\ell+1}| \}$.
\end{enumerate}

Let us agree once and for all regarding Sections \ref{sec:jacobi.toda.reduction}-\ref{sec:stable.solutions}, that each $\Gamma_\ell$ is endowed with the same coordinates $(y^1, \ldots, y^{n-1})$ as $\Sigma$ via the diffeomorphism $\Pi|_{\Gamma_\ell} : \Gamma_\ell \xrightarrow{\approx} \Sigma$.

Set $\Omega' \triangleq B_1^{n-1}(0) \times [-2\eta,2\eta] \subset \Omega$. Consider arbitrary $C^2$ functions
\[ h_\ell : \Gamma_\ell \cap C_1(0) \to (-\tfrac{\eta}{2}, \tfrac{\eta}{2}), \; \ell \in \{1, \ldots, Q \}. \]
Let $\ve{h} = (h_1, \ldots, h_n)$, From $\ve{h}$, we construct an approximate critical point $U(\ve{h})$ of $E_\varepsilon \restr \Omega'$,
\begin{equation} \label{eq:approximate.critical.point}
	U[\ve{h}] \triangleq \frac{(-1)^{Q+1}-1}{2} + \sum_{\ell=1}^Q \overline{\mathbb{H}}_{\eps,\ell}.
\end{equation}
Here, each $\overline{\mathbb{H}}_{\eps,\ell}$ is given by
\begin{multline} \label{eq:approximate.critical.point.model}
	((Z_{\Gamma_\ell})^* \overline{\mathbb{H}}_{\eps,\ell})(y, z) \triangleq \overline{\mathbb{H}}{}^{3 |\log \eps|} \big((-1)^{\ell-1} \eps^{-1} (z-h_\ell(y))\big) \\
	\iff \overline{\mathbb{H}}_{\eps,\ell} = \overline{\mathbb{H}}^{3|\log \eps|}((-1)^{\ell-1} \eps^{-1}(d_\ell - h_\ell \circ \Pi_\ell)),
\end{multline}
with $\overline{\mathbb{H}}{}^\Lambda : \RR \to [-1,1]$ (here, $\Lambda = 3 |\log \eps|$) being
\begin{equation} \label{eq:HLambda-cutoff-def}\overline{\mathbb{H}}{}^\Lambda(t) \triangleq \chi(\Lambda^{-1} t) \mathbb{H}(t) \pm (1-\chi(\Lambda^{-1} t)), \end{equation}
($\pm$ depending on $t > 0$ or $t < 0$). Here, $\chi (t) = 1$ for $t \in (-1,1)$ and $\support \chi \subset (-2,2)$ is a fixed cutoff function. These functions, $\overline{\mathbb{H}}{}^{3|\log \eps|}$, are truncations of $\mathbb{H}$ that coincide with it on $(-3 |\log \varepsilon|, 3 |\log \varepsilon|)$, with $\pm 1$ outside $(-6 |\log \varepsilon|, 6 |\log \varepsilon|)$, and such that
\begin{equation} \label{eq:approximate.heteroclinic.behavior}
	|(\overline{\mathbb{H}}{}^{3|\log \varepsilon|})'' - W'(\overline{\mathbb{H}}{}^{3 |\log \varepsilon|})|_{C^2(\RR)} = O(\varepsilon^3).
\end{equation}
See \cite[Section 9.1]{WangWei} for more details.

\begin{rema}
	The components of $\ve{h}$ represent the vertical offset  of the heteroclinic solutions we're superimposing relative to the nodal set of $u$.
\end{rema}

One can show (see \cite[Subsection 9.1]{WangWei}) that there exists $\ve{h}$ such that for every $\ell \in \{1, \ldots, Q \}$, $y \in \Gamma_\ell$, we have the orthogonality relation:
\begin{equation}\label{eq:h.defn.orth}
\int_{-\eta}^{\eta} ((Z_{\Gamma_\ell})^* (u - U[\ve{h}]))(y, z) \partial_{z} ((Z_{\Gamma_\ell})^* \overline{\mathbb{H}}_{\eps,\ell})(y,z) \, dz = 0. 
\end{equation}
Moreover (see \cite[Remark 9.2]{WangWei}):
\[ \sum_{j=0}^3 \varepsilon^{j-1} \Vert \nabla^j \ve{h} \Vert_{C^0(B_1^{n-1}(0))} = o(1) \text{ as } \varepsilon \to 0. \]
It will prove useful to introduce the notation
\begin{equation} \label{eq:discrepancy.function}
	\phi \triangleq u - U[\mathbf{h}],
\end{equation}
seeing as to how we can conveniently bound $\ve{h}$ in terms of $\phi$, as Lemma \ref{lemm:h.phi.comparison} below shows. 

\begin{lemm}[{\cite[Lemma 9.6]{WangWei}}] \label{lemm:h.phi.comparison}
	For $\ell \in \{1, \ldots, Q\}$, $y \in \Gamma_\ell(\tfrac{9}{10})$,
	\begin{align*}
		\varepsilon^{-1} |h_\ell(y)|
			& \leq c \left( |\phi |_{\Gamma_\ell}(y)| + \exp( -\sqrt{2} \varepsilon^{-1} D_\ell(y)) \right), \\
		|\nabla_{\Gamma_\ell} h_\ell(y)| 
			& \leq c \Big( \varepsilon |\nabla_{\Gamma_\ell} (\phi|_{\Gamma_\ell})(y)| + o(1) \exp (-\sqrt{2}\varepsilon^{-1} D_\ell(y)) \Big), \\
		\varepsilon |\nabla^2_{\Gamma_\ell} h_\ell(y)|
			& \leq c \Big( \varepsilon^2 |\nabla^2_{\Gamma_\ell} (\phi|_{\Gamma_\ell})(y)| + \eps^2 |\nabla_{\Gamma_\ell} (\phi|_{\Gamma_\ell})(y)|^2 + o(1) \exp (-\sqrt{2}\varepsilon^{-1} D_\ell(y)) \Big), \\
		\varepsilon^{1+\theta} [\nabla^2_{\Gamma_\ell}  h_\ell]_{\theta}
			& \leq c' \Big( \varepsilon^{2+\theta} [\nabla^2_{\Gamma_\ell} (\phi|_{\Gamma_\ell})]_{\theta}  + \eps^{2+\theta} \Vert \nabla_{\Gamma_\ell} (\phi|_{\Gamma_\ell}) \Vert_{C^0} [ \nabla_{\Gamma_\ell} (\phi|_{\Gamma_\ell}) ]_\theta + \Vert \exp (-\sqrt{2}\varepsilon^{-1} D_\ell ) \Vert_{C^0} \Big),
	\end{align*}
	where $c = c(n,c_0,E_0,\eta,\beta)$,  $c' = c'(n,c_0,E_0,\eta,\beta,\theta)$, and $o(1)$ is taken as $\eps \to 0$ with the remaining parameters held fixed. In the last inequality, the H\"older seminorms and the $C^k$ norms are taken over all $y' \in \Gamma_\ell \cap C_\eps(\Pi(y))$.
\end{lemm}

Wang--Wei deduce (see \cite[(10.2)]{WangWei}) the following Jacobi-Toda-like system; for $y \in \Gamma_\ell(\tfrac{9}{10})$,
\begin{align}
	& \varepsilon (\Delta_{\Gamma_\ell} h_\ell(y) - H_{\Gamma_\ell}(y)) \label{eq:jacobi.toda} \\ 
	& \qquad = \frac{4(\expansioncoeff)^{2}}{\energyunit} \left( \exp(-\sqrt{2}\varepsilon^{-1} |d_{\ell-1}(y)|) - \exp(-\sqrt{2}\varepsilon^{-1} |d_{\ell+1}(y)|) \right) \nonumber \\
	& \qquad  + O \Big( \varepsilon^{-1} |h_\ell(y)| + \varepsilon^{-1} \Vert (h_{\ell-1} \circ \Pi_{\ell-1} \circ Z_{\Gamma_\ell})(y, \cdot )\Vert_{C^0} + \varepsilon^{\tfrac{1}{3}} \Big)  \exp(-\sqrt{2} \varepsilon^{-1} |d_{\ell-1}(y)|) \nonumber \\
	& \qquad  + O \Big( \varepsilon^{-1} |h_\ell(y)| + \varepsilon^{-1} \Vert (h_{\ell+1} \circ \Pi_{\ell+1} \circ Z_{\Gamma_\ell})(y, \cdot )\Vert_{C^0} + \varepsilon^{\tfrac{1}{3}} \Big)  \exp(-\sqrt{2} \varepsilon^{-1} |d_{\ell+1}(y)|) \nonumber \\
	& \qquad  + O(\exp(-(\tfrac{3}{2}\sqrt{2}) \varepsilon^{-1} |d_{\ell-1}(y)|))  + O(\exp(-(\tfrac{3}{2}\sqrt{2}) \varepsilon^{-1} |d_{\ell+1}(y)|)) \nonumber \\
	& \qquad  + O(\exp(-\sqrt{2} \varepsilon^{-1} |d_{\ell-2}(y)|))  + O(\exp(-\sqrt{2} \varepsilon^{-1} |d_{\ell+2}(y)|)) \nonumber \\
	& \qquad  + \sum_{m\neq \ell} \varepsilon^{-1} |d_m(y)| \exp(-\sqrt{2}\varepsilon^{-1}|d_m(y)|) \Big[ \varepsilon \Vert \Delta_{\Gamma_m} h_m - H_{\Gamma_m} \Vert_{C^0} + \Vert \nabla_{\Gamma_m} h_m \Vert^2_{C^0} \Big] \nonumber \\
	& \qquad  + \sup_{|t| < 6 \varepsilon |\log \varepsilon|} \Big[ \varepsilon^4 |(\nabla^2_{\Gamma_{\ell,t}} (\phi|_{\Gamma_{\ell,t}}))(Z_{\Gamma_\ell}(y,t))|^2  + \varepsilon^2 |(\nabla_{\Gamma_{\ell,t}} (\phi|_{\Gamma_{\ell,t}}))(Z_{\Gamma_\ell}(y,t))|^2 + |\phi(Z_{\Gamma_\ell}(y, t))|^2 \Big] \nonumber \\
	& \qquad + O(\varepsilon^2). \nonumber
\end{align}
The $C^0$ norms appearing in the second and third term of the right hand side is taken over $|t| < 6\varepsilon|\log \varepsilon|$, and the $C^0$ norms appearing in the third term from the end is taken over $\Gamma_m \cap C_{\eps^{4/3}}(\Pi(y))$.

\begin{rema}
	$\Gamma_{\ell,t}$ denote $t$-level sets in Fermi coordinates $(y,t)$ relative to $\Gamma_\ell$, i.e., $\Gamma_{\ell,t} = \{ d_\ell = t \}$.
\end{rema}

\begin{rema}
	Notice the sign difference in the mean curvature terms between \eqref{eq:jacobi.toda} and \cite[(10.2)]{WangWei}. For us, the mean curvature is the divergence of the upper pointing unit normal. For instance, the ambient Laplace-Beltrami operator expands as
	\[ \Delta_g = \Delta_{\Gamma_{\ell,z}} + \partial_z^2 + H_{\Gamma_{\ell,z}} \partial_z. \]
	For this reason, all instances of the mean curvature in this work have to have the opposite sign relative to \cite{WangWei}.
\end{rema}

It will also be convenient to introduce the notation
\begin{equation} \label{eq:sup.exp.distance}
	A_\ell(r) \triangleq \sup \Big\{ \exp(-\sqrt{2} \varepsilon^{-1} D_\ell(y)) : 
	y \in \Gamma_\ell(r) \Big\}.
\end{equation}
We record \cite[(12.4)]{WangWei}, which will help estimate terms involving $h$, $\phi$, and the mean curvature:
\begin{equation} \label{eq:phi.c2a.estimate.full}
	\Vert \phi \Vert_{C^{2,\theta}_\eps(\cM_\ell(r))} + \eps \Vert \Delta_{\Gamma_\ell} h_\ell - H_{\Gamma_\ell} \Vert_{C^{0,\theta}_\eps(\Gamma_\ell(r))}  \leq c' \varepsilon^2 + c' \sum_{m=1}^Q A_m(r+K\varepsilon |\log \varepsilon|),
\end{equation}
where we're using the weighted H\"older space notation from \eqref{eq:dirichlet.data.ckalpha.eps} (see Section \ref{sec:dirichlet.data}), and
\begin{equation*}
\cM_\ell(r) \triangleq \{ X \in C_r(0) : |d_\ell(X)| < 1, 
- d_{\ell-1}(X) < d_\ell(X) < - d_{\ell+1}(X) \}.
\end{equation*}
Likewise, we record \cite[(13.6)]{WangWei}:
\begin{equation} \label{eq:phi.improved.c2a.estimate.full}
	\varepsilon \Vert ((Z_{\Gamma_\ell})_* \partial_{y_i}) \phi \Vert_{C^{1,\theta}_\eps(\cM_\ell(r))} \leq c' \varepsilon^2 + c' \sum_{m=1}^Q A_m(r + 2K \varepsilon |\log \varepsilon|)^{1+\kappa}  + c' \varepsilon^\kappa \sum_{m=1}^Q A_m(r+2K \varepsilon |\log \varepsilon|),
\end{equation}
with $\kappa > 0$.

The expressions above, \eqref{eq:phi.c2a.estimate.full}-\eqref{eq:phi.improved.c2a.estimate.full}, are true for all $\ell \in \{1, \ldots, Q\}$, $r \leq 8/10$, $\theta \in (0,1)$, $\varepsilon \leq \varepsilon'$, where $c'$, $\varepsilon'$, $K$, $\kappa$, depend on $n$, $c_0$, $E_0$, $\eta$, $\beta$, $\theta$.

\begin{rema} \label{rema:major.goal}
	In the remainder of Sections \ref{sec:jacobi.toda.reduction}-\ref{sec:stable.solutions}, we'll be actively interested in estimating the vertical distances $D_\ell$ from below. This is because Lemma \ref{lemm:h.phi.comparison}, \eqref{eq:sup.exp.distance},  \eqref{eq:phi.c2a.estimate.full}, and interior Schauder estimates together imply that, with $r$, $\theta$ as above:
	\begin{equation} \label{eq:rema.major.goal}
		\min_{\ell \in \{ 1, \ldots, Q \}} \inf_{\Gamma_\ell(r)} D_\ell \geq \tfrac{1+\theta}{2} \sqrt{2} \eps |\log \eps| \implies \Gamma_{\ell}(r') \text{ is uniformly } C^{2,\theta}
	\end{equation}
	for all $\ell \in \{ 1, \ldots, Q \}$, $r' \leq \sigma r$, $\sigma \in (0, 1)$, $\eps \leq \eps' = \eps'(n, c_0, E_0, \eta, \beta, \theta, \sigma)$.
\end{rema}

\subsection{Bootstrapping regularity via sheet distance lower bounds} \label{subsec:bootstrapping.sheet.distance.bounds}

We recall the following lemma from \cite{WangWei}. (See \cite[Appendix C]{Mantoulidis} for necessary modifications for the Riemannian setting.)

\begin{lemm}[{\cite[Section 14]{WangWei}}] \label{lemm:stationary.estimates}
	If $\ell \in \{1, \ldots, Q\}$, $y \in \Gamma_\ell(\tfrac{8.5}{10})$, and  $\varepsilon \leq \varepsilon_1$, then
	\[ D_\ell(y) \geq \tfrac{1}{2} \sqrt{2} \varepsilon |\log \varepsilon| - c_1 \varepsilon, \]
	where $\varepsilon_1 = \varepsilon_1(n, c_0, E_0, \eta, \beta)$, $c_1 = c_1(n, c_0, E_0, \eta, \beta)$.
\end{lemm}

As a corollary of Lemma \ref{lemm:stationary.estimates}, we can bootstrap the proof of Lemma \ref{lemm:h.phi.comparison} and obtain the following \emph{improved} estimates:

\begin{lemm} \label{lemm:h.phi.comparison.improved}
	For $\ell \in \{1, \ldots, Q\}$, $y \in \Gamma_\ell(\tfrac{8}{10})$,
	\begin{align*}
		\varepsilon^{-1} |h_\ell(y)|
			& \leq c \left( |\phi|_{\Gamma_\ell}(y)| + \exp( -\sqrt{2} \varepsilon^{-1} D_\ell(y)) \right), \\
		|\nabla_{\Gamma_\ell} h_\ell(y)| 
			& \leq c \left( \varepsilon |\nabla_{\Gamma_\ell} (\phi|_{\Gamma_\ell})(y)| +  \eps^\kappa \exp (-\sqrt{2}\varepsilon^{-1} D_\ell(y)) \right), \\
		\varepsilon |\nabla^2_{\Gamma_\ell} h_\ell(y)|
			& \leq c \Big( \varepsilon^2 |\nabla^2_{\Gamma_\ell} (\phi|_{\Gamma_\ell})(y)| + \eps^2 |\nabla_{\Gamma_\ell} (\phi|_{\Gamma_\ell})(y)|^2  + \eps^\kappa \exp (-\sqrt{2}\varepsilon^{-1} D_\ell(y)) \Big), \\
		\varepsilon^{1+\theta} [\nabla^2_{\Gamma_\ell}  h_\ell]_{\theta}
			& \leq c' \Big( \varepsilon^{2+\theta} [\nabla^2_{\Gamma_\ell} (\phi|_{\Gamma_\ell})]_{\theta} + \eps^{2+\theta} \Vert \nabla_{\Gamma_\ell} (\phi|_{\Gamma_\ell}) \Vert [ \nabla_{\Gamma_\ell} (\phi|_{\Gamma_\ell}) ]_\theta + \eps^{\kappa'} \Vert \exp (-\sqrt{2}\varepsilon^{-1} D_\ell ) \Vert_{C^0} \Big),
	\end{align*}
	where $c = c(n,c_0,E_0,\eta,\beta)$, $c' = c'(n,c_0,E_0,\eta,\beta,\theta)$, $\kappa = \kappa(n,c_0,E_0,\eta,\beta)$, $\kappa' = \kappa'(n,c_0,E_0,\eta,\beta,\theta)$. The norms and seminorms in the last inequality are taken over all $y' \in \Gamma_\ell$ with $\Pi(y') \in B_\eps^{n-1}(\Pi(y))$.
\end{lemm}
\begin{proof}
See Appendix \ref{app:proof.lem.comp.improved}.
\end{proof}

We now indicate how the enhanced second fundamental form tensor is affected by these estimates. 

Fix $\ell \in \{1, \ldots, Q\}$. We see from \eqref{eq:phi.improved.c2a.estimate.full} that
\begin{align}
	\varepsilon \Vert \nabla \phi -  \langle \nabla \phi, \nabla d_\ell \rangle \nabla d_\ell  \Vert_{C^0(\cM_\ell(r))} & \leq c' \varepsilon^2 + c' \sum_{m=1}^Q A_m(r + 2K \varepsilon |\log \varepsilon|)^{1+\kappa} \label{eq:grad.phi.estimate} \\
\intertext{for some $\kappa = \kappa(n, c_0, E_0, \eta, \beta) > 0$. Likewise, from \eqref{eq:approximate.critical.point}, \eqref{eq:approximate.critical.point.model}, Lemmas \ref{lemm:stationary.estimates}-\ref{lemm:h.phi.comparison.improved}, and \eqref{eq:phi.improved.c2a.estimate.full}:}
	\varepsilon \Vert \nabla U[\mathbf{h}] - \langle \nabla U[\mathbf{h}], \nabla d_\ell \rangle \nabla d_\ell  \Vert_{C^0(\cM_\ell(r))} & \leq c' \varepsilon^2 + c' \sum_{m=1}^Q A_m(r + 2K \varepsilon |\log \varepsilon|)^{1+\kappa}. \label{eq:grad.gluedU.estimate} \\
\intertext{Combining \eqref{eq:discrepancy.function}, \eqref{eq:grad.phi.estimate}, and \eqref{eq:grad.gluedU.estimate}, we get:}
	\varepsilon \Vert \nabla u - \langle \nabla u, \nabla d_\ell \rangle \nabla d_\ell \Vert_{C^0(\cM_\ell(r))} & \leq c' \varepsilon^2 + c' \sum_{m=1}^Q A_m(r + 2K \varepsilon |\log \varepsilon|)^{1+\kappa}. \label{eq:grad.u.estimate} \\
\intertext{Combining \eqref{eq:sheets.lower.density.bound} and  \eqref{eq:grad.u.estimate}, we get:}	
	\Vert \nu - (-1)^{\ell-1} \nabla d_\ell  \Vert_{C^0(\cM_\ell(r) \cap \{ |u| \leq 1-\beta \})} & \leq c' \varepsilon^2 + c' \sum_{m=1}^Q A_m(r + 2K \varepsilon |\log \varepsilon|)^{1+\kappa}, \label{eq:normal.direction.estimate}	
\end{align}
where $\nu = |\nabla u|^{-1} \nabla u$ denotes the normal to the level set of $u$ through each point. (The level set is smooth on $\{ |u| \leq 1-\beta \}$ in view of \eqref{eq:sheets.lower.density.bound}.) 

For the remainder of this section we choose to work in Fermi coordinates $(y, t)$ relative to $\Gamma_\ell$; note that $t = d_\ell$. It is not hard to see that the only nontrivial Christoffel symbols in this coordinate system are $\Gamma_{ij}^t$, $\Gamma_{jt}^i$, $\Gamma_{tj}^i$, and $\Gamma_{ij}^k$.  Set
\begin{equation} \label{eq:christoffel.symbol.sup}
\widehat{\Gamma}_\ell(r) \triangleq \sup_{\cM_\ell(r) \cap \{ |u| \leq 1-\beta \}} |\Gamma_{ij}^t| + |\Gamma_{jt}^i| + |\Gamma_{tj}^i| + |\Gamma_{ij}^k|.
\end{equation}
By arguing as above, and relying on \eqref{eq:phi.improved.c2a.estimate.full}, we find that:
\begin{align}
	& \varepsilon^2 \Vert \nabla^2 u - \nabla^2 u(\partial_t, \partial_t) \, dt^2 \Vert_{C^0(\cM_\ell(r) \cap \{|u| \leq 1-\beta\})} \label{eq:hessu.estimate.i} \\
	& \qquad \leq \varepsilon^2 \sum_{i=1}^{n-1} \Vert \nabla (((Z_{\Gamma_\ell})_* \partial_{y_i}) u) \Vert_{C^0(\cM_\ell(r) \cap \{ |u| \leq 1-\beta \})}  + \varepsilon^2 \widehat{\Gamma}_\ell(r) \Vert \nabla u \Vert_{C^0(\cM_\ell(r) \cap \{ |u| \leq 1-\beta \})} \nonumber \\
	& \qquad \leq c' \varepsilon^2 + c' \sum_{m=1}^Q A_m(r + 2K \varepsilon |\log \varepsilon|)^{1+\kappa} + c' \varepsilon \widehat{\Gamma}_\ell(r). \nonumber \\
\intertext{Using \eqref{eq:normal.direction.estimate} (note that $\partial_t = \nabla d_\ell$),}
	& \varepsilon^2 \Vert \nabla^2 u(\partial_t, \partial_t) \, dt \otimes (dt - \langle dt, \nu \rangle \nu^\flat) \Vert_{C^0(\cM_\ell(r) \cap \{ |u| \leq 1-\beta \})} \label{eq:hessu.estimate.ii} \\
	& \qquad \leq c' \Vert \nu - \partial_t \Vert_{C^0(\cM_\ell(r) \cap \{|u| \leq 1-\beta \})} \leq c' \varepsilon^2 + c' \sum_{m=1}^Q A_m(r + 2K \varepsilon |\log \varepsilon|)^{1+\kappa}, \nonumber \\
\intertext{where $\nu^\flat$ denotes $\nu$'s dual 1-form. Finally,   \eqref{eq:sheets.lower.density.bound}, \eqref{eq:hessu.estimate.i}, and  \eqref{eq:hessu.estimate.ii} give:}
	& \Vert \cA \Vert_{C^0(\cM_\ell(r) \cap \{ |u| \leq 1-\beta \})} \leq c' \varepsilon + c' \varepsilon^{-1} \sum_{m=1}^Q A_m(r + 2K \varepsilon |\log \varepsilon|)^{1+\kappa} + c' \widehat{\Gamma}_\ell(r). \label{eq:enhanced.sff.estimate} \\
\intertext{Now, we turn to estimating $H_{\Gamma_{\ell}}$. From Lemma \ref{lemm:h.phi.comparison.improved} and \eqref{eq:phi.improved.c2a.estimate.full} we have, for $y \in \Gamma_\ell(\tfrac{8}{10})$,}
	& \varepsilon |\Delta_{\Gamma_\ell} h_\ell(y)| \label{eq:bootstrapped.i} \\
	& \qquad \leq \varepsilon^2 |\nabla^2_{\Gamma_\ell} (\phi|_{\Gamma_\ell})(y)| +  \eps^\kappa \exp(-\sqrt{2} \varepsilon^{-1} D_\ell(y)) \nonumber \\
	& \qquad \leq c' \varepsilon^2 + c' \varepsilon^{\kappa} \sum_{m=1}^Q A_m(|y| + 2K \varepsilon|\log \varepsilon|) + c' \sum_{m=1}^Q  A_m(|y| + 2K \varepsilon |\log \varepsilon|)^{1+\kappa}. \nonumber
\end{align}

We're going to estimate the terms in \eqref{eq:jacobi.toda} from above by a function of $\varepsilon$ and the quantities in \eqref{eq:sup.exp.distance}. Fix $\ell \in \{1, \ldots, Q\}$, $y \in \Gamma_\ell(\tfrac{7}{10})$.

From Lemma \ref{lemm:stationary.estimates}, Lemma \ref{lemm:h.phi.comparison.improved}, and \eqref{eq:phi.c2a.estimate.full}, we have
\begin{align} 
& (\varepsilon^{-1} |h_\ell| + \varepsilon^{-1} |h_{\ell-1} \circ \Pi_{\ell-1} \circ Z_{\Gamma_\ell}| + \varepsilon^{\tfrac{1}{3}}) \exp(-\sqrt{2} \varepsilon^{-1} |d_{\ell-1}(y)|) \nonumber \\
& + (\varepsilon^{-1} |h_\ell| + \varepsilon^{-1} |h_{\ell+1} \circ \Pi_{\ell+1} \circ Z_{\Gamma_\ell}|  + \varepsilon^{\tfrac{1}{3}}) \exp(-\sqrt{2} \varepsilon^{-1} |d_{\ell+1}(y)|) \nonumber \\
& + \exp(-(\tfrac{3}{2}\sqrt{2}) \varepsilon^{-1} |d_{\ell-1}(y)|) + \exp(-(\tfrac{3}{2}\sqrt{2}) \varepsilon^{-1} |d_{\ell+1}(y)|) \nonumber \\
& + \exp(-\sqrt{2} \varepsilon^{-1} |d_{\ell-2}(y)|)) + \exp(-\sqrt{2} \varepsilon^{-1} |d_{\ell+2}(y)|) \nonumber \\
& \qquad \leq c' \varepsilon^2 + c'  \varepsilon^{\kappa} \sum_{m=1}^Q A_m(|y|+K\varepsilon |\log \varepsilon|) + c' \sum_{m=1}^Q  A_m(|y|+K \varepsilon |\log \varepsilon|)^{1+\kappa}. \label{eq:bootstrapped.ii}
\end{align}
By Lemma \ref{lemm:h.phi.comparison}, \eqref{eq:phi.c2a.estimate.full}, \eqref{eq:phi.improved.c2a.estimate.full}, every $m \neq \ell$ satisfies
\begin{multline} \label{eq:bootstrapped.iii}
\varepsilon^{-1} |d_m(y)| \exp(-\sqrt{2}\varepsilon^{-1}|d_m(y)|)  \Big[ \varepsilon \Vert \Delta_{\Gamma_m} h_m - H_{\Gamma_m}  \Vert_{C^0} + \Vert \nabla_{\Gamma_m} h_m \Vert^2_{C^0} \Big] \\
\leq  c' A_m(|y| + 2K \varepsilon |\log \varepsilon|)^{1-\rho}\sum_{m'\not = m} A_{m}(|y|+2K\varepsilon |\log \varepsilon|)
+ c' \varepsilon^2 A_m(|y| + 2K \varepsilon |\log \varepsilon|)^{1-\rho}
\end{multline}
for small $\rho > 0$, $\varepsilon \leq \varepsilon'$. The $C^0$ norms are taken over $\Gamma_m \cap C_{\eps^{4/3}}(\Pi(y))$. By Lemma \ref{lemm:stationary.estimates} and \eqref{eq:phi.c2a.estimate.full},
\begin{multline} \label{eq:bootstrapped.iv}
	\sup_{|t| < 6 \eps |\log \eps|} \Big[ \varepsilon^4 |(\nabla^2_{\Gamma_{m,t}} (\phi|_{\Gamma_{m,t}}))(Z_{\Gamma_m}(y,t))|^2 + \varepsilon^2 |(\nabla_{\Gamma_{m,t}} (\phi|_{\Gamma_{\ell,t}})(Z_{\Gamma_m}(y,z))|^2 + |\phi(Z_{\Gamma_m}(y, z))|^2 \Big] \\
	\leq c' \varepsilon^2 + c' \sum_{m'=1}^Q A_{m'}(|y| + K \varepsilon |\log \varepsilon|)^2.
\end{multline}
Combined, \eqref{eq:jacobi.toda} and \eqref{eq:bootstrapped.i}-\eqref{eq:bootstrapped.iv} give
\begin{align} 
	-\varepsilon H_{\Gamma_\ell}(y) 
	& = \frac{4(\expansioncoeff)^{2}}{\energyunit} \Big( \exp(-\sqrt{2} \varepsilon^{-1} |d_{\ell-1}(y)|) - \exp( -\sqrt{2} \varepsilon^{-1} |d_{\ell+1}(y)|) \Big) + \cR_{\ell} \label{eq:bootstrapped.v} \\
\intertext{for all $y \in \Gamma_\ell(\tfrac{7}{10})$, where}
	|\cR_\ell(y)| & \leq c' \varepsilon^2 + c'  \eps^\kappa \sum_{m=1}^Q A_m(|y| + 2K \varepsilon |\log \varepsilon|) + c' \sum_{m=1}^Q A_m(|y| + 2K \varepsilon |\log \varepsilon|)^{1+\kappa} \label{eq:bootstrapped.error.i}.
\end{align}

\begin{lemm} \label{lemm:mean.curvature.laplacian}
	Let $f : B_1^{n-1}(0) \to \RR$ be as in \eqref{eq:sheets.graph.apriori.C1.bounds}-\eqref{eq:sheets.graph.apriori.C2.bounds}. If $G[f]$ is the normal graph of $f$ over $\Gamma_{\ell}$, i.e., $G[f] = \{ Z_{\Gamma_\ell}(y, f(y)) : y \in B_1^{n-1}(0) \}$, then
	\[ H_{G[f]} - H_{\Gamma_\ell} = - (\cL + |\sff_{\Gamma_\ell}|^2 + \ricc_g(\nu_{\Gamma_\ell},\nu_{\Gamma_\ell})|_{\Gamma_\ell})f + \cQ(f), \]
	where $\cL$ is the linear uniformly elliptic operator
	\begin{equation} \label{eq:mean.curvature.laplacian.operator}
		\cL \varphi = \cL_{\Gamma_\ell,G[f]} \varphi \triangleq a(y)^{-1} \divg_{\Gamma_\ell} \left( a(y) \langle (Z_{\Gamma_\ell})_* \nu_{\Gamma_\ell}, \nu_{G[f]} \rangle \nabla_{G[f]} \varphi \right),
	\end{equation}
	with
	\begin{equation} \label{eq:mean.curvature.laplacian.operator.a}
		a(y) = a_{\Gamma_\ell,G[f]}(y) \triangleq \frac{\sqrt{g_{\Gamma_\ell}}}{\sqrt{g_{f(y)}}}.
	\end{equation}
	Here $(Z_{\Gamma_\ell})_* \nu_{\Gamma_\ell}$, $\nu_{G[f]}$ are upward pointing unit normal in Fermi coordinates and the upward pointing unit normal to $G[f]$, both evaluated at $Z_{\Gamma_\ell}(y, f(y))$. Note that the elliptic symbol coefficients are uniformly bounded away from $0$ and $\infty$ depending on \eqref{eq:sheets.graph.apriori.C1.bounds}. The (nonlinear) error term $\cQ(f)$ satisfies
	\begin{equation*}
		|\cQ(f)| \leq c' (|f|^2 + |\nabla_{\Gamma_\ell} f|^2).
	\end{equation*}
\end{lemm}
\begin{proof}
	This is a restatement of Lemma \ref{lemm:mean.curv.graphical.quad.error.new} from Appendix \ref{app:mean.curvature.graphs}.
\end{proof}

Notice that, by \eqref{eq:sheets.graph.apriori.C1.bounds}-\eqref{eq:sheets.graph.apriori.C2.bounds}, $\Gamma_{\ell+1}$ can be viewed as a normal graph of some function $f_{\ell,\ell+1}$ over $\Gamma_{\ell}$ that satisfies the conditions of Lemma \ref{lemm:mean.curvature.laplacian}. Let
\[ y' \triangleq Z_{\Gamma_\ell}(y,f_{\ell,\ell+1}(y)) \in \Gamma_{\ell+1}. \]
Applying \eqref{eq:bootstrapped.v} to $y$ at $\Gamma_\ell$ and to $y'$ at $\Gamma_{\ell+1}$,  subtracting, and invoking Lemma \ref{lemm:mean.curvature.laplacian}, we see that:
\begin{align}
	& \varepsilon (\cL + |\sff_{\Gamma_\ell}|^2 + \ricc_g(\nu,\nu)|_{\Gamma_\ell} + \cQ) f_{\ell,\ell+1}(y) \label{eq:bootstrapped.vi} \\
	& \qquad = \varepsilon(H_{\Gamma_\ell}(y) - H_{\Gamma_{\ell+1}}(y')) \nonumber \\
	& \qquad = \frac{4(\expansioncoeff)^{2}}{\energyunit} \Big( \exp(-\sqrt{2} \varepsilon^{-1} f_{\ell,\ell+1}(y))  - \exp(-\sqrt{2} \varepsilon^{-1} |d_{\ell+2}(y')|) \nonumber \\
	& \qquad \qquad \qquad - \exp(-\sqrt{2} \varepsilon^{-1} |d_{\ell-1}(y)|) + \exp(-\sqrt{2} \varepsilon^{-1} |d_{\ell+1}(y)|) \Big) \nonumber \\
	& \qquad \qquad - \cR_{\ell}(y) + \cR_{\ell+1}(y'). \nonumber
\end{align}
Here, $\cL$ is the second order linear operator defined in \eqref{eq:mean.curvature.laplacian.operator}, and which depends on $\Gamma_\ell$, $\Gamma_{\ell+1}$.  Note that (see Lemma \ref{lemm:WW.tilting.comparison.d}):
\begin{equation} \label{eq:distance.exchange.error}
\exp(-\sqrt{2} \varepsilon^{-1} |d_{\ell+1}(y)|) = \exp(-\sqrt{2} \varepsilon^{-1} f_{\ell,\ell+1}(y)) 
+ O(\varepsilon^{\tfrac{1}{3}}) \exp(-\sqrt{2} \varepsilon^{-1} D_{\ell}(y)).
\end{equation}
Absorbing the last term above into $\cR_\ell$ in view of \eqref{eq:bootstrapped.error.i}, we conclude:
\begin{align}
	& \varepsilon (\cL + |\sff_{\Gamma_\ell}|^2 + \ricc_g(\nu,\nu)|_{\Gamma_\ell} + \cQ) f_{\ell,\ell+1}(y) \label{eq:bootstrapped.vii} \\
	&  = \frac{4(\expansioncoeff)^{2}}{\energyunit} \Big( 2 \exp(-\sqrt{2} \varepsilon^{-1} f_{\ell,\ell+1}(y))   - \exp(-\sqrt{2} \varepsilon^{-1} |d_{\ell+2}(y')|)  - \exp(-\sqrt{2} \varepsilon^{-1} |d_{\ell-1}(y)|) \Big) \nonumber \\
	&  \qquad - \cR_{\ell}(y) + \cR_{\ell+1}(y'). \nonumber
\end{align}
Finally, dropping the negative terms gives:
\begin{equation} \label{eq:bootstrapped.viii} 
	\varepsilon (\cL + |\sff_{\Gamma_\ell}|^2 + \ricc_g(\nu, \nu)|_{\Gamma_\ell} + \cQ) f_{\ell,\ell+1}(y) 
	\leq \frac{8(\expansioncoeff)^{2}}{\energyunit} \exp(-\sqrt{2} \varepsilon^{-1} f_{\ell,\ell+1}(y)) + c' |\cR_\ell(y)| + |\cR_{\ell+1}(y')|; 
\end{equation}
the error terms $\cR_\ell$, $\cR_{\ell+1}$, are still as in \eqref{eq:bootstrapped.error.i}.

\section{Stable phase transitions ($n=3$)}


\label{sec:stable.solutions}

In this section, we use the Allen--Cahn stability inequality and bootstrap the distance estimates from the previous section until they become sufficiently sharp. Specifically, we combine three things: (i) an $L^2$ estimate on the height function of $\{ u = 0 \}$ (following an observation of Wang--Wei \cite[(19.7)]{WangWei}), (ii) a subtle application of Moser's Harnack inequality, and (iii) the nonexistence of nontrivial entire stable critical points of the Toda system on $\mathbf{R}^2$ (cf. the stable Bernstein problem for minimal surfaces in $\mathbf{R}^3$).

\subsection{Strong sheet distance lower bounds} \label{subsec:distance.estimates}

We continue to adopt the conventions and notation laid out in Section \ref{sec:jacobi.toda.reduction}. In particular, we emphasize that we continue to assume \eqref{eq:sheets.sff.bound}-\eqref{eq:sheets.enhanced.sff.bound} as well as assuming that $u$ is a stable critical point of $E_{\eps}\restr\Omega$ (cf.\ Definition \ref{def:ac.morse.index}). 

In \cite[(19.7)]{WangWei}, Wang--Wei derive the following stability inequality (in a slightly different setting) from the usual Allen--Cahn stability inequality.
\begin{equation} \label{eq:toda.stability.estimate}
	\int_{\Gamma_\ell(7/10)} \zeta^2 \Big[ \exp(-\sqrt{2} \varepsilon^{-1} |d_{\ell-1}|) + \exp(-\sqrt{2} \varepsilon^{-1} |d_{\ell+1}|) \Big] \leq c' \int_{\Gamma_\ell(7/10)} \varepsilon^2 |\nabla_{\Gamma_\ell} \zeta|^2 + c'   \eps^{1+\kappa}\int_{\Gamma_\ell(7/10)} \zeta^2
\end{equation}
for all $\ell \in \{1, \ldots, Q\}$, $\zeta \in C^\infty_c(\Gamma_\ell(\tfrac{7}{10}))$, $\varepsilon \leq \varepsilon'$, where $\varepsilon'$, $c'$, $\kappa$ depend on $c_0$, $E_0$, $\eta$, $\beta$. In fact, by a careful inspection of Wang--Wei's derivation of \eqref{eq:toda.stability.estimate} from \cite[Section 19]{WangWei}, we see that the following stronger inequality is true here:
\begin{equation} \label{eq:toda.stability.estimate.sharper}
	\int_{\Gamma_\ell(7/10)} \zeta^2 \Big[ \exp(-\sqrt{2} \varepsilon^{-1} |d_{\ell-1}|) + \exp(-\sqrt{2} \varepsilon^{-1} |d_{\ell+1}|) \Big] \leq c' \int_{\Gamma_\ell(7/10)} \varepsilon^2 |\nabla_{\Gamma_\ell} \zeta|^2 + |\cE_{\zeta}| \int_{\Gamma_\ell(7/10)} \zeta^2 
\end{equation}
with
\begin{equation} \label{eq:toda.stability.estimate.sharper.coefficient}
	|\cE_{\zeta}| \leq c' \varepsilon^2 +  c' \sum_{m=1}^Q \sup\left\{\exp(-\sqrt{2} (1+\kappa) \eps^{-1}D_{m}(y')) : y' \in \Gamma_m \cap \Pi_\ell^{-1}(B_{2K\eps |\log \eps|}^2(\support h)) \right\};
\end{equation}
here, $c', \kappa$ are independent of $\zeta$. We prove  \eqref{eq:toda.stability.estimate.sharper} in Appendix \ref{app:proof.stab.inproved} in a general $n$-dimensional setting, $n \geq 3$. (Below, we use it for $n=3$.) Note that, by Lemma \ref{lemm:stationary.estimates}, this recovers \eqref{eq:toda.stability.estimate}.

Our first main result is the following sheet-distance estimate. (cf. Remark \ref{rema:major.goal}.) 

\begin{prop}[Stable sheet distances, I]  \label{prop:bootstrapped.stable.estimates}
	If $u$ is a stable critical point of $E_\varepsilon \restr \Omega$, $\varepsilon \leq \varepsilon_3$, and $\nu \in (0, \tfrac{1}{2})$, then 
	\[ D_\ell \geq (1-\nu) \sqrt{2} \varepsilon |\log \varepsilon| \text{ on } \Gamma_\ell(\tfrac{1}{3}) \]
	for all $\ell \in \{ 1, \ldots, Q \}$, where $\varepsilon_3 =  \varepsilon_3(c_0,E_0,\eta,\beta,\nu)$.
\end{prop}

\begin{proof}
		Take $\nu \in (0,\tfrac{1}{2})$ and assume, for contradiction, that
		\begin{equation} \label{eq:bootstrapped.contradiction.assumption}
			A_{\ell_0}(r) \geq A_{\ell_0}(\tfrac{1}{3}) > \varepsilon^{2(1-\nu)} \text{ for all } r \in [\tfrac{1}{3}, \tfrac{1}{2}] \text{ and some } \ell_0 \in \{1, \ldots, Q\}.
		\end{equation}
		We will aim to prove
		\begin{equation} \label{eq:bootstrapped.claim}
			\max_{\ell \in \{1, \ldots, Q\}} A_{\ell}(r-K_\nu\varepsilon^\nu) < \tfrac{1}{2} \max_{\ell \in \{1, \ldots, Q\}} A_{\ell}(r) \text{ for all } r \in [\tfrac{1}{3}, \tfrac{1}{2}]
		\end{equation}
		where $K_\nu = K_\nu(c_0, E_0, \eta, \beta, \nu) > 0$; this will in turn prove our claim by a simple iteration. (We denote the dependence of $K_\nu$ on $\nu$ explicitly to disambiguate with the previous constant $K$. Let's assume $K_\nu > 2K$.)
		
		Let $r \in [\tfrac{1}{3}, \tfrac{1}{2}]$, $\alpha \triangleq \max \{ A_{\ell}(r) : \ell \in \{1, \ldots, Q\} \}$. Since
		\begin{equation} \label{eq:bootstrapped.claim.i}
			\alpha > \varepsilon^{2(1-\nu)} 
		\end{equation}
		by \eqref{eq:bootstrapped.contradiction.assumption}, it follows that to prove  \eqref{eq:bootstrapped.claim} it will suffice to prove
		\begin{equation} \label{eq:bootstrapped.claim.ii}
			A_{\ell}(r-\varepsilon K_\nu\alpha^{-\tfrac{1}{2}}) <  \tfrac{1}{2} \alpha \text{ for all } \ell \in \{1, \ldots, Q\}.
		\end{equation}
		Suppose, by way of contradiction, that \eqref{eq:bootstrapped.claim.ii} is violated at some $\ell_0 \in \{1, \ldots, Q\}$ and $y \in \Gamma_{\ell_0}(r - \eps K_\nu \alpha^{-\tfrac{1}{2}})$. From now on let's work in the coordinate chart induced on $\Gamma_{\ell_0}$ by  $\Pi|_{\Gamma_{\ell_0}} \approx \Sigma$. For $\widetilde{y} \in B_{K_\nu/2}^2(0)$, define:
		\begin{equation} \label{eq:bootstrapped.claim.ftilde}
			\widetilde{f}(\widetilde{y}) \triangleq \varepsilon^{-1} f_{\ell_0,\ell_0+1}(y + \varepsilon \alpha^{-\tfrac{1}{2}} \widetilde{y}) - \frac{1}{\sqrt{2}} |\log \alpha|.
		\end{equation}
		If $\widetilde{\cL}$ denotes the translation and rescaling of $\cL$ that respects the stretched coordinate, $\widetilde{y}$, then from \eqref{eq:bootstrapped.viii} we find
		\begin{align*}
			\widetilde{\cL} \widetilde{f}(\widetilde{y}) & = \varepsilon \alpha^{-1} \cL f_{\ell_0,\ell_0+1}(y+\varepsilon\alpha^{-\tfrac{1}{2}}\widetilde{y}) \\
			&  \leq  \frac{8(\expansioncoeff)^{2}\alpha^{-1}}{\energyunit} \exp(-\sqrt{2}\varepsilon^{-1}f_{\ell_0,\ell_0+1}(y+\varepsilon \alpha^{-\tfrac{1}{2}}\widetilde{y})) \\
			&  \qquad + \alpha^{-1} c' |\cR_{\ell_0}(y+\varepsilon \alpha^{-\tfrac{1}{2}}\widetilde{y})| + \alpha^{-1} |\cR_{\ell_0+1}((y+\varepsilon \alpha^{-\tfrac{1}{2}}\widetilde{y})')| \\
			&  \qquad - \varepsilon\alpha^{-1} (|\sff_{\Gamma_{\ell_0}}|^2+\ricc_g(\nu,\nu)|_{\Gamma_{\ell_0}} + \cQ) f_{\ell_0,\ell_0+1}(y+\varepsilon \alpha^{-\tfrac{1}{2}} \widetilde{y} ).
		\end{align*}
		Recalling \eqref{eq:bootstrapped.claim.ftilde}, the computation above readily implies that
		\begin{align} \label{eq:bootstrapped.claim.iii}
			\widetilde{\cL} \widetilde{f}(\widetilde{y}) &  \leq \frac{8(\expansioncoeff)^{2}}{\energyunit} \exp(-\sqrt{2} \widetilde{f}(\widetilde{y}))  \\
			&  \qquad + \alpha^{-1} c' |\cR_{\ell_0}(y+\varepsilon \alpha^{-\tfrac{1}{2}}\widetilde{y})|  + \alpha^{-1} |\cR_{\ell_0+1}((y+\varepsilon \alpha^{-\tfrac{1}{2}}\widetilde{y})')| \nonumber \\
			&  \qquad - \varepsilon \alpha^{-1} (|\sff_{\Gamma_{\ell_0}}|^2+\ricc_g(\nu,\nu)|_{\Gamma_{\ell_0}} + \cQ) f_{\ell_0,\ell_0+1}(y + \varepsilon \alpha^{-\tfrac{1}{2}} \widetilde{y}). \nonumber
		\end{align}
		From \eqref{eq:bootstrapped.error.i} and  \eqref{eq:bootstrapped.claim.i} 
		we have: 
		\begin{equation} \label{eq:bootstrapped.claim.iv}
			\alpha^{-1} c' |\cR_{\ell_0}| + \alpha^{-1} |\cR_{\ell_0+1}| \leq c' (\varepsilon^2 \alpha^{-1} + \eps^\kappa + \alpha^\kappa) 
			\leq c'(\alpha^{\tfrac{\nu}{1-\nu}} + \alpha^{\tfrac{\kappa}{2(1-\nu)}}) \leq c'.
		\end{equation}
		Now define the auxiliary function $\psi \triangleq \exp(-\sqrt{2} \widetilde{f}) > 0$. From the chain rule, \eqref{eq:bootstrapped.claim.iii}, and \eqref{eq:bootstrapped.claim.iv}, we have
		\begin{align} \label{eq:psi.differential.inequality}
			\widetilde{\cL} \psi & = - \sqrt{2} (\widetilde{\cL} \widetilde{f}) \psi + 2 |\widetilde{\nabla} \widetilde{f}|^2 \psi \nonumber \\
			& \geq - \frac{8\sqrt{2} (\expansioncoeff)^{2}}{\energyunit} \psi^2 - c' \psi  + \sqrt{2} \varepsilon \alpha^{-1} (|\sff_{\Gamma_{\ell_0}}|^2 + \ricc_g(\nu, \nu)|_{\Gamma_{\ell_0}}) f_{\ell_0,\ell_0+1} \psi \nonumber \\ 
			& \qquad + \alpha^{-1} (\sqrt{2} \varepsilon  \cQ (f_{\ell_0,\ell_0+1}) + |\nabla_{\Gamma_{\ell_0}} f_{\ell_0,\ell_0+1}|^2) \psi \nonumber \\
			& \geq - \frac{8\sqrt{2} (\expansioncoeff)^{2}}{\energyunit} \psi^2 - c' \psi  + \sqrt{2} \varepsilon \alpha^{-1} (|\sff_{\Gamma_{\ell_0}}|^2 + \ricc_g(\nu, \nu)|_{\Gamma_{\ell_0}}) f_{\ell_0,\ell_0+1} \psi \nonumber \\ 
			& \qquad - \alpha^{-1} \left[ \sqrt{2} \varepsilon  \cQ (f_{\ell_0,\ell_0+1}) + |\nabla_{\Gamma_{\ell_0}} f_{\ell_0,\ell_0+1}|^2 \right]_- \psi
		\end{align}
		on $B_{K_\nu}^2(0)$. Here, $[\cdot]_-$ denotes the negative part of a real number (and is a nonnegative quantity). Using a logarithmic cutoff function in \eqref{eq:toda.stability.estimate}, which is $1$ on $B_{\eps \alpha^{-1/2}\sqrt{K_\nu}}^2(0)$ and $0$ outside $B_{\eps \alpha^{-1/2} K_\nu/2}^2(0)$, we get
		\begin{equation} \label{eq:bootstrapped.claim.L1.bound}
			\int_{B_{\sqrt{K_\nu}}^2(0)} \psi \leq c' (\log K_\nu)^{-1} + c'  \alpha^{\tfrac{\kappa+2\nu-1}{2(1-\nu)}} K_\nu^2
		\end{equation}
		in the scale of $\psi$. By Moser's weak maximum principle on $B_1$ for \eqref{eq:psi.differential.inequality} (see, e.g.\ \cite[Theorem 4.1]{HanLin}), the $L^1$ bound in  \eqref{eq:bootstrapped.claim.L1.bound} implies the $L^\infty$ bound
		\begin{equation} \label{eq:bootstrapped.claim.sup.bound}
			\psi(0) \leq C_\star  \int_{B_1^2(0)} \psi \leq C_\star \left( (\log K_\nu)^{-1} + \alpha^{\tfrac{\kappa+2\nu-1}{2(1-\nu)}} K_\nu^2 \right),
		\end{equation}
		for a constant $C_\star$ that depends on the constants in \eqref{eq:sheets.constants} and the $L^\infty$ norm of the coefficients in the  differential inequality \eqref{eq:psi.differential.inequality} on $B_1^2(0)$. We're assuming that \eqref{eq:bootstrapped.claim.ii} fails at $y$, so together with \eqref{eq:sheets.sff.bound}, \eqref{eq:sheets.graph.apriori.C1.bounds}, \eqref{eq:sheets.graph.apriori.C2.bounds}, and Lemma \ref{lemm:mean.curvature.laplacian}, we have
		\begin{align}
			& \sup_{\widetilde{y} \in B_1^2(0)} | \varepsilon \alpha^{-1} (|\sff_{\Gamma_{\ell_0}}|^2 + \ricc_g(\nu, \nu)|_{\Gamma_{\ell_0}}) f_{\ell_0,\ell_0+1}(y + \varepsilon \alpha^{-\tfrac{1}{2}} \widetilde{y})| \nonumber \\
			& \qquad \leq c' \varepsilon \alpha^{-1} (|f_{\ell_0,\ell_0+1}(y)| + \operatorname{osc} \{ f_{\ell_0,\ell_0+1} : \Gamma_{\ell_0} \cap C_{\varepsilon \alpha^{-1/2}}(\Pi(y)) \})  \leq c' \varepsilon^2 \alpha^{-\tfrac{3}{2}} \leq c' \alpha^{\tfrac{3\nu-1}{2(1-\nu)}}. \label{eq:bootstrapped.claim.osc.bound}
		\end{align}
		Likewise, using Lemma \ref{lemm:mean.curvature.laplacian}, we can estimate
		\[ \sqrt{2} \varepsilon |\cQ (f_{\ell_0,\ell_0+1})| \leq c' \varepsilon (|f_{\ell_0,\ell_0+1}|^2 + |\nabla_{\Gamma_{\ell_0}} f_{\ell_0,\ell_0+1}|^2). \]
		By absorbing the gradient term and estimating $f_{\ell_0,\ell_0+1}$ with the same argument as in \eqref{eq:bootstrapped.claim.osc.bound}, we also estimate
		\begin{equation} \label{eq:bootstrapped.claim.quad.error.bound}
			\alpha^{-1} \left[ 2 |\nabla_{\Gamma_{\ell_0}} f_{\ell_0,\ell_0+1}|^2 + \sqrt{2} \varepsilon \cQ f_{\ell_0,\ell_0+1}(y + \varepsilon \alpha^{-\tfrac{1}{2}} \widetilde{y}) \right]_- \leq c' \eps \alpha^{-1} f_{\ell_0,\ell_0+1}^2 \leq  c' \eps^3 \alpha^{-2} \leq c' \alpha^{\tfrac{4\nu-1}{2(1-\nu)}}. 
		\end{equation}
		Thus, ignoring the unimportant dependencies on \eqref{eq:sheets.constants}, we have
		\begin{equation} \label{eq:bootstrapped.claim.constant.dependence}
			C_\star =  C_\star(1 + \alpha^{\tfrac{\kappa+2\nu-1}{2(1-\nu)}} +    \alpha^{\tfrac{3\nu-1}{2(1-\nu)}} + \alpha^{\tfrac{4\nu-1}{2(1-\nu)}}),
		\end{equation}
		which, as long as $\nu > \max \{ \tfrac 1 3, \tfrac{1-\kappa}{2} \}$, can be taken to be uniformly bounded independently of $\alpha$---though certainly depending on the constants in \eqref{eq:sheets.constants}---since $\alpha \leq 1$ by definition.
		
		Since $C_\star$ is uniformly bounded per \eqref{eq:bootstrapped.claim.constant.dependence}, it follows from \eqref{eq:bootstrapped.claim.sup.bound} that by choosing suitably large $K_\nu = K_\nu(c_0, E_0, \eta, \beta, \nu) > 0$, $\psi(0)$ will become less than $\tfrac{1}{2}$ for small $\alpha$, contradicting our assumption that \eqref{eq:bootstrapped.claim.ii} is violated. Specifically, recalling \eqref{eq:bootstrapped.claim.sup.bound}, we may simply pick $K_\nu$ large enough that $C_\star (\log K_\nu)^{-1} < \tfrac14$, in which case $\psi(0) < \tfrac12$ as long as $\alpha$ is small enough that $C_\star \alpha^{\tfrac{\kappa+2\nu-1}{2(1-\nu)}} K_\nu^2 < \tfrac14$.

		This concludes the proof of Proposition \ref{prop:bootstrapped.stable.estimates} for
		\[ \nu \in (\nu_0, \tfrac{1}{2}), \text{ where } \nu_0 =  \min \{ \tfrac{1}{3}, \tfrac{1-\kappa}{2} \}. \]
		The next step is to show that $\nu_0$ can be taken to be arbitrarily small, at the expense of possibly having to rescale our domain a finite number of times.
		
		Retracing the proof above, it's not hard to see that what one needs to improve are:
		\begin{enumerate}
			\item the exponent of $\alpha$ in \eqref{eq:bootstrapped.claim.L1.bound}, \eqref{eq:bootstrapped.claim.sup.bound}, and
			\item the oscillation bounds in \eqref{eq:bootstrapped.claim.osc.bound}, \eqref{eq:bootstrapped.claim.quad.error.bound}.
		\end{enumerate}
	
		For the prior, we may use  \eqref{eq:toda.stability.estimate.sharper}-\eqref{eq:toda.stability.estimate.sharper.coefficient} instead of \eqref{eq:toda.stability.estimate}; we get
		\[ \psi(0) \leq C_\star \left( (\log K_\nu)^{-1} + (\alpha^{\tfrac{\nu}{1-\nu}} + \alpha^\kappa) K_\nu^2 \right), \]
		a sufficient bound.
		
		For the latter, we need to use a Harnack-type inequality on the elliptic \emph{equation} satisfied by $f_{\ell_0,\ell_0+1}$, \eqref{eq:bootstrapped.vii}. Recalling \eqref{eq:bootstrapped.error.i}, and using the fact that we now know Proposition \ref{prop:bootstrapped.stable.estimates} to hold for $\nu' \in (\nu_0, \tfrac{1}{2})$, we see that the right hand side of \eqref{eq:bootstrapped.vii} can be bounded in $L^\infty$ by
		\[ c' \varepsilon^2 + c' \sum_{m=1}^Q A_m(|y| + 2K \varepsilon |\log \varepsilon|) \leq c' \varepsilon^{2(1-\nu')} \]
		for some $\nu' \in (\nu_0, \tfrac{1}{2})$. Diving \eqref{eq:bootstrapped.vii} through by $\varepsilon$, we thus get a uniformly elliptic equation
		\begin{equation} \label{eq:bootstrapped.ix}
			(\cL + |\sff_{\Gamma_\ell}|^2 + \ricc_g(\nu,\nu)|_{\Gamma_\ell} + \cQ) f_{\ell,\ell+1}(y) = O(\varepsilon^{1-2\nu'}).
		\end{equation}
		Now we apply the inhomogeneous Harnack-type inequality found in \cite[Theorems 8.17, 8.18]{GilbargTrudinger01} to \eqref{eq:bootstrapped.vii}, multiplied through by the $a(y)$ in \eqref{eq:mean.curvature.laplacian.operator.a}, with some $q > 2$, $R = \varepsilon \alpha^{-\tfrac{1}{2}}$, and $g = O(\varepsilon^{1-2\nu'})$ (in the $L^\infty$ sense) and we get
		\begin{equation*}
			\sup \left\{ f_{\ell_0,\ell_0+1} : \Gamma_{\ell_0} \cap C_{\varepsilon \alpha^{-1/2}}(\Pi(y)) \right\} \leq c'\Big( f_{\ell_0,\ell_0+1}(y) + \varepsilon^{2} \alpha^{-1} \cdot \varepsilon^{1-2\nu'} \Big) = c' \Big( f_{\ell_0,\ell_0+1}(y) + \eps^{3-2\nu'} \alpha^{-1} \Big).
		\end{equation*}
		Recall that we are assuming, by contradiction, that \eqref{eq:bootstrapped.claim.ii} is violated at our $y$, implying that $f_{\ell_0,\ell_0+1}(y)$ is an error term relative to the last term of the right hand side. Together with \eqref{eq:bootstrapped.claim.i}, this gives
		\begin{align}
			\sup \left\{ \varepsilon \alpha^{-1} f_{\ell_0,\ell_0+1} : \Gamma_{\ell_0} \cap C_{\varepsilon \alpha^{-1/2}}(\Pi(y)) \right\} & \leq c'  \varepsilon^{4-2\nu'} \alpha^{-2} \leq c'  \alpha^{\tfrac{2-\nu'}{1-\nu} - 2} = c' \alpha^{\tfrac{2\nu-\nu'}{1-\nu}}. \label{eq:bootstrapped.x} 
		\end{align}
		This is $\leq c' \alpha^\delta$ for some $\delta > 0$ as long as $\nu > \nu_0' \triangleq \tfrac{1}{2} \nu'$. This gives the improved oscillation bound that we sought in place of \eqref{eq:bootstrapped.claim.osc.bound}, and Proposition \ref{prop:bootstrapped.stable.estimates} follows in full by iteratively pushing $\nu$, $\nu_0'$ down to zero. 
\end{proof}

\begin{prop}[Stable sheet distances, II]  \label{prop:ultimate.stable.estimates}
	If $u$ is as in Proposition \ref{prop:bootstrapped.stable.estimates}, then
	\[ \lim_{\eps \to 0} \frac{\exp(-\sqrt{2} \eps^{-1} D_\ell)}{\eps^2 |\log \eps|} = 0 \text{ on } \Gamma_\ell(\tfrac{1}{6}) \]
	for all $\ell \in \{ 1, \ldots, Q \}$, uniformly in terms of $c_0$, $E_0$, $\eta$, $\beta$.
\end{prop}
\begin{proof}
	The proof follows along the same lines as the bootstrap portion of the proof of Proposition \ref{prop:bootstrapped.stable.estimates}. However, the modifications are somewhat delicate so we give the argument here. 
	
	We first prove a weaker bound. We argue by contradiction, assuming that there exists $\ell \in \{1, \ldots, Q\}$ such that
	\begin{equation} \label{eq:ultimate.i}
		A_{\ell}(r) \geq A_{\ell}(1/5) > \varepsilon^2 |\log \varepsilon|^2 \text{ for all } r \in [\tfrac{1}{5}, \tfrac{1}{4}].
	\end{equation}
	Let $r \in [\tfrac{1}{5}, \tfrac{1}{4}]$, and then let $\alpha \triangleq \max \{ A_{\ell}(r) : \ell \in \{1, \ldots, Q\} \}$. Then
	\begin{equation} \label{eq:ultimate.ii}
		\alpha > \varepsilon^2 |\log \varepsilon|^2.
	\end{equation}
	We claim that
	\begin{equation} \label{eq:ultimate.iii}
		\max_{\ell \in \{1, \ldots, Q\}} A_{\ell}(r - \varepsilon K_0 \alpha^{-\tfrac{1}{2}}) < \tfrac{1}{2} \alpha
	\end{equation}
	for a constant $K_0 = K_0(c_0, E_0, \eta, \beta) > 0$. 
	
	Suppose that \eqref{eq:ultimate.iii} fails for $\ell_0 \in \{1, \ldots, Q\}$ and $y \in \Gamma_{\ell_0}(\varepsilon K_0 \alpha^{-\tfrac{1}{2}})$. Define
	\begin{equation} \label{eq:ultimate.iv}
		\widetilde{f}(\widetilde{y}) \triangleq \varepsilon^{-1} f_{\ell_0,\ell_0+1}(y + \varepsilon \alpha^{-\tfrac{1}{2}}\widetilde{y}) - \tfrac{1}{\sqrt{2}} |\log \alpha|,
	\end{equation}
	for $\widetilde{y} \in B_{K_0/2}^2(0)$. Proceeding as in \eqref{eq:bootstrapped.claim.iii}, we find that
	\begin{align}
		& \widetilde{\cL} \widetilde{f}(\widetilde{y}) \label{eq:ultimate.v}   \leq \frac{8\expansioncoeff}{\energyunit} \exp(-\sqrt{2} \widetilde{f}(\widetilde{y}))  + \alpha^{-1} c' |\cR_{\ell_0}(y+\varepsilon \alpha^{-\tfrac{1}{2}} \widetilde{y})| + \alpha^{-1} |\cR_{\ell_0+1}((y+\varepsilon \alpha^{-\tfrac{1}{2}}\widetilde{y})')| \nonumber \\
		& \qquad \qquad - \varepsilon \alpha^{-1} (|\sff_{\Gamma_{\ell_0}}|^2 + \ricc_g(\nu, \nu)|_{\Gamma_{\ell_0}} + \cQ) f_{\ell_0,\ell_0+1}(y + \varepsilon \alpha^{-\tfrac{1}{2}} \widetilde{y}). \nonumber
	\end{align}
	We also still have an estimate of the form
	\begin{equation} \label{eq:ultimate.vi}
		\alpha^{-1} c' |\mathcal{R}_{\ell_0}| + \alpha^{-1} |\cR_{\ell_0+1}| \leq c',
	\end{equation}
	and the function $\psi \triangleq \exp(-\sqrt{2} \widetilde{f})$	still satisfies a differential inequality of the form
	\begin{equation} \label{eq:ultimate.vii}
		\widetilde{\cL} \psi \geq - \frac{8\sqrt{2} (\expansioncoeff)^{2}}{\energyunit} \psi^2 - c' \psi 
		+ \sqrt{2} \varepsilon \alpha^{-1}(|\sff_{\Gamma_{\ell_0}}|^2 + \ricc_g(\nu, \nu)|_{\Gamma_{\ell_0}} + \cQ) f_{\ell_0,\ell_0+1}(y+\varepsilon \alpha^{-\tfrac{1}{2}} \widetilde{y}) \psi.
	\end{equation}
	Applying the same inhomogeneous Harnack-type inequality that led to \eqref{eq:bootstrapped.x} before, we obtain
	\begin{align*}
		\sup \left\{ f_{\ell_0,\ell_0+1} : \Gamma_{\ell_0} \cap C_{\varepsilon \alpha^{-1/2}}(\Pi(y)) \right\} & \leq c' \Big( f_{\ell_0,\ell_0+1}(y) + \varepsilon^{2} \alpha^{-1}  \cdot \varepsilon^{-1}(\varepsilon^2 + \alpha) \Big) 
		 \leq c' \Big( \varepsilon |\log \alpha| + \varepsilon^3 \alpha^{-1} + \eps \Big).
	\end{align*}
	Thus, we have the following $L^\infty$ estimate on the coefficient in front of $\psi$ in the last term of \eqref{eq:ultimate.vii} on the domain $B_{\varepsilon \alpha^{-1/2}}^2(y)$:
	\begin{equation} \label{eq:ultimate.viii}
		\sup \left\{ \varepsilon \alpha^{-1} f_{\ell_0,\ell_0+1} : \Gamma_{\ell_0} \cap C_{\varepsilon \alpha^{-1/2}}(\Pi(y)) \right\}
		\leq c' \Big( \varepsilon^2 \alpha^{-1} |\log \alpha| +  \varepsilon^4 \alpha^{-2} + \varepsilon^{2} \alpha^{-1} \Big) \leq c',
	\end{equation}
	where we've used the simple fact that
	\begin{equation} \label{eq:ultimate.aux.fact}
		\eqref{eq:ultimate.ii} \iff \alpha > \varepsilon^2 |\log \varepsilon|^2 \implies \varepsilon^2 \alpha^{-1} |\log \alpha| = o(1).
	\end{equation}
	Thus, \eqref{eq:ultimate.vii} implies the uniformly elliptic partial differential inequality
	\begin{equation} \label{eq:ultimate.ix}
		\widetilde{\cL} \psi \geq - \frac{8\sqrt{2} (\expansioncoeff)^{2}}{\energyunit} \psi^2 - c' \psi.
	\end{equation}
	From Moser's weak maximum principle (see, e.g., \cite[Theorem 4.1]{HanLin}) applied to \eqref{eq:ultimate.ix} on $B_1^2(0)$, combined with \eqref{eq:toda.stability.estimate.sharper.coefficient}, we get the $L^\infty$ bound
	\[ \psi(0) \leq c' \int_{B_1^2(0)} \psi \leq c'\Big( (\log K_0)^{-1} + (\eps^2 \alpha^{-1} + \alpha^\kappa) K_0^2 \Big) \leq c' \Big( (\log K_0)^{-1} + (o(|\log \alpha|^{-1}) + \alpha^\kappa) K_0^2 \Big), \]
	violating the assumption that  \eqref{eq:ultimate.iii} fails, provided we take $K_0$ large and $\alpha$ small.
	
	Thus, \eqref{eq:ultimate.iii} holds true with a fixed $K_0$. Notice, then, that
	\[ \varepsilon K_0 \alpha^{-\tfrac{1}{2}} \leq K_0 |\log \varepsilon|^{-1}, \]
	A backward iteration of \eqref{eq:ultimate.iii} from $r = \tfrac{1}{4}$ to $r = \tfrac{1}{5}$, followed by an application of Proposition \ref{prop:bootstrapped.stable.estimates} at radius $r = 1/4$ with $\nu < \tfrac{\log 2}{20K_0}$
	yields
	\begin{align*}
		\log A_{\ell_0}(\tfrac{1}{5}) & \leq \log A_{\ell_0}(\tfrac{1}{4}) - \tfrac{\log 2}{20K_0} |\log \varepsilon| \leq 2(\nu-1) |\log \varepsilon| - \tfrac{\log 2}{20K_0} |\log \varepsilon|  < -2 |\log \varepsilon| = \log \varepsilon^2,
	\end{align*}
	violating \eqref{eq:ultimate.i}.
	
	We now prove the main claim. We argue by contradiction again assuming that there exists $\ell \in \{1, \ldots, Q\}$ such that
	\begin{equation} \label{eq:ultimate.x}
		A_{\ell}(r) \geq A_{\ell}(\tfrac{1}{5}) > \mu \varepsilon^2 |\log \varepsilon| \text{ for all } r \in  [\tfrac{1}{6}, \tfrac{1}{5}],
	\end{equation}
 	for some $\mu > 0$. Let $r \in [\tfrac{1}{6}, \tfrac{1}{5}]$, $\alpha \triangleq \max \{ A_{\ell}(r) : \ell \in \{ 1, \ldots, Q \}\}$. Then
	\begin{equation} \label{eq:ultimate.xi}
		\alpha >   \mu \varepsilon^2 |\log \varepsilon|.
	\end{equation}
	We claim that
	\begin{equation} \label{eq:ultimate.xii}
		A_{\ell}(r - \varepsilon K_0' \alpha^{-\tfrac{1}{2}}) < \tfrac{1}{2} \alpha \text{ for every } \ell \in \{1, \ldots, Q\},
	\end{equation}
	for a constant $K_0' = K_0'(c_0, E_0, \eta, \beta) > 0$. This indeed follows from the same argument as above, modulo the fact that one needs to replace \eqref{eq:ultimate.aux.fact} with
	\[ \eqref{eq:ultimate.xi} \iff \alpha >  \mu \varepsilon^2 |\log \varepsilon| \implies \varepsilon^2 \alpha^{-1} |\log \alpha| \leq  \mu^{-1} (2 + o(1)). \]
	Notice, now, that
	\[ \varepsilon K_0' \alpha^{-\tfrac{1}{2}} \leq  \mu^{-\tfrac{1}{2}} K_0' |\log \varepsilon|^{-\tfrac{1}{2}}, \]
	so that a backward iteration of \eqref{eq:ultimate.xii} from $r = 1/5$ to $r = 1/6$, together with the weaker assertion verified above, yields
	\[ \log A_{\ell_0}(\tfrac{1}{6})  \leq \log A_{\ell_0}(\tfrac{1}{5}) - \mu^{\tfrac{1}{2}} \tfrac{\log 2}{20K_0'} |\log \varepsilon|^{\tfrac{1}{2}} \leq -2 |\log \varepsilon| + 2 \log |\log \varepsilon| - \mu^{\tfrac{1}{2}}  \tfrac{\log 2}{20K_0'} |\log \varepsilon|^{\tfrac{1}{2}}. \]
	However,
	\[ \lim_{\eps \to 0} \left( \log |\log \varepsilon| - \mu^{\tfrac{1}{2}} \tfrac{\log 2}{20 K_0'} |\log \varepsilon|^{\tfrac{1}{2}} \right) = -\infty, \]
	so, for sufficiently small $\eps$ (depending on $K_0$, $\mu$), this quantity is $< \log \mu$. Thus, for small $\eps$,
	\[ \log A_{\ell_0}(\tfrac{1}{6}) \leq \log \mu - 2 |\log \eps| + \log |\log \eps| = \log (\mu \eps^2 |\log \eps|), \]
	which violates \eqref{eq:ultimate.x}. The result follows.
\end{proof}

In fact, Proposition  \ref{prop:ultimate.stable.estimates} and \eqref{eq:bootstrapped.v}-\eqref{eq:bootstrapped.error.i} establish the following:

\begin{coro} \label{coro:ultimate.stable.estimates}
	If $u$ is as in Proposition  \ref{prop:ultimate.stable.estimates}, then
	\[ \frac{H_{\Gamma_\ell}}{\varepsilon |\log \varepsilon|} \to 0 \text{ uniformly on } \Gamma_\ell(\tfrac{1}{6}) \]
	as $\varepsilon \to 0$, for all $\ell \in \{1, \ldots, Q\}$.
\end{coro}

This estimate is key for our geometric applications, since it says that the mean curvature of the zero sets $u$ dominates the effect of interactions between the sheets. This will allow us to treat the sheets (essentially) like disjoint minimal surfaces.

\subsection{Curvature estimates} \label{subsec:stable.curvature.estimates}

In what follows, we let $B^n_r(0)$ be a smooth $n$-ball equipped with a Riemannian metric $g$ so that $B_{r}^{n}(0)$ is a geodesic $r$-ball centered at $0$ (with respect to $g$).

\begin{theo} \label{theo:curvature.estimate}
	Suppose $\inj_g \geq 3$ and $|\riem_g| + |\nabla_g \riem_g| \leq 1$ on $B^3_1(0)$. If $\varepsilon \leq \varepsilon_1$, $u \in C^\infty(B^3_1(0); (-1,1))$ is a stable critical point of $E_\varepsilon \restr B^3_1(0)$, and $(E_\varepsilon \restr B^3_1(0))[u] \leq E_0$, then 
	\[ |\cA(x)| \leq c_1 \text{ for all } x \in B^3_{1/2}(0) \cap \{ |u| \leq 1-\beta \}, \]
	where $\varepsilon_1 = \varepsilon_1(n, E_0, \beta, W)$, $c_1 = c_1(n, E_0, \beta, W)$.
\end{theo}

\begin{rema}
We emphasize that, in one dimension lower, Wang--Wei \cite{WangWei} have proven that stable critical points of $E_{\eps}$ satisfy curvature bounds even without the assumption of uniformly bounded energy. 
\end{rema}

\begin{rema}
	It's not immediately obvious that the enhanced second fundamental form $\cA$ is well-defined on $B^3_{3/4}(0) \cap \{ |u| \leq 1-\beta \}$. This can be seen, for instance, by applying the following proposition with $n=3$. Its ``nonexistence'' condition, when $n=3$, is guaranteed in view of the work of Ambrosio--Cabr\'e \cite{AmbrosioCabre00} (see also the work of Farina--Mari--Valdinoci \cite{FarinaMariValdinoci13}).
\end{rema}

\begin{prop} \label{prop:lower.density.estimate}
	Let $u : B^n_1(0) \to (-1,1)$ be a stable critical point of $E_\varepsilon \restr B^n_1(0)$ with $(E_\varepsilon \restr B^n_1(0))[u] \leq E_0$. If $\varepsilon \leq \varepsilon_0$ and $\RR^n$ with the standard metric does not carry any nontrivial (i.e., heteroclinic or $\pm 1$) entire stable solutions with Euclidean energy growth then
	\[ \varepsilon |\nabla u_i| \geq c_0^{-1} > 0 \text{ for all } x \in B^n_{3/4}(0) \cap \{ |u| \leq 1-\beta \}, \]
	where $\varepsilon_0$, $c_0$ depend on $E_0$, $\beta$, $W$.
\end{prop}
\begin{proof}
	We argue by contradiction. If the assertion were false, there would exist a sequence
	\[ \{ (u_i, \varepsilon_i) \}_{i=1,2,\ldots} \subset C^\infty(B^n_1(0); (-1,1)) \times (0,\infty), \; \lim_i \varepsilon_i = 0, \]
	where each $u_i : B^n_1(0) \to [-1,1]$ is a stable critical point for $E_{\varepsilon_i} \restr B^n_1(0)$, with $(E_{\varepsilon_i} \restr B^n_1(0))[u_i] \leq E_0$, and so that $\lim_i \varepsilon_i \nabla u_i(q_i) = 0$ along some $\{ q_i \}_{i=1,2,\ldots} \subset B^n_{3/4}(0)$. The rescaled functions
	\[ v_i(x) \triangleq u_i(\varepsilon_i(x - q_i)) \]
	are all stable critical points of $E_1 \restr B^n_{(1-|q_i|)/\varepsilon_i(0)}$ with Euclidean energy growth. Since the ellipticity constants are uniform at this scale, we may pass to a subsequence with $\lim_i v_i = v_\infty$ in $C^\infty_{\loc}(\RR^n)$, where $v_\infty$ is a stable critical point of $E_1 \restr \RR^n$ with Euclidean area growth, $|v_\infty(0)| \leq 1-\beta$, and $\nabla v_\infty(0) = 0$. No such $v_\infty$ exists: the only entire stable solutions on $\RR^n$ with Euclidean energy growth are the constants $\pm 1$ and the one-dimensional heteroclinic solution.
\end{proof}

We are now in a position to prove Theorem \ref{theo:curvature.estimate}.

\begin{proof}[Proof of Theorem \ref{theo:curvature.estimate}]
	If the assertion were false, there would exist a sequence
	\[ \{ (u_i, \varepsilon_i) \}_{i=1,2,\ldots} \subset C^\infty(B_1^3(0); (-1,1)) \times (0,\infty), \; \lim_i \varepsilon_i = 0, \] 
	where each $u_i : B_1^3(0) \to (-1,1)$ is a stable critical point for $E_{\varepsilon_i} \restr B_1^3(0)$, with $(E_{\varepsilon_i} \restr B_1^3(0))[u_i] \leq E_0$, and so that the maximum value 
	\[ \max \left\{ \dist(x, \RR^3 \setminus B_{3/4}^3(0)) |\cA(x)| : x \in B^3_1(0) \cap \{ |u| \leq 1-\beta \} \right\} \]
	is attained at some $q_i \in B_{3/4}^3(0)$ with
	\[ \lim_i \dist(q_i, \partial B_{3/4}^3(0)) |\cA_i(q_i)| = \infty. \]
	Next, let $\lambda_i \triangleq |\cA_i(q_i)|$, which we note also satisfies $\lim_i \lambda_i = \infty$.
	
	\begin{clai}
		$\liminf_i \varepsilon_i \lambda_i = 0$.
	\end{clai}
	\begin{proof}[Proof of claim]
		Rescale to $v_i(x) \triangleq u_i(\varepsilon_i(x - q_i))$, a stable critical point of $E_1 \restr B^3_{(1-|q_i|)/\varepsilon_i}(0)$ with quadratic energy growth and $|v_i| \leq 1-\beta$. Since our ellipticity constants are uniform at this scale, we can pass to a subsequence such that $\lim_i v_i = v_\infty$ in $C^\infty_{\loc}(\RR^3)$, where $v_\infty$ is a stable critical point of $E_1 \restr \RR^3$ with $|v_\infty(0)| \leq 1-\beta$. The only such $v_\infty$ is the one-dimensional heteroclinic solution, for which $\cA_\infty \equiv 0$, and, therefore $\liminf_i \varepsilon_i \lambda_i = |\cA_\infty(0)| = 0$. This completes the proof of the claim.
	\end{proof}

	Pass to a subsequence for which $\liminf_i \varepsilon_i \lambda_i = 0$ is attained, and rescale to $\widetilde{u}_i(x) \triangleq u_i(\lambda_i^{-1}(x-q_i))$. This is a critical point of $E_{\varepsilon_i \lambda_i} \restr B^3_{(3/4 - |q_i|) \lambda_i}(0)$. We note that
	\begin{equation} \label{eq:curvature.estimate.eps.lambda.limit}
		\lim_i \varepsilon_i \lambda_i = 0, \; \lim_i (3/4 - |q_i|) \lambda_i = \infty.
	\end{equation}
	Moreover, for every $R \geq 1$, 
	\begin{equation} \label{eq:curvature.estimate.energy.growth}
		(E_{\eps_i \lambda_i} \restr B^3_{(3/4 - |q_i|) \lambda_i}(0))(B^3_R(0)) \leq cR^2
	\end{equation}
	for all sufficiently large $i$. Here, $c > 0$ is fixed. Combining \eqref{eq:curvature.estimate.eps.lambda.limit}, \eqref{eq:curvature.estimate.energy.growth}, together with the works of \cite[Theorem 1]{HutchinsonTonegawa00} and  \cite[Appendix B]{Guaraco} for the Riemannian modifications, the $2$-varifolds
	\begin{equation*} 
		V_{\eps_i \lambda_i}[\widetilde{u}_i](\varphi) \triangleq \int \varphi(x, T_x \{ \widetilde{u}_i = \widetilde{u}_i(x) \}) \, \eps_i \lambda_i |\nabla \widetilde{u}_i(x)|^2, 
		 \text{ for }  \varphi \in C^0_c(\operatorname{Gr}_2(B_{(3/4-|q_i|)\lambda_i}^3(0))),
	\end{equation*}
	converge weakly to a stationary integral varifold $V_\infty \in \mathbf{I}_2(\mathbf{R}^3)$. The enhanced second fundamental form estimates, moreover, imply that $\support \Vert V_\infty \Vert$ is $C^{1,1}$ and, therefore, a smooth minimal surface. 
	
	\begin{rema} \label{rema:HT.simplified}
		We note that the most technical elements of \cite{HutchinsonTonegawa00}, such as proving that the limit varifold is integral, can be proven (in the setting at hand) in a much more direct manner given the curvature estimates we now know to be true.
	\end{rema}
	
	The stability of $\widetilde{u}_i$ is also known to imply stability of $\support \Vert V_\infty \Vert$. Indeed, one may plug $\zeta = \psi |\nabla \widetilde{u}_i|$, $\psi \in C^\infty_c(\RR^3)$, into the second variation operator $\delta^2 E_{\eps_i \lambda_i}|_{\widetilde{u}_i}$ and let $i \to \infty$, and recover the second variation operator for $\support \Vert V_\infty \Vert$ with $\psi|_{\support \Vert V_\infty \Vert}$ being the test function. See also \cite{Tonegawa05}.
	
	Summarizing, $\support \Vert V_\infty \Vert$ is a smooth, stable, embedded minimal surface in $\RR^3$. (In fact, with quadratic area growth.) Therefore, the limit is a a disjoint union of planes $P_1, \ldots, P_k \subset \RR^3$ with integer multiplicities $m_1, \ldots, m_k \in \{1, 2, \ldots\}$. Without loss of generality, $P_j = \RR^2 \times \{z_j\}$, with $0 = z_1 < z_2 < \ldots < z_k$. 
	
	We will only need to focus on one of these planes, e.g., $P_1$. Writing
	\[ \{ \widetilde{u}_i = 0 \} \cap (B_1^2(0) \times [-z_2/2,z_2/2]) = \bigcup_{\ell=1}^{m_1} \graph f_{i,\ell}, \]
	it follows from our rescaling that $f_{i,\ell} : B_1^2(0) \to \RR$ all converge, in the $C^{1,\alpha}$ sense on $B_{1/2}^2(0)$, to the zero function as $i \to \infty$. In fact, by dilating as needed, we find ourselves in the setup of Sections \ref{subsec:jacobi.toda.setup}-\ref{subsec:distance.estimates}.
	
	Therefore, by employing Proposition  \ref{prop:ultimate.stable.estimates} (in fact, Proposition \ref{prop:bootstrapped.stable.estimates} suffices), it follows from \eqref{eq:enhanced.sff.estimate} that
	\begin{equation} \label{eq:curvature.estimate.i}
		\limsup_{i \to \infty} \Vert \widetilde{\cA}_i \Vert_{C^0(\sM_\ell(1/6) \cap \{ |\widetilde{u}_i| \leq 1-\beta\})} \leq c' \limsup_{i \to \infty} \widehat{\Gamma}_\ell(1/6),
	\end{equation}
	for all $\ell \in \{1, \ldots, Q\}$, where $\widehat{\Gamma}_\ell$ is as in \eqref{eq:christoffel.symbol.sup}. 

	\begin{clai}
		The right hand side of \eqref{eq:curvature.estimate.i} is zero.
	\end{clai}

	Notice that this claim violates the fact that our dilations were such that $|\widetilde{\cA}_i(0)| = 1$ for all $i = 1, 2, \ldots$, and Theorem \ref{theo:curvature.estimate} follows.
	
	\begin{proof}[Proof of claim]
		From the Riccati equation, \eqref{eq:mean.curv.ddt.sff}, it suffices to check that the second fundamental form of $\{ |\widetilde{u}_i| = 0 \}$ converges to zero. This follows from our H\"older estimate on the mean curvatures from \eqref{eq:phi.improved.c2a.estimate.full}, Lemma \ref{lemm:h.phi.comparison.improved}, and the fact that our graphing functions converge to zero in $C^1$.
	\end{proof}

	This concludes the proof of the curvature estimates.
\end{proof}

\begin{coro} \label{coro:curvature.estimates}
	Let $(M, g)$, $u$, $\varepsilon$, $\varepsilon_1$ be as in Theorem \ref{theo:curvature.estimate}, and $\theta \in (0, 1)$. Then, 
	\[ [\sff_{\{u=t\}}]_{\theta, \{u=t\} \cap B^3_{1/3}(0)} \leq c_1' \text{ for all } |t| \leq 1-\beta, \]
	where $c_1' = c_1'(n, E_0, \beta, \theta, W)$.
\end{coro}
\begin{proof}
	\eqref{eq:bootstrapped.iii}, \eqref{eq:bootstrapped.iv}, Lemma \ref{lemm:h.phi.comparison.improved}, and Proposition \ref{prop:ultimate.stable.estimates} together give $C^\theta$ bounds on the mean curvatures of $\{ u = 0 \}$. The improvement to $C^{2,\theta}$ bounds on the level sets comes from (quasilinear) Schauder theory and Theorem \ref{theo:curvature.estimate}.
\end{proof}

\section{Phase transitions with bounded Morse index ($n=3$)} 


\label{sec:bounded.index}

\subsection{Multiplicity and Jacobi fields} \label{subsec:multiplicity.jacobi.fields}

In this section we prove that uniform bounds on the Morse index generically prevent multiplicity from occurring in the Allen-Cahn setting. Specifically:

\begin{theo} \label{theo:bounded.index}
	Suppose $(M^3, g)$ is a compact Riemannian 3-manifold possibly with $\partial M \neq \emptyset$, and that $u_i \in C^\infty(M; [-1,1])$, $\eps_i > 0$, where each $u_i$ is a critical point of $E_{\varepsilon_i}$, and
	\begin{equation} \label{eq:bounded.index.i}
		E_{\varepsilon_i}[u_i] \leq E_0, \; \ind(u_i) \leq I_0 \text{ for all } i = 1, 2, \ldots
	\end{equation}
	Suppose $\lim_i \eps_i = 0$. Passing to a subsequence, write $V \triangleq \lim_i \energyunit^{-1} V_{\varepsilon_i}[u_i]$ for the limit $2$-varifold. Then $V$ is a stationary integral varifold, $\support \Vert V \Vert$ is smooth in the interior of $M$, and if $\Sigma$ denotes a connected component of $\support \Vert V \Vert$ that is a compact submanifold without boundary, then one of the following is true:
	\begin{enumerate}
		\item $\Sigma$ is two-sided and  $\Theta^2(V, \cdot) = 1$ on $\Sigma$ (i.e., $\Sigma$ has multiplicity $1$);
		\item $\Sigma$ is two-sided, $\Theta^2(V, \cdot) \geq 2$ on $\Sigma$ (i.e., multiple interfaces have converged), it is stable (see \eqref{eq:stable.min.surf}) and carries a smooth positive Jacobi field; or
		\item $\Sigma$ is one-sided, and the two-sided double cover of $\Sigma$ is stable and carries a smooth positive Jacobi field. 
	\end{enumerate}
\end{theo}
\begin{proof}
	For $p \in M$, $i = 1, 2, \ldots$, define the index concentration scale by
	\begin{equation} \label{eq:bounded.index.instability.radius}
		\cR(p, i) \triangleq \inf \{ r > 0 : \ind(u_i; B_r(p)) \geq 1 \},
	\end{equation}
	and then further let
	\[ \mathring{\Sigma} \triangleq \{ p \in M : \liminf_{i \to \infty} \cR(p, i) > 0 \}. \]
	By passing to an appropriate subsequence at the beginning of the proof, an elementary covering argument allows us to assume that $\cH^0(\Sigma \setminus \mathring{\Sigma}) \leq I_0$.

	The curvature estimates from Theorem \ref{theo:curvature.estimate} combined with the varifold convergence of $V_{\eps_{i}}[u_{i}]$ (from\footnote{See Remark \ref{rema:HT.simplified}.} \cite[Theorem 1]{HutchinsonTonegawa00}, and  \cite[Appendix B]{Guaraco}) show that along $\mathring \Sigma$, the limit varifold is supported with integer multiplicity (possibly greater than one) on a $C^{1,1}$ (and thus smooth) minimal surface. At this point, we may argue that $\Sigma$ extends smoothly across the index concentration set $\Sigma \setminus \mathring \Sigma$ exactly as in \cite[Proposition 3.10]{Guaraco}. We emphasize here that by using to our curvature estimates, we give an proof of the regularity of $\Sigma$ that does not rely on the deep works of Wickramasekera \cite{Wickramasekera14} and Tonegawa--Wickramasekera \cite{TonegawaWickramasekera12} (cf.\ \cite{Guaraco,Hiesmayr}). 
	
	We now assume that $\Sigma$ is connected (in general, one can apply the following argument to each component of the support of the limit varifold $V$). 
	
	First, suppose $\Sigma$ is two-sided and denote
	\begin{equation*}
		U \triangleq \text{ tubular neighborhood of } \Sigma \text{ such that } (\Sigma \cup \partial M) \cap U = \emptyset.
	\end{equation*}
	We may suppose that $U = Z_\Sigma(\Sigma \times (-1, 1))$.
	
	By the Constancy Theorem \cite[Theorem 41.1]{Simon83}, $\Theta^{2}(V, \cdot)$ is constant on $\Sigma$. If $\Theta^{2}(V, \cdot) = 1$ somewhere on $\Sigma$, then $\Sigma$ occurs entirely with multiplicity one as claimed. 
	
	In what follows we may assume, then, that $\Theta^{2}(V, \cdot) \equiv m \in \{2, 3, \ldots\}$ on $\Sigma$.
	
	Let us assume, for the time being, that $I_0 = 0$, i.e., that the critical points $u_i$ are all stable. The general case will be dealt with later.
	
	It follows from \eqref{eq:bounded.index.i}, Corollary \ref{coro:curvature.estimates}, and the two-sidedness of $\Sigma$, that the level sets $\{ u_i = 0 \} \cap U$ converge graphically in $C^{2,\theta}$ to $\Sigma$. In the case that $\{u_{i}=0\} \cap U$ were minimal surfaces, it is standard to produce a positive Jacobi field on $\Sigma$ out of this setup. We recall the argument here, with the necessary modifications for our lower regularity situation. 
	
	Since $\Sigma$ is two-sided, the level sets $\{ u_i = 0 \} \cap U$ (which are \emph{smoothly} embedded) can be ordered by their signed distance to $\Sigma$ in a fashion that is consistent across $\Sigma$. Without loss of generality, we may assume that there are $Q = 2$ level sets\footnote{Otherwise, we apply the same argument verbatim to the \emph{top} and \emph{bottom} level sets, ignoring all intermediate ones.}. To stay consistent with Section \ref{sec:jacobi.toda.reduction}, let's label the level sets
	\[ \Gamma_{i,1}, \Gamma_{i,2} \subset \{ u_i = 0 \} \cap U. \]
	Denote their corresponding height functions (over $\Sigma$) as $f_{i,1}$, $f_{i,2} : \Sigma \to \RR$, $\ell \in \{1,2\}$, so that $f_{i,1} < f_{i,2}$ on $\Sigma$. We recall \eqref{eq:mean.curv.graphical} from Appendix \ref{app:mean.curvature.graphs}, which tells us that:
	\begin{multline} \label{eq:bounded.index.h}
		H_{\Gamma_{i,\ell}} = - \divg_{g_{f_{i,\ell}}} \left( \frac{\nabla_{g_{f_{i,\ell}}} f_{i,\ell}}{(1 + g^{pq}_{f_{i,\ell}} (f_{i,\ell})_p (f_{i,\ell})_q)^{1/2}} \right) 
			- \frac{\sff_{f_{i,\ell}}^{pq} (f_{i,\ell})_p (f_{i,\ell})_q}{(1 + g^{pq}_{f_{i,\ell}} (f_{i,\ell})_p (f_{i,\ell})_q)^{1/2}} \\
			+ (1 + g^{pq}_{f_{i,\ell}} (f_{i,\ell})_p (f_{i,\ell})_q)^{1/2} H_{f_{i,\ell}} 
	\end{multline} 
	for $\ell = 1$, $2$. Here we're using notation from the appendix, where, e.g., $g = g_z + dz^2$ on $U$.
	
	We now claim that $H_{\Gamma_{i,2}} - H_{\Gamma_{i,1}}$ satisfies a \emph{linear} uniformly elliptic equation in $f_{i,2} - f_{i,1}$, whose parameters (obviously) depend on $f_{i,1}$, $f_{i,2}$. Indeed, \eqref{eq:bounded.index.h} tells us that
	\begin{equation} \label{eq:bounded.index.h.abstract.pde}
		H_{\Gamma_{i,\ell}} = - A(f_{i,\ell}) \divg_{\Sigma} \left( \sB(f_{i,\ell}, \nabla_{\Sigma} f_{i,\ell}) \nabla_{\Sigma} f_{i,\ell} \right) + C(f_{i,\ell}, \nabla_\Sigma f_{i,\ell})
	\end{equation}
	for \emph{smooth} functions (for each $p \in \Sigma$)
	\begin{align*} A & = A_p : \RR \to \RR, \\
	\sB & = \sB_p : \RR \times T_p \Sigma \to \operatorname{End}(T_p \Sigma),\\
	C & = C_p : \RR \times T_p \Sigma \to \RR,
	\end{align*}
	which, additionally, satisfy: $A > 0$, $\sB$ is positive definite. More specifically, at each point $p \in \Sigma$:
	\begin{align*}
		A(z) &  \triangleq \frac{\sqrt{g_0}}{\sqrt{g_z}}, \; z \in \RR,\\
		\sB(z, \mathbf{v}) \mathbf{w} & \triangleq \frac{\sqrt{g_z}}{\sqrt{g_0}} \frac{g_z^{ij} g^0_{jk} \mathbf{w}^j \partial_{y_k}}{(1 + g_z^{pq} g^0_{pk} g^0_{q\ell} \mathbf{v}^k \mathbf{v}^\ell)^{1/2}}, \; z \in \RR, \; \mathbf{v}, \mathbf{w} \in T_p \Sigma,\\
		C(z, \mathbf{v}) & \triangleq - \frac{\sff^{pq}_z g^0_{ip} g^0_{jq} \mathbf{v}^i \mathbf{v}^j}{(1 + g_z^{pq} g^0_{pk} g^0_{q\ell} \mathbf{v}^k \mathbf{v}^\ell)^{1/2}} + (1 + g_z^{pq} g^0_{pk} g^0_{q\ell} \mathbf{v}^k \mathbf{v}^\ell)^{1/2} H_z, \; z \in \RR, \; \mathbf{v} \in T_p \Sigma.
	\end{align*}
	
	From the fundamental theorem of calculus, as well as the fact that the two divergences (for the two cases $\ell = 1$, $2$) are \emph{pointwise} bounded (because the two mean curvatures are bounded), it follows that
	\begin{equation} \label{eq:bounded.index.h.diff.abstract.pde}
		H_{\Gamma_{i,2}} - H_{\Gamma_{i,1}} = - A \divg_\Sigma( \widehat{\sB} \nabla_\Sigma f_i + f_i \widehat{\mathbf{C}}) 
		+ \langle \mathbf{\widehat{D}}, \nabla_\Sigma f_i \rangle_{\Sigma} + \widehat{E}f_i \text{ on } \Sigma,
	\end{equation}
	where $f_i \triangleq f_{i,2} - f_{i,1} > 0$ on $\Sigma$, with coefficients
	\begin{align*} \widehat{\sB} & = \widehat{\sB}_p : \RR^2 \times (T_p \Sigma)^2 \to \operatorname{End}(T_p \Sigma), \\
	\widehat{\mathbf{C}} & = \widehat{\mathbf{C}}_p, \widehat{\mathbf{D}} = \widehat{\mathbf{D}}_p : \RR^2 \times (T_p \Sigma)^2 \to T_p \Sigma, \\
	\widehat{E} & = \widehat{E}_p : \RR^2 \times (T_p \Sigma)^2 \to \RR, 
	\end{align*}
	whose arguments are $(f_{i,1}, f_{i,2}, \nabla_\Sigma f_{i,1}, \nabla_\Sigma f_{i,2}) \in \RR^2 \times (T_p \Sigma)^2$. These coefficients are \emph{uniformly bounded} and satisfy
	\[ A \geq \mu, \; \langle B\mathbf{v}, \mathbf{v}\rangle_\Sigma \geq \mu \Vert \mathbf{v} \Vert_\Sigma^2, \; \mathbf{v} \in T_p \Sigma, \]
	for a fixed $\mu > 0$, provided
	\[ \limsup_{i \to \infty} \Vert f_{i,1} \Vert_{C^1(\Sigma)} + \Vert f_{i,2} \Vert_{C^1(\Sigma)} < \infty. \]
	
	It will be convenient to carry out the exact computation, as that will allow us to study a particular rescaled limit as $i \to \infty$. It will also be convenient to denote
	\[ \zeta_i^{(t)} \triangleq f_{i,1} + t (f_{i,2} - f_{i,1}) \equiv f_{i,2} + tf_i, \; t \in [0,1]. \]
	Note that
	\[ \zeta_i^{(0)} \equiv f_{i,1}, \; \zeta_i^{(1)} \equiv f_{i,2}, \text{ and } \tfrac{\partial}{\partial t} \zeta_i^{(t)} \equiv f_i \text{ on } \Sigma. \]
	
	Let us define $\widehat{\sB}$, $\widehat{\mathbf{C}}$, $\widehat{\mathbf{D}}$, $\widehat{E}$. The easiest term to deal with in \eqref{eq:bounded.index.h.abstract.pde} is the low order term, $C$. Indeed
	\begin{align*}
		& C(f_{i,2}, \nabla_\Sigma f_{i,2}) - C(f_{i,1}, \nabla_\Sigma f_{i,1})  = \underbrace{\left[ \int_0^1 D_z C(\zeta_i^{(t)}, \nabla_\Sigma \zeta_i^{(t)}) \, dt \right]}_{\widehat{E} \text{, term 1 out of 2}} f_i  + \left\langle \underbrace{\int_0^1 D_{\mathbf{v}} C(\zeta_i^{(t)}, \nabla_\Sigma \zeta_i^{(t)}) \, dt}_{\widehat{\mathbf{D}}}, \nabla_\Sigma f_i \right\rangle.
	\end{align*}
	We study the higher order term in two steps. First:
	\begin{align*}
		& \sB(f_{i,2}, \nabla_\Sigma f_{i,2}) \nabla_\Sigma f_{i,2} - \sB(f_{i,1}, \nabla_\Sigma f_{i,1}) \nabla_\Sigma f_{i,1} \\
		& = \underbrace{\left( \left[ \int_0^1 D_z \sB(\zeta_i^{(t)}, \nabla_\Sigma \zeta_i^{(t)}) \nabla_\Sigma \zeta_i^{(t)} \, dt \right] \right)}_{\widehat{\mathbf{C}}} f_i  + \left\langle \underbrace{\left[ \int_0^1 D_{\mathbf{v}} \sB(\zeta_i^{(t)}, \nabla_\Sigma \zeta_i^{(t)}) \nabla_\Sigma \zeta_i^{(t)} \, dt \right]}_{\widehat{\sB} \text{, term 1 out of 2}}, \nabla_\Sigma f_i \right\rangle \\
		& \qquad + \underbrace{\left[ \int_0^1 \sB(\zeta_i^{(t)}, \nabla_\Sigma \zeta_i^{(t)}) \, dt \right]}_{\widehat{\sB} \text{, term 2 out of 2}} \nabla_\Sigma f_i.
	\end{align*}
	Second,
	\begin{align*}
		& (A(f_{i,2}) - A(f_{i,1})) \divg_\Sigma (\sB(f_{i,2}, \nabla_\Sigma f_{i,2}) \nabla_\Sigma f_{i,2}) \\
		& \qquad = \underbrace{\left( \left[ \int_0^1 D_z A(\zeta_i^{(t)}) \, dt \right] \divg_\Sigma (\sB(f_{i,1}, \nabla_\Sigma f_{i,1}) \nabla_\Sigma f_{i,1}) \right)}_{\widehat{E} \text{, term 2 out of 2}} f_i.
	\end{align*}
	
	We now return to the qualitative study of $f_i$. Applying the Harnack inequality in divergence form to \eqref{eq:bounded.index.h.diff.abstract.pde} (after multiplying through by $A^{-1}$), we get
	\begin{equation} \label{eq:bounded.index.harnack}
		\sup_{\Sigma} f_i \leq c \inf_{\Sigma} f_i \text{ for } i = 1, 2, \ldots
	\end{equation}
	with a constant $c > 0$ that doesn't depend on $i$. From Proposition \ref{prop:ultimate.stable.estimates} and Corollary \ref{coro:ultimate.stable.estimates}, we know that
	\begin{align} 
		\lim_{i \to \infty} \frac{\Vert H_{\Gamma_{i,\ell}} \Vert_{C^0(\Gamma_{i,\ell})}}{\varepsilon_i |\log \varepsilon_i|} & = 0 \text{ for } \ell = 1, 2, \label{eq:bounded.index.estimate.meancurv} \\
		\liminf_{i \to \infty} \frac{\inf_{\Sigma} f_i}{\varepsilon_i |\log \varepsilon_i|} & > 0. \label{eq:bounded.index.estimate.height}
	\end{align}
	Define the normalizations
	\begin{equation} \label{eq:bounded.index.normalized.function}
		\widehat{f}_i \triangleq (\sup_\Sigma f_i)^{-1} f_i : \Sigma \to [\tfrac{1}{c}, 1],
	\end{equation}
	where $c$ is as in \eqref{eq:bounded.index.harnack}. In view of \eqref{eq:bounded.index.h.diff.abstract.pde}, ${\widehat{f}_i}$ satisfies the linear, uniformly elliptic equation (note that we've multiplied through by $A^{-1}$, which is uniformly bounded):
	\begin{equation} \label{eq:bounded.index.h.diff.norm.abstract.pde}
		\frac{H_{\Gamma_{i,2}} - H_{\Gamma_{i,1}}}{A \cdot \sup_\Sigma f_i} = - \divg_\Sigma( \widehat{\sB} \nabla_\Sigma \widehat{f}_i + \widehat{f}_i \widehat{\mathbf{C}})
		+ \langle A^{-1} \mathbf{\widehat{D}}, \nabla_\Sigma \widehat{f}_i \rangle_{\Sigma} + A^{-1} \widehat{E}\widehat{f}_i \text{ on } \Sigma.
	\end{equation}
	We will \emph{test} this PDE by multiplying through by some $\zeta \in C^\infty(\Sigma)$ and integrating by parts. By testing, first, with $\zeta = \widehat{f}_i$, we get uniform energy estimates
	\[ \limsup_{i \to \infty} \int_\Sigma |\nabla_\Sigma \widehat{f}_i|^2 < \infty. \]
	Moreover, since $\widehat{f}_i$ is (trivially) bounded, it follows from Rellich's theorem that there exist $\widehat{f} \in W^{1,2}(\Sigma)$ and a subsequence such that
	\[ \widehat{f}_i \weaklyto \widehat{f} \text{ in } W^{1,2}(\Sigma), \; \widehat{f}_i \to \widehat{f} \text{ in } L^2(\Sigma). \]
	Therefore, since the coefficients in \eqref{eq:bounded.index.h.diff.norm.abstract.pde} are all uniformly bounded as $i \to \infty$, it follows that we can test \eqref{eq:bounded.index.h.diff.norm.abstract.pde} with arbitrary $\zeta \in C^\infty(\Sigma)$, and pass to a subsequential limit $i \to \infty$.
	
	The left hand side of  \eqref{eq:bounded.index.h.diff.norm.abstract.pde} converges to zero uniformly as $i \to \infty$ because of  \eqref{eq:bounded.index.estimate.meancurv}-\eqref{eq:bounded.index.estimate.height} above.  Thus, $\widehat{f}$ is a $W^{1,2}$-weak solution of
	\begin{equation} \label{eq:bounded.index.h.diff.norm.limit.pde}
		-\divg_\Sigma (\widehat{\sB}_\infty \nabla_\Sigma \widehat{f} + \widehat{f} \widehat{\mathbf{C}}_\infty) + \langle A^{-1}_\infty \widehat{\mathbf{D}}_\infty, \nabla_\Sigma \widehat{f} \rangle + A^{-1}_\infty \widehat{E}_\infty \widehat{f} = 0 \text{ on } \Sigma,
	\end{equation}
	where $A_\infty$, $\widehat{\sB}_\infty$, $\widehat{\mathbf{C}}_\infty$, $\widehat{\mathbf{D}}_\infty$, $\widehat{E}_\infty$ are just the same coefficients, except now they are evaluated at the limiting configuration of $(0, 0, \mathbf{0}, \mathbf{0})$. It is not hard to see, using the evolutions in Appendix \ref{app:mean.curvature.graphs}, that
	\begin{equation*}
		A_\infty \equiv 1, \; \widehat{\sB}_\infty \equiv \operatorname{Id}, \; \widehat{\mathbf{C}}_\infty \equiv \widehat{\mathbf{D}}_\infty \equiv \mathbf{0}, 
		\text{ and } \widehat{E}_\infty \equiv - (|\sff_\Sigma|^2 + \ricc_g(\nu_\Sigma, \nu_\Sigma)).
	\end{equation*}
	Thus, $\widehat{f}$ is $W^{1,2}$-weak solution of the Jacobi equation,
	\begin{equation} \label{eq:bounded.index.h.jacobi.equation}
		(\Delta_\Sigma + |\sff_\Sigma|^2 + \ricc_g(\nu_\Sigma, \nu_\Sigma)|_\Sigma) h = 0 \text{ on } \Sigma.
	\end{equation}
	It must be smooth---and thus classically a solution---by elliptic regularity. Moreover
	\[ \tfrac{1}{c} \leq \widehat{f}_i \leq 1 \text{ for all } i = 1, 2, \ldots \implies \tfrac{1}{c} \leq \widehat{f} \leq 1. \]
	In particular, the function is positive. It follows that the principal eigenvalue of the Jacobi operator is zero, so $\Sigma$ is stable.{\footnote{The fact that $\ind(u_i) = 0$ for all $i = 1, 2, \ldots$ implies the stability of $\Sigma$ is not new: see \cite{Tonegawa05, TonegawaWickramasekera12, Hiesmayr, Gaspar}. Nonetheless, by appropriately generalizing the argument given here, we are going to be able to extend the conclusion that $\Sigma$ is stable in the case where $\ind(u_i) \leq I_0$ for $i = 1, 2, \ldots$, $I_0 \in \{0, 1, \ldots \}$ and convergence occurs with multiplicity $\geq 2$.}} The result follows.
	
	We now drop the stability assumption and proceed to the general case of $I_0 \in \{0, 1, 2, \ldots\}$. We continue to assume that $\Sigma$ is two-sided. Without loss of generality, we'll assume $I_0 = 1$ from this point on. The general case is similar.
	
	The index concentration set is either empty (in which case, we can argue as in the previous case) or satisfies $\mathring{\Sigma} = \Sigma \setminus \{ P_\star \}$ for some $P_\star \in \Sigma$, and the convergence of $\{ u_i = 0 \} \cap U$ to $\mathring{\Sigma}$ is graphical $C^{2,\theta}_{\loc}$ on $\Sigma \setminus \{ P_\star \}$. Notice that, by definition, for every $r > 0$ there exists a subsequence along which
	\begin{equation} \label{eq:bounded.index.concentration.puncture}
		\ind(u_i; M \setminus B_{r/2}(P_\star)) = 0.
	\end{equation}
	Our previous discussion regarding the stable case applies verbatim to $M \setminus C_{r}(P_\star)$, where, in exponential normal coordinates,
	\[ C_{\rho}(P_\star) \triangleq B^2_{\rho}(P_\star) \times (-1,1), \]
	and yields functions $f_{i,1}$, $f_{i,2} : \Sigma \setminus B_{r}^2(P_\star) \to \RR$ representing the \emph{incomplete} properly embedded surfaces\emph{-with-boundary}
	\begin{equation} \label{eq:bounded.index.graphing.punctured.sheets}
		\Gamma_{i,1}, \Gamma_{i,2} \subset \{ u_i = 0 \} \cap U \setminus C_{r}(P_\star).
	\end{equation}
	
	\begin{rema} \label{rema:bounded.index.notation.i}
		Recall that we assumed $U$ is the image of the normal exponential map of $\Sigma$ restricted to $\Sigma \times (-1, 1)$. Then, $\partial C_\rho(P_\star) \cap U = \partial B_\rho^2(P_\star) \times (-1, 1)$ for every sufficiently small $\rho > 0$.
	\end{rema}

	All of \eqref{eq:bounded.index.h}-\eqref{eq:bounded.index.estimate.height} continue to hold over $\Sigma \setminus B_{r}^2(P_\star)$ instead of $\Sigma$.
	All the constants inevitably depend on our choice of $r > 0$, which is yet to be determined. We note that, trivially, the energy estimate
	\[ \limsup_{i \to \infty} \int_{\Sigma \setminus B_{2r}^2(P_\star)} |\nabla_\Sigma \widehat{f}_i|^2 < \infty \]
	holds true for any fixed $r > 0$ by our previous discussion. In fact, because $\Gamma_{i,1} \setminus C_{r}(P_\star)$, $\Gamma_{i,2} \setminus C_{r}(P_\star)$ converge in $C^{2,\theta}$ to $\Sigma \setminus B^2_{r}(P_\star)$ as $i \to \infty$, a subset of the fixed surface $\Sigma$, the coefficients of \eqref{eq:bounded.index.h.diff.abstract.pde} will satisfy
	\begin{multline*}
		\limsup_{r \to 0} \Big[ \limsup_{i \to \infty} \Vert A \Vert_{C^0(\Sigma \setminus B_{3r/2}^2(P_\star))} + \Vert \widehat{\sB} \Vert_{C^0(\Sigma \setminus B_{3r/2}^2(P_\star))} \\
		+ \Vert \widehat{\mathbf{C}} \Vert_{C^0(\Sigma \setminus B_{3r/2}^2(P_\star))} + \Vert \widehat{\mathbf{D}} \Vert_{C^0(\Sigma \setminus B_{3r/2}^2(P_\star))} 
		+ \Vert \widehat{E} \Vert_{C^0(\Sigma \setminus B_{3r/2}^2(P_\star))} \Big] < \infty,
	\end{multline*}
	and, therefore, we'll have the \emph{uniform} energy estimate
	\[ \limsup_{r \to 0} \left[ \limsup_{i \to \infty} \int_{\Sigma \setminus B_{2r}^2(P_\star)} |\widehat{f}_i|^2 \right] < \infty. \]
	This means we can pass to a limiting $\widehat{f}$ in the following sense:
	\begin{equation} \label{eq:bounded.index.renormalized.h}
		\widehat{f}_i \weaklyto \widehat{f} \text{ in } W^{1,2}_{\loc}(\mathring{\Sigma}), \; \widehat{f}_i \to \widehat{f} \text{ in } L^2_{\loc}(\mathring{\Sigma}).
	\end{equation}
	Now, \eqref{eq:bounded.index.harnack}-\eqref{eq:bounded.index.estimate.height} also hold for each fixed $r > 0$, with the $\sup$ and $\inf$ taken over $\Sigma \setminus B^2_r(P_\star)$, the $C^0$ norm of $H_{\Gamma_{i,\ell}}$ taken over $\Gamma_{i,\ell} \setminus C_r(P_\star)$; the constant $c$ and rates of convergence of the limits, a priori, depend on $r$. Nonetheless, $\widehat{f} \in W^{1,2}_{\loc}(\mathring{\Sigma})$ is a weak solution of \eqref{eq:bounded.index.h.jacobi.equation} on $\mathring{\Sigma}$. By elliptic regularity, $\widehat{f}$ is smooth and solves \eqref{eq:bounded.index.h.jacobi.equation} classically on $\mathring{\Sigma}$.
	
	\begin{prop} \label{prop:bounded.index.nontrivial.limit}
		$\widehat{f} \in L^\infty(\mathring{\Sigma})$, $\widehat{f} \not \equiv 0$ a.e. on $\mathring{\Sigma}$.
	\end{prop}

	We defer the proof of Proposition \ref{prop:bounded.index.nontrivial.limit} to the next section, since the argument is of independent interest.

	This proposition, once verified, completes the proof of Theorem \ref{theo:bounded.index}: by standard removable singularity results for elliptic PDE, $\widehat{f}$ must extend to a smooth nonnegative solution of \eqref{eq:bounded.index.h.jacobi.equation} on $\Sigma$, which is not identically zero, and the result follows as it did in the stable setting.

	Finally, we explain the necessary modifications when $\Sigma$ is one-sided. Assume, as above, that $I_{0}=1$ (the general case is similar). As before, we can define $\mathring \Sigma$ to be the complement of the index concentration set. Considering a tubular neighborhood $U$, of $\Sigma$, we can use the normal exponential map to lift $\Sigma$ and $u : U\to \RR$ to $\check \Sigma \subset \check U$, where $\check\Sigma$ is the orientable double cover of $\Sigma$ and $\check U$ is the associated lift of $U$. We can assume that $\check U$ is diffeomorphic (via the normal exponential map) to $\check \Sigma \times (-1,1)$. Let $\check{\mathring{\Sigma}}$ be the lift of $\mathring{\Sigma}$. Observe that  $\check{\Sigma} \setminus \check{\mathring{\Sigma}}$ contains at most two points (more generally $2I_{0}$ points). 

	Note that the covering map $\pi : \check U \to U$ admits an deck transformation $\tau: \check U \to \check U$ with $\tau^{2}$ equal to the identity. Define $\check u \triangleq u \circ \pi$, which is still a critical point of $E_{\varepsilon_{i}}$. Clearly $\check u \circ \tau = \check u$. 

	We claim that the convergence of $\check u$ to $\check \Sigma$ occurs with even multiplicity. If not, (up to switching the normal vector) we can assume that $\check u \to -1$ on $\check\Sigma \times (-1,0)$ and $\check u \to 1$ on $\check\Sigma \times (0,1)$ (this is clear on $\check{\mathring{\Sigma}}$, which then implies that it holds for all $p\in\check\Sigma$). Note, however, that $\tau(\{p\} \times (0,1)) = \{\tau(p) \} \times (-1,0)$ (otherwise, we would find that $\Sigma$ was two-sided). This contradicts the fact that $\check u$ is invariant under $\tau$. 

	Thus, the convergence of $\check u$ occurs with even multiplicity (and thus multiplicity at least two). Now, the argument above can be applied verbatim to $\check \Sigma$ and $\check u$ to produce a smooth positive Jacobi field on $\check \Sigma$ (we emphasize that it is not clear what the index of $\check u$ is; here, we rely on the index bounds of $u$ to bound the cardinality of $\check{\Sigma} \setminus \check{\mathring{\Sigma}}$; after this step, the definition of $\check{\mathring \Sigma}$, rather than the index bounds is all that is used). As above, this implies that $\check\Sigma$ is stable. 
\end{proof}

\subsection{Sliding heteroclinic barriers} \label{subsec:bounded.index.height.bounds}

For the reader's convenience we start by recalling the following important result of White on local foliations by minimal surfaces.

\begin{prop}[{\cite[Appendix]{White:curvature}}] \label{prop:white.foliation}
	Let $\Phi$ be an even elliptic integrand, where $\Phi$ and $D_2 \Phi$ are $C^{2,\theta}$. Let $\Phi_r$ be the integrand defined by $\Phi_r(x, v) = \Phi(rx, v)$. There is an $\eta > 0$ such that if $r < \eta$ and if
	\[ w : B_1 \subset \RR^2 \to \RR, \; \Vert w \Vert_{C^{2,\theta}} < \eta, \]
	then for each $t \in [-1,1]$, there is a $C^{2,\theta}$ function $v_t : B_1 \to \RR$ whose graph is $\Phi_r$-stationary and such that 
	\[ v_t(x) = w(x) + t \; \text{ if } x \in \partial B_1. \]
	Furthermore, $v_t$ depends in a $C^1$ way on $t$ so that the graphs of the $v_t$ foliate a region of $\RR^3$. If $M$ is a $C^1$ properly immersed $\Phi_r$-stationary surface in $B_{1/2}(0)$ with $\partial M \subset \graph v_t$, then $M \subset \graph v_t$.
\end{prop}

We will use the minimal disks constructed by this proposition to construct barriers (using Theorem \ref{theo:dirichlet.data.construction}) that will allow us to control the height of the top and bottom $\{ u_i = 0 \}$ sheets near $P_\star$. This can be thought of as a variant of the moving planes method adapted to the Riemannian Allen--Cahn setting.

\begin{proof}[Proof of Proposition \ref{prop:bounded.index.nontrivial.limit}]
	We continue with the same notation as in the previous section. Let $\rho > 2r$, with still being such that \eqref{eq:bounded.index.concentration.puncture}-\eqref{eq:bounded.index.graphing.punctured.sheets} apply.
	
	Let $w_i : B^2_{2\rho}(P_\star) \to \RR$ be a harmonic function (defined on $B^2_{2\rho}(P_\star) \subset \Sigma$) with boundary data
	\[ w_i = f_{i,2} \text{ on } \partial B^2_{\rho}(P_\star). \]
	Recalling, from Corollary \ref{coro:curvature.estimates}, that $f_{i,2} \to 0$ in $C^{2,\theta}(B^2_{2\rho}(P_\star) \setminus B^2_{\rho/2}(P_\star))$, it follows that (by potentially going farther down the sequence of $i = 1, 2, \ldots$) $\Vert w_i \Vert_{C^{2,\theta}}$ --suitably scaled-- is small enough for Proposition \ref{prop:white.foliation} to apply. Once we're in that regime, Proposition \ref{prop:white.foliation} guarantees a foliation 
	\[ t \mapsto D_{i,\rho}(t), \; t \in [-\delta, \delta], \]
	consisting of minimal disks that all project to $B^2_{\rho}(P_\star) \subset \Sigma$. Without loss of generality, we may suppose that the foliated region $\cup_{|t| < \delta}  D_{i,\rho}(t)$ lies entirely within $U$. Note that:
	\begin{enumerate}
		\item the curves $t \mapsto \partial D_{i,\rho}(t)$ move at unit vertical speed in $\partial C_{\rho}(P_\star)$;
		\item the second fundamental form of the disks $D_{i,\rho}(t)$ is bounded in $C^{0,\theta}$ uniformly over $i = 1, 2, \ldots$, $t \in [-\delta, \delta]$,
		\begin{equation} \label{eq:bounded.index.disk.c2alpha}
			|\sff_{D_{i,\rho}(t)}| + [\sff_{D_{i,\rho}(t)}]_{\theta} \leq \eta,
		\end{equation} 
		and $\eta > 0$ can be made arbitrarily small.
	\end{enumerate}

	As a consequence of \eqref{eq:bounded.index.disk.c2alpha}, \eqref{eq:sheets.enhanced.sff.grad.bound}, and minimal surface curvature estimates, the disks also satisfy the following weaker $C^{3,\theta}$ bound uniformly over $i = 1, 2, \ldots$, $t \in [-\delta, \delta]$,
	\begin{equation} \label{eq:bounded.index.disk.c3alpha}
		\eps |\nabla_{D_{i,\rho}(t)} \sff_{D_{i,\rho}(t)}| + \eps^{1+\theta} [\nabla_{D_{i,\rho}(t)} \sff_{D_{i,\rho}(t)}]_{\theta} \leq \eta,
	\end{equation}
	after possibly relaxing $\eta > 0$, which can still nevertheless be made arbitrarily small.

	We'll now use a sliding/moving planes  argument that relies on the barrier construction in Section \ref{sec:dirichlet.data}, adopting relevant notation from therein. We assume, without loss of generality, that
	\begin{equation} \label{eq:bounded.index.sign.u}
		u > 0 \text{ above } \Gamma_{i,2} \text{ in } U \setminus C_{\rho/2}(P_\star).
	\end{equation}
	
	Our constructions below will take place for $\eps = \eps_i$, $i = 1, 2, \ldots$, but we suppress the dependence on $i$ for the sake of notational brevity.
	
	Define $\hat{\chi} : \RR \to [0,1]$ to be a cutoff function such that
	\begin{equation} \label{eq:bounded.index.chi.hat}
		\hat{\chi}(s) = \begin{cases} 1 & |s| \leq B \eps |\log \eps| \\ 0 & |s| \geq 2B \eps |\log \eps|, \end{cases}
	\end{equation}
	where $B \gg 1$ is to be chosen later. This can be constructed so that
	\[ |\hat{\chi}^{(k)}| = O((\eps |\log \eps|)^{-k}) \text{ for } k \geq 1, \; \eps \to 0. \]
	Let's very briefly run through some notation which is introduced later, in Section \ref{sec:dirichlet.data}; we will need to use some of it here in invoking that section's main theorem. In Section \ref{sec:dirichlet.data} we consider $\delta_* \in (0, 1)$ fixed and a H\"older exponent $\alpha \in (0, 1)$, $\alpha \leq \theta$, where $\theta$ is as in \eqref{eq:bounded.index.disk.c2alpha}, \eqref{eq:bounded.index.disk.c3alpha}. (We will eventually choose $\alpha$ near $0$ and $\theta$ near $1$.) In \eqref{eq:dirichlet.data.cutoff}, cutoff functions $\chi_j$ are introduced that are supported on strips of width $O(\eps^{\delta_*})$ (while $\hat{\chi}$ is supported on a thinner strip of size $O(\eps |\log \eps|)$). Finally, in \eqref{eq:dirichlet.data.approximate.heteroclinic}, $\chi_1$ is used to define a suitably truncated approximate heteroclinic solution $\widetilde{\mathbb{H}}_\eps$ that is constant  outside a strip of size $O(\eps^{\delta_*})$. (See Remark \ref{rema:dirichlet.data.trivialization.window}.)
		
	Given this notation, let's set:
	\[ \hat{v}^\sharp(s) \triangleq \gamma \hat{\chi}(s) \mathbb{H}'(\eps^{-1} s) + (1-\hat{\chi}(s)) \begin{cases} 1 - \eps^3 - \widetilde{\mathbb{H}}_\eps(s) & s > 0 \\ - 1 - \widetilde{\mathbb{H}}_\eps(s) & s < 0, \end{cases} \]
	where $\gamma \in \RR$ is chosen so that the orthogonality constraint
	\begin{equation} \label{eq:bounded.index.v.sharp.orthogonality}
		\int_{-\infty}^\infty \hat{v}^\sharp(s) \mathbb{H}'(\eps^{-1} s) \, ds = 0
	\end{equation}
	holds. Recalling \eqref{eq:heteroclinic.expansion.i}-\eqref{eq:heteroclinic.expansion.ii}, and that $\delta_* \in (0, 1)$,  \eqref{eq:bounded.index.v.sharp.orthogonality} is equivalent to
	\begin{align}
		\gamma(\energyunit - o(1)) & = O(\eps^{-1}) \int_{B \eps |\log \eps|}^\infty (1-\mathbb{H}(\eps^{-1} s)) \mathbb{H}'(\eps^{-1} s) \, ds   + O(\eps^2) \int_{2B \eps |\log \eps|}^{\tfrac{47}{50} \eps^{\delta_*}} |\mathbb{H}'(\eps^{-1} s)| \, ds \nonumber \\
		& = O(1) \int_{B |\log \eps|}^\infty (1-\mathbb{H}(s)) \mathbb{H}'(s) \, ds   + O(\eps^3) \int_{2B |\log \eps|}^{\tfrac{47}{50} \eps^{\delta_*-1}} |\mathbb{H}'(s)| \, ds \nonumber \\
		& = O(1) \exp(-2\sqrt{2} B |\log \eps|)  + O(\eps^3) \exp(-2\sqrt{2} B |\log \eps|) =  O(\eps^{2\sqrt{2} B}). \label{eq:bounded.index.v.sharp.estimate.i}
	\end{align}
	Also,
	\begin{equation} \label{eq:bounded.index.v.sharp.estimate.ii}
		\Vert \hat{\chi}(s) \mathbb{H}'(\eps^{-1} s) \Vert_{C^{2,\alpha}_\eps(\RR)} = O(1) \text{ as } \eps \to 0.
	\end{equation}
	Taking $B$ sufficiently large, \eqref{eq:bounded.index.v.sharp.estimate.i}-\eqref{eq:bounded.index.v.sharp.estimate.ii} together imply
	\begin{equation} \label{eq:bounded.index.v.sharp.estimate}
		\Vert \hat{v}^\sharp(s) \Vert_{C^{2,\alpha}_\eps(\RR)} = O(\eps^3).
	\end{equation}
	Next, for $(y, s) \in \partial (B_\rho^2(P_\star) \times [-\tfrac12, \tfrac12])$, define
	\[ \hat{v}^\flat(y,s) \triangleq (1-\chi_4(s)) \begin{cases} 1 - \eps^3 - \widetilde{\mathbb{H}}_\eps(s) & s > 0 \\ -1 - \widetilde{\mathbb{H}}_\eps(s) & s < 0. \end{cases} \]
	Recall that $\chi_{4}$ is defined in \eqref{eq:dirichlet.data.cutoff}. It is easy to see that $\Vert \hat{v}^\flat \Vert_{C^{2,\alpha}_\eps} = O(\eps^3)$. In fact, $\chi_5 \hat{v}^\flat = 0$, so
	\begin{equation} \label{eq:bounded.index.v.flat.estimate}
		\Vert \hat{v}^\flat \Vert_{\widetilde{C}^{2,\alpha}_\eps} = O(\eps^3)
	\end{equation}
	as well (see \eqref{eq:dirichlet.data.ckalpha.eps.modified} for the definition of $\widetilde{C}^{2,\alpha}_\eps$). 
	
	We emphasize that everything from \eqref{eq:bounded.index.chi.hat} to \eqref{eq:bounded.index.v.flat.estimate} above is \emph{agnostic} of our particular solutions with bounded Morse index. They will serve as prescribed boundary data for solutions of the Allen-Cahn equation on the fixed product manifold $B_{\rho}^2 \times [-\tfrac12, \tfrac12]$, albeit with varying interior metric that will depend on $g$, $i = 1, 2, \ldots$, $\rho$, and $t \in [-\delta, \delta]$. Indeed, we let
	\begin{multline} \label{eq:bounded.index.pullback.metric}
		g_{i,\rho}(t) \triangleq \text{ pullback metric from } Z_{D_{i,\rho}(t)}(D_{i,\rho}(t) \times [-\tfrac12, \tfrac12]) \subset U \\
		\text{ to } B_{\rho}^2 \times [-\tfrac12, \tfrac12] \text{ under Fermi coordinates } (y,s) \text{  with respect to } D_{i,\rho}(t).
	\end{multline}
	
	We may apply Theorem \ref{theo:dirichlet.data.construction} with $\hat{v}^\sharp$, $\hat{v}^\flat$ as above, $\hat{\zeta} \equiv 0$, and with the H\"older exponents $\alpha$ near $0$ and $\theta$ near $1$ per the theorem, to $\Omega \triangleq B_{\rho}^2 \times [-\tfrac12, \tfrac12]$ 	and the nonconstant Riemannian metric $g_{i,\rho}(t)$. Note that the conditions of the theorem are met trivially for $\hat{\zeta}$, and are also met for $\hat{v}^\sharp$, $\hat{v}^\flat$ by  \eqref{eq:bounded.index.v.sharp.estimate}-\eqref{eq:bounded.index.v.flat.estimate}.
	
	The theorem yields   $\mathfrak{b}_{i,\rho,t} : \Omega \to \RR$ such that
	\begin{equation} \label{eq:bounded.index.barrier.pde}
		\eps_i^2 \Delta_{g_{i,\rho}(t)} \mathfrak{b}_{i,\rho,t} = W'(\mathfrak{b}_{i,\rho,t})
	\end{equation}
	and, for all $(y, s) \in \partial \Omega$,
	\begin{equation} \label{eq:bounded.index.barrier.boundary.data}
		\mathfrak{b}_{i,\rho,t}(y,s) = \widetilde{\mathbb{H}}_\eps(s) + \chi_4(s) \hat{v}^\sharp(s) + \hat{v}^\flat(y,s).
	\end{equation}
	We constructed $\hat{v}^\sharp$, $\hat{v}^\flat$ specifically so that:
	\begin{equation} \label{eq:bounded.index.barrier.boundary.data.explicit}
		\mathfrak{b}_{i,\rho,t}(y,s) = \begin{cases} 1-\eps_i^3 & (y,s) \in \partial \Omega, \; s \geq 2B \eps_i |\log \eps_i| \\ -1 & (y, s) \in \partial \Omega, \; s \leq -2B \eps_i |\log \eps_i|. \end{cases}
	\end{equation}
	
	\begin{clai}
		For every $\beta > 0$, $\eps_i \leq 1$, we have 
		\begin{equation} \label{eq:bounded.index.barrier.strip}
			|\mathfrak{b}_{i,\rho,t}(y,s)| \leq 1-\beta \implies |s| \leq c' \eps_i,
		\end{equation}
		where $c' = c'(W, \beta, \eta, c_0) > 0$.
	\end{clai}
	\begin{proof}[Proof of Claim]
		This is a straightforward consequence of the ansatz $\mathfrak{b}_{i,\rho,t} = (\widetilde{\mathbb{H}}_\epsilon + \chi_4 v^\sharp + v^\flat) \circ D_\zeta$, $\Vert v^\sharp \Vert_{C^0}$, $\Vert v^\flat \Vert_{C^0} = o(1)$, $\Vert \zeta \Vert = O(\eps_i^{2-2\alpha})$, and \eqref{eq:heteroclinic.expansion.i}, at least provided we take $\alpha$ small enough.
	\end{proof}

	\begin{clai}
		For sufficiently large $i$, 
		\begin{equation} \label{eq:bounded.index.barrier.start}
			\mathfrak{b}_{i,\rho,\delta} < (Z_{D_{i,\rho}(\delta)})^* u_i \text{ on } \Omega.
		\end{equation}
		Recall that $\delta > 0$ represents the ``top'' of the foliation $D_{i,\rho}(\delta)$. 
	\end{clai}
	\begin{proof}[Proof of Claim]
		Let's agree, for the remainder of the proof of this claim, to write $u_i$ instead of $(Z_{D_{i,\rho}(\delta)})^* u_i$. We seek to show that $\cG \triangleq \{ x \in \Omega :  \mathfrak{b}_{i,\rho,\delta}(x) < u_i(x) \}$ coincides with $\Omega$. (Recall: we're assuming \eqref{eq:bounded.index.sign.u}.)
		
		Fix $\beta \in (0,1)$ so that $W'' \geq 2\kappa^2 > 0$ on $[-1,-1+\beta] \cup [1-\beta, 1]$ for some $\kappa > 0$.
		
		From \eqref{eq:bounded.index.barrier.strip}, $\{ |\mathfrak{b}_{i,\rho,\delta}| \leq 1-\beta \}$ is contained in an $O(\eps_i)$-neighborhood of $D_{i,\rho}(\delta)$. From \cite[Theorem 1]{HutchinsonTonegawa00}, $\{ |u_i| \leq 1-\beta \}$ converges, in the Hausdorff sense, to $\Sigma$. In particular, for sufficiently large $i$,
		\[ (\Omega \cap \{ |u_i| \leq 1-\beta \}) \cup \{ |\mathfrak{b}_{i,\rho,\delta}| \leq 1-\beta \} \subset \cG. \]
		Note that
		\begin{equation*}
			\eps_i^2 \Delta_g (1-u_i) = - W'(u_i) = \tfrac{W'(1)-W'(u_i)}{1-u_i} (1-u_i) 
			\geq 2\kappa^2 (1-u_i) \text{ on } \{ u_i > 1-\beta \},
		\end{equation*}
		\begin{equation*}
			\eps_i^2 \Delta_g (1+u_i) = W'(u_i) = \tfrac{W'(u_i)-W'(-1)}{u_i-(-1)} (1+u_i) 
			\geq 2\kappa^2 (1+u_i) \text{ on } \{ u_i < -1+\beta \},
		\end{equation*}
		so by an application of the barrier principle together with the saddle property of $W$ at zero (see \cite[Lemma 4.1]{KowalczykLiuPacard12}) we get:
		\begin{equation} \label{eq:bounded.index.exp.decay}
			|u_i^2 - 1| = O\big( \exp(- \kappa \eps_i^{-1} \dist_g(\cdot, \{ u_i = 0 \})) \big).
		\end{equation}
		Combined with  \eqref{eq:bounded.index.barrier.boundary.data.explicit}, this shows $\partial \Omega \subset \cG$ for sufficiently large $i$. Thus:
		\begin{equation} \label{eq:bounded.index.barrier.strip.i}
			\Omega \setminus \cG \subset \Omega \setminus (\partial \Omega \cup \{ |u_i| \leq 1-\beta \} \cup \{ |\mathfrak{b}_{i,\rho,\delta}| \leq 1-\beta \}).
		\end{equation}
		Subtracting from \eqref{eq:bounded.index.barrier.pde} the PDE satisfied by $u_i$, we see that
		\[ \eps_i^2 \Delta_g (\mathfrak{b}_{i,\rho,\tau} - u_i) = c(x) (\mathfrak{b}_{i,\rho,\tau} - u_i) \]
		for $c(x) \triangleq  (W'(\mathfrak{b}_{i,\rho,t}(x)) - W'(u_i(x)))/(\mathfrak{b}_{i,\rho,\tau}(x) - u_i(x))$. This is \emph{negative} on $\Omega \setminus \cG$ by \eqref{eq:bounded.index.barrier.strip.i}, and this violates the maximum principle unless $\cG = \Omega$. The claim follows.
	\end{proof}
		
	Next, since:
	\begin{enumerate}
		\item $\mathfrak{b}_{i,\rho,t}$ and $(Z_{D_{i,\rho}(t)})^* u_i$ both vary continuously in $t \in [-\delta, \delta]$ by Theorem \ref{theo:dirichlet.data.construction},
		\item \eqref{eq:bounded.index.barrier.start} holds true, and
		\item $\mathfrak{b}_{i,\rho,-\delta} \not \leq (Z_{D_{i,\rho}(-\delta)})^* u_i$,
	\end{enumerate}
	there will exist exactly one $\tau_i \in (-\delta,\delta)$, and at least one $Q_i^\star \in \Omega$, such that
	\begin{equation} \label{eq:bounded.index.tau.definition}
		\mathfrak{b}_{i,\rho,t} < (Z_{D_{i,\rho}(t)})^* u_i \text{ on } \Omega \text{ for all } t \in (\tau_i, \delta], 
		\text{ and } \mathfrak{b}_{i,\rho,\tau_i}(Q_i^\star) = [(Z_{D_{i,\rho}(\tau_i)})^* u](Q_i^\star).
	\end{equation}
	
	Our goal is to estimate $\tau_i$. Abusing notation again, we'll write $u_i$ instead of   $(Z_{D_{i,\rho}(\tau_i)})^* u_i$, and $g$ instead of $g_{i,\rho}(\tau_i)$. Thus:
	\begin{equation} \label{eq:bounded.index.max.principle.conditions}
		u_i - \mathfrak{b}_{i,\rho,\tau_i} \geq 0 \text{ on } \Omega, \; (u_i - \mathfrak{b}_{i,\rho,\tau_i})(Q_i^\star) = 0.
	\end{equation}
	Subtracting \eqref{eq:bounded.index.barrier.pde} from the PDE satisfied by $u$, we see that
	\[ \eps_i^2 \Delta_g (u_i - \mathfrak{b}_{i,\rho,\tau_i}) = c(x) (u_i - \mathfrak{b}_{i,\rho,\tau_i}) \]
	for $c(x) \triangleq (W'(u_i(x)) - W'(\mathfrak{b}_{i,\rho,\tau_i}(x)))/(u_i(x) - \mathfrak{b}_{i,\rho,\tau_i}(x))$. The maximum principle, then, tells us that
	\begin{enumerate}
		\item either $Q_i^\star \in \partial \Omega$, or
		\item $u_i \equiv \mathfrak{b}_{i,\rho,\tau_i}$ on $\Omega$.
	\end{enumerate}
	We only consider the first case here, since the second reduces to it by replacing $Q_i^\star$ with another point on $\partial \Omega$. Note that \eqref{eq:bounded.index.tau.definition}, the fact that $\mathfrak{b}_{i,\rho,0}|_{\partial D_{i,\rho}(0)} \equiv 0$, and the uniqueness of $\tau_i$ give a lower bound on $\tau_i$:
	\begin{equation} \label{eq:bounded.index.tau.lower.bound}
		\tau_i \geq 0.
	\end{equation}
	
	The upper bound is more subtle. We claim that
	\begin{equation} \label{eq:bounded.index.tau.upper.bound}
		\tau_i < 7B \eps_i |\log \eps_i|,
	\end{equation} 
	provided $B$ is chosen (independently of $i$) such that
	\begin{equation} \label{eq:bounded.index.B.requirement}
		\dist_g(x, \{ u_i = 0 \}) > 3B \eps_i |\log \eps_i| \implies |u_i(x)| > 1-\eps_i^3.
	\end{equation}
	The existence of a $B$ that satisfies \eqref{eq:bounded.index.B.requirement} is guaranteed by \eqref{eq:bounded.index.exp.decay}.
	
	It will be convenient to introduce the notation (here, $\lambda \geq 0$ is some parameter):
	\begin{align*}
		\overline{\partial \Omega}[\lambda] & \triangleq \{ (y, s) \in \partial \Omega : s \in [\lambda, \tfrac12] \}, \\
		\underline{\partial \Omega}[\lambda] & \triangleq \{ (y,s) \in \partial \Omega : s \in [-\tfrac12, -\lambda] \}.
	\end{align*}
	
	To start, let's estimate the height of $Q_i^\star$ from below. We have
	\[ u_i > -1 \text{ on } (M^n, g) \implies  \mathfrak{b}_{i,\rho,\tau_i}(Q_i^\star) = u_i(Q_i^\star) > -1, \]
	so, from  \eqref{eq:bounded.index.barrier.boundary.data.explicit}:
	\[ Q_i^\star \in \partial \Omega \setminus \underline{\partial \Omega}[2B\eps_i |\log \eps_i|]. \]
	Equivalently, the image $\widetilde{Q}_i^\star$ of $Q_i^\star$ to $(M^n, g)$ under $Z_{D_{i,\rho}(\tau_i)}$ satisfies
	\[ \widetilde{Q}_i^\star \in Z_{D_{i,\rho}(\tau_i)}(\partial \Omega \setminus \underline{\partial \Omega}[2B \eps_i |\log \eps|]). \]
	In particular, $\widetilde{Q}_i^\star$  belongs to the open tubular neighborhood of the image $Z_{D_{i,\rho}(\tau_i)}(\overline{\partial \Omega}[0]) \subset (M^n, g)$ with  radius $3B \eps_i |\log \eps_i|$:
	\begin{equation} \label{eq:bounded.index.height.Q.est.i}
		\widetilde{Q}_i^\star \in B_{3 B \eps_i |\log \eps_i|}\big( Z_{D_{i,\rho}(\tau_i)}(\overline{\partial \Omega}[0]) \big).
	\end{equation}

	We now prove \eqref{eq:bounded.index.tau.upper.bound} by contradiction. We'll show that
	\begin{equation} \label{eq:bounded.index.height.Q.est.ii}
		\dist_g(Z_{D_{i,\rho}(\tau_i)}(\overline{\partial \Omega}[0]), \{u_i=0\}) > 6B \eps_i |\log \eps_i|
	\end{equation}
	when \eqref{eq:bounded.index.tau.upper.bound} fails, i.e., when $\tau_i \geq 7B\eps_i |\log \eps_i|$.
	
	Since $D_{i,\rho}(\tau_i)$ is an $o(1)$-Lipschitz graph over $\Sigma$ (note that the argument used to prove \eqref{eq:bounded.index.barrier.start} shows that $\tau_{i}\to 0$), and $Z_{D_{i,\rho}(\tau_i)}(\partial \Omega) \perp D_{i,\rho}(\tau_i)$, there will exist $\eta > 0$ (independent of $i$) such that
	and
	\[ Z_{D_{i,\rho}(\tau_i)}(\overline{\partial \Omega}[0] \setminus \overline{\partial \Omega}[\eta]) \subset C_{3\rho/2}(P_\star) \setminus C_{\rho/2}(P_\star) \]
	for sufficiently large $i$. Moreover $\lim_{i \to \infty} \{ u_i = 0 \} = \Sigma$ in the Hausdorff topology (\cite[Theorem 1]{HutchinsonTonegawa00}, \cite[Appendix B]{Guaraco}), so
	\[ \liminf_{i \to \infty} \dist_g(Z_{D_{i,\rho}(\tau_i)}(\overline{\partial \Omega}[\eta]), \{ u_i = 0 \}) > 0 \]
	because $\tau_i \geq 0$ by \eqref{eq:bounded.index.tau.lower.bound}. Thus, \eqref{eq:bounded.index.height.Q.est.ii} will follow from
	\begin{equation*}
		\dist_g(Z_{D_{i,\rho}(\tau_i)}(\overline{\partial \Omega}[0] \setminus \overline{\partial \Omega}[\eta]), \{ u_i = 0 \} \cap C_{2\rho}(P_\star) \setminus C_{\rho/2}(P_\star)) > 6B \eps_i |\log \eps_i|.
	\end{equation*}
	when $\tau_i \geq 7B \eps_i |\log \eps_i|$. Since the components of $\{ u_i = 0 \} \cap C_{2\rho}(P_\star) \setminus C_{\rho/2}(P_\star)$ are well-ordered $o(1)$-Lipschitz graphs over $\Sigma$, with $\Gamma_{i,2}$ being the topmost, we may equivalently show
	\begin{equation*}
		\dist_g(Z_{D_{i,\rho}(\tau_i)}(\overline{\partial \Omega}[0] \setminus \overline{\partial \Omega}[\eta]), \Gamma_{i,2}) > 6B \eps_i |\log \eps_i|.
	\end{equation*}
	Because $D_{i,\rho}(t)$, $t \in [0,\tau_i]$, are all $o(1)$-Lipschitz graphs over $\Sigma$ as well, we have
	\begin{equation*}
		\nabla_g (\dist^\pm_g(\cdot; \Gamma_{i,2})), \nabla_g (\dist^\pm_g(\cdot; D_{i,\rho}(t)) \rangle \geq 1-o(1), \; t \in [0,\tau_i] 
	\end{equation*} 
	in a small (but definite) neighborhood of $\Sigma$ in $C_{2\rho}(P_\star) \setminus C_{\rho/2}(P_\star)$. Here $\dist_g^\pm$ denotes the signed distance. From it follows that for every $P \in Z_{D_{i,\rho}(\tau_i)}(\overline{\partial \Omega}[0] \setminus \overline{\partial \Omega}[\eta])$,
	\begin{multline*}
		\dist^\pm_g(P; \Gamma_{i,2}) \geq (1-o(1)) \dist_g^\pm(P; D_{i,\rho}(0)) \geq (1-o(1))(\tau_i + \dist_g^\pm(P; D_{i,\rho}(\tau_i))) \\
			\geq (1-o(1)) \tau_i > (1-o(1)) 7B \eps_i |\log \eps_i| > 6B \eps_i |\log \eps_i|,
	\end{multline*}
	as claimed, and \eqref{eq:bounded.index.height.Q.est.ii} follows. It is now an automatic consequence of  \eqref{eq:bounded.index.height.Q.est.i}-\eqref{eq:bounded.index.height.Q.est.ii} that:
	\[ \dist_g(\widetilde{Q}_i^\star, \{ u_i = 0 \}) > 3B \eps_i |\log \eps_i|. \]
	Recalling  \eqref{eq:bounded.index.B.requirement}, we find: $|u_i(Q_i^\star)| > 1-\eps_i^3$. Combined with $\dist_g^\pm(\widetilde{Q}_i^\star; \Gamma_{i,2}) > 6B \eps_i |\log \eps_i| > 0$, which guarantees that $u_i(Q_i^\star) > 0$, we conclude $u_i(Q_i^\star) > 1-\eps_i^3$. This contradicts  \eqref{eq:bounded.index.barrier.boundary.data.explicit}. Thus, \eqref{eq:bounded.index.tau.upper.bound} is true.
	
	Summarizing  \eqref{eq:bounded.index.tau.lower.bound}-\eqref{eq:bounded.index.tau.upper.bound}: $0 \leq \tau_i < 7B \eps_i |\log \eps_i|$. Combined with the defining property \eqref{eq:bounded.index.max.principle.conditions} of $\tau_i$, we get the following height estimate over $\Sigma$:
	\[ f_{i,2} \leq h^{D_{i,\rho}(\tau_i)} \leq h^{D_{i,\rho}(7B \eps_i |\log \eps_i|)} \text{ on } \Sigma \cap B^2_{2\rho}(P_\star) \setminus B^2_r(P_\star), \]
	where:
	\begin{enumerate}
		\item $f_{i,2} : \Sigma \setminus B_r^2(P_\star) \to \RR$ is the height of $\Gamma_{i,2}$ over $\Sigma$, with $r \in (0,\rho/2)$ as in \eqref{eq:bounded.index.concentration.puncture}-\eqref{eq:bounded.index.graphing.punctured.sheets}, and
		\item $h^{D_{i,\rho}(t)}$ denotes the height of the minimal disk $D_{i,\rho}(t)$ over $\Sigma$.
	\end{enumerate} 
	
	The same sliding argument, carried out below the bottom-most sheet $\Gamma_{i,1}$ of $\{ u_i = 0 \}$, similarly gives:
	\[ f_{i,1} \geq h^{D'_{i,\rho}(-7B \eps_i |\log \eps_i|)} \text{ on } \Sigma \cap B^2_{2\rho}(P_\star) \setminus B^2_{r}(P_\star). \]
	Notice that we're denoting the disks by $D'_{i,\rho}(-7B\eps_i |\log \eps_i|)$, since they come from a different foliation, namely, the one generated by applying Proposition \ref{prop:white.foliation} to $w_i = f_{i,1}$. Therefore, by the regularity of the foliation guaranteed by Proposition \ref{prop:white.foliation}:
	\begin{align*}
		f_i  = f_{i,2} - f_{i,1} & \leq h^{D_{i,\rho}(7B \eps_i |\log \eps_i|)} - h^{D'_{i,\rho}(-7B\eps_i |\log \eps_i|)} \leq c \left( 7B \eps_i |\log \eps_i| + h^{D_{i,\rho}(0)} - h^{D'_{i,\rho}(0)} \right) \\
			& \leq c' \left( \eps_i |\log \eps_i| + \max_{B^2(\rho)(P_\star)} (h^{D_{i,\rho}} - h^{D'_{i,\rho}}) \right) \leq c' \left( \eps_i |\log \eps_i| + \max_{\partial B^2_\rho(P_\star)} (f_{i,2} - f_{i,1}) \right)
	\end{align*}
	on $\Sigma \cap B^2_{2\rho}(P_\star) \setminus B^2_{r}(P_\star)$. The last inequality follows from the maximum principle. We emphasize that $c'$ is independent of $i$ and $r$. 
	
	The proof of Proposition \ref{prop:bounded.index.nontrivial.limit} is essentially done. Indeed, fix $0 < r < \rho/2$. By what we've shown so far, we have
	\[ \sup_{\Sigma \setminus B^2_r(P_\star)} f_i \leq c' \left( \eps_i |\log \eps_i| + \sup_{\Sigma \setminus B^2_{\rho}(P_\star)} f_i \right). \]
	By the Harnack inequality \eqref{eq:bounded.index.harnack} and sheet separation lower bound \eqref{eq:bounded.index.estimate.height} on $\Sigma \setminus B^2_{2\rho}(P_\star)$:
	\[ \sup_{\Sigma \setminus B^2_r(P_\star)} f_i \leq c'' \inf_{\Sigma \setminus B^2_{\rho}(P_\star)} f_i. \]
	This holds independently of $i$, $r$, so the renormalized limit $\widehat{f}$ taken in \eqref{eq:bounded.index.renormalized.h} (first with $i \to \infty$ and then with $r \to 0$) is nontrivial. This completes the proof of Proposition \ref{prop:bounded.index.nontrivial.limit}.
\end{proof}

\section{Phase transitions with multiplicity one}


\label{sec:multiplicity.one}

In this section we return to working in arbitrary dimension $n \geq 3$, and we consider a compact Riemannian manifold $(M^{n},g)$ and a sequence $u_i \in C^\infty(M; (-1,1))$ of critical points of $E_{\eps_i}$, $\eps_i > 0$, $E_{\varepsilon_i}[u_i] \leq E_0$, for all $i = 1, 2, \ldots$, with $\lim_i \eps_i = 0$. Let $V \triangleq \lim_{i}\energyunit^{-1}V_{\eps_{i}}[u_{i}]$ denote the limit stationary integral varifold, which exists by \cite[Theorem 1]{HutchinsonTonegawa00} (see \cite[Appendix B]{Guaraco} for Riemannian modifications). In this section we will \emph{assume} that:
\begin{equation} \label{eq:multiplicity.one.assumption}
	\Theta^{n-1}(V, \cdot) = 1 \text{ on } \Sigma \triangleq \support \Vert V \Vert, \text{ which is a smooth minimal surface} \subset M \setminus \partial M.
\end{equation}
In other words, we assume that the limit $V$ is \emph{smooth} and that it occurs with \emph{multiplicity one}. (We are not assuming any bounds on the Morse index.)

\begin{rema} 
	We recall that this is \emph{automatically} the case when: (i) $3 \leq n \leq 7$ and each $u_i$ 	minimizes $E_{\eps_i}$ among compact perturbations in $M$ 
	(by \cite[Theorem 2]{HutchinsonTonegawa00}), or (ii) $n = 3$, $\limsup_i \ind(u_i) < \infty$, and $\Sigma$ carries no positive Jacobi fields (this follows from Theorem \ref{theo:bounded.index}). We emphasize that this section requires only the multiplicity one assumption \eqref{eq:multiplicity.one.assumption}, not (i) or (ii).
\end{rema}

The main goal of this section is to prove Theorem \ref{theo:index.lower.bounds}. Roughly, it says that the Morse index is \emph{upper semicontinuous}. Note that, in general, one only expects the index to be \emph{lower} semicontinuous. This has been recently confirmed in the work of Hiesmayr \cite{Hiesmayr}; see also Gaspar's generalization to one-sided limit surfaces \cite{Gaspar}. Upper semicontinuity hinges strongly on the multiplicity one assumption, as the following example suggests:
	
\begin{exam}\label{exam:upper.semi.fails.index}
	Let $(u_i, \eps_i)$, $i=1,2,\ldots$, $\lim_i \eps_i = 0$, be a sequence constructed by \cite{delPinoKowalczykWeiYang:interface} to converge, with multiplicity $\geq 2$, to a two-sided minimal surface $\Sigma$ in a closed Riemannian 3-manifold $(M^3, g)$ with positive Ricci curvature. Then, by Theorem \ref{theo:bounded.index}, $\liminf_i \ind(u_i) = \infty$, because $\Sigma$ cannot be stable and there there are no stable two-sided minimal surfaces in the presence of positive Ricci curvature.
\end{exam}

In order to study the semicontinuity of the Morse index, we need to obtain a detailed understanding of the convergence of the $u_i$ and their level sets to $\Sigma$. Somewhat surprisingly, the regularity estimates in Section \ref{sec:jacobi.toda.reduction} (or \cite[Section 15]{WangWei}) do not seem to suffice for our purposes. Instead, we must upgrade the estimates so that we have an explicit understanding of the $O(\eps^2)$ term in \eqref{eq:jacobi.toda}. We use an ansatz inspired by the work of del~Pino--Kowalczyk--Wei \cite{delPinoKowalczykWei}, although our setting is different: rather than having constructed $u$, we are \emph{given} an arbitrary solution $u$ converging with multiplicity one. This technique does not seem to have been previously considered in the context of regularity in Allen--Cahn.

\subsection{Improved convergence} \label{subsec:multiplicity.one.convg}

Note that by scaling $M$, we can arrange that \eqref{eq:sheets.sff.bound}-\eqref{eq:sheets.metric.bound} hold; we will do so without further remark in the sequel. Note that then, due Lemma \ref{lemm:multiplicity.one.lower.gradient.bound} below,  \eqref{eq:sheets.eps.bound}-\eqref{eq:sheets.enhanced.sff.grad.bound} hold as well. Thus, Section \ref{sec:jacobi.toda.reduction} applies (as does \cite[Section 15]{WangWei} in the flat setting).

\begin{lemm} \label{lemm:multiplicity.one.lower.gradient.bound}
	Let $U \subset \subset M \setminus \partial M$ be a neighborhood of $\Sigma$, and $\beta \in (0, 1)$. Then, for sufficiently large $i$, $\eps_i |\nabla u_i| \geq c > 0$ on $U \cap \{ |u_i| \leq 1-\beta \}$.
\end{lemm}
\begin{proof}
We argue by contradiction. If the result were false, we'd be able to pick a subsequence (labeled the same) along which there would exist $x_i \in U \cap \{ |u_i| \leq 1-\beta \}$ with $\eps_i |\nabla u_i(x_i)| \to 0$. After rescaling by $\eps_i^{-1}$ around $x_i$, the rescaled critical points $\widetilde{u}_i$ converge to a nontrivial critical point of $E_1$ on $\mathbf{R}^n$ with $|\widetilde{u}(0)| \leq 1-\beta$, $\nabla \widetilde{u}(0) = 0$. By the monotonicity formula (see  \cite[Section 3]{HutchinsonTonegawa00} and \cite[Appendix B]{Guaraco}) and multiplicity one convergence at the original scale, we see that the tangent cone at infinity of $\widetilde{u}$ is a multiplicity one plane. Hence, by  \cite[Theorem 11.1]{Wang:Allard} (cf.\ \cite[Theorem 3.6]{Mantoulidis}), $\widetilde{u}$ has flat level sets. This contradicts $|\widetilde{u}(0)| \leq 1-\beta$, $\nabla \widetilde{u}(0) = 0$.
\end{proof}

Combined with the multiplicity one analysis in \cite[Section 15]{WangWei} (cf. Section \ref{sec:jacobi.toda.reduction} and Remark \ref{rema:major.goal} above), we may argue as in the proof of Theorem \ref{theo:bounded.index} to conclude that $\Sigma=\supp V$ is a smooth two-sided embedded minimal hypersurface and the convergence of the level sets of $u_{i}$ to $\Sigma$ occurs in $C^{2,\theta}$. (Of course, convergence in the Hausdorff sense follows immediately from \cite[Theorem 1]{HutchinsonTonegawa00}.)

\begin{lemm} \label{lemm:multiplicity.one.convergence}
		If $U \subset \subset M \setminus \partial M$ is a neighborhood of $\Sigma$, and $\theta,\beta \in (0, 1)$, then $U \cap \{ u_i = t \}$ converges uniformly in $C^{2,\theta}$ to $\Sigma$, for every $t \in (-1+\beta, 1-\beta)$.
\end{lemm}	
\begin{proof}
	By Section \ref{sec:jacobi.toda.reduction}, it suffices to check that the level sets are bounded in $C^2$. One uses a blow-up argument again, as in the proof of Theorem \ref{theo:curvature.estimate}. Suppose that the enhanced second fundamental form weren't bounded. Pick $x_i \in U \cap \{ |u_i| \leq 1-\beta \}$ such that $\lambda_i \triangleq |\cA_i(x_i)|$ are within a factor of $\tfrac{1}{2}$ from $\sup_{U \cap \{|u_i| \leq 1-\beta\}} |\cA_i|$; thus, $\lambda_i \to \infty$. Note that $\limsup_i \lambda_i \eps_i < \infty$ by elliptic regularity. Moreover, we in fact have that $\limsup_i \lambda_i \eps_i = 0$ because (by \cite[Theorem 11.1]{Wang:Allard} and monotonicity) there are no nontrivial (i.e., nonconstant and nonheteroclinic) entire critical points of $E_1$ in $\mathbf{R}^n$ with a planar tangent cone at infinity. In particular, rescaling by $\lambda_i^{-1}$ around $x_i$, we get a sequence $(\widetilde{u}_i, \widetilde{\eps}_i)$ with $\widetilde{\eps}_i \to 0$ and uniformly bounded enhanced second fundamental form, $|\widetilde{\cA}_i(0)| = 1$, and which therefore converges to a $C^{1,1}$ minimal surface in $\mathbf{R}^n$. However, by monotonicity, this minimal surface is a plane; this contradicts $|\widetilde{\cA}_i(0)| = 1$ by Remark \ref{rema:major.goal}.
\end{proof}

Let's return to the notation and conventions used in Section \ref{sec:jacobi.toda.reduction}. Also, we drop the subscript $i$.

Because of the multiplicity one assumption, we have reasonably strong estimates on $\phi,h,$ and $H_{\Gamma}$; see \eqref{eq:discrepancy.function}. We will write $h$ for $\ve{h}$, $U$ for $U[\ve{h}]$, $\Gamma$ for $\Gamma_1$, and $d$ for $d_1$, since $Q=1$. We record the specialization of \eqref{eq:phi.c2a.estimate.full} and Lemma \ref{lemm:h.phi.comparison} here (cf.\ \cite[Section 15]{WangWei}, and \cite[Theorem 3.6]{Mantoulidis}):
\begin{equation}\label{eq:mult.one.initial.bds}
\Vert \phi \Vert_{C^{2,\theta}_{\eps}(\cM)} + \eps \Vert \Delta _{\Gamma} h - H_{\Gamma}\Vert_{C^{0,\theta}_{\eps}(\Gamma)} + \eps^{-1} \Vert h \Vert_{C^{2,\theta}_{\eps}(\Gamma)}\leq c' \eps^{2},
\end{equation}
where $\cM \triangleq \{X \in M : |d(X)| < 1\}$. 
As we have already indicated, we must upgrade our estimates for $\Delta_{\Gamma} h - H_{\Gamma}$ in \eqref{eq:mult.one.initial.bds} as well as determine the $O(\eps^{2})$ behavior of $\phi$. 

Let us work in Fermi coordinates around $\Gamma$ so as not to write the diffeomorphism $Z_{\Gamma}$ explicitly below. We will also denote $\Gamma_z \triangleq \{ X \in \cM : d(X) = z \}$, and will write $\overline{\mathbb{H}}$ for $\overline{\mathbb{H}}{}^{3 |\log \eps|}$. 

We can compute the equation for $\phi$ as follows. Using \eqref{eq:mean.curv.ddt.sff}, \eqref{eq:mean.curv.ddt.h}, \eqref{eq:mean.curv.ddt.laplace}, as well as \eqref{eq:sheets.eps.bound}-\eqref{eq:sheets.enhanced.sff.grad.bound}, \eqref{eq:approximate.heteroclinic.behavior}, and \eqref{eq:mult.one.initial.bds}, one computes (see \cite[(9.4)]{WangWei}) in $\cM$
\begin{align} \label{eq:mult.one.phi.eqn.prelim}
	\eps^2 \Delta_g \phi 
		& = \eps^{2} \Delta_{\Gamma_z} \phi + \eps^{2} H_{\Gamma_z}\partial_{z}\phi + \eps^2 \partial^{2}_{z}\phi \nonumber\\
		& = W'(u) - \eps^2 \Delta_{\Gamma_z} U - \eps^{2} H_{\Gamma_z} \partial_{z} U - \eps^2 \partial^{2}_{z} U \nonumber \\
		& = W'(U + \phi) - (W'(U) + O(\eps^3)) + \eps (\Delta_{\Gamma_z} h -  H_{\Gamma_z}) \cdot  \overline{\mathbb{H}}'(\eps^{-1}(z-h(y))) \nonumber \\
		& \qquad - |\nabla_{\Gamma_z} h|^{2} \cdot  \overline{\mathbb{H}}''(\eps^{-1}(z-h(y))) \nonumber \\
		& = W''(U)\phi +  \eps ((\Delta_{\Gamma} h - H_{\Gamma}) \circ \Pi_\Gamma) \cdot  \overline{\mathbb{H}}'(\eps^{-1}(z-h(y)))  \nonumber \\ 
		& \qquad +   \eps ((|\sff_{\Gamma}|^{2} +\ricc_g(\partial_{z},\partial_{z}))\circ \Pi_\Gamma) \cdot z \cdot \overline{\mathbb{H}}'(\eps^{-1}(z-h(y))) \nonumber \\
		& \qquad +  \eps^2 O(|z|) \cdot \overline{\mathbb{H}}'(\eps^{-1}(z-h(y))) + O( \eps^3).
\end{align}
By using \eqref{eq:mult.one.initial.bds}, \eqref{eq:mult.one.phi.eqn.prelim}, and the multiplicity one assumption, one may revisit \cite[Appendix B]{WangWei} and establish the following bounds. 
\begin{lemm}\label{lemm:mult.one.imp.hHeqn}
We can improve the estimate in \eqref{eq:mult.one.initial.bds} to $\eps \Vert \Delta_{\Gamma} h - H_{\Gamma}\Vert_{C^{0}(\Gamma)} \leq c' \eps^{3}$. 
\end{lemm}
\begin{proof}
Multiply \eqref{eq:mult.one.phi.eqn.prelim} by $\overline{\mathbb{H}}'(\eps^{-1}(z-h(y)))$ and integrate over $z \in [-\eta,\eta]$. We find (at $y \in \Sigma$ fixed)
\begin{align*}
& \int_{-\eta}^{\eta} (\eps^2 (\Delta_{\Gamma_z} \phi + H_{\Gamma_z}\partial_{z}\phi + \partial^{2}_{z}\phi )- W''(U)\phi) \cdot  \overline{\mathbb{H}}'(\eps^{-1}(z-h(y)))) \, dz\\
& =  \eps^2 (\energyunit-o(1)) (\Delta_{\Gamma} h - H_{\Gamma}) + \eps (|\sff_{\Gamma}|^{2} + \ricc_g(\partial_{z},\partial_{z})) \int_{-\eta}^{\eta} z  \overline{\mathbb{H}}'(\eps^{-1}(z-h(y)))^{2} \, dz\\
& \qquad + \int_{-\eta}^{\eta}  \eps^2 O(|z|) \cdot  \overline{\mathbb{H}}'(\eps^{-1}(z-h(y)))^{2}  \, dz + O(\eps^4)\\
& = \eps^2 (\energyunit - o(1)) (\Delta_{\Gamma} h - H_{\Gamma}) + O( \eps^{4}).
\end{align*}
We have used \eqref{eq:heteroclinic.expansion.ii} together with $\int_{-\infty}^{\infty} t \mathbb{H}'(t)^{2} dt = 0$ (which holds by parity).

Twice differentiating the orthogonality relation \eqref{eq:h.defn.orth} used to define $h$ (see Section \ref{subsec:approximate.solutions} and \cite[Appendix B]{WangWei}) and using \eqref{eq:mult.one.initial.bds}, we have
\[
\int_{-\eta}^{\eta}  \eps^2 (\Delta_{\Gamma_z} \phi) \cdot \overline{\mathbb{H}}'(\eps^{-1}(z-h(y))) \, dz = O( \eps^{4}). 
\]
From \eqref{eq:mult.one.initial.bds}, we have
\[
\int_{-\eta}^{\eta} \eps^2 H_{\Gamma_z} \partial_{z} \phi \cdot \overline{\mathbb{H}}'(\eps^{-1}(z-h(y))) \, dz = O( \eps^{5}).
\]
Finally, an integration by parts shows that
\begin{align*}
& \int_{-\eta}^{\eta} \left(\eps^2 \partial^{2}_{z} \phi \cdot  \overline{\mathbb{H}}'(\eps^{-1}(z-h(y))) -  W''(u) \phi \overline{\mathbb{H}}'(\eps^{-1}(z-h(y))) \right) \, dz \\
& = \int_{-\eta}^{\eta} \left(  \overline{\mathbb{H}}'''(\eps^{-1}(z-h(y))) - W''(u)  \overline{\mathbb{H}}'(\eps^{-1}(z-h(y))) \right) \phi \, dz.
\end{align*}
Using \eqref{eq:approximate.heteroclinic.behavior} here, combined with the previous expressions, we conclude the proof.
\end{proof}

Thus, returning to \eqref{eq:mult.one.phi.eqn.prelim} we find that in $\cM$, we have:
\begin{equation}\label{eq:mult.one.phi.upgrade}
	\eps^2 \Delta_g \phi - W''(U)\phi = \eps((|\sff_{\Gamma}|^{2} +\ricc_g(\partial_{z},\partial_{z}))\circ \Pi_\Gamma) \cdot z \cdot \overline{\mathbb{H}}'(\eps^{-1}(z-h(y))) + O(\eps^3).
\end{equation}
We have used the fact that $z \overline{\mathbb{H}}'(\eps^{-1}(z - h(y))) = O(\eps)$.

Observe that the right hand side of \eqref{eq:mult.one.phi.upgrade} is only bounded in $O(\eps^2)$. Thus, we expect this to represent the leading term of $\phi$, after inverting $ \eps^2 \Delta_g - W''(U)$. To make this precise, we first define (cf.\ \cite[Section 3.2]{delPinoKowalczykWei}) a function $\mathbb{J}(t)$ to be the unique bounded solution of the ODE
\begin{equation}\label{eq:ODE.for.J}
\mathbb{J}''(t) = W''(\mathbb{H}(t)) \mathbb{J}(t) + t \mathbb{H}'(t), \text{ with } \mathbb{J}(0) = 0. 
\end{equation}
Indeed, we even have the explicit expression (cf.\ \cite[p.\ 82]{delPinoKowalczykWei})
\[
\mathbb{J}(t) = \mathbb{H}'(t) \int_{0}^{t} \int_{-\infty}^{s} \tau \mathbb{H}'(s)^{-2} \, \mathbb{H}'(\tau)^{2} d\tau ds
\]
which shows that $\mathbb{J}$ is well defined and decays exponentially as $t\to\pm\infty$. It will be important in the sequel to observe that $\mathbb{J}(-t) = -\mathbb{J}(t)$, which follows from the parity of $\mathbb{H}(t)$ and either the uniqueness of solutions to the ODE, or the explicit integral expression.

Observe that $|\sff_{\Gamma}|^{2} +\ricc_g(\partial_{z},\partial_{z})$ converges to $|\sff_{\Sigma}|^{2} +\ricc_g(\nu,\nu)$ in $C^{0,\theta}$ because  $\Gamma$ converges to $\Sigma$ in $C^{2,\theta}$  by Lemma \ref{lemm:multiplicity.one.convergence}. We fix functions $V :\Gamma \to\RR$ with the property that $V$ still converges to $|\sff_{\Sigma}|^{2} +\ricc_g(\nu,\nu)$ in $C^{0}$ and $\Vert V \Vert_{C^{2}(\Gamma)} \leq C$. For definiteness we choose $V(y) = (|\sff_{\Sigma}|^{2} +\ricc_g(\nu,\nu))\circ \Pi_{\Sigma}(y)$, where $\Pi_{\Sigma}$ is the nearest point projection to $\Sigma$. 

We claim that $\eps^{2}V(y) \mathbb{J}(\eps^{-1}(z-h(y)))$ represents the leading order term in $\phi$. To this end, in $\cM$, we define a refined discrepancy function
\[ \widetilde{\phi}(y,z) \triangleq \phi(y,z) - \eps^{2}(V \circ \Pi_\Gamma)(y, z) \cdot \mathbb{J}(\eps^{-1}(z-h(y))). \]
We compute (using the $C^{2}$ bounds for $V$, as well as \eqref{eq:mult.one.initial.bds} and Lemma \ref{lemm:mult.one.imp.hHeqn}) that on $\cM$, we have
\begin{align*}
& \eps^{2} \Delta_g \widetilde{\phi} - W''(U) \widetilde{\phi} \\
& =  \eps ((|\sff_{\Gamma}|^{2} +\ricc_g(\partial_{z},\partial_{z}))\circ \Pi_\Gamma) \cdot z \cdot \overline{\mathbb{H}}'(\eps^{-1}(z-h(y))) \\
& \qquad - \eps^{2} (V \circ \Pi_\Gamma) \big[ \mathbb{J}''(\eps^{-1}(z-h(y))) - W''(U) \cdot \mathbb{J}(\eps^{-1}(z-h(y))) \big] + O(\eps^3)\\
& =  \eps \big[ (|\sff_{\Gamma}|^{2} +\ricc_g(\partial_{z},\partial_{z}))\circ \Pi_\Gamma - V\circ \Pi_\Gamma \big] \cdot z \cdot \overline{\mathbb{H}}'(\eps^{-1}(z-h(y))) \\
& \qquad - \eps^{2} \big[ W''(\mathbb{H}(\eps^{-1}(z-h(y)))) - W''(U) \big] (V\circ \Pi_\Gamma) \cdot \mathbb{J}(\eps^{-1}(z-h(y))) + O(\eps^3)\\
& = o(\eps^{2}).
\end{align*}
We again used that $z \overline{\mathbb{H}}'(\eps^{-1}(z-h(y))) = O(\eps)$ as well as the definition of $V$. We now use the defining property of $h$ to invert $\eps^2 \Delta_g - W''(U)$.

\begin{prop}\label{prop:mult.one.improved.phi.behav}
	We have that $\widetilde{\phi} = o(\eps^{2})$ on $\cM$. 
\end{prop}
\begin{proof}
For contradiction, suppose that $\lambda \triangleq  \sup_{\cM} |\widetilde{\phi}| \geq \gamma \eps^{2}$ for some $\gamma > 0$. Note that $\widetilde{\phi}$ is exponentially small at points that are uniformly bounded away from $\Gamma$, so it is clear that this supremum is achieved at some $X^* \in \cM$ with $d(X^*) \to 0$. We can assume that $\widetilde{\phi}(X^{*}) = \lambda$. Write $X^* = (y^*,z^*)$ in Fermi coordinates over $\Gamma$. We split the argument into two cases: (i) $\eps^{-1} |z^*|$ is uniformly bounded or (ii) $\eps^{-1}|z^*| \to \infty$. 

First we consider case (i). We can assume that $\eps^{-1} z^* \to z_{\infty}$. Define $\widehat{\phi}(\widehat{X}) = \lambda^{-1} \widetilde{\phi}(X^* + \eps \widehat{X})$, which, in blown up Fermi coordinates $\widehat{X} = (\widehat{y}, \widehat{z})$, satisfies:
\[
\Delta_{\widehat{g}} \widehat{\phi}(\widehat{y}, \widehat{z}) - W''(\overline{\mathbb{H}}(\eps^{-1}z^* + \widehat{z} - \eps^{-1}h(y^*+\eps \widehat{y}))) \widehat{\phi}(\widehat{y}, \widehat{z}) = o(1),
\]
for $\widehat{z} \in (-\eps^{-1} \eta, \eps^{-1} \eta)$ and $\widehat{y} \in \Sigma$, and where $\widehat{g}$ is converging smoothly to the Euclidean metric. Moreover, $\widehat{\phi}(0) = 1$ and $|\widehat{\phi}|$ is uniformly bounded on compact sets. Interior Schauder estimates yield uniform bounds for $\widehat{\phi}$ in $C^{1,\theta}_{\loc}$. Thus, $\widehat{\phi}$ converges in $C^{1}$ to a weak (and thus strong, by elliptic regularity) solution of 
\[ \Delta \widehat{\phi}(\widehat{y}, \widehat{z}) - W''({\mathbb{H}}(z_{\infty} + \widehat{z})) \widehat{\phi}(\widehat{y}, \widehat{z}) = 0 \]
on $\RR^{n-1}\times \RR$. By \cite[Lemma 3.7]{Pacard12} (see also \cite{PacardRitore03}), we have that
\[ \widehat{\phi}(\widehat{y}, \widehat{z}) = \rho {\mathbb{H}}'(z_{\infty} + \widehat{z}) \text{ for some } \rho \in\RR, \]
because $\widehat{\phi} \in L^{\infty}(\RR^{n-1}\times \RR)$. In fact, $\widehat{\phi}(0) = 1$ implies that $\rho = \mathbb{H}'(z_\infty)^{-1}$. At the original scale, write $X = (y, z)$ in Fermi coordinates over $\Gamma$. Then, for $K$ fixed sufficiently large, if $|z| \leq K\eps$, we have:
\[ \widetilde{\phi}(y, z) = \lambda \big[ \mathbb{H}'(z_\infty)^{-1} {\mathbb{H}}'(z_{\infty} + \eps^{-1} (z-z^*)) + o(1) \big]. \]
Therefore,
\begin{equation} \label{eq:mult.one.improved.phi.i}
	\phi(y, z) = \lambda \big[ \mathbb{H}'(z_\infty)^{-1} {\mathbb{H}}'(z_{\infty} + \eps^{-1} (z-z^*)) + o(1) \big] + \eps^2 V(y) \mathbb{J}(\eps^{-1}(z-h(y))).
\end{equation}
By estimating the exponential tail using \eqref{eq:heteroclinic.expansion.ii}, and then using the definition of $\phi$ and $h$, and also \eqref{eq:mult.one.initial.bds}, we have:
\begin{equation} \label{eq:mult.one.improved.phi.ii}
	\int_{-K\eps}^{K\eps} \phi(y, z) \cdot  \overline{\mathbb{H}}'(\eps^{-1}(z-h(y))) \, dz = O(\eps^3 e^{-\sqrt{2} K}).
\end{equation}
By parity ($\mathbb{H}'$ is even, $\mathbb{J}$ is odd) and similarly estimating an exponential tail we also have
\begin{equation} \label{eq:mult.one.improved.phi.iii}
	\int_{-K\eps}^{K\eps} \mathbb{J}(\eps^{-1}(z - h(y))) \cdot  \overline{\mathbb{H}}'(\eps^{-1}(z-h(y))) \, dz = O(\eps e^{-\sqrt{2} K}).
\end{equation}
Finally:
\begin{equation} \label{eq:mult.one.improved.phi.iv}
	\int_{-K\eps}^{K\eps} \mathbb{H}'(z_\infty+ \eps^{-1}(z - z^*)) \cdot  \overline{\mathbb{H}}'(\eps^{-1}(z - h(y))) \, dz \geq (\energyunit - O(e^{-\sqrt{2}K})) \eps.
\end{equation}
Altogether, \eqref{eq:mult.one.improved.phi.i}-\eqref{eq:mult.one.improved.phi.iv} imply $\lambda = \energyunit^{-1} O(\eps^2 e^{-\sqrt{2} K})$, which (for large $K$) contradicts our assumption that $\lambda \geq \gamma \eps^2$ for a fixed $\gamma > 0$. This is a contradiction, completing the proof of case (i).

We now turn to case (ii). The proof here is analogous (and simpler). By rescaling as above, we find a non-zero smooth function $\widehat{\phi} \in L^{\infty}(\RR^{n-1}\times \RR)$ solving $\Delta \widehat{\phi}  - W''(\pm 1) \widehat{\phi} = 0$. An integration by parts, using $W''(\pm 1) > 0$, shows that $\widehat{\phi} = 0$. This is a contradiction, completing the proof of case (ii). 
\end{proof}

\subsection{Relating the second variations and index upper semicontinuity}

We now can give the fundamental computation linking the index of $u$ as a critical point of $E_{\eps}$ with the index of $\Sigma$ as a critical point of area. Our argument is closely related to the proof of \cite[Lemma 9.2]{delPinoKowalczykWei}. Recall from \eqref{eq:second.var.AC} that the second variation of $E_{\eps}$ is given by
\begin{equation*}
\cQ_{u}(\zeta,\psi) \triangleq \delta^{2}E_{\eps}[u]\{\zeta,\xi\}  = \int_{M} \left( \eps \langle\nabla \zeta,\nabla \xi \rangle + \frac{W''(u)}{\eps} \zeta \xi\right) \, d\mu_{g}.
\end{equation*}
Similarly, we recall that the second variation of area at $\Sigma$ is given by
\[
\cQ_{\Sigma}(\zeta,\xi) \triangleq \delta^{2} \operatorname{Area}[\Sigma]\{\zeta,\xi\} = \int_{\Sigma} \left( \langle\nabla\zeta,\nabla \xi \rangle - (|\sff_{\Sigma}|^{2} + \ricc_g(\nu,\nu))\zeta\xi \right) \, d\mu_\Sigma
\]
\begin{lemm}\label{lemm:index.Hprime.comp}
For $f \in C^{2}(\Sigma)$, setting
\[ \psi(y, z) = f(y) \cdot  \overline{\mathbb{H}}'(\eps^{-1}(z-h(y))) \]
for $(y,z)$ Fermi coordinates with respect to $\Gamma$ (the nodal set of $u$), and $\psi = 0$ far from $\Gamma$, we have that 
\begin{align*}
\cQ_{u}(\psi,\psi) 
	& = \eps^{2}(\energyunit-o(1)) \int_{\Gamma} \left( |\nabla_{\Gamma} f|^{2} - \left( (|\sff_{\Gamma}|^{2} +\ricc_{g}(\partial_{z},\partial_{z}))\circ\Pi_\Gamma\right) f^{2} \right) d\mu_\Gamma \\
	& \qquad + o(\eps^{2}) \int_{\Gamma} \left( |\nabla_{\Gamma} f|^{2} + f^{2}\right) d\mu_\Gamma.
\end{align*}
Here, $\Pi_\Gamma$ denotes the nearest point projection onto $\Gamma$.
\end{lemm}
\begin{proof}
We compute, using \eqref{eq:approximate.heteroclinic.behavior}:
\begin{align*}
& \cQ_{u}(\psi,\psi) \\
& = \int_{M} \left( -\eps \psi \Delta_{g} \psi + \eps^{-1} W''(u) \psi^{2} \right)\, d\mu_{g}\\
& = \int_{-\eta}^{\eta} \int_{\Gamma}\left( -\eps \psi \Delta_{\Gamma_{z}} \psi - \eps H_{\Gamma_{z}} \psi \partial_{z} \psi - \eps \psi \partial^{2}_{z} \psi + \eps^{-1} W''(u) \psi^{2} \right)\, d\mu_{g_{z}} dz\\
& = \int_{-\eta}^{\eta} \int_{\Gamma}\left( \eps |\nabla_{\Gamma_{z}} \psi|^{2} - \eps H_{\Gamma_{z}} \psi \partial_{z} \psi - \eps \psi \partial^{2}_{z} \psi + \eps^{-1} W''(u) \psi^{2} \right)\, d\mu_{g_{z}} dz\\
& = \int_{-\eta}^{\eta} \int_{\Gamma}\Big( \eps |(\nabla_{\Gamma_{z}}f) \cdot  \overline{\mathbb{H}}'(\eps^{-1}(z-h(y))) - \eps^{-1}f(y)(\nabla_{\Gamma_{z}}h) \cdot  \overline{\mathbb{H}}''(\eps^{-1}(z-h(y))) |^{2}\\
& \qquad -  H_{\Gamma_{z}} f(y)^{2} \cdot  \overline{\mathbb{H}}'(\eps^{-1}(z-h(y)))  \overline{\mathbb{H}}''(\eps^{-1}(z-h(y)))\\
& \qquad - \eps^{-1} f(y)^{2} \cdot  \overline{\mathbb{H}}'(\eps^{-1}(z-h(y))) \overline{\mathbb{H}}'''(\eps^{-1}(z-h(y))) \\
& \qquad + \eps^{-1} W''(u) f(y)^{2} \cdot \overline{\mathbb{H}}'(\eps^{-1}(z-h(y)))^{2} \Big)\, d\mu_{g_{z}} dz.
\end{align*}
Using, additionally, \eqref{eq:mult.one.initial.bds}, our $C^2$ bounds on $\Gamma$, \eqref{eq:mean.curv.ddt.metric}, \eqref{eq:mean.curv.ddt.sff}, and \eqref{eq:mean.curv.ddt.h}:
\begin{align*}
& \cQ_u(\psi, \psi) \\
& = \int_{-\eta}^{\eta} \int_{\Gamma}\Big( \eps |(\nabla_{\Gamma_{z}}f) \cdot  \overline{\mathbb{H}}'(\eps^{-1}(z-h(y))) - \eps^{-1}f(y)(\nabla_{\Gamma_{z}}h) \cdot  \overline{\mathbb{H}}''(\eps^{-1}(z-h(y))) |^{2}\\
& \qquad -   H_{\Gamma} f(y)^{2} \cdot \overline{\mathbb{H}}'(\eps^{-1}(z-h(y)))  \overline{\mathbb{H}}''(\eps^{-1}(z-h(y)))\\
& \qquad +  \left( (|\sff_{\Gamma}|^{2} +\ricc_{g}(\partial_{z},\partial_{z})) \circ \Pi_\Gamma \right) f(y)^{2} \cdot z \cdot \overline{\mathbb{H}}'(\eps^{-1}(z-h(y)))  \overline{\mathbb{H}}''(\eps^{-1}(z-h(y)))\\
& \qquad + \eps^{-1} (W''(U+\phi)-W''(U)) f(y)^{2} \cdot \overline{\mathbb{H}}'(\eps^{-1}(z-h(y)))^{2} \Big)\, d\mu_{g_{z}} dz\\
& \qquad + O(\eps^{3}) \int_{\Gamma} f(y)^{2} d\mu_\Gamma \\
& = \int_{-\eta}^{\eta} \int_{\Gamma}\Big( \eps |\nabla_{\Gamma_{z}}f|^{2} \cdot  \overline{\mathbb{H}}'(\eps^{-1}(z-h(y)))^{2}\\
& \qquad +  \left( (|\sff_{\Gamma}|^{2} +\ricc_{g}(\partial_{z},\partial_{z})) \circ \Pi_\Gamma \right) f(y)^{2} \cdot z \cdot  \overline{\mathbb{H}}'(\eps^{-1}(z-h(y)))  \overline{\mathbb{H}}''(\eps^{-1}(z-h(y)))\\
& \qquad + \eps^{-1} W'''(U) \phi  f(y)^{2} \cdot \overline{\mathbb{H}}'(\eps^{-1}(z-h(y)))^{2} \Big)\, d\mu_{g_{z}} dz\\
& \qquad + O(\eps^{3}) \int_{\Gamma} (|\nabla_{\Gamma}f|^{2} + f^{2}) \, d\mu_\Gamma \\
& = \eps^{2}(\energyunit-o(1)) \int_{\Gamma} \left( |\nabla_{\Gamma} f|^{2} - \left( (|\sff_{\Gamma}|^{2} +\ricc_{g}(\partial_{z},\partial_{z})) \circ \Pi_\Gamma \right) f(y)^{2} \right) d\mu_\Gamma \\
& \qquad + o(\eps^{2}) \int_{\Gamma} \left( |\nabla_{\Gamma} f|^{2} + f^{2}\right) d\mu_\Gamma.
\end{align*}
In the final equality, we have used 
\[
\int_{-\infty}^{\infty} t \mathbb{H}'(t) \mathbb{H}''(t) dt = \tfrac 12 \int_{-\infty}^{\infty} t \tfrac{d}{dt} \mathbb{H}'(t)^{2}  dt = - \tfrac 12 \energyunit
\]
on the second term. We have also used $\phi = \eps^{2} V(y) \mathbb{J}(\eps^{-1}(z-h(y))) +o(\eps^{2})$, $V(y) = (|\sff_{\Gamma}|^{2} +\ricc_{g}(\partial_{z},\partial_{z}))\circ\Pi_\Gamma + o(1)$, and the following identity that follows by differentiating \eqref{eq:ODE.for.J} once and integrating by parts:
\begin{align*}
\int_{-\infty}^{\infty} W'''(\mathbb{H}(t)) \mathbb{J}(t) \mathbb{H}'(t)^{2} dt & = \int_{-\infty}^{\infty} \left( \mathbb{J}'''(t)\mathbb{H}'(t) - W''(\mathbb{H}(t)) \mathbb{J}'(t)\mathbb{H}'(t) - \mathbb{H}'(t)^{2} - t\mathbb{H}'(t)\mathbb{H}''(t)
 \right) dt\\
& = \int_{-\infty}^{\infty} \left( \mathbb{J}'(t)\mathbb{H}'''(t) - W''(\mathbb{H}(t)) \mathbb{J}'(t)\mathbb{H}'(t) - \mathbb{H}'(t)^{2} - t\mathbb{H}'(t)\mathbb{H}''(t)
 \right) dt\\
& = \int_{-\infty}^{\infty} \left(  - \mathbb{H}'(t)^{2} - t\mathbb{H}'(t)\mathbb{H}''(t)
 \right) dt = -\tfrac 12 \energyunit.
\end{align*}
This completes the proof. 
\end{proof}

Let $\Omega$ denote the $\eta$-tubular neighborhood of $\Gamma$ and consider the restriction $\cQ_u^\Omega$ of $\cQ_{u}$ to $\Omega$:
\[
\cQ_{u}^{\Omega}(\zeta,\xi) \triangleq \delta^2 (E_\eps \restr \Omega)[u]\{ \zeta, \xi \} = \int_{\Omega} \left( \eps \langle\nabla \zeta,\nabla \xi \rangle + \frac{W''(u)}{\eps} \zeta\xi\right) \, d\mu_{g}, \; \zeta, \xi \in C^{\infty}(\Omega).
\]

Consider $w \in C^{\infty}(\Omega)$.
We decompose $w$ as
\begin{equation} \label{eq:index.w.decomposition}
	w(y,z) = f(y) \cdot  \overline{\mathbb{H}}'(\eps^{-1}(z-h(y))) + w^{\perp}(y,z)
\end{equation}
where
\begin{equation}\label{eq:index.w.perp.defn}
\int_{-\eta}^{\eta} w^{\perp}(y,z) \overline{\mathbb{H}}'(\eps^{-1}(z-h(y)))  \, dz = 0.
\end{equation}
It is useful to write
\begin{equation} \label{eq:index.w.parallel}
	\psi(y,z) = f(y) \cdot  \overline{\mathbb{H}}'(\eps^{-1}(z-h(y))).
\end{equation}

Note that
\begin{align}\label{eq:index.decomp.L2.norm}
\int_{\Omega} w^{2} d\mu_{g} 
	& = \int_{-\eta}^{\eta} \int_{\Gamma}  f^{2} \cdot  \overline{\mathbb{H}}'(\eps^{-1}(z-h(y)))^{2} \, d\mu_{g_{z}} \, dz + \int_{\Omega} (w^{\perp})^{2} \, d\mu_{g} \nonumber \\
	& \qquad + 2 \int_{-\eta}^{\eta} \int_{\Gamma} f w^{\perp} \cdot \overline{\mathbb{H}}'(\eps^{-1}(z-h(y))) \, d\mu_{g_{z}} \, dz \nonumber \\
	& = (1+o(1))\int_{-\eta}^{\eta} \int_{\Gamma}  f^{2} \cdot  \overline{\mathbb{H}}'(\eps^{-1}(z-h(y)))^{2} \, d\mu_\Gamma \, dz + (1+o(1))\int_{\Omega} (w^{\perp})^{2} \, d\mu_{g} \nonumber \\
	& = \eps (\energyunit-o(1)) \int_{\Gamma} f ^{2}  \, d\mu_\Gamma + (1+o(1))\int_{\Omega} (w^{\perp})^{2} \, d\mu_{g}.
\end{align}

We now use this decomposition to estimate $\cQ^{\Omega}_{u}(w,w)$. 
\begin{lemm}\label{lemm:index.perp.comp}
For $w^{\perp}$ as in \eqref{eq:index.w.perp.defn}, there is $\gamma>0$ so that for $\eps>0$ sufficiently small,
\[
\cQ_{u}^{\Omega}(w^{\perp},w^{\perp}) \geq \gamma \int_{\Omega} \eps |\nabla w^{\perp}|^{2} + \eps^{-1} (w^{\perp})^{2} \, d\mu_{g}.
\]
\end{lemm}
\begin{proof}
Recall that there is some $\gamma = \gamma(W)>0$ so that if $f(t)$ satisfies $\int_{-\infty}^{\infty}f(t) \mathbb{H}'(t) dt =0$, then
\[
\int_{-\infty}^{\infty} f'(t)^{2} + W''(\mathbb{H}(t)) f(t)^{2} \, dt \geq 4\gamma \int_{-\infty}^{\infty} f'(t)^{2} + f(t)^{2} \, dt.
\]
(See, e.g.\ \cite[(9.28)]{delPinoKowalczykWei}.) A change of variables and a compactness argument imply that
\[
\int_{-\eta}^{\eta} \eps (\partial_{z}w^{\perp}(y,z))^{2} + \eps^{-1}W''(U) (w^{\perp}(y,z))^{2} \, dz \geq 3 \gamma \int_{-\eta}^{\eta} \eps (\partial_{z}w^{\perp}(y,z))^{2} + \eps^{-1}(w^{\perp}(y,z))^{2} \, dz
\]
as long as $\eps>0$ is sufficiently small. From this, and \eqref{eq:mult.one.initial.bds}, we find
\begin{align*}
\cQ_{u}^{\Omega}(w^{\perp},w^{\perp}) & = \int_{-\eta}^{\eta} \int_{\Gamma} \left( \eps |\nabla_{\Gamma_{z}}w^{\perp}|^{2} + \eps (\partial_{z}w^{\perp})^{2} + \eps^{-1} W''(u) (w^{\perp})^{2} \right)d\mu_{g_{z}} dz\\
& = \int_{-\eta}^{\eta} \int_{\Gamma} \left( \eps |\nabla_{\Gamma_{z}}w^{\perp}|^{2} + \eps (\partial_{z}w^{\perp})^{2} + \eps^{-1} W''(U) (w^{\perp})^{2} \right)d\mu_{g_{z}} dz\\
& \qquad + O(\eps) \int_{\Omega} (w^{\perp})^{2} d\mu_{g} \\
& \geq 2\gamma \int_{-\eta}^{\eta} \int_{\Gamma} \left( \eps (\partial_{z}w^{\perp})^{2} + \eps^{-1} (w^{\perp})^{2} \right)d\mu_\Gamma dz\\
& \qquad + \int_{-\eta}^{\eta} \int_{\Gamma}  \eps |\nabla_{\Gamma_{z}}w^{\perp}|^{2} d\mu_{g_{z}} dz + O(\eps) \int_{\Omega} (w^{\perp})^{2} d\mu_{g} \\
& \geq \gamma \int_{\Omega} \eps |\nabla w^{\perp}|^{2} + \eps^{-1}(w^{\perp})^{2} \, d\mu_{g}.
\end{align*}
This completes the proof. 
\end{proof}
\begin{lemm}\label{lemm:index.mixed.terms}
For $\psi$, $f$, $w^{\perp}$ as in \eqref{eq:index.w.decomposition}-\eqref{eq:index.w.parallel}, we have
\begin{align*}
 \cQ_{u}^{\Omega}(\psi, w^{\perp}) & \geq - o(\eps^{2}) \int_{\Gamma} |\nabla_\Gamma f|^{2} + f^{2}\, d\mu_\Gamma -  o(1) \int_{\Omega} \eps |\nabla w^{\perp}|^{2} + \eps^{-1} (w^{\perp})^{2} \, d\mu_{g}.
\end{align*}
\end{lemm}
\begin{proof}
Repeatedly using \eqref{eq:heteroclinic.expansion.ii}, \eqref{eq:heteroclinic.expansion.iii},  \eqref{eq:mult.one.initial.bds}, Lemma \ref{lemm:mult.one.imp.hHeqn},  \eqref{eq:mean.curv.ddt.h}, \eqref{eq:mean.curv.ddt.grad}:
\begin{align*}
& \cQ_{u}^{\Omega}(\psi,w^{\perp}) \\
& = \int_{\Omega} \left( - \eps (\Delta_{g}\psi) w^{\perp} + \eps^{-1} W''(u) \psi w^{\perp} \right) d\mu_{g}\\
& = \int_{-\eta}^{\eta} \int_{\Gamma} \Big( - \eps (\Delta_{\Gamma_{z}} \psi) w^{\perp} - H_{\Gamma_{z}} f w^{\perp} \cdot \overline{\mathbb{H}}''(\eps^{-1}(z-h(y))) \\
& \qquad - \eps^{-1} f w^{\perp} \cdot \overline{\mathbb{H}}'''(\eps^{-1}(z-h(y))) + \eps^{-1} W''(u) \psi w^{\perp} \Big) d\mu_{g_{z}}dz \\
& = \int_{-\eta}^{\eta} \int_{\Gamma} \Big( - \eps (\Delta_{\Gamma_{z}}f) w^{\perp} \cdot \overline{\mathbb{H}}'(\eps^{-1}(z-h(y))) + 2  \langle \nabla_{\Gamma_{z}}f,\nabla_{\Gamma_{z}} h\rangle w^{\perp} \cdot \overline{\mathbb{H}}''(\eps^{-1}(z-h(y))) \\
& \qquad + (\Delta_{\Gamma_{z}}h - H_{\Gamma_{z}}) f w^{\perp} \cdot \overline{\mathbb{H}}''(\eps^{-1}(z-h(y))) \\
& \qquad - \eps^{-1} f w^\perp |\nabla_{\Gamma_z} h|^2 \cdot  \overline{\mathbb{H}}'''(\eps^{-1}(z - h(y))) \\
& \qquad  - \eps^{-1} f w^{\perp} \cdot \overline{\mathbb{H}}'''(\eps^{-1}(z-h(y))) + \eps^{-1} W''(u) \psi w^{\perp}  \Big) d\mu_{g_z} dz \\
& = \int_{-\eta}^{\eta} \int_{\Gamma} \Big( - \eps (\Delta_{\Gamma_{z}}f) w^{\perp} \cdot \overline{\mathbb{H}}'(\eps^{-1}(z-h(y))) + 2  \langle \nabla_{\Gamma_{z}}f,\nabla_{\Gamma_{z}} h\rangle w^{\perp} \cdot \overline{\mathbb{H}}''(\eps^{-1}(z-h(y))) \\
& \qquad + (\Delta_{\Gamma_{z}}h - H_{\Gamma_{z}}) f w^{\perp} \cdot \overline{\mathbb{H}}''(\eps^{-1}(z-h(y))) \\
& \qquad +  \eps^{-1}(W''(U+\phi) - W''(U) + O(\eps^3)) \psi w^\perp \Big) d\mu_{g_z} dz \\
& = \int_{-\eta}^{\eta} \int_{\Gamma} \Big( - \eps (\Delta_{\Gamma_{z}}f) w^{\perp} \cdot \overline{\mathbb{H}}'(\eps^{-1}(z-h(y))) + 2  \langle \nabla_{\Gamma_{z}}f,\nabla_{\Gamma_{z}} h\rangle w^{\perp} \cdot \overline{\mathbb{H}}''(\eps^{-1}(z-h(y))) \\
& \qquad + (\Delta_{\Gamma_{z}}h - H_{\Gamma_{z}}) f w^{\perp} \cdot \overline{\mathbb{H}}''(\eps^{-1}(z-h(y)))) d\mu_{g_z} dz -  O(\eps) \int_\Omega |f w^\perp| \, d\mu_g. \\
& = \int_{-\eta}^\eta \int_\Gamma \Big( \eps \langle \nabla_{\Gamma_z} f, \nabla_{\Gamma_z} w^\perp \rangle \cdot \overline{\mathbb{H}}'(\eps^{-1}(z - h(y))) + \langle \nabla_{\Gamma_z} f, \nabla_{\Gamma_z} h \rangle w^\perp \cdot \overline{\mathbb{H}}''(\eps^{-1}(z-h(y)))  \\
& \qquad + (\Delta_{\Gamma_z} h - H_{\Gamma_z}) f w^\perp \cdot \overline{\mathbb{H}}''(\eps^{-1}(z-h(y))) \Big) d\mu_{g_z} dz - O(\eps) \int_\Omega |f w^\perp| \, d\mu_g \\
& = \int_{-\eta}^\eta \int_\Gamma \Big( \eps \langle \nabla_{\Gamma_z} f, \nabla_{\Gamma_z} w^\perp \rangle \cdot \overline{\mathbb{H}}'(\eps^{-1}(z - h(y))) + (\Delta_{\Gamma_z} h - H_{\Gamma_z}) f w^\perp \cdot \overline{\mathbb{H}}''(\eps^{-1}(z-h(y))) \Big) d\mu_{g_z} dz \\
& \qquad - O(\eps) \int_\Omega |f w^\perp| \, d\mu_g - O(\eps^2) \int_\Omega |\nabla_{\Gamma_z} f| |w^\perp| \, d\mu_g \\
& = \int_{-\eta}^\eta \int_\Gamma \Big( \eps \langle \nabla_{\Gamma_z} f, \nabla_{\Gamma_z} w^\perp \rangle\cdot \overline{\mathbb{H}}'(\eps^{-1}(z - h(y))) + (\Delta_{\Gamma_z} h - H_{\Gamma_z}) f w^\perp \cdot \overline{\mathbb{H}}''(\eps^{-1}(z-h(y))) \Big) d\mu_{g_z} dz \\
& \qquad - o(\eps^3) \int_\Gamma |\nabla_{\Gamma} f|^2 + f^2 \, d\mu_\Gamma - o(1) \int_\Omega \eps^{-1} (w^\perp)^2 \, d\mu_g.
\end{align*}
In the last inequality, we estimated, using Cauchy--Schwarz, $2 ab \leq \eps^{1-\sigma} a^2 + \eps^{-1+\sigma} b^2$, for $\sigma \in (0, 1)$, $(a, b) = (|f|, |w^\perp|)$, $(|\nabla_{\Gamma_z} f|, |w^\perp|)$. We can further estimate, using \eqref{eq:mult.one.initial.bds},  \eqref{eq:mean.curv.ddt.h}, \eqref{eq:mean.curv.ddt.grad}, and \eqref{eq:mean.curv.ddt.laplace}:
\[ \Delta_{\Gamma_z} h - H_{\Gamma_z} = \Delta_\Gamma h - H_\Gamma + O(|z|) = O(\eps + |z|), \]
and
\[ \langle \nabla_{\Gamma_z} f, \nabla_{\Gamma_z} w^\perp \rangle = \langle \nabla_\Gamma f, \nabla_\Gamma w^\perp \rangle + O(\eps + |z|) |\nabla_\Gamma f| |\nabla_\Gamma w^\perp|. \]
By the same Cauchy--Schwarz estimate applied to $(a, b) = (|f|, |w^\perp|)$, $(|\nabla_\Gamma f|, |\nabla_\Gamma w^\perp|)$ we get:
\begin{align*}
	\cQ_u^\Omega(\psi, w^\perp) 
		& = \int_{-\eta}^\eta \int_\Gamma \eps \langle \nabla_\Gamma f, \nabla_\Gamma w^\perp \rangle\cdot \overline{\mathbb{H}}'(\eps^{-1}(z - h(y))) \, d\mu_{g_z} dz \\
		& \qquad - o(\eps^2) \int_\Gamma |\nabla_\Gamma f|^2 + f^2 \, d\mu_\Gamma - o(1) \int_\Omega \eps |\nabla w^\perp|^2 + \eps^{-1} (w^\perp)^2 \, d\mu_g.
\end{align*}
Estimating $|d\mu_{g_z} - d\mu_\Gamma| = O(|z|) d\mu_\Gamma$ and using the same Cauchy--Schwarz inequality, we deduce:
\begin{align*}
	\cQ_u^\Omega(\psi, w^\perp) 
		& = \int_{-\eta}^\eta \int_\Gamma \eps \langle \nabla_\Gamma f, \nabla_\Gamma w^\perp \rangle\cdot \overline{\mathbb{H}}'(\eps^{-1}(z - h(y))) \, d\mu_\Gamma dz \\
		& \qquad - o(\eps^2) \int_\Gamma |\nabla_\Gamma f|^2 + f^2 \, d\mu_\Gamma - o(1) \int_\Omega \eps |\nabla w^\perp|^2 + \eps^{-1} (w^\perp)^2 \, d\mu_g \\
		& = \int_\Gamma \int_{-\eta}^\eta \eps \langle \nabla_\Gamma f, \nabla_\Gamma w^\perp \rangle\cdot \overline{\mathbb{H}}'(\eps^{-1}(z - h(y))) \, dz \, d\mu_\Gamma\\
		& \qquad - o(\eps^2) \int_\Gamma |\nabla_\Gamma f|^2 + f^2 \, d\mu_\Gamma - o(1) \int_\Omega \eps |\nabla w^\perp|^2 + \eps^{-1} (w^\perp)^2 \, d\mu_g.
\end{align*}
Since $\langle \nabla_\Gamma f, \nabla_\Gamma w^\perp \rangle = g^{ij}_\Gamma \partial_{y_i} f \partial_{y_j} w^\perp$, whose first two factors are independent of $z$, we can use
\[
\int_{-\eta}^{\eta} \partial_{y_j} w^{\perp}\overline{\mathbb{H}}'(\eps^{-1}(z-h(y))) \, dz =  \eps^{-1} \int_{-\eta}^{\eta} (\partial_{y_j} h) w^{\perp} \overline{\mathbb{H}}''(\eps^{-1}(z-h(y))) \, dz,
\]
which follows from differentiating \eqref{eq:index.w.perp.defn} once horizontally. We thus have:
\begin{align*}
	\cQ_u^\Omega(\psi, w^\perp)
		& = \int_{-\Gamma} \int_{-\eta}^\eta \langle \nabla_\Gamma f, \nabla_\Gamma h \rangle w^\perp \cdot \overline{\mathbb{H}}''(\eps^{-1}(z-h(y))) \, dz \, d\mu_\Gamma\\
		& \qquad - o(\eps^2) \int_\Gamma |\nabla_\Gamma f|^2 + f^2 \, d\mu_\Gamma - o(1) \int_\Omega \eps |\nabla w^\perp|^2 + \eps^{-1} (w^\perp)^2 \, d\mu_g.
\end{align*}
This completes the proof, since we're already estimated terms of this form with the correct error term.
\end{proof}

\begin{lemm}\label{lemm:first.eig.stability.Q.low.bd}
There is $\sigma =\sigma(M,g,W,\Sigma)>0$ so that for $\eps > 0$ sufficiently small and any $w \in C^{\infty}(\Omega)$, we have
\[
\cQ^{\Omega}_{u}(w,w) \geq - \eps \sigma \int_{\Omega} w^{2} d\mu_{g}.
\]
\end{lemm}
\begin{proof}
Because $\Gamma$ converges to $\Sigma$ in $C^{2,\theta}$, we find that for $\delta=\delta(M,g,\Sigma) \in (0, 1)$ and $\eps>0$ sufficiently small, we have that 
\[
\int_{\Gamma} |\nabla_{\Gamma} f|^{2} - \left( (|\sff_{\Gamma}|^{2} +\ricc_{g}(\partial_{z},\partial_{z}))\circ\Pi_\Gamma \right) f^{2} d\mu_\Gamma \geq \int_{\Gamma} \delta |\nabla_{\Gamma}f|^{2} -\delta^{-1}f^{2} \, d\mu_\Gamma. 
\]
Thus, using \eqref{eq:index.w.decomposition}-\eqref{eq:index.w.parallel}, Lemmas \ref{lemm:index.Hprime.comp}, \ref{lemm:index.perp.comp}, and \ref{lemm:index.mixed.terms}, we find that
\begin{align*}
\cQ^{\Omega}_{u}(w,w) & = \cQ^{\Omega}_{u}(\psi,\psi) + \cQ_u^{\Omega}(w^{\perp},w^{\perp}) + 2 \cQ^{\Omega}_{u}(\psi,w^{\perp})\\
& \geq \eps^{2} (\energyunit - o(1)) \int_{\Gamma} \delta |\nabla_{\Gamma} f|^{2} - \delta^{-1} f^{2} \, d\mu_\Gamma + \gamma \int_{\Omega} \eps |\nabla w^{\perp}|^{2} + \eps^{-1} (w^{\perp})^{2} \, d\mu_{g}\\
& \qquad - o(\eps^{2}) \int_{\Gamma} |\nabla f|^{2} + f^{2}\, d\mu_\Gamma - o(1) \int_{\Omega} \eps |\nabla w^{\perp}|^{2} + \eps^{-1} (w^{\perp})^{2} \, d\mu_{g}\\
& \geq - \eps^{2} \delta^{-1}(\energyunit-o(1)) \int_{\Gamma} f^{2} d\mu_\Gamma \geq - \eps \delta^{-1} (1+o(1)) \int_{\Omega} w^{2} d\mu_{g} .
\end{align*}
In the last inequality we used \eqref{eq:index.decomp.L2.norm}. This completes the proof. 
\end{proof}

We are now able to prove the main theorem. In what follows:
\begin{itemize}
	\item $\ind(u)$, $\nul(u)$ denote the index and nullity of the second variation of Allen--Cahn energy functional (see \eqref{eq:second.var.AC}), and 
	\item $\ind(\Sigma)$, $\nul(\Sigma)$ denote the index and nullity of the second variation of the area functional for the limiting multiplicity one smooth minimal surface (recall  \eqref{eq:multiplicity.one.assumption}).
\end{itemize}
For simplicity, we will assume that $\partial M = \emptyset$, although we expect that the general strategy used here should extend to Dirichlet or Neumann boundary conditions with appropriate modifications.

\begin{theo} \label{theo:index.lower.bounds}
	If $(M^n, g)$, $u$, $\Sigma$ are as above, and $\partial M = \emptyset$, then, for sufficiently small $\eps > 0$,
	\[ \ind(\Sigma) + \nul(\Sigma) \geq \ind(u) + \nul(u). \]
\end{theo}
\begin{proof}
	For brevity, let's set $I_\Sigma \triangleq \ind(\Sigma) + \nul(\Sigma)$, $I_0 \triangleq \ind(u) + \nul(u)$.	First, we show:

	\begin{clai}
		There are smooth functions $f_1, \ldots, f_{I_{\Sigma}} : \Gamma \to \RR$ and a constant $\delta > 0$ so that if $f \in C^{1}(\Gamma)$ satisfies $\langle f, f_i \rangle_{L^2(\Gamma)} = 0$ for all $i = 1, \ldots, I_\Sigma$, then
		\begin{equation}\label{eq:ind.gamma.pos.dir}
			\cQ_{\Gamma}(f,f)  \geq  \delta \int_{\Gamma} |\nabla_{\Gamma} f|^{2} + f^{2} \, d\mu_\Gamma.
		\end{equation}
	\end{clai}
	\begin{proof}[Proof of claim]	
		Because the nodal set $\Gamma$ converges to $\Sigma$ in $C^{2,\theta}$ (by Lemma \ref{lemm:multiplicity.one.convergence}), it is not hard to see that is a uniform lower bound $\nu > 0$ for the first \emph{positive} eigenvalue of the second variation of area of $\Gamma$. Take $f_1, \ldots, f_{I_\Sigma}$ to be the first $I_\Sigma$ eigenfunctions of $\cQ_\Gamma$. Then:
		\[ \cQ_\Gamma(f, f) = \int_\Gamma |\nabla_\Gamma f|^2 - (|\sff_\Gamma|^2 + \ricc_g(\partial_z, \partial_z)) f^2 \, d\mu_\Gamma \geq \nu \int_\Gamma f^2 \, d\mu_\Gamma, \]
		for $f \in C^1(\Gamma)$, $\langle f, f_1 \rangle_{L^2(\Gamma)} = \ldots = \langle f, f_{I_\Sigma} \rangle_{L^2(\Gamma)} = 0$. If $|\sff_\Gamma|^2 + \ricc_g(\partial_z, \partial_z)\leq C$, then
		\[ \tfrac{\nu}{2C} \cQ_\Gamma(f, f) = \tfrac{\nu}{2C} \int_\Gamma |\nabla_\Gamma f|^2 - (|\sff_\Gamma|^2 + \ricc_g(\partial_z, \partial_z)) f^2 \, d\mu_\Gamma \geq \int_\Gamma \tfrac{\nu}{2C} |\nabla_\Gamma f|^2 - \tfrac{\nu}{2} f^2 \, d\mu_\Gamma. \]
		The claim follows by adding these two inequalities.
	\end{proof}

	We define the linear functional $\Pi_\eps : L^2(M) \to L^2(\Gamma)$:
	\[ \Pi_\eps(w)(y) \triangleq \eps^{-1} \int_{-\eta}^\eta w(y, z) \cdot \overline{\mathbb{H}}'(\eps^{-1}(z-h(y))) \, dz, \]
	and another linear functional $\cI_\Gamma : C^1(\Gamma) \to \RR^{I_\Sigma}$:
	\[ \cI_{\Gamma}(f) \triangleq \big( \langle f, f_1 \rangle_{L^2(\Gamma)}, \ldots, \langle f, f_{I_\Sigma} \rangle_{L^2(\Gamma)} \big), \]
	so that $f \in \ker \cI_\Gamma$ precisely implies \eqref{eq:ind.gamma.pos.dir}. We note one more property of elements of $\ker \cI_\Gamma$:
	\begin{clai}
		Let $w \in C^{\infty}(\Omega)$ be such that $\Pi_\eps(w) \in \ker \cI_\Gamma$. Then, 
		\begin{equation} \label{eq:orth.ind.Q.big} 
			\cQ^{\Omega}_{u} (w,w) \geq \eps \sigma' \int_{\Omega} w^{2} \, d\mu_{g}
		\end{equation}
		for $\sigma'=\sigma'(M,g,W,\Sigma)>0$ and $\eps>0$ sufficiently small.
	\end{clai}
	\begin{proof}[Proof of claim]
		We proceed as in Lemma \ref{lemm:first.eig.stability.Q.low.bd} but we use the improved lower bound for $\cQ_{\Gamma}(f,f)$, \eqref{eq:ind.gamma.pos.dir}, for $f = \Pi_\eps(w)$. Write $\psi(y,z) = \Pi_\eps(w) \overline{\mathbb{H}}'(\eps^{-1}(z-h(y)))$. Then, using Lemmas \ref{lemm:index.Hprime.comp}, \ref{lemm:index.perp.comp}, and \ref{lemm:index.mixed.terms}:
		\begin{align*}
			\cQ^{\Omega}_{u}(w,w) 
				& = \cQ^{\Omega}_{u}(\psi,\psi) + \cQ^{\Omega}(w^{\perp},w^{\perp}) + 2 \cQ^{\Omega}_{u}(\psi,w^{\perp})\\
				& \geq \eps^{2} \delta (\energyunit - o(1)) \int_{\Gamma}  |\nabla_{\Gamma} f|^{2} + f^{2} \, d\mu_\Gamma + \gamma \int_{\Omega} \eps |\nabla w^{\perp}|^{2} + \eps^{-1} (w^{\perp})^{2} \, d\mu_{g}\\
				& - o(\eps^{2}) \int_{\Gamma} |\nabla f|^{2} + f^{2}\, d\mu_\Gamma -  o(1) \int_{\Omega} \eps |\nabla w^{\perp}|^{2} + \eps^{-1} (w^{\perp})^{2} \, d\mu_{g}\\
				& \geq \eps \sigma' \int_{\Omega} w^{2} \, d\mu_{g}
		\end{align*}
		The claim follows.
	\end{proof}
	
	\begin{clai}
		If $w \in C^{\infty}(M)$ satisfies $\cQ_{u}(w,w) \leq 0$, then 
		\begin{equation} \label{eq:neg-dir-concentrate}
			\int_{M\setminus \Omega} w^{2} \, d\mu_{g} \leq C \eps^{2} \int_{\Omega} w^{2} \, d\mu_g,
		\end{equation}
		for $C=C(M,g,W,\Sigma,\eta)>0$ and $\eps > 0$ sufficiently small.
	\end{clai}
	\begin{proof}[Proof of claim]
		Using Lemma \ref{lemm:first.eig.stability.Q.low.bd} and that $W''(u) \geq \kappa > 0$ on $M\setminus\Omega$ for $\eps>0$ small, we compute 
		\begin{equation*}
			0 \geq \cQ_{u}(w,w) \geq \cQ_{u}^{\Omega}(w,w) + \eps^{-1} \kappa \int_{M\setminus\Omega} w^{2} d\mu_{g} \geq -\eps \sigma \int_{\Omega} w^{2} d\mu_{g} + \eps^{-1} \kappa \int_{M\setminus\Omega} w^{2} d\mu_{g}.
		\end{equation*}
		Rearranging this completes the proof. 
	\end{proof}

	Now, let $w_{1}, \ldots, w_{I_{0}} \in C^{\infty}(M)$ denote an $L^{2}(M)$-orthonormal set of eigenfunctions for $\cQ_{u}$ with non-positive eigenvalue, and
	\[ W_\Omega \triangleq \lspan\{ w_1|_\Omega, \ldots, w_{I_0}|_\Omega \} \subset C^\infty(\Omega), \; W_\Gamma \triangleq \left\{ \Pi_\eps(w) : w \in \lspan\{ w_1, \ldots, w_{I_0} \} \right\}. \] 
	We emphasize that
	\begin{equation} \label{eq:negative.eigenfunctions.cQ}
		\cQ_u(w, w) \leq 0 \text{ for all } w \in \lspan\{ w_1, \ldots, w_{I_0} \} \subset C^\infty(M).
	\end{equation}
		
	\begin{clai}
		$\dim W_\Omega = \dim W_{\Gamma} = I_0$ for $\eps > 0$ sufficiently small.
	\end{clai}
	\begin{proof}[Proof of claim]
		To see $\dim W_\Omega = I_0$, it suffices to note that no nontrivial linear combination $w$ of $w_1, \ldots, w_{I_0}$ can vanish on $\Omega$ because of \eqref{eq:neg-dir-concentrate} and \eqref{eq:negative.eigenfunctions.cQ}.
		
		Likewise, to see $\dim W_\Gamma = I_0$, it suffices to note that no nontrivial linear combination $w$ of $w_1, \ldots, w_{I_0}$ has $\Pi_\eps(w) = 0$ because of \eqref{eq:orth.ind.Q.big}, \eqref{eq:neg-dir-concentrate}, and \eqref{eq:negative.eigenfunctions.cQ}.
	\end{proof}

	Finally, suppose, for the sake of contradiction, that $I_{\Sigma} < I_{0}$. Because $\dim W_\Gamma = I_{0} > I_{\Sigma}$, it must hold that there exists $w \in \lspan\{ w_1, \ldots, w_{I_0} \} \setminus \{ 0 \}$ such that $\cI_\Gamma(\Pi_\eps(w)) = 0$. For $\eps>0$ sufficiently small so that $W''(u) \geq 0$ on $M\setminus\Omega$,
	\[ 0 \geq \cQ_{u}(w,w) = \cQ_{u}^{\Omega}(w,w) + \int_{M\setminus \Omega} \eps |\nabla w|^{2} + \eps^{-1} W''(u) w^{2} \, d\mu_{g} \geq \cQ_{u}^{\Omega}(w,w)  \geq \eps \sigma' \int_{\Omega} w^{2} d\mu_{g}. \]
	We used \eqref{eq:orth.ind.Q.big} in the last step. Thus, $w \equiv 0$ on $\Omega$, so $w \equiv 0$ on $M$ by \eqref{eq:neg-dir-concentrate}, \eqref{eq:negative.eigenfunctions.cQ}, a contradiction. 
\end{proof}

\section{Geometric applications} 


\label{sec:applications}

\begin{coro}[Multiplicity one, two-sidedness, and index of Allen--Cahn limits for bumpy or positive Ricci curvature metrics] \label{coro:mult.one.conj}
	Let $(M^{3},g)$ denote a closed $3$-manifold with a bumpy metric (see Definition \ref{def:bumpy.metric}) or with positive Ricci curvature. Suppose that $u_i \in C^\infty(M; [-1,1])$, $\eps_i > 0$, $u_i$ is a critical point of $E_{\varepsilon_i}$, $E_{\varepsilon_i}[u_i] \leq E_0$, $\ind(u_i) \leq I_0$ for all $i = 1, 2, \ldots$ and $\lim_i \eps_i = 0$. Passing to a subsequence, denote by $V \triangleq \lim_i \energyunit^{-1} V_{\varepsilon_i}[u_i]$ the limit varifold. Then:
	\begin{itemize}
		\item The support $\Sigma$ of $V$ is a smooth, embedded, two-sided minimal surface with $\ind(\Sigma) \leq I_{0}$.
		\item The limiting varifold $V$ is equal to the varifold associated to $\Sigma$ with multiplicity one. 
		\item For $\beta \in (0,1)$ fixed, the level sets $u_{i}^{-1}(t)$, 		$|t| < 1-\beta$, converge in $C^{2,\theta}$ with multiplicity one to $\Sigma$. 
		\item For $i$ sufficiently large, $\nul(\Sigma) + \ind(\Sigma) \geq \nul(u_{i}) + \ind(u_{i})$.
	\end{itemize}
\end{coro}
\begin{proof}
By Theorem \ref{theo:bounded.index}, any component of $\Sigma$ that does not satisfy the conclusion at hand must admit a two-sided double cover with a positive Jacobi field. This cannot happen if $g$ is bumpy (irrespective of the sign of the Jacobi field). Similarly, because a positive Jacobi field implies that the two-sided double cover is stable, this cannot occur for positive Ricci curvature. The $C^{2,\theta}$ convergence follows from Lemma \ref{lemm:multiplicity.one.convergence}. The index upper bounds for $\Sigma$ follow from \cite{Hiesmayr} (also from \cite{Gaspar}). Finally, the index lower bounds follow from Theorem \ref{theo:index.lower.bounds}. 
\end{proof}

Finally, we note that Corollary \ref{coro:mult.one.conj} proves Yau's conjecture for bumpy metrics (or those with positive Ricci curvature) on a $3$-manifold. In fact, we establish the following strengthened version of Yau's conjecture, which describes certain geometric properties of the minimal surfaces. That a generic Riemannian manifold contains an embedded two-sided minimal surface of each positive Morse index was conjectured by Marques and Neves (cf.\ \cite[p.\ 17]{neves:ICM}, \cite[Conjecture 6.2]{MarquesNeves:spaceOfCycles}).

\begin{coro}[Yau's conjecture for bumpy metrics and geometric properties of the minimal surfaces]\label{coro:Yau.conj}
	Let $(M^3, g)$ be a closed $3$-manifold with a bumpy metric. There is $C=C(M,g,W)>0$ and a smooth embedded minimal surface $\Sigma_{p}$ for each positive integer $p$ so that
	\begin{itemize}
		\item each component of $\Sigma_{p}$ is two-sided,
		\item the area of $\Sigma_{p}$ satisfies $C^{-1} p^{\frac 1 3}\leq \area_{g}(\Sigma_{p}) \leq C p^{\frac 1 3}$, 
		\item the index of $\Sigma_{p}$ is satisfies $\ind(\Sigma_{p}) = p$, 
		and
		\item the genus of $\Sigma_{p}$ satisfies $\genus(\Sigma_{p}) \geq \frac p 6 - C p^{\frac 13}$.
	\end{itemize}
In particular, thanks to the index estimate, all of the $\Sigma_{p}$ are geometrically distinct. 
\end{coro}

\begin{proof}
Gaspar--Guaraco set up a min-max procedure for the Allen--Cahn energy functional and showed \cite[Theorems 3, 4]{GasparGuaraco} that there is $C=C(M,g,W)>1$ so that for each integer $p>0$, there exists $\varepsilon_{0}(p)>0$ so that for $\varepsilon \in (0,\varepsilon_{0})$, there exists ${u_{p,\eps}}$, a critical point of $E_{\eps}$ with
\[
C^{-1}p^{\frac 1 3} \leq E_{\eps}[{u_{p,\eps}}]\leq C p^{\frac 13}, \; \ind({u_{p,\eps}}) \leq p, \; \nul({u_{p,\eps}}) + \ind({u_{p,\eps}}) \geq p 
\]
(see \cite[Theorem 3.3(2)]{GasparGuaraco}). Now, the first three bullet points follow from Corollary \ref{coro:mult.one.conj} applied to an arbitrary sequence $({u_{p,\eps_{i}}},\eps_{i})$ with $\eps_{i}\to 0$. 

The genus bounds follow from an estimate of Ejiri--Micallef \cite[Theorem 4.3]{EjiriMicallef} who prove that there is a constant $C=C(M,g)$ so that writing $\Sigma_{p} = \cup_{m=1}^{N}\Sigma_{p}^{(m)}$, where $\Sigma_{p}^{(m)}$ are connected and ${N} =|\pi_{0}(\Sigma_{p})|$ is the number of connected components of $\Sigma_{p}$, we have
\[
\ind(\Sigma_{p}^{(m)}) \leq C \area(\Sigma_{p}^{(m)}) + r(\genus(\Sigma_{p}^{(m)})), \; {m = 1, \ldots, N,}
\]
where $r(g)$ is the dimension of the space of conformal structures on a genus $g$ surface, i.e., 
\[
r(g) = \begin{cases} 0 & g = 0, \\ 2 & g = 1, \\ 6(g-1) & g>1.\end{cases}
\]
Thus, we find that
\[
p = \sum_{m=1}^{N} \ind(\Sigma_{p}^{(m)}) \leq C \area_{g}(\Sigma_{p}) + \sum_{m=1}^{N} r(\genus(\Sigma_{p}^{(m)})).
\]
Using $r(g) \leq 6g$ and $\area(\Sigma_{p})\leq C p^{\frac 13}$ (for some $C=C(M,g)$ as explained above), we find that,
\[
{\frac p 6 - C p^{\frac 1 3}} \leq \sum_{m=1}^{M} \genus(\Sigma_{p}^{(m)}) = \genus(\Sigma_{p}),
\]
{for $C = C(M, g)$.} This proves the fourth bullet point, completing the proof. 
\end{proof}

\begin{rema}
We note that for $(M^{3},g)$ with positive Ricci curvature, the same conclusion as in Corollary \ref{coro:Yau.conj} holds, except the third bullet point is replaced by $\ind(\Sigma_{p}) + \nul(\Sigma_{p}) = p$. The genus bound still holds by the same result of Ejiri--Micallef \cite[Theorem 4.3]{EjiriMicallef}.
\end{rema}

When $(M^{3},g)$ does not have positive Ricci, $\Sigma_{p}$ might have several components. We can use the discrepancy between the linear index growth and sublinear area growth to prove that at least one of the components has large index and genus (note that this discrepancy has been leveraged in a rather different manner by Marques--Neves \cite{MarquesNeves:posRic} in their proof of Yau's conjecture in positive Ricci curvature). 

\begin{coro}[Connected components of the $p$-width having large index and genus] \label{coro:Yau.conj.components}
{ Let $(M^3, g)$ denote a closed 3-manifold with a bumpy metric.} There is $C=C(M,g,W)>0$ so that some connected component $\Sigma_{p}'$ of the minimal surface $\Sigma_{p}$ discussed in Corollary \ref{coro:Yau.conj} has 
{$\genus(\Sigma_p') \geq C^{-1} \ind(\Sigma_p') \geq C^{-1} p^{\frac 23}$.}
\end{coro}
\begin{proof}
Write the surfaces $\Sigma_p$ obtained in Corollary \ref{coro:Yau.conj} above as a union of their connected components, i.e., $\Sigma_{p} = \cup_{m=1}^{N} \Sigma_{p}^{(m)}$. By the monotonicity formula, there is $c=c(M,g)>0$ so that any closed minimal surface $\Sigma'$ in $M$ has $\area_{g}(\Sigma') \geq c$. Thus,
\begin{equation} \label{eq:connected.component.bound}
{N} c \leq \sum_{m=1}^{N} \area_{g}(\Sigma_{p}^{(m)}) = \area_{g}(\Sigma_{p}) \leq C p^{\frac 13}. 
\end{equation}
{Because} $p = \sum_{m=1}^{ N}\ind(\Sigma_{p}^{(m)})$, 
{ \eqref{eq:connected.component.bound} implies that $\ind(\Sigma_{p}^{m})\geq C p^{\frac 2 3}$ for some} $m \in \{1,\dots, {N} \}$. {For this particular $m$,} $\area_{g}(\Sigma_{p}^{(m)}) \leq C p^{\frac 13}$ and the estimate of Ejiri--Micallef \cite{EjiriMicallef} used above implies that
\[
\genus(\Sigma_{p}^{(m)}) \geq {C^{-1} \ind(\Sigma_p^{(m)}) \geq C^{-1}} p^{\frac 2 3}.
\]
This completes the proof. 
\end{proof}

\section{Barriers with Dirichlet data}


\label{sec:dirichlet.data}

\subsection{Setup} \label{subsec:dirichlet.data.setup}

The heteroclinic solution from Section \ref{subsec:heteroclinic.solution} lifts trivially to a solution of the Allen-Cahn PDE, \eqref{eq:ac.pde}, on $\RR^n$, for any $n \geq 1$; indeed, one may just take $u(x^1, \ldots, x^n) \triangleq \mathbb{H}_\eps(x^n)$. Notice that this solution is ``centered'' on the $\{x^n = 0\}$ hyperplane. One may just as easily center it on any hyperplane in $\RR^n$ by a suitable translation and rotation.

The question of centering approximate heteroclinic solutions on arbitrary minimal $\Sigma^{n-1} \subset (M^n, g)$ has been well-studied in the compact setting; see, e.g., \cite{PacardRitore03} for the boundary-less case and the geometrically natural case of  Neumann conditions at the boundary when $\partial M$, $\partial \Sigma \neq \emptyset$, or see \cite{Pacard12} for a more general survey with a faster construction than \cite{PacardRitore03}, albeit only presented in the boundary-less case.

In this section we establish a corresponding existence theorem similar in spirit to those in \cite{PacardRitore03, Pacard12}, except we prescribe Dirichlet data. This theorem provides the barriers that were a crucial ingredient in the final ``sliding'' argument of Section \ref{sec:bounded.index}. 

The setup is as follows. Define $C^{k,\alpha}_\eps$, $\alpha \in (0,1)$, $\eps > 0$, to be the standard H\"older space after rescaling by $\eps$, i.e., whose Banach norm is
\begin{equation} \label{eq:dirichlet.data.ckalpha.eps}
	\Vert v \Vert_{C^{k,\alpha}_\eps} \triangleq \sum_{j=0}^k \eps^j \Vert \nabla^j v \Vert_{L^\infty} + \eps^{k+\alpha} [\nabla^k v]_\alpha.
\end{equation}
Various choices of domain and metric will be specified below. See Remarks \ref{rema:dirichlet.data.regularity.product.vs.omega}, \ref{rema:dirichlet.data.various.norms}.

Next, suppose that $D^{n-1}$ is an $(n-1)$-dimensional manifold with nonempty boundary, over which we take a topological cylinder $\Omega \triangleq D \times [-1,1]$, whose coordinates we label $X = (y, z) \in D \times [-1,1]$. Let $g$ be a smooth metric on $\Omega$, given in $(y, z)$ coordinates (Fermi coordinates) by
\[ g = g_z + dz^2. \]

We require that
\begin{equation} \label{eq:dirichlet.data.sigma.minimal}
	\Sigma \triangleq D \times \{0\} \subset (\Omega, g) \text{ is a minimal surface}
\end{equation}
whose second fundamental form is uniformly bounded  in $C^{0,\theta}$, for some $\theta \in (0, 1)$ that will be eventually chosen to be near $1$ (see Theorem \ref{theo:dirichlet.data.construction}):
\begin{equation} \label{eq:dirichlet.data.sigma.c2alpha}
	|\sff_\Sigma| + [\sff_\Sigma]_{\theta} \leq \eta,
\end{equation}
and also\footnote{It is crucial for Section \ref{sec:bounded.index} that we only work with the weaker bounds on derivatives of $\sff$ given in \eqref{eq:dirichlet.data.sigma.c2alpha}, \eqref{eq:dirichlet.data.sigma.c3alpha}, which are precisely the types of estimates we derived in Section \ref{sec:stable.solutions}.} in $C^{1,\theta}_\eps$:
\begin{equation} \label{eq:dirichlet.data.sigma.c3alpha}
	\eps |\nabla_{\Sigma} \sff_\Sigma| + \eps^{1+\theta} [\nabla_{\Sigma} \sff_\Sigma]_{\theta} \leq \eta,
\end{equation}
with $\eta > 0$ small. We furthermore assume that there are $C^{2,\theta}$-coordinate charts on $\Sigma$ so that the induced metric $g_{0}$ is $C^{0,\theta}$ and $C^{1,\theta}_{\varepsilon}$-close to the Euclidean metric in the sense that 
\begin{equation}\label{eq:dirichlet.data.sigma.induced.c2alpha}
	|(g_{0})_{ij} - \delta_{ij}| + [(g_{0})_{ij}]_{\theta} \leq \eta,
\end{equation}
\begin{equation}\label{eq:dirichlet.data.sigma.induced.c3alpha}
	\varepsilon|\partial_{k}(g_{0})_{ij}| + \varepsilon^{1+\theta} [\partial_{k}(g_{0})_{ij}]_\theta \leq \eta,
\end{equation}
where $i$, $j$, $k$ run through the coordinates $(y^{1},\dots,y^{n-1})$ on $\Sigma$ in the given coordinate chart.

Note that \eqref{eq:dirichlet.data.sigma.c2alpha} implies that Fermi coordinates $(y, z)$ with respect to $\Sigma$ are a diffeomorphism which is $C^{1,\theta}$-close to the identity, so in particular, together with \eqref{eq:dirichlet.data.sigma.induced.c2alpha}, it follows that the metric $g$ is $C^{0,\theta}$-close to being Euclidean in Fermi coordinates:
\begin{equation} \label{eq:dirichlet.data.g.c0alpha}
	|g_{\kappa \lambda} - \delta_{\kappa \lambda}| + [g_{\kappa \lambda}]_{\theta} \leq \eta',
\end{equation}
for small $\eta' = \eta'(\eta, n) > 0$. Here, $\kappa$, $\lambda$ run through all $n$ Fermi coordinates $(y^1, \ldots, y^{n-1}, z)$.

Likewise, \eqref{eq:dirichlet.data.sigma.c3alpha} and \eqref{eq:dirichlet.data.sigma.induced.c3alpha} imply that Fermi coordinates are $C^{2,\theta}_\eps$-close to the identity and
\begin{equation} \label{eq:dirichlet.data.g.c1alpha}
	\eps |\partial_{\mu} g_{\kappa \lambda}| + \eps^{1+\theta} [\partial_{\mu} g_{\kappa \lambda}]_{\theta} \leq \eta'.
\end{equation}
Here, $\kappa$, $\lambda$, $\mu$ run through all $n$ Fermi coordinates.

We also require that $\Sigma$ carries no nontrivial Jacobi fields with Dirichlet boundary conditions in the following quantitative sense:
\begin{equation} \label{eq:dirichlet.data.sigma.nondegenerate}
	\int_\Sigma (J_\Sigma f)^2 \, d\mu_{g_0} \geq \eta \int_\Sigma f^2 \, d\mu_{g_0} \text{ for every } f \in C^\infty_c(\Sigma \setminus \partial \Sigma).
\end{equation}
where
\begin{equation} \label{eq:dirichlet.data.sigma.jacobi.operator}
	J_\Sigma f \triangleq -\Delta_{g_0} f - (|\sff_\Sigma|^2 + \ricc_{g}(\partial_z, \partial_z)|_{\Sigma})f
\end{equation}
denotes the Jacobi operator on $\Sigma$. (Note that our sign convention for the Jacobi operator differs from the one in \cite{Pacard12}.)

Let's also fix $\delta_* \in (0, 1)$, and define cutoff functions $\chi_j : \RR \to [0,1]$, with $\chi_j' \geq 0$ on $[0, \infty)$, so that
\begin{equation} \label{eq:dirichlet.data.cutoff}
	\chi_j(t) = \begin{cases} 1 & |t| \leq \eps^{\delta_*} \Big( 1 - \tfrac{2j-1}{100} \Big) \\ 0 & |t| \geq \eps^{\delta_*} \Big( 1 - \tfrac{2j-2}{100} \Big). \end{cases}
\end{equation}
as well as $\Vert \chi_j \Vert_{C^{3}_{\eps^{\delta_*}}(\RR)} \leq 200$. We further require that the $\chi_j$ be even functions. 

For $\eps > 0$, set
\begin{equation} \label{eq:dirichlet.data.approximate.heteroclinic}
\widetilde{\mathbb{H}}_\eps(t) \triangleq \chi_1(t) \mathbb{H}_\eps(t) \pm (1 - \chi_1(t)),
\end{equation}
where the $\pm$ corresponds to $t > 0$, $t < 0$, respectively, and $\mathbb{H}_\eps$ is as in \eqref{eq:heteroclinic.eps}. This is a truncation of the one-dimensional solution, $\mathbb{H}_\eps$, which coincides with $\mathbb{H}_\eps$ near $\Sigma$ and with $\pm 1$ away from $\Sigma$.

The functions $\chi_j$, $\mathbb{H}_\eps$, $\widetilde{\mathbb{H}}_\eps$ lift trivially to $\Sigma \times \RR$. We also set
\[ \Omega_j \triangleq \{ (y, z) \in \Sigma \times \RR : z \in \support \chi_j \}. \]
Using the Fermi coordinates $(y, z)$, $\chi_j$, $\mathbb{H}_\eps$, $\widetilde{\mathbb{H}}_\eps$ also give functions on $\Omega$ that depend only on $z$. By \eqref{eq:dirichlet.data.sigma.induced.c3alpha}, \eqref{eq:dirichlet.data.cutoff}, these functions are uniformly $C^{2,\theta}_\eps$ in $\Sigma \times \RR$ with respect to the product metric $g_0 + dz^2$ and also in $\Omega$ with respect to the metric $g$. Likewise, by \eqref{eq:dirichlet.data.sigma.c2alpha}, \eqref{eq:dirichlet.data.sigma.c3alpha}, the slab $\Omega_j$ can also be viewed as a subset of $(\Omega, g)$ whose boundary is $C^{1,\theta}$ and $C^{2,\theta}_\eps$-close to being totally geodesic.

\begin{rema} \label{rema:dirichlet.data.regularity.product.vs.omega}
	By \eqref{eq:dirichlet.data.sigma.induced.c3alpha}, \eqref{eq:dirichlet.data.g.c1alpha}, there exists a constant $C = C(\eta)$ such that
	\[ C^{-1} \Vert f \Vert_{C^{k,\alpha}_\eps(\Omega)} \leq \Vert f \Vert_{C^{k,\alpha}_{\eps}(\Sigma \times \RR)} \leq C \Vert f \Vert_{C^{k,\alpha}_\eps(\Omega)}, \; k = 0, 1, 2, \; \alpha \in (0, \theta], \]
	for any function $f : \Omega \to \RR$ with support in the interior of $\Omega$. The norms above are taken with respect to the product metric $g_0 + dz^2$ on $\Sigma \times \RR$ and the metric $g$ on $\Omega$.
\end{rema}

\begin{rema} \label{rema:dirichlet.data.trivialization.window}
	We cannot reuse the truncation from Section \ref{sec:jacobi.toda.reduction}, because we now need a truncation that trivializes outside a polynomial window instead of a logarithmic window.
\end{rema}

For subsets $S \subset \Sigma$, let's define 
\[ \Pi_\eps : L^2(S \times \RR) \to L^2(S), \; \Pi_\eps^\perp : L^2(S \times \RR) \to L^2(S \times \RR) \]
to be given by
\begin{align} 
	\Pi_\eps(f)(y) & \triangleq  \eps^{-1} \energyunit^{-1}  \int_{-\infty}^\infty f(y, z) \cdot  \mathbb{H}'(\eps^{-1} z) \, dz, \label{eq:dirichlet.data.proj} \\
	\Pi_\eps^\perp(f)(y,z) & \triangleq f(y,z) - \Pi_\eps(f)(y) \mathbb{H}'(\eps^{-1} z). \label{eq:dirichlet.data.proj.perp}
\end{align}
We note two things:
\begin{enumerate}
	\item $S$ does not appear in the projection notation, but it will clear from the context when it is  relevant.
	\item Our normalization is such that $\Pi_\eps( \{ z \mapsto  \mathbb{H}'(\eps^{-1} z) \} ) = \eps \Pi_\eps (\mathbb{H}_\eps') =  1$.
\end{enumerate} 

From this point forward we also consider another H\"older exponent, $\alpha \in (0, 1)$, which is such that
\[ \alpha \leq \theta \]
(with $\theta$ is as in \eqref{eq:dirichlet.data.sigma.c2alpha}-\eqref{eq:dirichlet.data.sigma.induced.c3alpha}). The exponent $\alpha$ will be eventually taken to be near $0$ (see Theorem \ref{theo:dirichlet.data.construction}).

We point out the following trivial lemma:

\begin{lemm} \label{lemm:dirichlet.data.proj.holder.norms}
	Both $\Pi_\eps$ and $\Pi_\eps^\perp$ lift to linear maps
	\[ \Pi_\eps : C^{0,\alpha}_\eps(S \times \RR) \to C^{0,\alpha}_\eps(S), \; \Pi_\eps^\perp : C^{0,\alpha}_\eps(S \times \RR) \to C^{0,\alpha}_\eps(S \times \RR). \]
	The $C^{0,\alpha}_\eps(S \times \RR)$ norm is taken with respect to the product metric $g_{0} + dz^{2}$. Viewed as linear maps over these H\"older spaces, we have $\sup_{\eps > 0} \big(  \Vert \Pi_\eps \Vert  + \Vert \Pi_\eps^\perp \Vert \big) < \infty$. 
\end{lemm}

For $\zeta \in C^{2,\alpha}(\Sigma)$, we define $D_\zeta$ to be the map
\begin{equation} \label{eq:dirichlet.data.offset.map}
D_\zeta(y, t) \triangleq (y, t - \chi_2(t) \zeta(y)).
\end{equation}

Finally, we introduce the modified H\"older norm:
\begin{equation} \label{eq:dirichlet.data.ckalpha.eps.modified}
	\Vert v \Vert_{\widetilde{C}^{k,\alpha}_\eps(\Omega)} \triangleq \eps^{-2} \Vert \chi_5 v \Vert_{C^{k,\alpha}_\eps(\Omega)} + \Vert v \Vert_{C^{k,\alpha}_\eps(\Omega)}.
\end{equation}
Recall that $\Vert \cdot \Vert_{C^{k,\alpha}_\eps}$ is as in \eqref{eq:dirichlet.data.ckalpha.eps}. As with Remark \ref{rema:dirichlet.data.regularity.product.vs.omega}, the $C^{k,\alpha}_{\eps}(\Omega)$ norm is taken with respect to $g$.

The main result of this section is:

\begin{theo} \label{theo:dirichlet.data.construction}
	If $\alpha \leq \alpha_0$, $\eps \leq \varepsilon_0$ and we're given boundary data
	\begin{enumerate}
		\item $\widehat{v}^\flat \in \widetilde{C}^{2,\alpha}_\eps(\partial \Omega)$, $\Vert \widehat{v}^\flat \Vert_{\widetilde{C}^{2,\alpha}_\eps(\partial \Omega)} \leq \mu \eps^2$, $\widehat{v}^\flat = 0$ on $\{ \chi_4 = 1 \} \cap \partial\Omega$,
		\item $\widehat{v}^\sharp \in C^{2,\alpha}_\eps(\partial \Sigma \times \RR)$, $\Vert \widehat{v}^\sharp \Vert_{C^{2,\alpha}_\eps(\partial \Sigma \times \RR)} \leq \mu \eps^2$, $\Pi_\eps(\widehat{v}^\sharp) \equiv 0$ on $\partial \Sigma$,
		\item $\widehat{\zeta} \in C^{2,\alpha}(\partial \Sigma)$, $\eps^{2\alpha} \Vert \widehat{\zeta} \Vert_{C^{2,\alpha}(\partial \Sigma)} \leq \mu \eps^2$,
	\end{enumerate}
	and a metric $g$ for which \eqref{eq:dirichlet.data.sigma.minimal}-\eqref{eq:dirichlet.data.sigma.induced.c3alpha} hold with $\theta \geq \theta_0 \geq \alpha_0$, there exist 
	\begin{enumerate}
		\item $v^\flat \in \widetilde{C}^{2,\alpha}(\Omega)$, $v^\flat|_{\partial \Omega} = \widehat{v}^\flat$, $\Vert v^\flat \Vert_{\widetilde{C}^{2,\alpha}_\eps(\Omega)} \leq C \eps^2$,
		\item $v^\sharp \in C^{2,\alpha}(\Sigma \times \RR)$, $v^\sharp|_{\partial \Sigma \times \RR} = \widehat{v}^\sharp$, $\Pi_\eps v^\sharp \equiv 0$, $\Vert v^\sharp \Vert_{C^{2,\alpha}_\eps(\Sigma \times \RR)} \leq C \eps^2$,
		\item $\zeta \in C^{2,\alpha}(\Sigma)$, $\zeta|_{\partial \Gamma} = \widehat{\zeta}$, $\eps^{2\alpha} \Vert \zeta \Vert_{C^{2,\alpha}(\Sigma)} \leq C \eps^2$,
	\end{enumerate}
	so that $\mathfrak{u} = (\widetilde{\mathbb{H}}_{\eps} + \chi_4 v^\sharp + v^\flat) \circ D_\zeta$ satisfies
	\begin{equation} \label{eq:dirichlet.data.pde}  
		\eps^2 \Delta_g \mathfrak{u} = W'(\mathfrak{u}) \text{ on } \Omega.
	\end{equation}
	The solution map $(\widehat{v}^\flat, \widehat{v}^\sharp, \widehat{\zeta}, g) \mapsto (v^\flat, v^\sharp, \zeta)$ is Lipschitz continuous, with Lipschitz constant $L$, as a map
	\begin{equation*}
		\widetilde{C}^{2,\alpha}_\eps(\partial \Omega) \times C^{2,\alpha}_\eps(\partial \Sigma \times \RR) \times C^{2,\alpha}(\partial \Sigma) \times \operatorname{Met}_{\eps,\eta}(\Omega) 
		\to \widetilde{C}^{2,\alpha}_\eps(\Omega) \times C^{2,\alpha}_\eps(\Sigma \times \RR) \times C^{2,\alpha}(\Sigma)
	\end{equation*}
	where $\operatorname{Met}_{\eps,\eta}(\Omega)$ denotes the set of metrics satisfying \eqref{eq:dirichlet.data.g.c0alpha}-\eqref{eq:dirichlet.data.g.c1alpha} with the obvious topology. The spaces $\widetilde{C}^{2,\alpha}_\eps(\Omega) \times C^{2,\alpha}_\eps(\Sigma \times \RR) \times C^{2,\alpha}(\Sigma)$, $\widetilde{C}^{2,\alpha}(\partial \Omega) \times C^{2,\alpha}_\eps(\partial \Sigma \times \RR) \times C^{2,\alpha}(\partial \Sigma)$ are topologized using the norms in \eqref{eq:dirichlet.data.interior.product.norm},  \eqref{eq:dirichlet.data.boundary.product.norm}, respectively. Here, $\eps_0 = \eps_0(n, \eta, W, \delta_*, \mu, \alpha)$, $\alpha_0 = \alpha_0(n, \eta, W, \delta_*, \mu)$, $\theta_0 = \theta_0(\delta_*)$, $C = C(n, \eta, W, \delta_*, \mu, \alpha)$, $L = L(n, \eta, W, \delta_*, \mu, \alpha, \theta)$.
\end{theo}

This follows along the lines of \cite[Section 3]{Pacard12}, provided one makes the necessary modifications to account for (possibly nonzero, but small) Dirichlet data as well as the important fact that our Fermi coordinate regularity is constrained by the weaker assumptions \eqref{eq:dirichlet.data.sigma.c2alpha}-\eqref{eq:dirichlet.data.sigma.c3alpha}. This lower regularity situation makes certain aspects of Theorem \ref{theo:dirichlet.data.construction} delicate, so we describe the proof in detail below.

\subsection{Linear scheme} \label{subsec:dirichlet.data.linear.scheme}

In this section we generalize linear estimates found in \cite[Section 3]{Pacard12} to allow Dirichlet boundary conditions, possibly with nonzero data. The operators we'll study are:
\begin{align}
	L_* & \triangleq \Delta_{\RR^n} + \partial_z^2 - W''(\mathbb{H}) \text{ on } \RR_+^n \times \RR, \label{eq:dirichlet.data.l.star} \\
	L_\eps & \triangleq \eps^2 (\Delta_{g_0} + \partial_z^2) - W''(\mathbb{H}_\eps) \text{ on } \Sigma \times \RR,	\label{eq:dirichlet.data.l.eps} \\
	\cL_\eps & \triangleq \eps^2 \Delta_g - W''(\pm 1) \text{ on } \Omega. \label{eq:dirichlet.data.cl.eps}
\end{align}

\begin{lemm}[cf. {\cite[Lemma 3.7]{Pacard12}}] \label{lemm:dirichlet.data.lemm.3.7}
	Assume that $w \in L^\infty(\RR^n_+ \times \RR)$ satisfies $L_* w = 0$ and $w \equiv 0$ on $\partial \RR^n_+ \times \RR$. Then $w \equiv 0$.
\end{lemm}
\begin{proof}
	The result follows from \cite[Lemma 3.7]{Pacard12} after an odd reflection of $w$ across $\partial \RR^{n}_{+}$. 
\end{proof}

The next results that need to be adapted pertain to $L_\eps$ and functions $\varphi \in L^\infty(\Sigma \times \RR)$ satisfying $\Pi_\eps(\varphi) \equiv 0$ on $\Sigma$, where $\Pi_\eps$ is as in \eqref{eq:dirichlet.data.proj}.

\begin{lemm}[cf. {\cite[Proposition 3.1]{Pacard12}}] \label{lemm:dirichlet.data.prop.3.1}
	If $\eps \leq \eps_0$, $w \in C^{2,\alpha}_\eps(\Sigma \times \RR)$, and $\Pi_\eps(w) \equiv 0$ on $\Sigma$, then
	\[ \Vert w \Vert_{C^{2,\alpha}_\eps(\Sigma\times \RR)} \leq C( \Vert L_\eps w \Vert_{C^{0,\alpha}_\eps(\Sigma \times \RR)} + \Vert w|_{\partial \Sigma \times \RR} \Vert_{C^{2,\alpha}_\eps(\partial \Sigma \times \RR)}). \]
	Here, $\eps_0 = \eps_0(n, \eta, W)$, $C = C(n, \eta, W, \alpha)$.
\end{lemm}
\begin{proof}
	This follows from the $C^{1,\alpha}_\eps$ control of $g_0$ by way of  \eqref{eq:dirichlet.data.sigma.induced.c3alpha}, \cite[Proposition 3.1]{Pacard12}, Lemma \ref{lemm:dirichlet.data.lemm.3.7}, and boundary Schauder estimates (e.g., \cite[Theorem 5]{Simon97}).
\end{proof}

\begin{lemm}[cf. {\cite[Proposition 3.2]{Pacard12}}] \label{lemm:dirichlet.data.prop.3.2}
	There exists $\eps_0 > 0$ depending on $n$, $\eta > 0$, $W$, such that for all $\eps \in (0, \eps_0)$, all $f \in C^{0,\alpha}_\eps(\Sigma \times \RR)$ with $\Pi_\eps(f) \equiv 0$ on $\Sigma$, and all $\hat f \in C^{2,\alpha}_\eps(\partial \Sigma \times \RR)$ with $\Pi_\eps(\hat f) \equiv 0$ on $\partial \Sigma$, there exists a unique function $w \in C^{2,\alpha}_\eps(\Sigma \times \RR)$, also with $\Pi_\eps(w) \equiv 0$ on $\Sigma$, such that
	\[ L_\eps w = f \text{ in } \Sigma \times \RR, \; w = \hat f \text{ on } \partial \Sigma \times \RR. \]
\end{lemm}
\begin{proof}
	When $\hat f \equiv 0$ this follows from the functional analytic argument already found in \cite[Proposition 3.2]{Pacard12} applied, instead, to $W^{1,2}_0(\Sigma \times \RR)$.
	
	When $\hat f \not \equiv 0$, this follows by extending $\hat f$ to $C^{2,\alpha}(\Sigma \times \RR)$, $\Pi_\eps(\hat f) \equiv 0$, and applying the previous existence result with zero boundary data to solve $L_\eps w = f - L_\eps \hat f$.
\end{proof}

Finally, \cite{Pacard12} deals with $\cL_\eps$. 

\begin{lemm}[cf. {\cite[Proposition 3.3]{Pacard12}}] \label{lemm:dirichlet.data.prop.3.3}
	If $\eps \in (0, 1)$, then 
	\[ \Vert w \Vert_{C^{2,\alpha}_\eps(\Omega)} \leq C(\Vert \cL_\eps w \Vert_{C^{0,\alpha}_\eps(\Omega)} + \Vert w|_{\partial \Omega} \Vert_{C^{2,\alpha}_\eps(\partial \Omega)}). \]
	Here, $C = C(n, \eta, W, \alpha)$. 
\end{lemm}
\begin{proof}
	The interior estimate follows from interior Schauder theory, since $g$ is $C^{1,\alpha}_\eps$ by \eqref{eq:dirichlet.data.g.c1alpha}. The boundary estimate on the regular portion of $\partial \Omega$ follows from boundary Schauder theory, because $\partial \Omega$ is $C^{2,\alpha}_\eps$ at those points by \eqref{eq:dirichlet.data.sigma.c3alpha}. Finally, the estimate at the corners of $\partial \Omega$ follows from the boundary theory as well. This is because we can carry out odd reflections across $D \times \{\pm 1\}$ since the angles at the corners are all $\pi/2$.
\end{proof}

We also derive an \emph{improved} estimate for functions satisfying $\cL_\eps w = 0$ on a strip of height $O(\eps^{\delta_*})$, and $w = 0$ on its lateral boundary. Recall the definition of the norm $\widetilde{C}^{2,\alpha}_{\varepsilon}$ in \eqref{eq:dirichlet.data.ckalpha.eps.modified}.

\begin{lemm}[cf. {\cite[(3.26)]{Pacard12}}] \label{lemm:dirichlet.data.eq.3.26}
	If $\eps \leq \eps_0$, $w \in C^{2,\alpha}_\eps(\Omega)$, and
	\[ \cL_\eps w = 0 \text{ on } \Omega_4, \text{ and } w = 0 \text{ on } \partial \Omega_4 \cap \partial \Omega, \]
	then
	\[ \Vert w \Vert_{\widetilde{C}^{2,\alpha}_\eps(\Omega)} \leq C ( \Vert \cL_\eps w \Vert_{C^{0,\alpha}_\eps(\Omega)} + \Vert w|_{\partial \Omega} \Vert_{C^{2,\alpha}_\eps(\partial \Omega)}). \]
	Here, $\eps_0 = \eps_0(n, \eta, W, \delta_*)$, $C = C(n, \eta, W, \delta_*, \alpha)$. 
\end{lemm}
\begin{proof}
	Considering Lemma \ref{lemm:dirichlet.data.prop.3.3}, it suffices to check that
	\begin{equation} \label{eq:dirichlet.data.eq.3.26.i}
		\Vert \chi_5 w \Vert_{C^{2,\alpha}_\eps(\Omega)} \leq C \eps^2 (\Vert \cL_\eps w \Vert_{C^{0,\alpha}_\eps(\Omega)} + \Vert w|_{\partial \Omega} \Vert_{C^{2,\alpha}_\eps(\partial \Omega)}).
	\end{equation}
	
	Since $\cL_\eps = 0$ on $\Omega_4$, $w = 0$ on $\partial \Omega_4 \cap \partial \Omega$, and $\delta_* \in (0,1)$, Schauder's \emph{interior} estimates estimates on $\partial\Omega_5 \setminus \partial \Omega$, Schauder's \emph{boundary} estimates near $\partial \Omega_5 \cap \partial \Omega$, \eqref{eq:dirichlet.data.sigma.c3alpha}, and \eqref{eq:dirichlet.data.g.c1alpha}, imply:
	\[ \Vert w \Vert_{C^{2,\alpha}_\eps(\Omega_5)} \leq C \Vert w \Vert_{L^\infty(\{ \chi_4 = 1 \})}. \]
	In particular, given the decay of the first and second derivatives of $\chi_j$ from \eqref{eq:dirichlet.data.cutoff} and $\delta_* \in (0,1)$, \eqref{eq:dirichlet.data.eq.3.26.i} will follow as long as
	\begin{equation} \label{eq:dirichlet.data.eq.3.26.ii}
		\Vert w \Vert_{L^\infty(\{ \chi_4 = 1 \})} \leq C \eps^2 \Vert w \Vert_{L^\infty(\Omega)}
	\end{equation}
	
	We use the same barrier argument as in \cite[Remark 3.2]{Pacard12}, paying closer attention to the boundary and to the regularity. Define
	\[ \varphi_{z_0}(z) \triangleq \cosh (\gamma \eps^{-1} (z-z_0)) \]
	with $|z_0| \leq \eps^{\delta_*}$ and $\gamma \in (0, (W''(\pm 1))^{\tfrac{1}{2}})$. If $H_z$ denotes the mean curvature of of a $z$-level set in Fermi coordinates, then:
	\begin{align*} 
		\eps^2 \Delta_g \varphi_{z_0}(z) 
			& = \gamma^2 \varphi_{z_0}(z) + H_z \gamma \eps \sinh( \gamma \eps^{-1} (z-z_0))  \leq (\gamma^2 + \gamma \eps |H_z|) \varphi_{z_0}(z).
	\end{align*}
	It follows from \eqref{eq:dirichlet.data.sigma.c2alpha} and  \eqref{eq:mean.curv.ddt.sff}-\eqref{eq:mean.curv.ddt.h} that $|H_z|$ is uniformly bounded. In particular, for sufficiently small $\eps$, depending on $\gamma$, $\eta$, $n$, we have
	\[ \eps^2 \Delta_g \varphi_{z_0}(z) \leq W''(\pm 1) \varphi_{z_0}(z), \]
	so $\varphi_{z_0}$ is a barrier, as it was in \cite{Pacard12}. It therefore follows from the maximum principle applied to $w - t \varphi_{z_0}$ that, for $(y, z_0) \in \Omega_4$,
	\[ |w(y,z_0)| \leq \left( \inf_{\Omega \setminus \Omega_4} \varphi_{z_0} \right)^{-1} \max_{\partial \Omega_4} |w|, \]
	which is trivially bounded by $ c \eps^2 \Vert w \Vert_{L^\infty(\Omega)}$ whenever $(y,z_0) \in \{ \chi_4 = 1 \}$, and $\eps > 0$ is small. This implies \eqref{eq:dirichlet.data.eq.3.26.ii} and, in turn, \eqref{eq:dirichlet.data.eq.3.26.i}.
\end{proof}

\subsection{Nonlinear scheme} \label{subsec:dirichlet.data.nonlinear.scheme}

We consider the following nonlinear functionals, originally defined in \cite[Section 3]{Pacard12}:
\begin{align} 
	\sE_\eps(\zeta) 
		& \triangleq \eps^2 \Delta_g (\widetilde{\mathbb{H}}_\eps \circ D_\zeta) \circ D_\zeta^{-1} - W'(\widetilde{\mathbb{H}}_\eps), \label{eq:dirichlet.data.E.eps} \\
	Q_\eps(v) 
		& \triangleq W'(\widetilde{\mathbb{H}}_\eps + v) - W'(\widetilde{\mathbb{H}}_\eps) - W''(\widetilde{\mathbb{H}}_\eps) v, \label{eq:dirichlet.data.Q.eps} \\
	M_\eps(v^\flat, v^\sharp, \zeta) 
		& \triangleq \chi_3 \Big[ L_\eps v^\sharp - \eps^2 \Delta_g (v^\sharp \circ D_\zeta) \circ D_\zeta^{-1} + W''(\mathbb{H}_\eps) v^\sharp \label{eq:dirichlet.data.M.eps} \\
		& \qquad - \eps^2 (\Delta_g (v^\flat \circ D_\zeta) \circ D_\zeta^{-1} - \Delta_g v^\flat) - \sE_\eps(\zeta) + \eps^2 (J_\Sigma \zeta)  \partial_z \mathbb{H}_\eps \nonumber \\
		& \qquad  - Q_\eps(\chi_4 v^\sharp + v^\flat) + (W''(\mathbb{H}_\eps) - W''(\pm 1)) v^\flat \Big], \nonumber \\
	N_\eps(v^\flat, v^\sharp, \zeta)
		& \triangleq (\chi_4-1) \Big[ \eps^2( \Delta_g( v^\flat \circ D_\zeta) \circ D_\zeta^{-1} - \Delta_g v^\flat) \label{eq:dirichlet.data.N.eps} \\ 
		& \qquad \qquad + (W''(\widetilde{\mathbb{H}}_\eps) - W''(\pm 1)) v^\flat - \sE_\eps(\zeta) - Q_\eps(\chi_4 v^\sharp + v^\flat) \Big] \nonumber \\
		& \qquad \qquad - \eps^2 (\Delta_g ((\chi_4 v^\sharp) \circ D_\zeta) - \chi_4 \Delta_g(v^\sharp \circ D_\zeta)) \circ D_\zeta^{-1} . \nonumber
\end{align}

These functionals allow us to pose \eqref{eq:dirichlet.data.pde} as a fixed point problem:
\begin{align} 
	\cL_\eps v^\flat & = N_\eps(v^\flat, v^\sharp, \zeta) \label{eq:dirichlet.data.pde.vsharp} \\
	L_\eps v^\sharp & = \Pi_\eps^\perp M_\eps(v^\flat, v^\sharp, \zeta) \label{eq:dirichlet.data.pde.vflat} \\
	J_\Sigma \zeta & = \eps^{-1} \Pi_\eps M_\eps(v^\flat, v^\sharp, \zeta) \label{eq:dirichlet.data.pde.zeta},
\end{align}
(cf. \cite[(3.31), (3.32), (3.33)]{Pacard12}). We impose, as does \cite[Section 3]{Pacard12}, the additional constraint:
\[ \Pi_\eps v^\sharp \equiv 0 \text{ on } \Sigma. \]

\begin{lemm}[cf. {\cite[Lemma 3.8]{Pacard12}}]  \label{lemm:dirichlet.data.lemm.3.8}
	The following estimates hold:
	\[ \Vert N_\eps(0, 0, 0) \Vert_{C^{0,\alpha}_\eps(\Omega)} + \Vert \Pi_\eps^\perp M_\eps(0, 0, 0) \Vert_{C^{0,\alpha}_\eps(\Sigma \times \RR)} +  \eps^{-1} \Vert \Pi_\eps M_\eps(0, 0, 0) \Vert_{C^{0,\alpha}(\Sigma)} \leq c_0 \eps^2. \]
	Here, $\eps \in (0, \tfrac12)$, $c_0 = c_0(n, \eta, W, \delta_*, \alpha)$.
\end{lemm}

\begin{proof}
	Note that
	\[ M_\eps(0, 0, 0) = - \chi_3 \sE_\eps(0), \; N_\eps(0, 0, 0) = (1-\chi_4) \sE_\eps(0). \]
	Straightforward computation shows $\sE_\eps(0) = \eps^2 \Delta_g \widetilde{\mathbb{H}}_\eps - W'(\widetilde{\mathbb{H}}_\eps)$. From \eqref{eq:dirichlet.data.approximate.heteroclinic}: 
	\begin{equation} \label{eq:dirichlet.data.lemm.3.8.Htilde.minus.H}
		\widetilde{\mathbb{H}}_\eps - \mathbb{H}_\eps = (1-\chi_1)(\pm 1 - \mathbb{H}_\eps),
	\end{equation}
	($\pm$ depends on $z > 0$ or $z < 0$), a quantity that decays exponentially to all orders with $\eps \to 0$. Since $\mathbb{H}_\eps$ does too on $\support (1-\chi_4)$, we in fact get
	\[ \Vert N_\eps(0, 0, 0) \Vert_{C^{0,\alpha}_\eps(\Omega)} \leq C_m \eps^m \]
	for all $m \in \NN$. (Taking $m=2$ will suffice.)

	To estimate $M_\eps(0, 0, 0)$, we proceed to further rewrite:
	\begin{align*}
		\sE_\eps(0) 
			& = \eps^2 \Delta_g \widetilde{\mathbb{H}}_\eps - W'(\widetilde{\mathbb{H}}_\eps) \\
			& = \eps^2 \Delta_g \mathbb{H}_\eps - W'(\mathbb{H}_\eps) + \eps^2 \Delta_g (\widetilde{\mathbb{H}}_\eps - \mathbb{H}_\eps) - (W'(\widetilde{\mathbb{H}}_\eps) - W'(\mathbb{H}_\eps)) \\
			& = \eps^2 H_z \partial_z \mathbb{H}_\eps - (\eps^2 \Delta_g - W''(\widetilde{\mathbb{H}}_\eps) - Q_\eps) ({\mathbb{H}}_\eps - \widetilde{\mathbb{H}}_\eps).
	\end{align*}
	Note that
	\[ \widetilde{\mathbb{H}}_\eps \equiv 1 \text{ on } \Omega \setminus \Omega_1 \implies \sE_\eps(0) \equiv 0 \text{ on } \Omega \setminus \Omega_1. \]
	If $\chi : \Omega \to [0,1]$ is the cutoff function $\chi(z) = \chi_1(z/2)$, then note that $\chi \equiv 1$ on $\support \sE_\eps(0)$ so that
	\[ \sE_\eps(0) = \chi \cdot \eps^2 H_z \partial_z \mathbb{H}_\eps -  \chi \cdot (\eps^2 \Delta_g - W''(\widetilde{\mathbb{H}}_\eps) - Q_\eps)({\mathbb{H}}_\eps - \widetilde{\mathbb{H}}_\eps). \]

	It follows from \eqref{eq:dirichlet.data.g.c1alpha},  \eqref{eq:dirichlet.data.cutoff}, and \eqref{eq:dirichlet.data.lemm.3.8.Htilde.minus.H} that
	\begin{equation} \label{eq:dirichlet.data.lemm.3.8.i}
		\Vert \chi \cdot (\eps^2 \Delta_g - W''(\widetilde{\mathbb{H}}_\eps) - Q_\eps)({\mathbb{H}}_\eps - \widetilde{\mathbb{H}}_\eps) \Vert_{C^{0,\alpha}_\eps(\Sigma \times \RR)} \leq C_m \eps^m,
	\end{equation}
	for $m \in \NN$. (Taking $m = 4$ will suffice.)

	Recalling \eqref{eq:mean.curv.ddt.h}:
	\begin{equation} \label{eq:dirichlet.data.lemm.3.8.ddt.h}
		\partial_z H_z = - |\sff_z|^2 + \ricc_g(\partial_z, \partial_z)|_{D \times \{z\}}, \; z \in [-1,1].
	\end{equation}
	Certainly, this already implies, since $\alpha \leq \theta$,
	\[ \sup_{|z| \leq 1} \Vert y \mapsto \partial_z H_z \Vert_{C^{0,\alpha}(\Sigma)} \leq C. \]
	Combining  \eqref{eq:dirichlet.data.lemm.3.8.ddt.h} with  \eqref{eq:dirichlet.data.sigma.c2alpha}, $\alpha \leq \theta$,  \eqref{eq:mean.curv.ddt.metric}, and  \eqref{eq:mean.curv.ddt.sff},
	we even find that
	\begin{equation} \label{eq:dirichlet.data.lemm.3.8.ddt.sq.h}
		\sup_{|z| \leq 1} \Vert y \mapsto \partial_z^2 H_z(y,z) \Vert_{C^{0,\alpha}(\Sigma)} \leq C.
	\end{equation} 
	In particular,  \eqref{eq:dirichlet.data.sigma.minimal}, \eqref{eq:dirichlet.data.lemm.3.8.ddt.sq.h} and Taylor's theorem imply
	\begin{equation} \label{eq:dirichlet.data.lemm.3.8.taylor}
		H_z = - (|\sff_0|^2 + \ricc_g(\partial_z, \partial_z)|_{\Sigma}) z + \cR(y,z) z^2,
	\end{equation}
	where
	\begin{equation} \label{eq:dirichlet.data.3.8.proj.remainder}
		\sup_{|z|\leq 1} \Vert y \mapsto \cR(y,z) \Vert_{C^{0,\alpha}(\Sigma)} \leq C.
	\end{equation}
	
	From the trivial estimate $|z| \partial_z \mathbb{H}_\eps \leq C$, \eqref{eq:dirichlet.data.cutoff}, and \eqref{eq:dirichlet.data.lemm.3.8.taylor}, we find that
	\begin{equation} \label{eq:dirichlet.data.lemm.3.8.ii}
		\Vert \chi \cdot \eps^2 H_z \partial_z \mathbb{H}_\eps \Vert_{C^{0,\alpha}_\eps(\Sigma \times \RR)} \leq C \eps^2.
	\end{equation}
	Put together, \eqref{eq:dirichlet.data.lemm.3.8.i}, \eqref{eq:dirichlet.data.lemm.3.8.ii}, and Lemma \ref{lemm:dirichlet.data.proj.holder.norms} imply:
	\[ \Vert \Pi_\eps^\perp M_\eps(0,0,0) \Vert_{C^{0,\alpha}_\eps(\Sigma \times \RR)} \leq C \eps^2. \]
	
	Finally, by \eqref{eq:dirichlet.data.lemm.3.8.taylor},
	\[ \Pi_\eps (\chi \cdot \eps^2 H_z \partial_z \mathbb{H}_\eps)  = \energyunit^{-1} \int_{-\infty}^{\infty} \chi(z) (\partial_z H_z(y,0) \cdot z + \cR(y,z) z^2) (\mathbb{H}'(\eps^{-1} z))^2 \, dz. \]
	Recalling that, from parity, (since $\chi(z)$ is even)
	\[ \int_{-\infty}^\infty \chi(z) z (\mathbb{H}'(\eps^{-1} z))^2 \, dz = 0 \]
	it follows that
	\[ \Pi_\eps(\chi \cdot \eps^2 H_z \partial_z \mathbb{H}_\eps) = \energyunit^{-1} \int_{-\infty}^\infty \chi(z) \cR(y,z) z^2 (\mathbb{H}'(\eps^{-1} z))^2 \, dz, \]
	at which point we can directly estimate using \eqref{eq:heteroclinic.expansion.ii}, \eqref{eq:heteroclinic.eps}, and \eqref{eq:dirichlet.data.3.8.proj.remainder}, and get:
	\[ \Vert \Pi_\eps (\chi \cdot \eps^2 H_z \partial_z \mathbb{H}_\eps) \Vert_{C^{0,\alpha}(\Sigma)} \leq C  \eps^3, \]
	Together with \eqref{eq:dirichlet.data.lemm.3.8.i} (with $m=4$), this implies
	\[ \Vert \Pi_\eps M_\eps(0, 0, 0) \Vert_{C^{0,\alpha}(\Sigma)} \leq C  \eps^3. \]
 	This completes the proof.
\end{proof}

\begin{lemm}[cf. {\cite[Lemma 3.9]{Pacard12}}] \label{lemm:dirichlet.data.lemm.3.9}
	For $\alpha \leq \alpha_0$, $\eps \leq \eps_0$:
	\begin{align} 
		& \Vert N_\eps(v_2^\flat, v_2^\sharp, \zeta_2) - N_\eps(v_1^\flat, v_1^\sharp, \zeta_1) \Vert_{C^{0,\alpha}_\eps(\Omega)} \label{eq:dirichlet.data.lemm.3.9.n} \\
			& \qquad \leq c_1 \eps^{\delta} \Big( \Vert v_2^\flat - v_1^\flat \Vert_{C^{2,\alpha}_\eps(\Omega)} + \Vert v_2^\sharp - v_1^\sharp \Vert_{C^{2,\alpha}_\eps(\Sigma \times \RR)} + \Vert \zeta_2 - \zeta_1 \Vert_{C^{2,\alpha}(\Sigma)} \Big),  \nonumber \\
		& \Vert \Pi_\eps^\perp(M_\eps(v_2^\flat, v_2^\sharp, \zeta_2) - M_\eps(v_1^\flat, v_1^\sharp, \zeta_1)) \Vert_{C^{0,\alpha}_\eps(\Sigma \times \RR)} \label{eq:dirichlet.data.lemm.3.9.m.perp} \\
			& \qquad \leq c_1 \eps^\delta \Big( \Vert v_2^\flat - v_1^\flat \Vert_{\widetilde{C}^{2,\alpha}_\eps(\Omega)} + \Vert v_2^\sharp - v_1^\sharp \Vert_{C^{2,\alpha}_\eps(\Sigma \times \RR)} + \Vert \zeta_2 - \zeta_1 \Vert_{C^{2,\alpha}(\Sigma)} \Big), \nonumber \\
		& \Vert \Pi_\eps( M_\eps(v_2^\flat, v_2^\sharp, \zeta_2) - M_\eps(v_1^\flat, v_1^\sharp, \zeta_1) ) \Vert_{C^{0,\alpha}(\Sigma)} \label{eq:dirichlet.data.lemm.3.9.m} \\
			& \qquad \leq c_1 \eps^{1 + \delta} \Vert v_2^\flat - v_1^\flat \Vert_{\widetilde{C}^{2,\alpha}_\eps(\Omega)}  + c_1 \eps^{1-\alpha} \Vert v_2^\sharp - v_1^\sharp \Vert_{C^{2,\alpha}_\eps(\Sigma \times \RR)} + c_1  \eps^{1 + \delta} \Vert \zeta_2 - \zeta_1 \Vert_{C^{2,\alpha}(\Sigma)}, \nonumber
	\end{align}
	provided \eqref{eq:dirichlet.data.sigma.c2alpha}-\eqref{eq:dirichlet.data.sigma.induced.c3alpha} hold with $\theta \geq \theta_0 \geq \alpha_0$, and
	\[ \sum_{j=1,2} \Vert v_j^\flat \Vert_{\widetilde{C}^{2,\alpha}_\eps(\Omega)} + \Vert v_j^\sharp \Vert_{C^{2,\alpha}_\eps(\Sigma \times \RR)} + \eps^{2\alpha} \Vert \zeta_j \Vert_{C^{2,\alpha}(\Sigma)} \leq C' \eps^2. \]
	Here, $\eps_0 = \eps_0(n, \eta, W, \delta_*)$, $\delta = \delta(\delta_*)$, $\theta_0  = \theta_0(\delta_*)$, $\alpha_0 = \alpha_0(\delta_*)$, $c_1 = c_1(n, \eta, W, \delta_*, C', \alpha)$.
\end{lemm}

\begin{rema} \label{rema:dirichlet.data.various.norms}
	We emphasize that three different norms are used:
	\begin{enumerate}
		\item On $v^\flat$, we use the \emph{modified} weighted H\"older norm
			\[ \Vert w \Vert_{\widetilde{C}^{2,\alpha}_\eps(\Omega)} = \Vert w \Vert_{C^{2,\alpha}_\eps(\Omega)} + \eps^{-2} \Vert \chi_5 w \Vert_{C^{2,\alpha}_\eps(\Omega)}. \]
			Here, the H\"older norms are measured with respect to the metric $g$.
		\item On $v^\sharp$, we use the standard weighted H\"older norm $C^{2,\alpha}_\eps(\Sigma \times \RR)$. Here, the H\"older norms are measured with respect to the product metric $g_{0}+dz^{2}$.
		\item On $\zeta$, we use the \emph{unweighted} H\"older norm $C^{2,\alpha}(\Sigma)$, which strictly dominates $C^{2,\alpha}_\eps(\Sigma)$:
			\[ \Vert \zeta \Vert_{C^{2,\alpha}_\eps(\Sigma)} \leq \Vert \zeta \Vert_{C^{2,\alpha}(\Sigma)}. \]
			Here, the H\"older norms are measured with respect to the metric $g_{0}$ induced on $\Sigma$.
	\end{enumerate}
\end{rema}

\begin{proof}[Proof of Lemma {\ref{lemm:dirichlet.data.lemm.3.9}}]
	In what follows we may assume that $\alpha_0 \leq \tfrac14$. 
	
	Note, from \eqref{eq:dirichlet.data.cutoff}, \eqref{eq:dirichlet.data.N.eps}, that
	\[ N_\eps(v_1^\flat, v_1^\sharp, \zeta_1) \equiv N_\eps(v_2^\flat, v_2^\sharp, \zeta_2) \equiv 0 \text{ on } \{ \chi_4 = 1 \}. \]
	Therefore, since $\delta_* \in (0,1)$,
	\begin{align*}
		 \Vert N_\eps(v_2^\flat, v_2^\sharp, \zeta_2) - N_\eps(v_1^\flat, v_1^\sharp, \zeta_1) \Vert_{C^{0,\alpha}_\eps(\Omega)} & = \Vert N_\eps(v_2^\flat, v_2^\sharp, \zeta_2) - N_\eps(v_1^\flat, v_1^\sharp, \zeta_1) \Vert_{C^{0,\alpha}_\eps(\{ \chi_4 \neq 1 \})} \\
		&  \leq \Vert N_\eps(v_2^\flat, v_2^\sharp, \zeta_2) - N_\eps(v_1^\flat, v_1^\sharp, \zeta_1) \Vert_{C^{0,\alpha}_\eps(\Omega \setminus \Omega_5)}.
	\end{align*}
	We'll estimate this by pairing up the terms, making sure to use use the fact that our H\"older norm is taken over $\Omega \setminus \Omega_5$ instead of over $\Omega$, in order to gain a factor of $\eps^\delta$, for some $\delta > 0$ that depends on $\delta_*$.
	
	In all that follows, we'll repeatedly (and implicitly) use that our Fermi coordinates (and thus also $D_\zeta$, $D_\zeta^{-1}$) are $C^{2,\alpha}_\eps$ close to the identity, and that our metric $g$ in Fermi coordinates is $C^{1,\alpha}_\eps$ close to Euclidean.  
	
	We start by estimating
	\begin{equation*}
		\Vert \eps^2 (\Delta_g (v_2^\flat \circ D_{\zeta_2}) \circ D_{\zeta_2}^{-1} - \Delta_g v_2^\flat) 
		- \eps^{2} (\Delta_g (v_1^\flat \circ D_{\zeta_1}) \circ D_{\zeta_1}^{-1} - \Delta_g v_1^\flat) \Vert_{C^{0,\alpha}_\eps(\Omega)}.
	\end{equation*}
	(We can deduce a good estimate on all of $\Omega$, not just on $\Omega \setminus \Omega_5$.) By working in Fermi coordinates in scale $O(\eps)$, we see that
	\begin{equation} \label{eq:dirichlet.data.lemm.3.9.F1}
		\cF_1(v, \zeta) \triangleq \eps^2 \Delta_g (v \circ D_\zeta) \circ D_\zeta^{-1}
	\end{equation}
	is a \emph{smooth} nonlinear Banach space functional $\cF_1 : C^{2,\alpha}_\eps(\Omega) \times C^{2,\alpha}_\eps(\Sigma) \to C^{0,\alpha}_\eps(\Omega)$, and is \emph{linear} in $v$. In particular,
	\begin{align*}
		& \eps^2 \big[ (\Delta_g (v_2^\flat \circ D_{\zeta_2}) \circ D_{\zeta_2}^{-1} - \Delta_g v_2^\flat) - (\Delta_g (v_1^\flat \circ D_{\zeta_1}) \circ D_{\zeta_1} - \Delta_g v_1^\flat) \big] \\
		& \qquad = (\cF_1(v_2^\flat, \zeta_2) - \cF_1(v_1^\flat, \zeta_1)) - (\cF_1(v_2^\flat, 0) - \cF_1(v_1^\flat, 0)) \\
		& \qquad = \int_0^1 \langle D_v \cF_1(v_1^\flat + t(v_2^\flat - v_1^\flat), \zeta_1 + t(\zeta_2 - \zeta_1)), v_2^\flat - v_1^\flat \rangle \\
		& \qquad \qquad + \langle D_\zeta \cF_1(v_1^\flat + t(v_2^\flat - v_1^\flat), \zeta_1 + t(\zeta_2 - \zeta_1)), \zeta_2 - \zeta_1 \rangle \, dt \\
		& \qquad - \int_0^1 \langle D_v \cF_1(v_1^\flat + t(v_2^\flat - v_1^\flat), 0), v_2^\flat - v_1^\flat \rangle dt \\
		& \qquad = \int_0^1 \int_0^1 \langle D_\zeta D_v \cF_1(v_1^\flat + t(v_2^\flat - v_1^\flat), s \zeta_1 + st (\zeta_2-\zeta_1)),  (\zeta_1 + t(\zeta_2 - \zeta_1)) \otimes (v_2^\flat - v_1^\flat)  \rangle \, ds \, dt \\
		& \qquad + \int_0^1 \langle D_\zeta \cF_1(v_1^\flat + t(v_2^\flat - v_1^\flat), \zeta_1 + t(\zeta_2 - \zeta_1)), \zeta_2 - \zeta_1 \rangle \, dt.
	\end{align*}
	Seeing as to how $\Vert v_j^\flat  \Vert_{C^{2,\alpha}_\eps(\Omega)} \leq C' \eps^2$, $\Vert \zeta_j \Vert_{C^{2,\alpha}(\Sigma)} \leq C' \eps^{2-2\alpha}$, and using the linearity in $v$ of $\cF_1$ (and thus of $D_\zeta \cF_1$), we can directly estimate:
	\begin{align}
		& \Vert \eps^2((\Delta_g (v_2^\flat \circ D_{\zeta_2}) \circ D_{\zeta_2}^{-1} - \Delta_g v_2^\flat) -  (\Delta_g (v_1^\flat \circ D_{\zeta_1}) \circ D_{\zeta_1} - \Delta_g v_1^\flat)) \Vert_{C^{0,\alpha}_\eps(\Omega)} \nonumber \\
		& \qquad \leq C (\Vert \zeta_1 \Vert_{C^{2,\alpha}_\eps(\Sigma)} + \Vert \zeta_2 \Vert_{C^{2,\alpha}_\eps(\Sigma)}) \Vert v_2^\flat - v_1^\flat \Vert_{C^{2,\alpha}_\eps(\Omega)} + C (\Vert v_1^\flat \Vert_{C^{2,\alpha}_\eps(\Omega)} + \Vert v_2^\flat  \Vert_{C^{2,\alpha}_\eps(\Omega)}) \Vert \zeta_2 - \zeta_1 \Vert_{C^{2,\alpha}_\eps(\Sigma)} \nonumber \\
		& \qquad \leq C (\Vert \zeta_1 \Vert_{C^{2,\alpha}(\Sigma)} + \Vert \zeta_2 \Vert_{C^{2,\alpha}(\Sigma)}) \Vert v_2^\flat - v_1^\flat \Vert_{C^{2,\alpha}_\eps(\Omega)}  + C (\Vert v_1^\flat \Vert_{C^{2,\alpha}_\eps(\Omega)} + \Vert v_2^\flat \Vert_{C^{2,\alpha}_\eps(\Omega)}) \Vert \zeta_2 - \zeta_1 \Vert_{C^{2,\alpha}(\Sigma)} \nonumber \\
		& \qquad \leq C \eps^{2-2\alpha} \Vert v_2^\flat - v_1^\flat \Vert_{C^{2,\alpha}_\eps(\Omega)} + C \eps^2 \Vert \zeta_2 - \zeta_1 \Vert_{C^{2,\alpha}(\Sigma)}. \label{eq:dirichlet.data.lemm.3.9.laplace.zeta.commutator}
	\end{align}
	This estimate is of the desired form.
	
	Next, we estimate
	\[ \Vert (W''(\widetilde{\mathbb{H}}_\eps) - W''(\pm 1))(v_2^\flat - v_1^\flat) \Vert_{C^{0,\alpha}_\eps(\Omega \setminus\Omega_5)} \]
	The desired estimate is a simple consequence of Remark \ref{rema:dirichlet.data.regularity.product.vs.omega} and how, on $\Omega \setminus \Omega_5$, we have
	\begin{equation} \label{eq:dirichlet.data.lemm.3.9.H.decay}
		\Vert W''(\widetilde{\mathbb{H}}_\eps) - W''(\pm 1) \Vert_{C^{0,\alpha}_\eps(\Omega \setminus \Omega_5)} \leq C_m \eps^m,
	\end{equation}
	for all $m \in \NN$; thus, any $\delta > 0$ will do.
	
	Next, we estimate
	\[ \Vert \sE_\eps(\zeta_2) - \sE_\eps(\zeta_1) \Vert_{C^{0,\alpha}_\eps(\Omega \setminus \Omega_5)}. \]
	We have
	\begin{align*}
		& \sE_\eps(\zeta_2) - \sE_\eps(\zeta_1) = \eps^2 ( \Delta_g (\widetilde{\mathbb{H}}_\eps \circ D_{\zeta_2}) \circ D_{\zeta_2}^{-1} - \Delta_g (\widetilde{\mathbb{H}}_\eps \circ D_{\zeta_1}) \circ D_{\zeta_1}^{-1})  = \cF_1'(\widetilde{\mathbb{H}}_\eps, \zeta_2) - \cF_1'(\widetilde{\mathbb{H}}_\eps, \zeta_1),
	\end{align*}
	where $\cF_1' : C^{2,\alpha}_\eps(\Omega \setminus \Omega_5) \times C^{2,\alpha}_\eps(\Sigma) \to C^{0,\alpha}_\eps(\Omega \setminus \Omega_5)$ is the restriction of $\cF_1$ from \eqref{eq:dirichlet.data.lemm.3.9.F1}. Arguing as before, we get
	\begin{align}
		& \Vert \sE_\eps(\zeta_2) - \sE_\eps(\zeta_1) \Vert_{C^{0,\alpha}_\eps(\Omega \setminus \Omega_5)}  \leq C \Vert \widetilde{\mathbb{H}}_\eps \Vert_{C^{2,\alpha}_\eps(\Omega \setminus \Omega_5)} \Vert \zeta_2 - \zeta_1 \Vert_{C^{2,\alpha}_\eps(\Sigma)} \leq C_m \eps^m \Vert \zeta_2 - \zeta_1 \Vert_{C^{2,\alpha}(\Sigma)}, \label{eq:dirichlet.data.lemm.3.9.Eeps.n}
	\end{align}
	for all $m \in \NN$, which implies what we want, for any $\delta > 0$.
	
	Next, we estimate
	\[ \Vert Q_\eps(\chi_4 v_2^\sharp + v_2^\flat) - Q_\eps(\chi_4 v_1^\sharp + v_1^\flat) \Vert_{C^{0,\alpha}_\eps(\Omega)}. \]
	Note that
	\begin{multline*}
		Q_\eps(\chi_4 v_2^\sharp + v_2^\flat) - Q_\eps(\chi_4 v_1^\sharp + v_1^\flat) \\
		= W'(\widetilde{\mathbb{H}}_\eps + \chi_4 v_2^\sharp + v_2^\flat) - W'(\widetilde{\mathbb{H}}_\eps + \chi_4 v_1^\sharp + v_1^\flat) 
		- W''(\widetilde{\mathbb{H}}_\eps)(\chi_4 (v_2^\sharp - v_1^\sharp) + (v_2^\flat - v_1^\flat)).
	\end{multline*}
	Define
	\begin{equation} \label{eq:dirichlet.data.lemm.3.9.F2}
		\cF_2(v) \triangleq W'(\widetilde{\mathbb{H}}_\eps + v),
	\end{equation}
	viewed as a \emph{smooth} nonlinear Banach space functional $\cF_2 : C^{0,\alpha}_\eps(\Omega) \to C^{0,\alpha}_\eps(\Omega)$.	Note that
	\[ \langle D_v \cF_2(v), w \rangle = W''(\widetilde{\mathbb{H}}_\eps + v) w, \; \langle D_v D_v \cF_2(v), w \otimes w' \rangle = W''(\widetilde{\mathbb{H}}_\eps + v) ww', \]
	for $w$, $w' \in C^{0,\alpha}_\eps(\Omega)$. In particular, the expression we're trying to bound equals
	\begin{align*}
		& = \cF_2(\chi_4 v_2^\sharp + v_2^\flat) - \cF_2(\chi_4 v_1^\sharp + v_1^\flat)  - \langle D_v \cF_2(0), \chi_4 (v_2^\sharp-v_1^\sharp) + v_2^\flat - v_1^\flat \rangle \\
		& = \int_0^1 \langle D_v \cF_2(\chi_4 v_1^\sharp + v_1^\flat + t(\chi_4(v_2^\sharp - v_1^\sharp) + v_2^\flat - v_1^\flat)),  \chi_4 (v_2^\sharp - v_1^\sharp) + v_2^\flat - v_1^\flat \rangle \, dt \\
		& \qquad - \langle  D_v \cF_2(0), \chi_4 (v_2^\sharp-v_1^\sharp) + v_2^\flat - v_1^\flat \rangle \\
		& = \int_0^1 \int_0^1 \langle D_v D_v \cF_2(s(\chi_4 v_1^\sharp + v_1^\flat + t(\chi_4(v_2^\sharp - v_1^\sharp) + v_2^\flat - v_1^\flat)), \\
		& \qquad \qquad \qquad (\chi_4 v_1^\sharp + v_1^\flat + t(\chi_4(v_2^\sharp - v_1^\sharp) + v_2^\flat - v_1^\flat)) \otimes (\chi_4(v_2^\sharp - v_1^\sharp) + v_2^\flat - v_1^\flat)) \rangle \, ds \, dt.
	\end{align*} 
	Recalling Remark \ref{rema:dirichlet.data.regularity.product.vs.omega}, \eqref{eq:dirichlet.data.cutoff}, and $\delta_* \in (0,1)$, we can estimate
	\begin{align*}
		\Vert Q_\eps(\chi_4 v_2^\sharp + v_2^\flat) - Q_\eps(\chi_4 v_1^\sharp + v_1^\flat) \Vert_{C^{0,\alpha}_\eps(\Omega)} & \leq C (\Vert v_1^\sharp  \Vert_{C^{0,\alpha}_\eps(\Sigma \times \RR)} + \Vert v_1^\flat \Vert_{C^{0,\alpha}_\eps(\Omega)} + \Vert v_2^\sharp \Vert_{C^{0,\alpha}_\eps(\Sigma \times \RR)} + \Vert v_2^\flat \Vert_{C^{0,\alpha}_\eps(\Omega)}) \\
		& \qquad \qquad \cdot (\Vert v_2^\sharp - v_1^\sharp \Vert_{C^{0,\alpha}_\eps(\Sigma \times \RR)} + \Vert v_2^\flat - v_1^\flat \Vert_{C^{0,\alpha}_\eps(\Omega)}).
	\end{align*}
	This gives
	\begin{equation} \label{eq:dirichlet.data.lemm.3.9.Qeps}
		\Vert Q_\eps(\chi_4 v_2^\sharp + v_2^\flat) - Q_\eps(\chi_4 v_1^\sharp + v_1^\flat) \Vert_{C^{0,\alpha}_\eps(\Omega)} 
		\leq C \eps^2 (\Vert v_2^\sharp - v_1^\sharp \Vert_{C^{0,\alpha}_\eps(\Sigma \times \RR)} + \Vert v_2^\flat - v_1^\flat \Vert_{C^{0,\alpha}_\eps(\Omega)}),
	\end{equation}
	using $\Vert v_j^\sharp \Vert_{C^{0,\alpha}_\eps(\Sigma \times \RR)}, \Vert v_j^\flat \Vert_{C^{0,\alpha}_\eps(\Omega)} \leq C' \eps^2$. 
	
	Next, we consider
	\begin{multline*}
		\Vert \eps^2 ((\Delta_g ((\chi_4 v_2^\sharp) \circ D_{\zeta_2}) - \chi_4 \Delta_g(v_2^\sharp \circ D_{\zeta_2})) \circ D_{\zeta_2}^{-1}) \\
		- (\Delta_g ((\chi_4 v_1^\sharp) \circ D_{\zeta_1}) - \chi_4 \Delta_g(v_1^\sharp \circ D_{\zeta_1})) \circ D_{\zeta_1}^{-1}) \Vert_{C^{0,\alpha}_\eps(\Omega)}
	\end{multline*}
	Define
	\[ \cF_3(v, \zeta) \triangleq \eps^2 (\Delta_g ((\chi_4 v) \circ D_\zeta) - \chi_4 \Delta_g (v \circ D_\zeta)) \circ D_{\zeta}^{-1}, \]
	which is, once again, viewed as a map $\cF_3 : C^{2,\alpha}_\eps(\Omega) \times C^{2,\alpha}_\eps(\Sigma) \to C^{0,\alpha}_\eps(\Omega)$, is a smooth nonlinear Banach space functional. We can then write
	\begin{align*}
		& \eps^2 ((\Delta_g ((\chi_4 v_2^\sharp) \circ D_{\zeta_2}) - \chi_4 \Delta_g(v_2^\sharp \circ D_{\zeta_2})) \circ D_{\zeta_2}^{-1})  - (\Delta_g ((\chi_4 v_1^\sharp) \circ D_{\zeta_1}) - \chi_4 \Delta_g(v_1^\sharp \circ D_{\zeta_1})) \circ D_{\zeta_1}^{-1}) \\
		& \qquad = \cF_3(v_2^\sharp, \zeta_2) - \cF_3(v_1^\sharp, \zeta_1) \\
		& \qquad = \int_0^1 \langle D_v \cF_3(v_1^\sharp + t(v_2^\sharp - v_1^\sharp), \zeta_1 + t(\zeta_2 - \zeta_1)), v_2^\sharp - v_1^\sharp \rangle \\
		& \qquad \qquad \qquad + \langle D_\zeta \cF_3(v_1^\sharp + t(v_2^\sharp - v_1^\sharp), \zeta_1 + t(\zeta_2 - \zeta_1)), \zeta_2 - \zeta_1 \rangle \, dt.
	\end{align*}
	The second term can be estimated by using the linearity in $v$ of $\cF_3$ (and thus of $D_\zeta \cF_3$), and Remark \ref{rema:dirichlet.data.regularity.product.vs.omega} to give:
	\begin{align*}
		& \Vert \langle D_\zeta \cF_3(v_1^\sharp + t(v_2^\sharp - v_1^\sharp), \zeta_1 + t(\zeta_2 - \zeta_1)), \zeta_2 - \zeta_1 \rangle \Vert_{C^{0,\alpha}_\eps(\Omega)} \\
		& \qquad \leq C (\Vert v_1^\sharp \Vert_{C^{2,\alpha}_\eps(\Sigma \times \RR)} + \Vert v_2^\sharp \Vert_{C^{2,\alpha}_\eps(\Sigma \times \RR)}) \Vert \zeta_2 - \zeta_1 \Vert_{C^{2,\alpha}_\eps(\Sigma)}  \leq C \eps^2 \Vert \zeta_2 - \zeta_1 \Vert_{C^{2,\alpha}(\Sigma)},
	\end{align*}
	which is of the desired form with $\delta = 2$. 
	
	The first term instead requires that we use the product rule on $\cF_3$ to recast it as
	\begin{equation*}
		\cF_3(v, \zeta) = \eps^2 (2 \langle \nabla_g (\chi_4 \circ D_\zeta), \nabla_g (v \circ D_\zeta) \rangle 
		+ (\Delta_g(\chi_4 \circ D_\zeta)) (v \circ D_\zeta)) \circ D_{\zeta}^{-1},
	\end{equation*}
	which can, in turn, be differentiated in $v$ to give
	\begin{equation*}
		\langle D_v \cF_3(v, \zeta), w \rangle = \eps^2 (2 \langle \nabla_g (\chi_4 \circ D_\zeta), \nabla_g (w \circ D_\zeta) \rangle 
			 + (\Delta_g(\chi_4 \circ D_\zeta)) (w \circ D_\zeta)) \circ D_{\zeta}^{-1}.
	\end{equation*}
	At this point, we note that there are no zero-order $\chi_4$'s remaining, so we use Remark \ref{rema:dirichlet.data.regularity.product.vs.omega},  \eqref{eq:dirichlet.data.cutoff}, $\delta_* \in (0,1)$ to get
	\begin{align*}
		& \Vert \langle D_v \cF_3(v_1^\sharp + t(v_2^\sharp - v_1^\sharp), \zeta_1 + t(\zeta_2 - \zeta_1)), v_2^\sharp - v_1^\sharp \rangle \Vert_{C^{0,\alpha}_\eps(\Omega)}\\
		& \qquad \leq C \eps^{1-\delta_*} \Vert v_2^\sharp - v_1^\sharp \Vert_{C^{1,\alpha}_\eps(\Sigma \times \RR)}  \leq C \eps^{1-\delta_*} \Vert v_2^\sharp - v_1^\sharp \Vert_{C^{2,\alpha}_\eps(\Sigma \times \RR)}.
	\end{align*}
	Summarizing, we have shown that
	\begin{multline} \label{eq:dirichlet.data.lemm.3.9.laplace.chi.commutator}
		\Vert \eps^2 ((\Delta_g ((\chi_4 v_2^\sharp) \circ D_{\zeta_2}) - \chi_4 \Delta_g(v_2^\sharp \circ D_{\zeta_2})) \circ D_{\zeta_2}^{-1}) 
		- (\Delta_g ((\chi_4 v_1^\sharp) \circ D_{\zeta_1}) - \chi_4 \Delta_g(v_1^\sharp \circ D_{\zeta_1})) \circ D_{\zeta_1}^{-1}) \Vert_{C^{0,\alpha}_\eps(\Omega)} \\
		\leq C \eps^{1-\delta_*} ( \Vert v_2^\sharp - v_1^\sharp \Vert_{C^{2,\alpha}_\eps(\Sigma \times \RR)} + \Vert \zeta_2 - \zeta_1 \Vert_{C^{2,\alpha}(\Sigma)}).
	\end{multline}
	
	The contraction estimate on $N_\eps$, \eqref{eq:dirichlet.data.lemm.3.9.n}, now follows from  \eqref{eq:dirichlet.data.lemm.3.9.laplace.zeta.commutator}, \eqref{eq:dirichlet.data.lemm.3.9.H.decay}, \eqref{eq:dirichlet.data.lemm.3.9.Eeps.n}, \eqref{eq:dirichlet.data.lemm.3.9.Qeps}, and \eqref{eq:dirichlet.data.lemm.3.9.laplace.chi.commutator}.
	
	We move on to the contraction estimates on $M_\eps$, \eqref{eq:dirichlet.data.lemm.3.9.m.perp} and \eqref{eq:dirichlet.data.lemm.3.9.m}. Before we derive those two precise estimates, we investigate several of the easier terms in $M_{\varepsilon}(v_{2}^{\flat},v_{2}^{\sharp},\zeta_{2})-M_{\varepsilon}(v_{1}^{\flat},v_{1}^{\sharp},\zeta_{1})$. 
	
	We note, right away, that we've already shown in \eqref{eq:dirichlet.data.lemm.3.9.laplace.zeta.commutator}:
	\begin{multline*}
		\Vert \eps^2 (\Delta_g (v_2^\flat \circ D_{\zeta_2}) \circ D_{\zeta_2}^{-1} - \Delta_g v_2^\flat) 
		- (\Delta_g (v_1^\flat \circ D_{\zeta_1}) \circ D_{\zeta_1}^{-1} - \Delta_g v_1^\flat) \Vert_{C^{0,\alpha}_\eps(\Omega_3)} \\
		\leq C \eps^{2-2\alpha} \Vert v_2^\flat - v_1^\flat \Vert_{C^{2,\alpha}_\eps(\Omega)} + C \eps^2 \Vert \zeta_2 - \zeta_1 \Vert_{C^{2,\alpha}(\Sigma)}.
	\end{multline*}
	In particular, Remark \ref{rema:dirichlet.data.regularity.product.vs.omega}, Lemma \ref{lemm:dirichlet.data.proj.holder.norms}, and $\Vert \cdot \Vert_{C^{0,\alpha}(\Sigma)} \leq \eps^{-\alpha} \Vert \cdot \Vert_{C^{0,\alpha}_\eps(\Sigma)}$  imply
	\begin{align} \label{eq:dirichlet.data.lemm.3.9.zeta.commutator.m}
		&  \eps^{\alpha} \Big\Vert \Pi_\eps \Big[ \eps^2 (\Delta_g (v_2^\flat \circ D_{\zeta_2}) \circ D_{\zeta_2}^{-1} - \Delta_g v_2^\flat)  - (\Delta_g (v_1^\flat \circ D_{\zeta_1}) \circ D_{\zeta_1}^{-1} - \Delta_g v_1^\flat) \Big] \Big\Vert_{C^{0,\alpha}(\Sigma)} \nonumber \\
		& + \Big\Vert \Pi_\eps^\perp \Big[ \eps^2 (\Delta_g (v_2^\flat \circ D_{\zeta_2}) \circ D_{\zeta_2}^{-1} - \Delta_g v_2^\flat)  - (\Delta_g (v_1^\flat \circ D_{\zeta_1}) \circ D_{\zeta_1}^{-1} - \Delta_g v_1^\flat) \Big] \Big\Vert_{C^{0,\alpha}_\eps(\Sigma \times \RR)} \nonumber \\
		& \qquad \leq C \eps^{2-2\alpha} \Vert v_2^\flat - v_1^\flat \Vert_{C^{2,\alpha}_\eps(\Omega)} + C \eps^2 \Vert \zeta_2 - \zeta_1 \Vert_{C^{2,\alpha}(\Sigma)}.
	\end{align}
	
	 Next, from Remark \ref{rema:dirichlet.data.regularity.product.vs.omega},  \eqref{eq:dirichlet.data.lemm.3.9.Qeps}, we conclude
	\begin{align} \label{eq:dirichlet.data.lemm.3.9.Qeps.m}
		&   \eps^{\alpha} \Big\Vert \Pi_\eps \big[ Q_\eps(\chi_4 v_2^\sharp + v_2^\flat) - Q_\eps(\chi_4 v_1^\sharp + v_1^\flat) \big] \Big\Vert_{C^{0,\alpha}(\Sigma)}  + \Big\Vert \Pi_\eps^\perp \big[ Q_\eps(\chi_4 v_2^\sharp + v_2^\flat) - Q_\eps(\chi_4 v_1^\sharp + v_1^\flat) \big] \Big\Vert_{C^{0,\alpha}_\eps(\Sigma \times \RR)} \nonumber \\
		& \qquad \qquad \leq C \eps^2 (\Vert v_2^\sharp - v_1^\sharp \Vert_{C^{0,\alpha}_\eps(\Sigma \times \RR)} + \Vert v_2^\flat - v_1^\flat \Vert_{C^{0,\alpha}_\eps(\Omega)}).
	\end{align}
	
		Next, we estimate
	\[ \Vert (W''(\mathbb{H}_\eps) - W''(\pm 1))(v_2^\flat - v_1^\flat) \Vert_{C^{0,\alpha}_\eps(\Omega_3)}. \]
	This is the only time we will use $\Vert \cdot \Vert_{\widetilde{C}^{2,\alpha}_\eps(\Omega)}$ for the purposes of \eqref{eq:dirichlet.data.lemm.3.9.m.perp}. We have
	\begin{align*}
		& \Vert (W''(\mathbb{H}_\eps) - W''(\pm 1))(v_2^\flat - v_1^\flat) \Vert_{C^{0,\alpha}_\eps(\Omega_3)} \\
		& \qquad \leq \Vert (W''(\mathbb{H}_\eps) - W''(\pm 1)) \chi_5(v_2^\flat - v_1^\flat) \Vert_{C^{0,\alpha}_\eps(\Omega)} + \Vert (W''(\mathbb{H}_\eps) - W''(\pm 1)) (1-\chi_5) (v_2^\flat - v_1^\flat) \Vert_{C^{0,\alpha}_\eps(\Omega)} \\
		& \qquad \leq \eps^2 \Vert v_2^\flat - v_1^\flat \Vert_{\widetilde{C}^{2,\alpha}_\eps(\Omega)} + \Vert (W''(\mathbb{H}_\eps) - W''(\pm 1)) (1-\chi_5) (v_2^\flat - v_1^\flat) \Vert_{C^{0,\alpha}_\eps(\Omega \setminus \Omega_5)}.
	\end{align*}
	Recalling $\Vert W''(\mathbb{H}_\eps) - W''(\pm 1) \Vert_{C^{0,\alpha}_\eps(\Omega \setminus \Omega_5)} \leq C_m \eps^m$ for all $m \in \NN$, e.g., as in \eqref{eq:dirichlet.data.lemm.3.9.H.decay}, we deduce
	\[ \Vert (W''(\mathbb{H}_\eps) - W''(\pm 1))(v_2^\flat - v_1^\flat) \Vert_{C^{0,\alpha}_\eps(\Omega_3)} \leq C \eps^2 \Vert v_2^\flat - v_1^\flat \Vert_{\widetilde{C}^{2,\alpha}_\eps(\Omega)}, \]
	so, combined with Remark \ref{rema:dirichlet.data.regularity.product.vs.omega}, Lemma \ref{lemm:dirichlet.data.proj.holder.norms}, \eqref{eq:dirichlet.data.cutoff}, $\delta_* \in (0, 1)$, this gives:
	\begin{multline} \label{eq:dirichlet.data.lemm.3.9.WH}
		 \eps^{\alpha} \Vert \Pi_\eps \big[ \chi_3 (W''(\mathbb{H}_\eps) - W''(\pm 1))(v_2^\flat - v_1^\flat) \big] \Vert_{C^{0,\alpha}(\Sigma)} 
		+ \Vert \Pi_\eps^\perp \big[ \chi_3 (W''(\mathbb{H}_\eps) - W''(\pm 1))(v_2^\flat - v_1^\flat) \big] \Vert_{C^{0,\alpha}_\eps(\Sigma \times \RR)} \\
		\leq C \eps^2 \Vert v_2^\flat - v_1^\flat \Vert_{\widetilde{C}^{2,\alpha}_\eps(\Omega)}. 
	\end{multline}
	
	We now proceed to the more involved contraction estimates pertaining to $M_\eps$. We will estimate:
	\begin{multline} \label{eq:dirichlet.data.lemm.3.9.Leps}
		\Vert (L_\eps v_2^\sharp - \eps^2 \Delta_g (v_2^\sharp \circ D_{\zeta_2}) \circ D_{\zeta_2}^{-1} + W''(\mathbb{H}_\eps) v_2^\sharp) \\
		- (L_\eps v_1^\sharp - \eps^2 \Delta_g (v_1^\sharp \circ D_{\zeta_1}) \circ D_{\zeta_1}^{-1} + W''(\mathbb{H}_\eps) v_1^\sharp) \Vert_{C^{0,\alpha}_\eps(\Omega_3)}.
	\end{multline}
	Note that, by repeating the argument carried out to obtain \eqref{eq:dirichlet.data.lemm.3.9.laplace.zeta.commutator}, except with $v_j^\sharp$ in place of $v_j^\flat$, and also using Remark \ref{rema:dirichlet.data.regularity.product.vs.omega}, we get
	\begin{align}
		& \Vert \eps^2 ((\Delta_g (v_2^\sharp \circ D_{\zeta_2}) \circ D_{\zeta_2}^{-1} - \Delta_g v_2^\sharp)  - (\Delta_g (v_1^\sharp \circ D_{\zeta_1}) \circ D_{\zeta_1}^{-1} - \Delta_g v_1^\sharp)) \Vert_{C^{0,\alpha}_\eps(\Omega_3)} \nonumber \\
		& \qquad \leq C \eps^{2-2\alpha} \Vert v_2^\sharp - v_1^\sharp \Vert_{C^{2,\alpha}_\eps(\Sigma \times \RR)} + C \eps^2 \Vert \zeta_2 - \zeta_1 \Vert_{C^{2,\alpha}(\Sigma)}. \label{eq:dirichlet.data.lemm.3.9.laplace.zeta.commutator.m}
	\end{align}
	In view of Remark \ref{rema:dirichlet.data.regularity.product.vs.omega} and Lemma \ref{lemm:dirichlet.data.proj.holder.norms}, this allows us to estimate
	\begin{align*}
		& (L_\eps v_2^\sharp - \eps^2 \Delta_g v_2^\sharp + W''(\mathbb{H}_\eps) v_2^\sharp) - (L_\eps v_1^\sharp - \eps^2 \Delta_g v_1^\sharp + W''(\mathbb{H}_\eps) v_1^\sharp) \\
		& \qquad = L_\eps (v_2^\sharp - v_1^\sharp) - \eps^2 \Delta_g (v_2^\sharp - v_1^\sharp) + W''(\mathbb{H}_\eps)(v_2^\sharp - v_1^\sharp) \\
		& \qquad = \eps^2 (\Delta_{g_0} + \partial_z^2 - \Delta_g) (v_2^\sharp - v_1^\sharp) 
	\end{align*}
	instead of \eqref{eq:dirichlet.data.lemm.3.9.Leps} in both \eqref{eq:dirichlet.data.lemm.3.9.m.perp} and \eqref{eq:dirichlet.data.lemm.3.9.m}. Let's denote
	\[ \cF_4(v) \triangleq \eps^2 (\Delta_g - \Delta_{g_0} - \partial_z^2) v, \]
	which is evidently a linear functional $\cF_4 : C^{2,\alpha}_\eps(\Omega_3) \to C^{0,\alpha}_\eps(\Omega_3)$. Because $\Delta_g = \Delta_{g_z} + \partial_z^2 + H_z \partial_z$ in Fermi coordinates, we can rewrite
	\[ \cF_4(v) = \eps^2 (\Delta_{g_z} - \Delta_{g_0}) v + \eps^2 H_z \partial_z v. \]
	We now make use of \eqref{eq:mean.curv.ddt.laplace} to write:
	\[ \cF_4(v) = \Big[ - \eps^2 \int_0^z (2 \langle \sff_t, \nabla^2_{g_t} v \rangle_{g_t} + \langle \nabla_{g_t} H_t, \nabla_{g_t} v \rangle_{g_t}) \, dt \Big] + \eps^2 H_z \partial_z v. \]
	First, let's derive $C^0$ bounds. Let $(y, z) \in \Omega_3$. It follows from \eqref{eq:dirichlet.data.sigma.c2alpha}, \eqref{eq:mean.curv.ddt.metric}, \eqref{eq:mean.curv.ddt.sff}, and \eqref{eq:mean.curv.ddt.hess} that
	\begin{equation} \label{eq:dirichlet.data.lemm.3.9.F4.c0.i}
		\left| 2 \eps^2 \int_0^z \langle \sff_t, \nabla^2_{g_t} v \rangle_{g_t} \, dt \right| \leq C |z| \Vert v \Vert_{C^2_\eps(\Omega_3)}.
	\end{equation}
	It follows from \eqref{eq:dirichlet.data.sigma.c3alpha}, \eqref{eq:mean.curv.ddt.metric}, \eqref{eq:mean.curv.ddt.sff}, \eqref{eq:mean.curv.ddt.grad}, and \eqref{eq:mean.curv.ddt.gradsff} that
	\begin{equation} \label{eq:dirichlet.data.lemm.3.9.F4.c0.ii}
		\left| \eps^2 \int_0^z \langle \nabla_{g_t} H_t, \nabla_{g_t} v \rangle_{g_t} \, dt \right| \leq C |z| \Vert v \Vert_{C^1_\eps(\Omega_3)}.	
	\end{equation}
	It follows from \eqref{eq:dirichlet.data.sigma.minimal}, \eqref{eq:dirichlet.data.sigma.c2alpha}, \eqref{eq:mean.curv.ddt.sff} that
	\begin{equation} \label{eq:dirichlet.data.lemm.3.9.F4.c0.iii}
		|\eps^2 H_z \partial_z v| \leq C \eps |z| \Vert v \Vert_{C^1_\eps(\Omega_3)}.
	\end{equation}
	Altogether, \eqref{eq:dirichlet.data.lemm.3.9.F4.c0.i}-\eqref{eq:dirichlet.data.lemm.3.9.F4.c0.iii}, 
	show:
	\begin{equation} \label{eq:dirichlet.data.lemm.3.9.F4.c0}
		 |\cF_4(v)| \leq C |z| \Vert v \Vert_{C^{2,\alpha}_\eps(\Omega_3)} \text{ on } \Omega_3. 
	\end{equation}
	
	Next, let's derive H\"older bounds. For fixed $z \in \Omega_3$, an analogous argument gives 
	\begin{equation} \label{eq:dirichlet.data.lemm.3.9.F4.calpha.y}
		\eps^\alpha [ y \mapsto \cF_4(v)(y,z) ]_\alpha \leq C |z| \Vert v \Vert_{C^{2,\alpha}_\eps(\Omega_3)}.  
	\end{equation}
	
	Now fix $y$. By \eqref{eq:dirichlet.data.sigma.c2alpha}, \eqref{eq:mean.curv.ddt.metric}, \eqref{eq:mean.curv.ddt.sff}, and \eqref{eq:mean.curv.ddt.hess}, we have the \emph{Lipschitz} bound
	\[ \left| \frac{\partial}{\partial z} \left( 2 \eps^2 \int_0^z \langle \sff_t, \nabla^2_{g_t} v \rangle_{g_t} \, dt \right) \right| \leq C \Vert v \Vert_{C^2_\eps(\Omega_3)} \]
	In view of the a priori height bound $|z| \leq \eps^{\delta_*}$, this trivially implies the H\"older bound
	\begin{equation} \label{eq:dirichlet.data.lemm.3.9.F4.calpha.z.i}
		\eps^\alpha \left[ z \mapsto \eps^2 \int_0^z \langle \sff_t, \nabla^2_{g_t} v \rangle_{g_t} \, dt \right]_\alpha \leq C \eps^\alpha \eps^{\delta_*(1-\alpha)} \Vert v \Vert_{C^2_\eps(\Omega_3)}.
	\end{equation}
	By \eqref{eq:dirichlet.data.sigma.c3alpha}, \eqref{eq:mean.curv.ddt.metric}, \eqref{eq:mean.curv.ddt.sff}, \eqref{eq:mean.curv.ddt.grad}, and \eqref{eq:mean.curv.ddt.gradsff}, we have another Lipschitz bound:
	\begin{equation} 
		\left| \frac{\partial}{\partial z} \left( \eps^2 \int_0^z \langle \nabla_{g_t} H_t, \nabla_{g_t} v \rangle_{g_t} \, dt \right) \right| \leq C \Vert v \Vert_{C^1_\eps(\Omega_3)},
	\end{equation}
	which, again by $|z| \leq \eps^{\delta_*}$, implies
	\begin{equation} \label{eq:dirichlet.data.lemm.3.9.F4.calpha.z.ii}
		\eps^\alpha \left[ z \mapsto \eps^2 \int_0^z \langle \nabla_{g_t} H_t, \nabla_{g_t} v \rangle_{g_t} \, dt \right]_\alpha \leq C \eps^\alpha \eps^{\delta_*(1-\alpha)} \Vert v \Vert_{C^1_\eps(\Omega_3)}.
	\end{equation}
	Finally, from \eqref{eq:mean.curv.ddt.h} we have the Lipschitz bound
	\[ \left| \frac{\partial}{\partial z} (\eps^2 H_z \partial_z v) \right| \leq C \Vert v \Vert_{C^2_\eps(\Omega_3)}, \]
	which, again by $|z| \leq \eps^{\delta_*}$, improves to
	\begin{equation} \label{eq:dirichlet.data.lemm.3.9.F4.calpha.z.iii}
		\eps^{\alpha} [ z \mapsto \eps^2 H_z \partial_z v ]_\alpha \leq C \eps^\alpha \eps^{\delta_*(1-\alpha)} \Vert v \Vert_{C^2_\eps(\Omega_3)}.
	\end{equation}
	Altogether, \eqref{eq:dirichlet.data.lemm.3.9.F4.calpha.z.i}-\eqref{eq:dirichlet.data.lemm.3.9.F4.calpha.z.iii} imply
	\begin{equation} \label{eq:dirichlet.data.lemm.3.9.F4.calpha.z}
		\eps^\alpha [ z \mapsto \cF_4(v)(y,z) ]_\alpha \leq C \eps^{\delta_* + \alpha(1-\delta_*)} \Vert v \Vert_{C^2_\eps(\Omega_3)}.
	\end{equation}
	Together, {\eqref{eq:dirichlet.data.lemm.3.9.F4.c0},} \eqref{eq:dirichlet.data.lemm.3.9.F4.calpha.y},  \eqref{eq:dirichlet.data.lemm.3.9.F4.calpha.z} imply
	\begin{equation} \label{eq:dirichlet.data.lemm.3.9.F4}
		\Vert \cF_4(v) \Vert_{C^{0,\alpha}_\eps(\Omega_3)} \leq C \eps^{\delta_*} \Vert v \Vert_{C^{2,\alpha}_\eps(\Omega_3)}.
	\end{equation}
	 Together with Remark \ref{rema:dirichlet.data.regularity.product.vs.omega}, Lemma \ref{lemm:dirichlet.data.proj.holder.norms}, this gives:
	\begin{equation} \label{eq:dirichlet.data.lemm.3.9.F4.proj.perp}
		\Vert \Pi_\eps^\perp \cF_4(v) \Vert_{C^{0,\alpha}_\eps(\Omega_3)} \leq C \eps^{\delta_*} \Vert v \Vert_{C^{2,\alpha}_\eps(\Omega_3)}.
	\end{equation}
	It remains to estimate $\Pi_\eps \cF_4 (v)$. Note the obvious inequality (which follows from \eqref{eq:heteroclinic.expansion.ii}, \eqref{eq:heteroclinic.eps})
	\[ \int_{-\infty}^\infty |z| \partial_z \mathbb{H}_\eps(z) \, dz = \eps \int_{-\infty}^\infty |t| \mathbb{H}'(t) \, dt \leq C \eps \]
	combined with \eqref{eq:dirichlet.data.lemm.3.9.F4.c0} and \eqref{eq:dirichlet.data.lemm.3.9.F4.calpha.y}, readily implies:
	\begin{equation} \label{eq:dirichlet.data.lemm.3.9.F4.proj}
		\Vert \Pi_\eps \cF_4(v) \Vert_{C^{0,\alpha}_\eps(\Sigma)} \leq C \eps \Vert v \Vert_{C^{2,\alpha}_\eps(\Omega_3)} 
		\implies \Vert \Pi_\eps \cF_4(v) \Vert_{C^{0,\alpha}(\Sigma)} \leq C \eps^{1-\alpha} \Vert v \Vert_{C^{2,\alpha}_\eps(\Omega_3)}. 
	\end{equation}
	This completes our study of $\cF_4$, as we have the desired estimates in view of Remark \ref{rema:dirichlet.data.regularity.product.vs.omega}. 
	
	 We proceed to the final contraction estimate pertaining to $M_\eps$, which involves $\Pi_\eps$, $\Pi_\eps^\perp$ of
	\[ \chi_3(\sE_\eps(\zeta_2) - \sE_\eps(\zeta_1) - \eps^2 J_\Sigma (\zeta_2 - \zeta_2) \partial_z \mathbb{H}_\eps). \]
	By \eqref{eq:dirichlet.data.cutoff}, $\delta_* \in (0,1)$, and Lemma \ref{lemm:dirichlet.data.proj.holder.norms}, we may just estimate $\sE_\eps(\zeta_2) - \sE_\eps(\zeta_1) - \eps^2 J_\Sigma (\zeta_2 - \zeta_2) \partial_z \mathbb{H}_\eps$ on $\Omega_3$. 
	
	Fix $(y, z) \in \Omega_3$. Recall the definition of $D_\zeta$ in \eqref{eq:dirichlet.data.offset.map}, and the estimate
	\begin{equation} \label{eq:dirichlet.data.lemm.3.9.M.eps.E.eps.buffer}
		\dist_g( \Omega_3, \{ \chi_2 \neq 1 \} ) = O(\eps^{\delta_*}) \gg \Vert \zeta_2 \Vert_{C^0(\Sigma)} + \Vert \zeta_1 \Vert_{C^0(\Sigma)}
	\end{equation}
	that follows from the a priori bound on $\zeta_1$, $\zeta_2$. Also recall that $\mathbb{H}_\eps \equiv \widetilde{\mathbb{H}}_\eps$ on $\Omega_3$. Then, in Fermi coordinates $(y,z)$, we have:
	\begin{align} \label{eq:dirichlet.data.lemm.3.9.M.eps.E.eps.expr}
		& \sE_\eps(\zeta_2)(y,z) - \sE_\eps(\zeta_1)(y,z) - \eps^2 J_\Sigma (\zeta_2 - \zeta_1)(y) \cdot \partial_z \mathbb{H}_\eps(z) \\
		& \qquad = \eps^2 \Delta_g (\mathbb{H}_\eps \circ D_{\zeta_2}) \circ D_{\zeta_2}^{-1}(y,z) - \eps^2 \Delta_g (\mathbb{H}_\eps \circ D_{\zeta_1}) \circ D_{\zeta_1}^{-1}(y,z) \nonumber \\
		& \qquad \qquad - \eps^2 J_\Sigma(\zeta_2 - \zeta_1)(y) \cdot \partial_z \mathbb{H}_\eps(z) \nonumber \\
		& \qquad = \eps^2 \Big[ \partial_z^2 \mathbb{H}_\eps(z) \big( |\nabla_{g_{z+\zeta_2(y)}} \zeta_2(y)|^2 - |\nabla_{g_{z+\zeta_1(y)}} \zeta_1(y)|^2 \big) \nonumber \\
		& \qquad \qquad - \partial_z \mathbb{H}_\eps(z) \big( (\Delta_{g_{z+\zeta_2(y)}} \zeta_2(y) - H_{z+\zeta_2(y)}(y)) - (\Delta_{g_{z+\zeta_1(y)}} \zeta_1(y) - H_{z+\zeta_1(y)}(y)) + J_\Sigma (\zeta_2 - \zeta_1)(y) \big) \Big]. \nonumber
	\end{align}
	Denote, for $\zeta \in C^{1,\alpha}(\Sigma)$,
	\[ \cF_5(\zeta)(y,z) \triangleq |\nabla_{g_{z + \zeta(y)}} \zeta(y)|^2 = g_{z + \zeta(y)}^{ij} \zeta_i(y) \zeta_j(y) \]
	to be the smooth nonlinear functional, $\cF_5 : C^{1,\alpha}(\Sigma) \to C^{0,\alpha}(\Omega_3)$. By virtue of \eqref{eq:mean.curv.ddt.metric}, we know that:
	\begin{equation} \label{eq:dirichlet.data.lemm.3.9.F5.derivative}
		\langle D_\zeta \cF_5(\zeta), w \rangle(y,z) = - 2 \sff_{z+\zeta(y)}^{ij} \zeta_i(y) \zeta_j(y) w(y) + 2 g^{ij}_{z + \zeta(y)} w_i(y) \zeta_j(y).
	\end{equation}
	By the fundamental theorem of calculus, 
	\[ \cF_5(\zeta_2) - \cF_5(\zeta_1) = \int_0^1 \langle D_\zeta \cF_5( \zeta_1 + t (\zeta_2 - \zeta_1)), \zeta_2 - \zeta_1 \rangle \, dt, \]
	so together with  \eqref{eq:dirichlet.data.sigma.c2alpha}, the a priori estimates on $\zeta_1$, $\zeta_2$,  \eqref{eq:dirichlet.data.lemm.3.9.F5.derivative}, \eqref{eq:mean.curv.ddt.metric}, and  \eqref{eq:mean.curv.ddt.sff}:
	\[ \Vert \cF_5(\zeta_2) - \cF_5(\zeta_1) \Vert_{C^{0,\alpha}(\Omega_3)} \leq C \eps^{2-2\alpha} \Vert \zeta_2 - \zeta_1 \Vert_{C^{1,\alpha}(\Sigma)}. \]
	Alongside \eqref{eq:heteroclinic.expansion.iii}, Remark \ref{rema:dirichlet.data.regularity.product.vs.omega}, Lemma \ref{lemm:dirichlet.data.proj.holder.norms}, \eqref{eq:dirichlet.data.cutoff}, $\delta_* \in (0,1)$, this implies:
	\begin{multline} \label{eq:dirichlet.data.lemm.3.9.M.eps.E.eps.i}
		\eps^{\alpha} \left\Vert \Pi_\eps \big( \chi_3 \eps^2 (\partial_z^2 \mathbb{H}_\eps) (\cF_5(\zeta_2) - \cF_5(\zeta_1)) \big) \right\Vert_{C^{0,\alpha}(\Sigma)} + \left\Vert \Pi_\eps^\perp \big( \chi_3 \eps^2 (\partial_z^2 \mathbb{H}_\eps) (\cF_5(\zeta_2) - \cF_5(\zeta_1)) \big) \right\Vert_{C^{0,\alpha}_\eps(\Sigma \times \RR)} \\
		\leq C \eps^{2-2\alpha} \Vert \zeta_2 - \zeta_1 \Vert_{C^{1,\alpha}(\Sigma)}.
	\end{multline}
	
	Finally, let's denote
	\[ \cF_6(\zeta)(y,z) \triangleq \eps \Big( \Delta_{z + \zeta(y)} \zeta(y) - H_{z + \zeta(y)}(y) + J_\Sigma \zeta(y) \Big) \]
	to be the smooth nonlinear Banach space functional $\cF_6 : C^{2,\alpha}(\Sigma) \to C^{0,\alpha}_\eps(\Omega_3)$. By \eqref{eq:mean.curv.ddt.h} and \eqref{eq:mean.curv.ddt.laplace},
	\begin{align*}
		\langle D_\zeta \cF_6(\zeta), w \rangle & = \eps \Big( \Delta_{z + \zeta} w + \big( -2 \langle \sff_{z+\zeta}, \nabla^2_{g_{z+\zeta}} \zeta \rangle_{g_{z+\zeta}} - \langle \nabla_{g_{z+\zeta}} H_{z+\zeta}, \nabla_{g_{z+\zeta}} \zeta \rangle_{g_{z+\zeta}} \big) w \\
		& \qquad + (|\sff_{z+\zeta}|^2 + \ricc_g(\partial_z, \partial_z)|_{D \times \{z+\zeta\}}) w + J_\Sigma w \Big) \\
		& = \eps \Big( \big( -2 \langle \sff_{z+\zeta}, \nabla^2_{g_{z+\zeta}} \zeta \rangle_{g_{z+\zeta}} - \langle \nabla_{g_{z+\zeta}} H_{z+\zeta}, \nabla_{g_{z+\zeta}} \zeta \rangle_{g_{z+\zeta}} \big) w \\
		& \qquad - \int_{0}^{z+\zeta} ( 2 \langle \sff_t, \nabla^2_{g_t} w \rangle_{g_t} + \langle \nabla_{g_t} H_t, \nabla_{g_t} w \rangle_{g_t} ) \, dt \\
		& \qquad + \Big( \int_{0}^{z+\zeta} \tfrac{\partial}{\partial t} (|\sff_t|^2 + \ricc_g(\partial_z, \partial_z)|_{D \times \{t\}}) \, dt \Big) w\Big).
	\end{align*}
	By the fundamental theorem of calculus,
	\[ \cF_6(\zeta_2) - \cF_6(\zeta_1) = \int_0^1 \langle D_\zeta \cF_6(\zeta_1 + t(\zeta_2 - \zeta_1)), \zeta_2 - \zeta_1 \rangle \, dt. \]
	We now estimate $\langle D_\zeta \cF_6(\zeta), w \rangle$ for $\zeta = \zeta_1 + t(\zeta_2 - \zeta_1)$ and $w = \zeta_2 - \zeta_1$. We will make repeated use of \eqref{eq:dirichlet.data.sigma.c2alpha}, \eqref{eq:dirichlet.data.sigma.c3alpha}, \eqref{eq:mean.curv.ddt.metric}, \eqref{eq:mean.curv.ddt.sff}, \eqref{eq:mean.curv.ddt.grad}, \eqref{eq:mean.curv.ddt.christoffel}, \eqref{eq:mean.curv.ddt.hess}, $\Vert \zeta \Vert_{C^{2,\alpha}(\Sigma)} \leq C' \eps^{2-2\alpha}$, and $\Vert \cdot \Vert_{C^{0,\alpha}_\eps(\Sigma)} \leq \Vert \cdot \Vert_{C^{2,\alpha}(\Sigma)}$. First, 
	\begin{equation} \label{eq:dirichlet.data.lemm.3.9.F6.i}
		\left\Vert \eps \big( 2 \langle \sff_{z+\zeta}, \nabla^2_{g_{z+\zeta}} \zeta \rangle_{g_{z+\zeta}} + \langle \nabla_{g_{z+\zeta}} H_{z+\zeta}, \nabla_{g_{z+\zeta}} \zeta \rangle_{g_{z+\zeta}} \big) w \right\Vert_{C^{0,\alpha}_\eps(\Omega_3)} \leq C \eps^{2-2\alpha} \Vert \zeta_2 - \zeta_1 \Vert_{C^{2,\alpha}(\Sigma)}.
	\end{equation}
	Additionally using the $O(\eps^{\delta_*})$ height bound on $\Omega_3$, we also have:
	\begin{equation} \label{eq:dirichlet.data.lemm.3.9.F6.ii}
		\left\Vert \eps \int_0^{z+\zeta} \langle \sff_{z+\zeta}, \nabla^2_{g_{z+\zeta}} w \rangle_{g_{z+\zeta}} \right\Vert_{C^{0,\alpha}_\eps(\Omega_3)} \leq C \eps^{1+\delta_*} \Vert \zeta_2 - \zeta_1 \Vert_{C^{2,\alpha}(\Sigma)}.
	\end{equation}
	Likewise: 
	\begin{equation} \label{eq:dirichlet.data.lemm.3.9.F6.iii}
		\left\Vert \eps \Big( \int_{0}^{z+\zeta} \tfrac{\partial}{\partial t} (|\sff_t|^2 + \ricc_g(\partial_z, \partial_z)|_{D \times \{t\}}) \, dt \Big) w \right\Vert_{C^{0,\alpha}_\eps(\Omega_3)} \leq C \eps^{1+\delta_*} \Vert \zeta_2 - \zeta_1 \Vert_{C^{0,\alpha}(\Sigma)}.
	\end{equation}
	It remains to estimate:
	\[ \left\Vert \eps \int_0^{z+\zeta} \langle \nabla_{g_t} H_t, \nabla_{g_t} w \rangle_{g_t} \, dt \right\Vert_{C^{0,\alpha}_\eps(\Omega_3)}. \]
	Now is the only place in the proof where we need to distinguish the H\"older exponents $\alpha \leq \theta$, taking the prior to be small and the latter to be large. From \eqref{eq:dirichlet.data.sigma.c2alpha}, \eqref{eq:dirichlet.data.sigma.c3alpha}, \eqref{eq:mean.curv.ddt.metric}, \eqref{eq:mean.curv.ddt.sff}, \eqref{eq:mean.curv.ddt.grad} and the interpolation of (unweighted) H\"older spaces $C^{1,\theta} \hookrightarrow C^{1,\alpha} \hookrightarrow C^{0,\theta}$ (Lemma \ref{lemm:holder.space.interpolation}), we have
	\[ \Vert \nabla_{g_z} H_z \Vert_{C^{0,\alpha}(\Omega_3)} \leq C \Vert H_z \Vert_{C^{0,\theta}(\Omega)}^{\theta-\alpha} \Vert H_z \Vert_{C^{1,\theta}(\Omega)}^{1+\alpha-\theta} \leq C \eps^{-2 (1+\alpha-\theta)} \leq C \eps^{-\frac12 \delta_*}, \]
	as long as $\alpha_0$, $\theta_0$ are chosen sufficiently close to $0$ and to $1$, respectively, depending on $\delta_*$. It is now easy to see, as before, that
	\begin{equation} \label{eq:dirichlet.data.lemm.3.9.F6.iv}
		\left \Vert \eps \int_0^{z+\zeta} \langle \nabla_{g_t} H_t, \nabla_{g_t} w \rangle_{g_t} \, dt \right\Vert_{C^{0,\alpha}_\eps(\Omega_3)} \leq C \eps^{1+\frac12 \delta_*} \Vert \zeta_2 - \zeta_1 \Vert_{C^{1,\alpha}(\Sigma)}.
	\end{equation}	
	Altogether, \eqref{eq:dirichlet.data.lemm.3.9.F6.i}, \eqref{eq:dirichlet.data.lemm.3.9.F6.ii}, \eqref{eq:dirichlet.data.lemm.3.9.F6.iii}, \eqref{eq:dirichlet.data.lemm.3.9.F6.iv} imply:
	\[ \Vert \cF_6(\zeta_2) - \cF_6(\zeta_1) \Vert_{C^{0,\alpha}_\eps(\Omega_3)} \leq C \eps^{1+\frac12 \delta_*} \Vert \zeta_2 - \zeta_1 \Vert_{C^{2,\alpha}(\Sigma)}.  \]
	Alongside \eqref{eq:heteroclinic.expansion.ii}, Remark \ref{rema:dirichlet.data.regularity.product.vs.omega}, Lemma \ref{lemm:dirichlet.data.proj.holder.norms}, \eqref{eq:dirichlet.data.cutoff}, $\delta_* \in (0,1)$, this implies:
	\begin{multline} \label{eq:dirichlet.data.lemm.3.9.M.eps.E.eps.ii}
	\eps^{\alpha} \left\Vert \Pi_\eps \big( \chi_3 \eps (\partial_z \mathbb{H}_\eps) (\cF_6(\zeta_2) - \cF_6(\zeta_1)) \big) \right\Vert_{C^{0,\alpha}(\Sigma)} + \left\Vert \Pi_\eps^\perp \big( \chi_3 (\eps \partial_z \mathbb{H}_\eps) (\cF_6(\zeta_2) - \cF_6(\zeta_1)) \big) \right\Vert_{C^{0,\alpha}_\eps(\Sigma \times \RR)} \\
	\leq C (\eps^{2-2\alpha} + \eps^{1+\frac12 \delta_*}) \Vert \zeta_2 - \zeta_1 \Vert_{C^{2,\alpha}(\Sigma)}.
	\end{multline}

	Together,  \eqref{eq:dirichlet.data.lemm.3.9.zeta.commutator.m},  \eqref{eq:dirichlet.data.lemm.3.9.Qeps.m}, \eqref{eq:dirichlet.data.lemm.3.9.WH}, \eqref{eq:dirichlet.data.lemm.3.9.laplace.zeta.commutator.m}, \eqref{eq:dirichlet.data.lemm.3.9.F4.proj.perp}, \eqref{eq:dirichlet.data.lemm.3.9.M.eps.E.eps.expr}, \eqref{eq:dirichlet.data.lemm.3.9.M.eps.E.eps.i}, and \eqref{eq:dirichlet.data.lemm.3.9.M.eps.E.eps.ii} imply  \eqref{eq:dirichlet.data.lemm.3.9.m.perp} for $\alpha_0$, $\theta_0$ depending on $\delta_*$.

	Likewise, \eqref{eq:dirichlet.data.lemm.3.9.zeta.commutator.m},  \eqref{eq:dirichlet.data.lemm.3.9.Qeps.m},  \eqref{eq:dirichlet.data.lemm.3.9.WH}, \eqref{eq:dirichlet.data.lemm.3.9.laplace.zeta.commutator.m},  \eqref{eq:dirichlet.data.lemm.3.9.F4.proj}, \eqref{eq:dirichlet.data.lemm.3.9.M.eps.E.eps.expr}, \eqref{eq:dirichlet.data.lemm.3.9.M.eps.E.eps.i}, \eqref{eq:dirichlet.data.lemm.3.9.M.eps.E.eps.ii} imply  \eqref{eq:dirichlet.data.lemm.3.9.m} for $\alpha_0$, $\theta_0$ depending on $\delta_*$.
\end{proof}

\begin{proof}[Proof of Theorem \ref{theo:dirichlet.data.construction}]
	As was already pointed out, we can rewrite \eqref{eq:dirichlet.data.pde} as the nonlinear fixed point problem \eqref{eq:dirichlet.data.pde.vsharp}-\eqref{eq:dirichlet.data.pde.zeta}. We'll take $\alpha$, $\theta$, $\delta$ as in Lemma \ref{lemm:dirichlet.data.lemm.3.9}, and $M \geq 1$.
	
	Consider $g$ as in Section \ref{sec:dirichlet.data}, and also define
	\begin{multline} \label{eq:dirichlet.data.contraction.interior.ball}
		\cU(\eps; M) \triangleq \Big\{ (v^\flat, v^\sharp, \zeta) \in \widetilde{C}^{2,\alpha}_\eps(\Omega) \times C^{2,\alpha}_\eps(\Sigma \times \RR) \times C^{2,\alpha}(\Sigma) : \\
		\Vert v^\flat \Vert_{\widetilde{C}^{2,\alpha}_\eps(\Omega)} + \Vert v^\sharp \Vert_{C^{2,\alpha}_\eps(\Sigma \times \RR)} + \eps^{2\alpha} \Vert \zeta \Vert_{C^{2,\alpha}(\Sigma)} \leq M \eps^2 \Big\}. 
	\end{multline}
	\begin{multline} \label{eq:dirichlet.data.contraction.boundary.ball}
		\cB(\eps; \mu) \triangleq \Big\{ (\widehat{v}^\flat, \widehat{v}^\sharp, \widehat{\zeta}) \in  C^{2,\alpha}_\eps(\partial \Omega) \times C^{2,\alpha}_\eps(\partial \Sigma \times \RR) \times C^{2,\alpha}(\partial \Sigma) : \\
		\widehat{v}^\flat \equiv 0 \text{ on } \{ \chi_4 = 1 \}, \; \Pi_\eps(\widehat{v}^\sharp) \equiv 0 \text{ on } \partial \Sigma, \\
		\Vert \widehat{v}^\flat \Vert_{C^{2,\alpha}_\eps(\partial \Omega)} + \Vert \widehat{v}^\sharp \Vert_{C^{2,\alpha}_\eps(\partial \Sigma \times \RR)} + \Vert \widehat{\zeta} \Vert_{C^{2,\alpha}(\partial \Sigma)} \leq \mu \eps^2 \Big\}.
	\end{multline}

	Lemmas \ref{lemm:dirichlet.data.lemm.3.8}, \ref{lemm:dirichlet.data.lemm.3.9}, guarantee that for every $(v^\flat, v^\sharp, \zeta) \in \cU(\eps; M)$, 
	\begin{align}
		\Vert N_\eps(v^\flat, v^\sharp, \zeta) \Vert_{C^{0,\alpha}_\eps(\Omega)} 
			& \leq  c_1' \eps^{2+\delta-2\alpha} + c_0 \eps^2, \label{eq:dirichlet.data.thm.est.n} \\
		\Vert \Pi_\eps^\perp M_\eps (v^\flat, v^\sharp, \zeta) \Vert_{C^{0,\alpha}_\eps(\Sigma \times \RR)} 
			& \leq  c_1'  \eps^{2+\delta-2\alpha} + c_0 \eps^2, \label{eq:dirichlet.data.thm.est.m.perp} \\
		\Vert \eps^{-1} \Pi_\eps M_\eps(v^\flat, v^\sharp, \zeta) \Vert_{C^{0,\alpha}(\Sigma)} 
			& \leq   c_1'  \eps^{2+\delta-2\alpha} + c_1'  \eps^{2-\alpha} + c_0 \eps^2, \label{eq:dirichlet.data.thm.est.m}
	\end{align}
	with $c_0$ as in Lemma \ref{lemm:dirichlet.data.lemm.3.8}, and with $c_1' = M \cdot c_1$, $\eps \leq \eps_0$ as in Lemma  \ref{lemm:dirichlet.data.lemm.3.9}.
	
	Let
	\[ \Phi : \cU(\eps; M) \times \cB(\eps; \mu) \times \operatorname{Met}_{\eps,\eta}(\Omega) \to \widetilde{C}^{2,\alpha}_\eps(\Omega) \times C^{2,\alpha}_\eps(\Sigma \times \RR) \times C^{2,\alpha}(\Sigma), \]
	be the solution map $\Phi : (v^\flat, v^\sharp, \zeta, \widehat{v}^\flat, \widehat{v}^\sharp, \widehat{\zeta}, g) \mapsto (V^\flat, V^\sharp, Z)$ for the \emph{linear} system
	\begin{equation} \label{eq:dirichlet.data.thm.pde.vflat}
		\cL_\eps V^\flat = N_\eps(v^\flat, v^\sharp, \zeta) \text{ on } \Omega, \; V^\flat|_{\partial \Omega} = \widehat{v}^\flat,
	\end{equation}
	\begin{equation} \label{eq:dirichlet.data.thm.pde.vsharp}
		L_\eps V^\sharp = \Pi_\eps^\perp M_\eps(v^\flat, v^\sharp, \zeta) \text{ on } \Sigma \times \RR, \; V^\sharp|_{\partial \Sigma \times \RR} = \widehat{v}^\sharp,
	\end{equation}
	\begin{equation} \label{eq:dirichlet.data.thm.pde.zeta}
		J_\Sigma Z = \eps^{-1} \Pi_\eps M_\eps(v^\flat, v^\sharp, \zeta) \text{ on } \Sigma, \; Z|_{\partial \Sigma} = \widehat{\zeta}.
	\end{equation}

	The existence of $V^\flat$ follows from Fredholm theory. In fact, together with Lemma \ref{lemm:dirichlet.data.eq.3.26},  \eqref{eq:dirichlet.data.thm.est.n}, we have
	\begin{align}
		\Vert V^\flat \Vert_{\widetilde{C}^{2,\alpha}_\eps(\Omega)} 
			& \leq C ( \Vert N_\eps (v^\flat, v^\sharp, \zeta) \Vert_{C^{0,\alpha}_\eps(\Omega)} + \Vert \widehat{v}^\flat \Vert_{C^{2,\alpha}_\eps(\Omega)})  \leq C c_1'  \eps^{2+\delta-2\alpha} + C (c_0 + \mu) \eps^2. \label{eq:dirichlet.data.thm.est.vflat}
	\end{align}
	
	The existence of $V^\sharp$ follows from Lemma \ref{lemm:dirichlet.data.prop.3.1}. In fact, together with Lemma \ref{lemm:dirichlet.data.prop.3.2},  \eqref{eq:dirichlet.data.thm.est.m.perp}, we have
	\begin{equation} \label{eq:dirichlet.data.thm.est.vsharp}
		\Vert V^\sharp \Vert_{C^{2,\alpha}_\eps(\Sigma \times \RR)} 
			\leq C ( \Vert \Pi_\eps^\perp M_\eps(v^\flat, v^\sharp, \zeta) \Vert_{C^{0,\alpha}_\eps(\Sigma \times \RR)} + \Vert \widehat{v}^\sharp \Vert_{C^{2,\alpha}_\eps(\partial \Sigma \times \RR)} 			\leq  C c_1'  \eps^{2+\delta-2\alpha} + C (c_0 + \mu) \eps^2. 
	\end{equation}
	
	Finally, the existence of $Z$ follows from Fredholm theory and \eqref{eq:dirichlet.data.sigma.nondegenerate}. In fact, by Schauder theory on the elliptic operator $J_\Sigma$ on $\Sigma$, and \eqref{eq:dirichlet.data.thm.est.m}, we find:
	\begin{align}
		\Vert Z \Vert_{C^{2,\alpha}(\Sigma)}
			& \leq C ( \Vert \eps^{-1} \Pi_\eps M_\eps (v^\flat, v^\sharp, \zeta) \Vert_{C^{0,\alpha}(\Sigma)} + \Vert \widehat{\zeta} \Vert_{C^{2,\alpha}(\partial \Sigma)}) \nonumber \\
			& \leq C c_1'  \eps^{2+\delta-2\alpha} + C c_1' \eps^{2-\alpha} + C c_0 \eps^2 + C \mu  \eps^{2-2\alpha}, \nonumber \\
		\implies \eps^{2\alpha} \Vert Z \Vert_{C^{2,\alpha}(\Sigma)} & \leq C c_1' \eps^{2+\delta} + Cc_1' \eps^{2+\alpha} + C c_0 \eps^{2+2\alpha} + C \mu \eps^2. \label{eq:dirichlet.data.thm.est.zeta}
	\end{align}
	
	We emphasize that the constant $C$ in \eqref{eq:dirichlet.data.thm.est.vflat}, \eqref{eq:dirichlet.data.thm.est.vsharp}, and \eqref{eq:dirichlet.data.thm.est.zeta} depends only on $n$, $\eta > 0$, and $W$.

	The expressions in \eqref{eq:dirichlet.data.thm.est.vflat}, \eqref{eq:dirichlet.data.thm.est.vsharp}, and \eqref{eq:dirichlet.data.thm.est.zeta} can all be made to be $\leq \tfrac{1}{3} M \eps^2$ as follows:

	\begin{enumerate}
		\item Choose $M$ large, depending on $c_0$, $C$, $\mu$, so that $C(c_0 + \mu) \leq \tfrac{1}{6} M$.
		\item Then, choose $\eps \leq \eps_0$ small depending on $C$, $c_1'$, $M$, so that
		\begin{equation} \label{eq:dirichlet.data.thm.contraction.constant}
			C c_1' \eps^{\alpha} \ll 1;
		\end{equation}
		note that, since $M \geq 1$, the left hand side is also $\leq \tfrac{1}{12} M$.
		\item Using $\alpha \in (0, \tfrac{\delta}{3})$ we find that $\eps^{\delta - 2\alpha} \leq \eps^{\alpha}$, so $C c_1' \eps^{\delta} \leq C c_1' \eps^{\delta-2\alpha} \leq \tfrac{1}{12} M$.
	\end{enumerate}
	Thus, for such a choice of $M = M(n, \eta, W, \delta_*, \mu)$, $\eps \leq \eps_0' = \eps_0'(n, \eta, W, \delta_*, \mu, \alpha)$, we have
	\[ \Phi \big( \cU(\eps; M) \times \cB(\eps; \mu) \times \operatorname{Met}_{\eps,\eta}(\Omega) \big) \subset \cU(\eps; M). \]
	
	We show that $\Phi(\cdot, \cdot, \cdot, \widehat{v}^\flat, \widehat{v}^\sharp, \widehat{\zeta}, g)$ is a \emph{contraction} with respect to the norm 
	\begin{equation} \label{eq:dirichlet.data.interior.product.norm}
		\Vert (v^\flat, v^\sharp, \zeta) \Vert_{\cU} \triangleq \Vert v^\flat \Vert_{\widetilde{C}^{2,\alpha}_\eps(\Omega)} + \Vert v^\sharp \Vert_{C^{2,\alpha}_\eps(\Sigma \times \RR)} + \eps^{2\alpha} \Vert \zeta \Vert_{C^{2,\alpha}(\Sigma)},
	\end{equation}
	\emph{uniformly} with respect to $\widehat{v}^\flat$, $\widehat{v}^\sharp, \widehat{\zeta}, g$. Let's also define
	\begin{equation} \label{eq:dirichlet.data.boundary.product.norm}
		\Vert (\widehat{v}^\flat, \widehat{v}^\sharp, \widehat{\zeta}) \Vert_{\cB} \triangleq \Vert \widehat{v}^\flat \Vert_{\widetilde{C}^{2,\alpha}_\eps(\partial \Omega)} + \Vert \widehat{v}^\sharp \Vert_{C^{2,\alpha}_\eps(\partial \Sigma \times \RR)} + \Vert \widehat{\zeta} \Vert_{C^{2,\alpha}(\partial \Sigma)}.
	\end{equation}
	
	Let's set
	\[ (V_1^\flat, V_1^\sharp, Z_1) \triangleq \Phi(v_1^\flat, v_1^\sharp, \zeta_1, \widehat{v}^\flat, \widehat{v}^\sharp, \widehat{\zeta}, g), \]
	\[ (V_2^\flat, V_2^\sharp, Z_2) \triangleq \Phi(v_2^\flat, v_2^\sharp, \zeta_2, \widehat{v}^\flat, \widehat{v}^\sharp, \widehat{\zeta}, g). \]
	
	By Lemma \ref{lemm:dirichlet.data.eq.3.26}, Lemma \ref{lemm:dirichlet.data.lemm.3.9}:
	\begin{align} \label{eq:dirichlet.data.thm.contraction.vflat}
			 \Vert V_2^\flat - V_1^\flat \Vert_{\widetilde{C}^{2,\alpha}_\eps(\Omega)} & \leq C \Vert \cL_\eps V_2^\flat - \cL_\eps V_1^\flat \Vert_{C^{0,\alpha}_\eps(\Omega)} \nonumber \\
			&  = C \Vert N_\eps(v_2^\flat, v_2^\sharp, \zeta_2) - N_\eps(v_1^\flat, v_1^\sharp, \zeta_1) \Vert_{C^{0,\alpha}_\eps(\Omega)} \nonumber \\
			&  \leq C c_1' \eps^{\delta} \Big( \Vert v_2^\flat - v_1^\flat \Vert_{\widetilde{C}^{2,\alpha}(\Omega)} + \Vert v_2^\sharp - v_1^\sharp \Vert_{C^{2,\alpha}_\eps(\Sigma \times \RR)}  + \Vert \zeta_2 - \zeta_1 \Vert_{C^{2,\alpha}(\Sigma)} \Big).
	\end{align}
	By Lemma \ref{lemm:dirichlet.data.prop.3.2}, Lemma \ref{lemm:dirichlet.data.lemm.3.9},
	\begin{align} \label{eq:dirichlet.data.thm.contraction.vsharp}
			 \Vert V_2^\sharp - V_1^\sharp \Vert_{C^{2,\alpha}_\eps(\Sigma \times \RR)}  & \leq C \Vert L_\eps V_2^\sharp - L_\eps V_1^\sharp \Vert_{C^{0,\alpha}_{\eps}(\Sigma \times \RR)} \nonumber \\
			&  = C \Vert \Pi_\eps^\perp M_\eps(v_2^\flat, v_2^\sharp, \zeta_2) - \Pi_\eps^\perp M_\eps(v_1^\flat, v_1^\sharp, \zeta_1) \Vert_{C^{0,\alpha}_\eps(\Sigma \times \RR)} \nonumber \\
			&  \leq  C c_1' \eps^\delta \Big( \Vert v_2^\flat - v_1^\flat \Vert_{\widetilde{C}^{2,\alpha}_\eps(\Omega)} + \Vert v_2^\sharp - v_1^\sharp \Vert_{C^{2,\alpha}_\eps(\Sigma \times \RR)}  + \Vert \zeta_2 - \zeta_1 \Vert_{C^{2,\alpha}(\Sigma)} \Big).
	\end{align}
	Finally, by Lemma \ref{lemm:dirichlet.data.lemm.3.9}, \eqref{eq:dirichlet.data.sigma.nondegenerate}, and Schauder theory,
	\begin{align} \label{eq:dirichlet.data.thm.contraction.zeta}
		\Vert Z_2 - Z_1 \Vert_{C^{2,\alpha}(\Sigma)} \nonumber & \leq C \Vert J_\Sigma Z_2 - J_\Sigma Z_1 \Vert_{C^{0,\alpha}(\Sigma)} \nonumber \\
		&  = C \Vert \eps^{-1} \Pi_\eps M_\eps(v_2^\flat, v_2^\sharp, \zeta_2) - \eps^{-1} \Pi_\eps M_\eps(v_1^\flat, v_1^\sharp, \zeta_1) \Vert_{C^{0,\alpha}(\Sigma)} \nonumber \\
		&  \leq  C c_1' \Big[ \eps^\delta \big( \Vert v_2^\flat - v_1^\flat \Vert_{\widetilde{C}^{2,\alpha}_\eps(\Omega)} + \Vert \zeta_2 - \zeta_1 \Vert_{C^{2,\alpha}(\Sigma)} \big) + \eps^{-\alpha} \Vert v_2^\sharp - v_1^\sharp \Vert_{C^{2,\alpha}_\eps(\Sigma \times \RR)} \Big].
	\end{align}
	Adding \eqref{eq:dirichlet.data.thm.contraction.vflat}, \eqref{eq:dirichlet.data.thm.contraction.vsharp}, and $\eps^{2\alpha}$ times \eqref{eq:dirichlet.data.thm.contraction.zeta}, using $\alpha < \tfrac{1}{3} \delta$, and the $\Vert \cdot \Vert_{\cU}$ norm on $\cU(\eps; M)$:
	\begin{equation} \label{eq:dirichlet.data.thm.contraction}
		\Vert (V_2^\flat, V_2^\sharp, Z_2) - (V_1^\flat, V_1^\sharp, Z_1) \Vert_{\cU} 
		\leq  C c_1' \eps^\alpha \Vert (v_2^\flat, v_2^\sharp, \zeta_2) - (v_1^\flat, v_1^\sharp, \zeta_1) \Vert_{\cU}.
	\end{equation}
	This implies that $\Phi(\cdot, \cdot, \cdot, \widehat{v}^\flat, \widehat{v}^\sharp, \widehat{\zeta}, g)$ is uniformly Lipschitz, with Lipschitz constant $\leq C c_1' \eps^\alpha$, and by  \eqref{eq:dirichlet.data.thm.contraction.constant} we conclude that it's, in fact, a contraction mapping. This readily implies the existence of a fixed point $(v^\flat, v^\sharp, \zeta)$, which therefore satisfies \eqref{eq:dirichlet.data.pde}.
	
	We finally move to prove the continuity of the solution map
	\[ \cS : \cB(\eps; \mu) \times \operatorname{Met}_{\eps,\eta}(\Omega) \to \cU(\eps; M). \]
 	For $(\widehat{v}_1^\flat, \widehat{v}_1^\sharp, \widehat{\zeta}_1, g_1)$, $(\widehat{v}_2^\flat, \widehat{v}_2^\sharp, \widehat{\zeta}_2, g_2) \in \cB(\eps; \mu) \times  \operatorname{Met}_{\eps,\eta}(\Omega)$, we have, by the fixed point property:
 	\begin{align*}
 		& \cS(\widehat{v}_2^\flat, \widehat{v}_2^\sharp, \widehat{\zeta}_2, g_2) - \cS(\widehat{v}_1^\flat, \widehat{v}_1^\sharp, \widehat{\zeta}_1, g_1) \\
 		& \qquad = \Big( \Phi(\cS(\widehat{v}_2^\flat, \widehat{v}_2^\sharp, \widehat{\zeta}_2, g_2), \widehat{v}_2^\flat, \widehat{v}_2^\sharp, \widehat{\zeta}_2, g_2)  - \Phi(\cS(\widehat{v}_2^\flat, \widehat{v}_2^\sharp, \widehat{\zeta}_2, g_2), \widehat{v}_1^\flat, \widehat{v}_1^\sharp, \widehat{\zeta}_1, g_1) \Big) \\
 		& \qquad \qquad - \Big( \Phi(\cS(\widehat{v}_1^\flat, \widehat{v}_1^\sharp, \widehat{\zeta}_1, g_1), \widehat{v}_1^\flat, \widehat{v}_1^\sharp, \widehat{\zeta}_1, g_1)  - \Phi(\cS(\widehat{v}_2^\flat, \widehat{v}_2^\sharp, \widehat{\zeta}_2, g_2), \widehat{v}_1^\flat, \widehat{v}_1^\sharp, \widehat{\zeta}_1, g_1) \Big).
 	\end{align*}
	
	The last parenthesis will be bounded using the contraction mapping property \eqref{eq:dirichlet.data.thm.contraction} on $(\widehat{v}_1^\flat, \widehat{v}_1^\sharp, \widehat{\zeta}_1, g_1)$. The second-to-last parenthesis will be bounded by varying the four slots of $\Phi(\cS(\widehat{v}_2^\flat, \widehat{v}_2^\sharp, \widehat{\zeta}_2, g_2), \cdot, \cdot, \cdot, \cdot)$ using the fundamental theorem of calculus. The $\widehat{v}^\flat$, $\widehat{v}^\sharp$, $\widehat{\zeta}$ derivatives of $\Phi(\cS(\widehat{v}_2^\flat, \widehat{v}_2^\sharp, \widehat{\zeta}_2, g_2), \cdot, \cdot, \cdot, \cdot)$ can be controlled using Lemma \ref{lemm:dirichlet.data.eq.3.26}, Lemma \ref{lemm:dirichlet.data.prop.3.1}, and Schauder theory on $J_\Sigma$, respectively. Likewise, it is not hard to see that for $g \in \operatorname{Met}_{\eps,\eta}(\Omega)$, the map 
	\begin{equation*}
		g\mapsto \Phi(v^{\flat},v^{\sharp},\zeta,\widehat{v}^{\flat}, \widehat{v}^{\sharp}, \widehat{\zeta},g)
	\end{equation*} 
	is uniformly Lipschitz with respect to $(v^{\flat},v^{\sharp},\zeta,\widehat{v}^{\flat}, \widehat{v}^{\sharp}, \widehat{\zeta}) \in  \cU(\eps; M) \times \cB(\eps; \mu)$. Altogether, we have
 	\begin{align*}
 		& \Vert \cS(\widehat{v}_2^\flat, \widehat{v}_2^\sharp, \widehat{\zeta}_2, g_2) - \cS(\widehat{v}_1^\flat, \widehat{v}_1^\sharp, \widehat{\zeta}_1, g_1)) \Vert_{\cU} \\
 		& \qquad \leq c \Big( \Vert (\widehat{v}_2^\flat, \widehat{v}_2^\sharp, \widehat{\zeta}_2) - (\widehat{v}_1^\flat, \widehat{v}_1^\sharp, \widehat{\zeta}_1) \Vert_{\cB} + d(g_2, g_1) \Big)  + C c_1'  \eps^\alpha \Vert \cS(\widehat{v}_2^\flat, \widehat{v}_2^\sharp, \widehat{\zeta}_2, g_2) - \cS(\widehat{v}_1^\flat, \widehat{v}_1^\sharp, \widehat{\zeta}_1, g_1)) \Vert_{\cU},
 	\end{align*}
 	and the result follows by rearranging.
\end{proof}

\appendix

\section{Mean curvature of normal graphs} 


\label{app:mean.curvature.graphs}

The purpose of this appendix is to record a form of the second variation of mean curvature that is convenient for our paper, since this computation is not easily found in the literature.

We consider Fermi coordinates $(y, z)$ near a hypersurface $\Sigma \subset M$ satisfying the conditions of Section \ref{subsec:jacobi.toda.setup}, where the normal graph of a function $f : \Sigma \to \RR$ will eventually look like
\[ G[f] \triangleq \{ (y, f(y)) : y \in \Sigma \}. \]
Before discussing the geometry of the graph over $\Sigma$, let's first discuss the geometry of the distance level sets $\{ z = \text{const} \}$ relative $\Sigma$. 

We'll denote the restriction of the metric to the parallel hypersurface $\{ (y, z) : y \in \Sigma \}$ by $g_z$, i.e., $g_z = Z_\Sigma(\cdot, z){}^* g$, and the corresponding upward pointing unit normal, area form, second fundamental form, mean curvature, divergence, gradient, Hessian, and Laplacian by $\partial_z$, $d\mu_{g_z}$, $\sff_z$, $H_z$, $\divg_{g_z}$, $\nabla_{g_z}$, $\nabla^2_{g_z}$, $\Delta_{g_z}$. We recall that the $\partial_z$ (Lie) derivative of $g_z$ is known to be
\begin{equation} \label{eq:mean.curv.ddt.metric}
	\sL_{\partial_z} g_z = 2 \sff_z,
\end{equation}
and also the corresponding derivative of the second fundamental form $\sff_z$ is
\begin{equation} \label{eq:mean.curv.ddt.sff}
	\sL_{\partial_z} \sff_z = \sff_z^2 - \riem_g(\cdot, \partial_z, \partial_z, \cdot),
\end{equation}
where $\sff_z^2$ denotes a single trace of $\sff_z \otimes \sff_z$, and our Riemann curvature convention is such that $(\riem_g)_{ijji}$ (suitably normalized) denotes a sectional curvature. From \eqref{eq:mean.curv.ddt.sff} we recover the well-known Jacobi equation
\begin{equation} \label{eq:mean.curv.ddt.h}
	\partial_z H_z = \partial_z (g_z^{ij} \sff^z_{ij}) = - g_z^{ik} g_z^{j\ell} (\sL_{\partial_z} g_z)_{k\ell} \sff^z_{ij} + g^{ij}_z \sL_{\partial_z} \sff^z_{ij} 
	= - |\sff_z|^2 - \ricc_g(\partial_z, \partial_z).
\end{equation}
From \eqref{eq:mean.curv.ddt.metric} we also find the evolution of the gradient operator:
\begin{equation} \label{eq:mean.curv.ddt.grad}
	\sL_{\partial_z} \nabla_{g_z} f = - 2\sff_z(\nabla_{g_z} f, \cdot), \; f \in C^\infty(\Sigma).
\end{equation}
Next, we seek the evolution of the divergence operator on 1-forms. To find it, we first need to find the evolution of the Christoffel symbols. 
Recall that the Christoffel symbols don't transform like tensors but that their difference does. In particular, $\partial_z \Gamma$ is a vector-valued 2-tensor given by (using the Codazzi equation to get the second form):
\begin{multline} \label{eq:mean.curv.ddt.christoffel}
	(\partial_z \Gamma)(\mathbf{X}, \mathbf{Y}) = [\nabla^{g_z}_{\mathbf{X}} \sff(\cdot, \mathbf{Y}) + \nabla^{g_z}_{\mathbf{Y}} \sff(\mathbf{X}, \cdot) - \nabla^{g_z}_{\cdot} \sff(\mathbf{X}, \mathbf{Y})]^\sharp \\
	= [\nabla^{g_z}_{\cdot} \sff(\mathbf{X}, \mathbf{Y}) + \riem_g(\partial_z, \mathbf{X}, \mathbf{Y}, \cdot) +  \riem_g(\partial_z, \mathbf{Y}, \mathbf{X}, \cdot)]^\sharp,
\end{multline}
where the indices are raised with the $\sharp$ operator using $g_z$. From this we find the evolution of the Hessians of scalar fields:
\begin{equation} \label{eq:mean.curv.ddt.hess}
	\sL_{\partial_z} \nabla^2_{g_z} f = - \nabla^{g_z}_{\nabla_{g_z} f} \sff_z,
\end{equation}
and their Laplacians:
\begin{equation} \label{eq:mean.curv.ddt.laplace}
	\sL_z \Delta_{g_z} f = - 2 \langle \sff_z, \nabla^2_{g_z} f \rangle_{g_z} - \langle \nabla_{g_z} H_z, \nabla_{g_z} f \rangle_{g_z}.
\end{equation}
Likewise, the evolution of the divergence of 1-forms is:
\begin{equation} \label{eq:mean.curv.ddt.divg}
	\sL_z \divg_{g_z} \omega = - 2 \langle \sff_z, \nabla_{g_z} \omega \rangle_{g_z} - \langle \nabla_{g_z} H + \ricc_g(\partial_z, \cdot), \omega \rangle_{g_z}, \qquad 
	\omega \in \Omega^1(\Sigma).
\end{equation}
Next, we seek to calculate the evolution of $\nabla_{g_z} \sff_z$. To do so, we pick coordinates so that the vectors $\partial_{y_i}$ are parallel (with respect to $\nabla_{g_z}$) at the base point where we're computing the derivative. Then
\begin{align*}
	\partial_z (\partial_{y_i} \sff^z_{jk}) = \partial_{y_i} (\partial_z \sff^z_{jk})& = \partial_{y_i} (\sL_{\partial_z} \sff^z_{jk}) \\
		& = \partial_{y_i} ((\sff_z^2)_{jk} - \riem_g(\partial_{y_j}, \nu_z, \nu_z, \partial_{y_k})) \\
		& = \partial_{y_i} (g_z^{\ell m} \sff^z_{j\ell} \sff^z_{km} - \riem_g(\partial_{y_j}, \partial_z, \partial_z, \partial_{y_k})) \\
		& = - g^{\ell m}_z \partial_{y_i} \sff^z_{j\ell} \sff^z_{km} - g^{\ell m}_z \sff^z_{j\ell} \partial_{y_i} \sff^z_{km}   - \partial_{y_i}(\riem_g(\partial_{y_j}, \partial_z, \partial_z, \partial_{y_k})) \\
		& = - g_z^{\ell m} (\nabla^{g_z}_i \sff^z_{j\ell}) \sff^z_{km} - g_z^{\ell m} (\nabla_i^{g_z} \sff^z_{km}) \sff^z_{j\ell}  - \nabla^g_{\partial_{y_i}} \riem_g(\partial_{y_j}, \partial_z, \partial_z, \partial_{y_k}) \\
		& \qquad + 2 \riem_g(\partial_{y_j}, g_z^{\ell m}\sff^z_{i\ell} \partial_{y_m}, \partial_z, \partial_{y_k})
\end{align*}
i.e.,
\begin{multline} \label{eq:mean.curv.ddt.gradsff.i}
	\partial_z(\partial_{y_i} \sff^z_{jk}) = - g_z^{\ell m} (\nabla^{g_z}_i \sff^z_{j\ell}) \sff^z_{km} - g_z^{\ell m} (\nabla_i^{g_z} \sff^z_{km}) \sff^z_{j\ell} \\
	- \nabla^g_{\partial_{y_i}} \riem_g(\partial_{y_j}, \partial_z, \partial_z, \partial_{y_k}) 
	+ 2 g_z^{\ell m} \sff^z_{i\ell} \riem_g(\partial_{y_j}, \partial_{y_m}, \partial_z, \partial_{y_k}).
\end{multline}
Moreover,
\begin{align*}
	\sL_{\partial_z} (\nabla^{g_z}_{\partial_{y_i}} \partial_{y_j})  = \sL_{\partial_z} (\nabla^g_{\partial_{y_i}} \partial_{y_j} + \sff^z_{ij} \partial_z) 
		& = \sL_{\partial_z} \nabla^g_{\partial_{y_i}} \partial_{y_j} + (\sL_{\partial_z} \sff_z)_{ij} \partial_z \\
		& = \nabla^g_{\partial_z} \nabla^g_{\partial_{y_i}} \partial_{y_j} + (\sff^2_z)_{ij} \partial_z - \riem_g(\partial_{y_i}, \partial_z, \partial_z, \partial_{y_j}) \partial_z \\
		& = \nabla^g_{\partial_z} \nabla^g_{\partial_{y_i}} \partial_{y_j} + (\sff^2_z)_{ij} \partial_z - \riem_g(\partial_{z}, \partial_{y_i}, \partial_{y_j}, \partial_{z}) \partial_z \\
		& = \nabla^g_{\partial_z} \nabla^g_{\partial_{y_i}} \partial_{y_j} + (\sff^2_z)_{ij} \partial_z  -(\nabla^g_{\partial_z} \nabla^g_{\partial_{y_i}} \partial_{y_j} - \nabla^g_{\partial_{y_i}} \nabla^g_{\partial_z} \partial_{y_j}) \\
		& = \nabla^g_{\partial_{y_i}} \nabla^g_{\partial_z} \partial_{y_j} + (\sff^2_z)_{ij} \partial_z 
		 = \nabla^g_{\partial_{y_i}} \nabla^g_{\partial_{y_j}} \partial_z + (\sff^2_z)_{ij} \partial_z.
\end{align*}
Recall that $\sL_{\partial_z} (\nabla^{g_z}_{\partial_{y_i}} \partial_{y_j})$ is tangential to $\{ z = \text{const}\}$, and so is $\nabla^g_{\partial_{y_j}} \partial_z = g_z^{k\ell} \sff^z_{jk} \partial_{y_\ell}$. By projecting onto $\{ z = \text{const} \}$, the expression above reduces to
\begin{equation} \label{eq:mean.curv.ddt.gradsff.ii}
	\sL_{\partial_z} (\nabla^{g_z}_{\partial_{y_i}} \partial_{y_j}) = \nabla^{g_z}_{\partial_{y_i}} (g^{k\ell}_z \sff^z_{jk} \partial_{y_\ell}) = g^{k\ell}_z (\nabla^{g_z}_i \sff^z_{jk}) \partial_{y_\ell}
\end{equation}
Combining with \eqref{eq:mean.curv.ddt.gradsff.i}, we deduce that
\begin{align*}
	\partial_z (\nabla^{g_z}_i \sff^z_{jk}) & = \partial_z(\partial_{y_i} \sff^z_{jk}) - \sff_z(\nabla^{g_z}_{\partial_{y_i}} \partial_{y_j}, \partial_{y_k}) - \sff_z(\partial_{y_j}, \nabla^{g_z}_{\partial_{y_i}} \partial_{y_k})) \\
		& = \partial_z (\partial_{y_i} \sff^z_{jk}) - \sff_z(\sL_{\partial_z} (\nabla^{g_z}_{\partial_{y_i}} \partial_{y_j}), \partial_{y_k})  - \sff_z(\partial_{y_j}, \sL_{\partial_z}(\nabla^{g_z}_{\partial_{y_i}} \partial_{y_k})) \\
		& = - g_z^{\ell m} (\nabla^{g_z}_i \sff^z_{j\ell}) \sff^z_{km} - g_z^{\ell m} (\nabla_i^{g_z} \sff^z_{km}) \sff^z_{j\ell}   - \nabla^g_{\partial_{y_i}} \riem_g(\partial_{y_j}, \partial_z, \partial_z, \partial_{y_k}) \\
		& \qquad  + 2 g_z^{\ell m} \sff^z_{i\ell} \riem_g(\partial_{y_j}, \partial_{y_m}, \partial_z, \partial_{y_k})   - \sff_z(g^{m\ell}_z (\nabla^{g_z}_i \sff^z_{jm}) \partial_{y_\ell}, \partial_{y_k}) \\
		& \qquad - \sff_z(\partial_{y_j}, g^{m\ell}_z (\nabla_i^{g_z} \sff^z_{km}) \partial_{y_\ell}) \\
		& = - g_z^{\ell m} (\nabla^{g_z}_i \sff^z_{j\ell}) \sff^z_{km} - g_z^{\ell m} (\nabla_i^{g_z} \sff^z_{km}) \sff^z_{j\ell}  - \nabla^g_{\partial_{y_i}} \riem_g(\partial_{y_j}, \partial_z, \partial_z, \partial_{y_k}) \\
		& \qquad  + 2 g_z^{\ell m} \sff^z_{i\ell} \riem_g(\partial_{y_j}, \partial_{y_m}, \partial_z, \partial_{y_k})  - g_z^{m\ell} \sff^z_{\ell k} \nabla^{g_z}_i \sff^z_{jm} - g_z^{m\ell} \sff^z_{j\ell} \nabla^{g_z}_i \sff^z_{km} \\
		& = - 2 g_z^{\ell m} (\nabla^{g_z}_i \sff^z_{j\ell}) \sff^z_{km} - 2 g_z^{\ell m} (\nabla_i^{g_z} \sff^z_{km}) \sff^z_{j\ell} \\
		& \qquad  - \nabla^g_{\partial_{y_i}} \riem_g(\partial_{y_j}, \partial_z, \partial_z, \partial_{y_k})  + 2 g_z^{\ell m} \sff^z_{i\ell} \riem_g(\partial_{y_j}, \partial_{y_m}, \partial_z, \partial_{y_k}).
\end{align*}
In particular,
\begin{equation} \label{eq:mean.curv.ddt.gradsff}
	\sL_{\partial_z} \nabla_{g_z} \sff_z = \nabla_{g_z} \sff_z \ast \sff_z + \nabla_g \riem_g + \sff_z \ast \riem_g
\end{equation}
as a symmetric 2-tensor on the $\{ z = \text{const} \}$ level sets.

We now proceed to use these evolution equations to understand the second variation of the mean curvature of a graph in Fermi coordinates. These computations are motivated by the ones in \cite{PacardSun}.

Continuing to work in Fermi coordinates $(y,z)$ relative to $\Sigma$, we write
\[ G[f] \triangleq \{ (y,f(y)) : y \in \Sigma \}. \]
Note that the induced metric on $G[f]$ is:
\[ g|_{G[f]} = g_{f(y)} + df(y)^2, \; y \in \Sigma. \]
The induced area form on $G[f]$ is, therefore:
\[ d\mu_{G[f]}(y) = (1 + g_{f(y)}^{ij} f_i(y) f_j(y))^{1/2} \, d\mu_{g_{f(y)}}(y), \; y \in \Sigma. \]
Thus,
\begin{equation} \label{eq:mean.curv.area}
	\area(G[f]) = \int_\Sigma (1 + g_f^{ij} f_i f_j)^{1/2} \, d\mu_{g_{f(\cdot)}}.
\end{equation}
We now consider the variation $f + t\varphi$, $\varphi \in C^2_c(\Sigma \setminus \partial \Sigma)$. We have (we're using integration by parts in the second step):
\begin{align*}
	\left[ \frac{d}{dt} \area(G[f+t\varphi]) \right]_{t=0} 
	&  = \int_\Sigma \frac{g_f^{ij} f_i \varphi_j}{(1 + g_f^{ij} f_i f_j)^{1/2}} \, d\mu_{g_f} - \int_\Sigma \frac{\sff_f^{ij} f_i f_j \varphi}{(1+g^{ij}_f f_j f_j)^{1/2}} \, d\mu_{g_f} \\
	& \qquad \qquad + \int_\Sigma (1 + g_f^{ij} f_i f_j)^{1/2} H_f \varphi \, d\mu_{g_f} \\
	& = - \int_\Sigma \divg_{g_f} \left( \frac{\nabla_{g_f} f}{(1+g_f^{ij} f_i f_j)^{1/2}} \right) \varphi \, d\mu_{g_f} - \int_\Sigma \frac{\sff_f^{ij} f_i f_j \varphi}{(1+g^{ij}_f f_j f_j)^{1/2}} \, d\mu_{g_f} \\
	& \qquad \qquad + \int_\Sigma (1 + g_f^{ij} f_i f_j)^{1/2} H_f \varphi \, d\mu_{g_f}.
\end{align*}
Note that, if $\nu$ denotes the normal to $G[f]$, then
\[ g(\nu, \partial_z) = (1 + g_f^{ij} f_i f_j)^{-1/2} \implies d\mu_{g_f} = g(\nu, \partial_z) d\mu_{G[f]}, \]
and, therefore,
\begin{align*}
	& \left[ \frac{d}{dt} \area(G[f+t\varphi]) \right]_{t=0} \\
	& = \int_\Sigma \Big[ - \divg_{g_f} \left( \frac{\nabla_{g_f} f}{(1+g_f^{ij}f_if_j)^{1/2}} \right) - \frac{\sff_f^{ij} f_i f_j}{(1+g^{ij}_f f_i f_j)^{1/2}}  + (1 + g_f^{ij} f_i f_j)^{1/2} H_f \Big] g(\nu, \partial_z) \varphi \, d\mu_{G[f]}.
\end{align*}
On the other hand, by definition,
\[ \left[ \frac{d}{dt} \area(G[f+t\varphi]) \right]_{t=0} = \int_\Sigma H_{G[f]} g(\nu, \partial_z) \varphi \, d\mu_{G[f]}, \]
so we conclude
\begin{equation} \label{eq:mean.curv.graphical}
	H[f] = - \divg_{g_f} \left( \frac{\nabla_{g_f} f}{(1+g_f^{ij} f_i f_j)^{1/2}} \right) - \frac{\sff_f^{ij} f_i f_j}{(1+g_f^{ij} f_i f_i)^{1/2}}
	+ (1+g_f^{ij} f_i f_j)^{1/2} H_f.
\end{equation}
We now claim that the quantity
\begin{equation} \label{eq:mean.curv.graphical.quad.error.new}
	\widetilde{\cQ} f \triangleq H[f] - H_0 + \tfrac{\sqrt{g_0}}{\sqrt{g_f}} \divg_{g_0} ( \tfrac{\sqrt{g_f}}{\sqrt{g_0}} (1 + g^{ij}_f f_i f_j)^{-1/2} \nabla_{g_f} f) 
	+ (|\sff_0|^2 + \ricc_g(\partial_z, \partial_z)) f
\end{equation}
is a quadratic error term in the Taylor expansion of $H[f]$ with respect to $\{ z = 0 \}$:
\begin{lemm} \label{lemm:mean.curv.graphical.quad.error.new}
	We have the pointwise estimate
	\[ |\widetilde{\cQ} f| \leq c (|f|^2 + |\partial f|^2) \]
	where $c = c(n, \Lambda) > 0$ and $\Lambda = \Lambda(f,y) > 0$ is such that
	\[ \sup_{|z| \leq |f(y)|} |\sff_z(y)| \leq \Lambda. \]
\end{lemm}
\begin{proof}
	First, note that
	\[ \divg_{g_f} \left( \frac{\nabla_{g_f} f}{(1+g^{ij}_f f_i f_j)^{1/2}} \right) = \frac{\sqrt{g_0}}{\sqrt{g_f}} \divg_{g_0} \left( \frac{\sqrt{g_f}}{\sqrt{g_0}} \frac{\nabla_{g_f} f}{(1+g^{ij}_f f_i f_j)^{1/2}} \right), \]
	which means that
	\begin{equation*}
		\widetilde{\cQ} f = - \frac{\sff_f^{ij} f_i f_j}{(1+g^{ij}_f f_i f_j)^{1/2}} + (1 + g^{ij}_f f_i f_j)^{1/2} H_f 
		- H_0 + (|\sff_0|^2 + \ricc_g(\partial_z, \partial_z)) f.
	\end{equation*}
	The result follows by adding and subtracting $H_f$,
	\[ \frac{|\sff_f^{ij} f_i f_j|}{(1+g^{ij}_f f_i f_j)^{1/2}} + |(1+g^{ij}_f f_i f_j)^{1/2} H_f - H_f| \leq c |\partial f|^2, \]
	and
	\[ |H_f - H_0 + (|\sff_0|^2 + \ricc_g(\partial_z, \partial_z)) f| \leq c |f|^2. \]
	We've used \eqref{eq:mean.curv.ddt.h} in the last estimate.
\end{proof}

\section{Some results of Wang--Wei}


\label{app:WW-results}

For completeness, we recall several results proven by Wang--Wei in \cite{WangWei}. We will assume \eqref{eq:sheets.eps.bound}-\eqref{eq:sheets.constants} and will use the notation  $\overline{\mathbb{H}} \triangleq \overline{\mathbb{H}}^{3|\log \eps|}$ and $\overline{\mathbb{H}}_{\eps,\ell}$ from Section \ref{subsec:jacobi.toda.setup} throughout this appendix. We emphasize (see Remark \ref{rema:scale-vs-WW}) that we are working at the original scale, rather than the $\eps$-scale as in \cite{WangWei}, so these expressions have changed relative to \cite{WangWei} by appropriate factors of $\eps$.

\begin{lemm}[{\cite[Lemma 8.3]{WangWei}}]\label{lemm:WW.tilting.comparison.d}
For $m\not = \ell \in \{1,\dots,Q\}$, consider $X \in Z_{\ell}(\Gamma_{\ell}(\frac 32)\times [-1,1])$ with $|d_{m}(X)|,|d_{\ell}(X)| \leq K \eps |\log \eps|$. Then,
\begin{align*}
d_{\Gamma_{m}}(\Pi_{m}\circ\Pi_{\ell}(X),\Pi_{m}(X)) & \leq C(K) \eps^{\frac 32} |\log \eps|^{\frac 32},\\
|d_{m}(\Pi_{\ell}(X)) + d_{\ell}(\Pi_{m}(X))| & \leq C(K)\eps^{\frac 32} |\log \eps|^{\frac 32}\\
|d_{m}(X) - d_{\ell}(X)+d_{m}(\Pi_{\ell}(X))| & \leq C(K) \eps^{\frac 32} |\log\eps|^{\frac 32}\\
|d_{\ell}(X) - d_{m}(X) - d_{\ell}(\Pi_{m}(X))| & \leq C(K) \eps^{\frac 32}|\log\eps|^{\frac32}\\
1 - \nabla d_{\ell}(X) \cdot \nabla d_{m}(X) & \leq C(K) \eps^{\frac 12} |\log \eps|^{\frac 32}. 
\end{align*}
\end{lemm}

Recall the definition of $\phi$ in \eqref{eq:discrepancy.function}. Wang--Wei compute \cite[(9.4)]{WangWei} in Fermi coordinates with respect to $\Gamma_\ell$ that
\begin{align}\label{eq:phi-eqn-from-WW}
& \eps^{2}( \Delta_{\Gamma_{\ell,z}} \phi + H_{\Gamma_{\ell,z}}\partial_{z} \phi + \partial^{2}_{z}\phi) \\
& = W''(U[\mathbf{h}]) \phi +\cR(\phi) + \left(W'(U[\mathbf{h}]) - \sum_{m=1}^{Q} W'(\overline{\mathbb{H}}_{\eps,m}) \right)  \nonumber \\
& \qquad + \eps^2 (\Delta_{\Gamma_{\ell,z}} h_{\ell} - H_{\Gamma_{\ell,z}}) \partial_z \overline{\mathbb{H}}_{\eps,\ell} - \eps^2 |\nabla_{\Gamma_{\ell,z}} h_{\ell}|^{2} \partial_z^2 \overline{\mathbb{H}}_{\eps,\ell}   \nonumber\\
& \qquad + \sum_{m\not = \ell} \big( \eps \cR_{m,1} ((Z_{\Gamma_m})_* \partial_z)  \overline{\mathbb{H}}_{\eps,m} - \eps^2 \cR_{m,2} ((Z_{\Gamma_m})_* \partial_z)^2 \overline{\mathbb{H}}_{\eps,m} \big) - \sum_{m=1}^{Q} \overline{\xi}((-1)^{m-1}\eps^{-1}(d_{m} - h_{m} \circ \Pi_{m})).\nonumber
\end{align}
Note the slight difference in signs relative to \cite[(9.4)]{WangWei}, which arise from our different sign convention on the mean curvature and our choice to avoid introducing extraneous notation for $(Z_{\Gamma_m})_* \partial_z$ derivatives of $\overline{\mathbb{H}}_{\eps,m}$ that introduce factors of $(-1)^m$ for $m = \ell$ or $m \neq \ell$ (cf. $g'_\alpha$, $g''_\alpha$ in \cite[Section 9]{WangWei}). Above, we have written
\begin{align*}
\cR(\phi) & \triangleq W'(U[\mathbf{h}]+\phi)-W'(U[\mathbf{h}])-W''(U[\mathbf{h}]) \phi = O(\phi^{2})\\
((Z_{\Gamma_m})^* \cR_{m,1})(y',z') & \triangleq \eps (\Delta_{\Gamma_{m,z'}} h_{m}(y') - H_{\Gamma_{m,z'}}(y')) \\
((Z_{\Gamma_m})^* \cR_{m,2})(y',z') & \triangleq |\nabla_{\Gamma_{m,z'}}h_{m}(y')|^{2}
\end{align*}
as well as (cf.\ \eqref{eq:approximate.heteroclinic.behavior})
\begin{equation}\label{eq:app.xi.defn.error.barH}
\overline{\xi}(t) \triangleq \overline{\mathbb{H}}''(t)-W'(\overline{\mathbb{H}}(t)).
\end{equation}

It is useful to remember that the terms involving $\cR_{m,1}$, $\cR_{m,2}$ in \eqref{eq:phi-eqn-from-WW} vanish when $d_{m} > 6 \eps |\log \eps|$. 

\begin{lemm}[{cf.\ \cite[Lemma A.1]{WangWei}}]\label{lemm:app.sheets.interaction.1d}
	For $\kappa>0$, we have
	\[ 
	\int_{-\infty}^{\infty} (W''(\mathbb{H}(t)) - 2) \mathbb{H}'(t-T) \mathbb{H}'(t) \, dt = -4\sqrt{2}(\expansioncoeff)^{2}e^{-\sqrt{2}T} + O(e^{-2(1-\kappa)\sqrt{2}T})
	\]
	as $T\to\infty$. 
\end{lemm}
\begin{proof}
	 Let's denote the left hand side as $I(T)$. Recall that $W(\pm1) = 2$. We can rewrite:
	\begin{align*}
		I(T) 
			& = \sqrt{2} \expansioncoeff e^{-\sqrt{2}T} \int_{-\infty}^{\infty} (W''(\mathbb{H}(t)) - 2) e^{\sqrt{2}t} \mathbb{H}'(t) \, dt \\
			& \qquad + \int_{-\infty}^{\infty} (W''(\mathbb{H}(t)) - 2) (\mathbb{H}'(t-T) - \sqrt{2}\expansioncoeff e^{\sqrt{2}(t-T)}) \mathbb{H}'(t) \, dt\\
			& = \sqrt{2}\expansioncoeff e^{-\sqrt{2}T} \int_{-\infty}^{\infty} (W''(\mathbb{H}(t)) - 2) e^{\sqrt{2}t} \mathbb{H}'(t) \, dt + O(1) \int_{-\infty}^\infty \Big| \mathbb{H}'(t-T) - \sqrt{2} \expansioncoeff e^{\sqrt{2}(t-T)} \Big| \mathbb{H}'(t)^2 \, dt,
	\end{align*}
	where we've used \eqref{eq:heteroclinic.expansion.i}-\eqref{eq:heteroclinic.expansion.ii} in the last step.
	
	We can directly evaluate the first integral by writing $W''(\mathbb{H}(t))\mathbb{H}'(t)=\mathbb{H}'''(t)$ and integrating by parts:
	\begin{align*}
		& \int_{-L}^{L} W''(\mathbb{H}(t)) e^{\sqrt{2}t} \mathbb{H}'(t) dt = \int_{-L}^{L}  \mathbb{H}'''(t) e^{\sqrt{2}t} dt\\
		& = \mathbb{H}''(L) e^{\sqrt{2}L} - \mathbb{H}''(-L) e^{-\sqrt{2}L} - \sqrt{2} \int_{-L}^{L}  \mathbb{H}''(t) e^{\sqrt{2}t} dt\\
		& = \mathbb{H}''(L) e^{\sqrt{2}L} - \mathbb{H}''(-L) e^{-\sqrt{2}L} - \sqrt{2} \mathbb{H}'(L)e^{\sqrt{2}L} + \sqrt{2}\mathbb{H}'(-L) e^{-\sqrt{2}L}+ 2 \int_{-L}^{L}  \mathbb{H}'(t) e^{\sqrt{2}t} dt.
	\end{align*}
	Recalling  \eqref{eq:heteroclinic.expansion.ii}-\eqref{eq:heteroclinic.expansion.iii}, sending $L \to \infty$ gives
	\[
	\int_{-\infty}^{\infty} \left(W''(\mathbb{H}(t)) - 2  \right) e^{\sqrt{2}t} \mathbb{H}'(t) \, dt = -4 \expansioncoeff.
	\]
	Plugging this into the expression for $I(T)$, we have:
	\begin{equation} \label{eq:ww.lemm.a.1.IofT}
		I(T) = -4 \sqrt{2} (\expansioncoeff)^2 e^{-\sqrt{2} T} + O(1) \int_{-\infty}^{\infty} (\mathbb{H}'(t-T) - \sqrt{2}\expansioncoeff e^{\sqrt{2}(t-T)}) \mathbb{H}'(t)^2 \, dt.
	\end{equation}
	It remains to show that the last integral is $O(e^{-2(1-\kappa)\sqrt{2}T})$. Let $\alpha \in (0, 1)$ be fixed. When $t \in (-\infty, \alpha T)$,
	\[ \mathbb{H}'(t-T) - \sqrt{2} \expansioncoeff e^{\sqrt{2}(t-T)} = O(e^{2\sqrt{2}(t-T)}), \]
	by \eqref{eq:heteroclinic.expansion.ii}, so
	\begin{equation} \label{eq:ww.lemm.a.1.IofT.i}
		\int_{-\infty}^{\alpha T} (\mathbb{H}'(t-T) - \sqrt{2} \expansioncoeff e^{\sqrt{2}(t-T)}) \mathbb{H}'(t)^2 \, dt = O(T e^{-2\sqrt{2} T}).
	\end{equation}
	To bound the integral over $[\alpha T, \infty)$, it suffices to observe the following bound on its dominant term:
	\begin{multline} \label{eq:ww.lemm.a.1.IofT.ii}
		\int_{\alpha T}^\infty e^{\sqrt{2}(t-T)} \mathbb{H}'(t)^2 \, dt = O(1) \int_{\alpha T}^\infty e^{-\sqrt{2}(t+T)} \, dt = O(e^{-\sqrt{2}(1+\alpha)T}) \\
		\implies \int_{\alpha T}^\infty (\mathbb{H}'(t-T) - \sqrt{2} \expansioncoeff e^{\sqrt{2}(t-T)}) \mathbb{H}'(t)^2 \, dt = O(e^{-\sqrt{2}(1+\alpha)T}).
	\end{multline}
	The result follows by plugging \eqref{eq:ww.lemm.a.1.IofT.i}, \eqref{eq:ww.lemm.a.1.IofT.ii} into \eqref{eq:ww.lemm.a.1.IofT}.
\end{proof}

\section{Proof of Lemma \ref{lemm:h.phi.comparison.improved}}


\label{app:proof.lem.comp.improved}

We follow the proof of \cite[Lemma 9.6]{WangWei}, using Lemma \ref{lemm:stationary.estimates} to gain improved estimates on the error terms. We continue to use the notation of Appendix \ref{app:WW-results}.

Fix $\ell\in\{1,\dots, Q\}$, $y \in \Gamma_\ell(\tfrac{8}{10})$. In what follows we work in Fermi coordinates with respect to $\Gamma_\ell$. Because $u(y) = 0$, we have
\begin{equation} \label{eq:lemm.2.8.phi}
	\phi(y,0) = -\overline{\mathbb{H}}((-1)^{\ell-1}\eps^{-1} h_{\ell}(y)) - \sum_{m\not = \ell} \left(  \overline{\mathbb{H}}_{\eps,m}(y, 0) \pm (-1)^{m-1} \right),
\end{equation}
where the ``$\pm$'' is a ``$-$'' for $m<\ell$ and ``$+$'' for $m>\ell$. This implies the first inequality immediately, using the fact that $|\mathbb{H}(\eps^{-1}h_{\ell}(y))| \simeq \eps^{-1}|h_{\ell}(y)|$ by a Taylor expansion.

Differentiating \eqref{eq:lemm.2.8.phi} once with respect to $y$, we find that (recalling \eqref{eq:approximate.critical.point.model})
\begin{align*}
\eps \nabla_{\Gamma_{\ell}} (\phi|_{\Gamma_\ell})(y, 0) 
	& = - (-1)^{\ell-1} \overline{\mathbb{H}}'((-1)^{\ell-1}\eps^{-1}h_{\ell}(y)) \nabla_{\Gamma_{\ell}}h_{\ell}(y) \\
	& \qquad - \eps \sum_{m\not = \ell} \partial_z ((Z_{\Gamma_m})^* \overline{\mathbb{H}}_{\eps,m}) (y,0) (\nabla_{\Gamma_{\ell}} d_{m}(y,0) - \nabla_{\Gamma_{\ell}} (h_{m} \circ \Pi_{m})(y,0)).
\end{align*}
Define
\[
\cI_{\ell} \triangleq \{ m \in \{1,\dots, Q\} \; : \; m\not =  \ell, \; d_{m}(y,0) \leq K \eps |\log \eps|\}
\]
for $K > 6$ fixed. Then, the exponential decay of $\overline{\mathbb{H}}'$ (and definition of $\overline{\mathbb{H}}$) gives
\begin{align*}
|\nabla_{\Gamma_{\ell}} h_{\ell}(y)| & \leq c \left( \eps |\nabla_{\Gamma_{\ell}} (\phi|_{\Gamma_{\ell}})(y)| + \sup_{m \in \cI_{\ell}(y)} \left( |\nabla_{\Gamma_{\ell}} d_{m}(y,0)| + |\nabla_{\Gamma_{\ell}}(h_{m} \circ \Pi_{m})(y,0)|  \right) \exp(-\sqrt{2} \eps^{-1} D_{\ell}(y))  \right)\\
& \leq c \left( \eps |\nabla_{\Gamma_{\ell}} (\phi|_{\Gamma_{\ell}})(y)| + \eps^{\kappa} \exp(-\sqrt{2} \eps^{-1} D_{\ell}(y))  \right)
\end{align*}
We have used Lemma \ref{lemm:WW.tilting.comparison.d} to bound the first term in the supremum and the bounds from Lemmas \ref{lemm:h.phi.comparison} and \ref{lemm:stationary.estimates} to bound the second term (note that in the proof of Lemma \ref{lemm:h.phi.comparison}, the second term was simply bounded by $o(1)$ since at that point Lemma \ref{lemm:stationary.estimates} was not available). 

Differentiating \eqref{eq:lemm.2.8.phi} again, we find
\begin{align*}
& \eps^{2} \nabla_{\Gamma_{\ell}}^{2} (\phi|_{\Gamma_\ell})(y,0)\\
& = - \eps (-1)^{\ell-1}  \overline{\mathbb{H}}'((-1)^{\ell-1}\eps^{-1}h_{\ell}(y)) \nabla_{\Gamma_{\ell}}^{2}h_{\ell}(y) - \overline{\mathbb{H}}''((-1)^{\ell-1}\eps^{-1}h_{\ell}(y)) \nabla_{\Gamma_{\ell}}h_{\ell}(y) \otimes \nabla_{\Gamma_{\ell}}h_{\ell}(y) \\
& - {\eps^2} \sum_{m\not = \ell}  \partial_z^2 ((Z_{\Gamma_m})^* \overline{\mathbb{H}}_{\eps,m}) (y,0) \\
& \qquad \cdot (\nabla_{\Gamma_{\ell}} d_{m}(y,0) - \nabla_{\Gamma_{\ell}} (h_{m} \circ \Pi_{m})(y,0))) \otimes (\nabla_{\Gamma_{\ell}} d_{m}(y,0) - \nabla_{\Gamma_{\ell}} (h_{m} \circ \Pi_{m})(y,0))) \\
& - {\eps^2} \sum_{m\not = \ell}  \partial_z ((Z_{\Gamma_m})^*  \overline{\mathbb{H}}_{\eps,m}) (y,0) (\nabla_{\Gamma_{\ell}}^{2} d_{m}(y,0) - \nabla_{\Gamma_{\ell}}^{2} (h_{m} \circ \Pi_{m})(y,0))).
\end{align*}
Because $\Gamma_{\ell}$, $\Gamma_m$ have bounded second fundamental form by \eqref{eq:sheets.enhanced.sff.bound}, \eqref{eq:mean.curv.ddt.sff} shows that 
\[
|\nabla_{\Gamma_{\ell}}^{2} d_{m}(y,0)| \leq c, \; m \in \cI_\ell.
\]
Thus, we find that, as claimed:
\begin{align*}
\eps |\nabla^{2}_{\Gamma_{\ell}}h(y)| & \leq c \left( \eps^{2}|\nabla^{2}_{\Gamma_{\ell}}(\phi|_{\Gamma_{\ell}})(y)| + \eps^{2}|\nabla_{\Gamma_{\ell}}(\phi|_{\Gamma_{\ell}})(y)|^{2} + \eps^{\kappa} \exp(-\sqrt{2}\eps^{-1}D_{\ell}(y)) \right).
\end{align*}

The H\"older estimate follows similarly, with one important change: we do not know (at this point) that $[\sff_{\Gamma_{\ell}}]_{\theta}$ is uniformly bounded, and thus cannot conclude that $[\nabla_{\Gamma_{\ell}}^{2} d_{m}(y,0)]_{\theta} \leq c$. Instead we use \eqref{eq:sheets.enhanced.sff.bound} and \eqref{eq:sheets.enhanced.sff.grad.bound} in conjunction with \eqref{eq:mean.curv.ddt.sff} and \eqref{eq:mean.curv.ddt.gradsff} to conclude that 
\[
\eps^{\theta}[\nabla^{2} d_{m}(y,0)]_{\theta} \leq c, \; m \in \cI_\ell.
\]
This, combined with the factor of $\eps$ in front of the last line suffices to complete the H\"older estimate.

\section{Proof of \eqref{eq:toda.stability.estimate.sharper}}


\label{app:proof.stab.inproved}

We follow \cite[Section 19]{WangWei}, except we keep track of how the error terms improve upon strengthened sheet separation estimates, as well as keeping track of the constant in front of the main term on the right hand side of the stability inequality. We assume that $\ell \in \{2,\dots, Q-1\}$, i.e., there are sheets above and below $\Gamma_{\ell}$ (when $\ell = 1$ or $Q$, the argument is similar). Similarly, we can assume that
\begin{equation} \label{eq:app.improved.stability.parity.assumption}
	(-1)^{\ell-1} = 1.
\end{equation}

Here, and throughout this appendix, we will write  $\cE_\zeta$ for any term that is bounded as follows:
\begin{equation} \label{eq:app.improved.stability.Eeps}
		|\cE_\zeta| \leq c' \varepsilon^2 +  c' \sum_{m=1}^Q \sup\left\{\exp(-\sqrt{2} ( 1+\kappa) \eps^{-1}D_m(y')) : y' \in \Gamma_m \cap \Pi_\ell^{-1}(B^{n-1}_{2K\eps |\log \eps|}(\support \zeta)) \right\},
\end{equation}
for some $\kappa > 0$ fixed throughout sufficiently small. We emphasize that the constant $c'$ is uniform in $\eps$ sufficiently small. Here, $\zeta$ is just the test function from the statement of \eqref{eq:toda.stability.estimate.sharper}.

We emphasize that Lemma \ref{lemm:stationary.estimates} holds, so by \eqref{eq:phi.c2a.estimate.full} and Lemma \ref{lemm:h.phi.comparison}, we have
\begin{multline}\label{eq:phi.est.stab.proof}
	\sum_{m=1}^Q \Vert \phi\Vert_{C^{2,\theta}_{\eps}(\mathcal{M}_{m}(r))} + \eps \Vert \Delta_{\Gamma_m} h_{m} - H_{\Gamma_{m}} \Vert_{C^{0,\theta}_\eps(\Gamma_m(r))}  + \eps^{-1} \Vert h_{m}\Vert_{C^{2,\theta}_{\eps}(\Gamma_{m}(r))} \\
	\leq c' \varepsilon^2 + c' \sum_{m=1}^Q A_m(r+K\varepsilon |\log \varepsilon|)  \leq c' \eps,
\end{multline}
and the improved estimate on the tangential derivatives of $\phi$ from \eqref{eq:phi.improved.c2a.estimate.full}, which we will write as
\begin{multline}\label{eq:phi.improved.c2a.estimate.full.app.ver}
\varepsilon \Vert (Z_{\Gamma_\ell})_* \partial_{y_i} \phi \Vert_{C^{1,\theta}_\eps(\cM_\ell(r))} \\
	\leq c' \varepsilon^2 + c' \sum_{m=1}^Q A_m(r + 2K \varepsilon |\log \varepsilon|)^{1+\kappa}  + c' \varepsilon^\kappa \sum_{m=1}^Q A_m(r+2K \varepsilon |\log \varepsilon|) \leq c'\eps^{1+\kappa}.
\end{multline}
In fact, we will often use the localized version of \eqref{eq:phi.est.stab.proof}, \eqref{eq:phi.improved.c2a.estimate.full.app.ver} on $\cM_\ell(1) \cap \Pi_\ell^{-1}(\support \zeta)$:
\begin{equation} \label{eq:app.improved.stability.phi.bounds}
	\Vert \phi \Vert_{C^{2,\theta}_\epsilon} + \eps \Vert \Delta_{\Gamma_\ell} h_\ell - H_{\Gamma_\ell} \Vert_{C^{0,\theta}_\eps} + \eps^{-1} \Vert h_\ell \Vert_{C^{2,\theta}_\eps} \leq c' \eps^2 + c' \sum_{m=1}^{Q} \sup \left\{ \exp(-\sqrt{2} \eps^{-1} D_m(\cdot)) \right\},
\end{equation}
\begin{equation} \label{eq:app.improved.stability.phi.bounds.2}
	\eps \Vert (Z_{\Gamma_\ell})_* \partial_{y_i} \phi \Vert_{C^{1,\theta}_\eps} \leq O(|\cE_{\zeta}|);
\end{equation}
the H\"older norms are taken over $\cM_\ell(1) \cap \Pi_\ell^{-1}(\support \zeta)$, $\Gamma_\ell(1) \cap \Pi_\ell^{-1}(\support \zeta)$, $\Gamma_\ell(1) \cap \Pi_\ell^{-1}(\support \zeta)$ for \eqref{eq:app.improved.stability.phi.bounds} and over $\cM_\ell(1) \cap \Pi_\ell^{-1}(\support \zeta)$ for \eqref{eq:app.improved.stability.phi.bounds.2}, and the $\sup$ is over $\cM_m(1) \cap \Pi_\ell^{-1}(B_{2K\eps |\log \eps|}^{n-1}(\support \zeta))$. Note how \eqref{eq:app.improved.stability.phi.bounds}, \eqref{eq:app.improved.stability.phi.bounds.2} imply \eqref{eq:phi.est.stab.proof}, \eqref{eq:phi.improved.c2a.estimate.full.app.ver} by Lemma \ref{lemm:stationary.estimates}.

We will write $\overline{\mathbb{H}}$ for $\overline{\mathbb{H}}^{3|\log\eps|}$ throughout this appendix, where $\overline{\mathbb{H}}^{3|\log \eps|}$ is as in \eqref{eq:HLambda-cutoff-def} with $\Lambda = 3 |\log \eps|$. We also recall the definition of $\overline{\xi}$ from \eqref{eq:app.xi.defn.error.barH}. We then define the following functions by their expression in $\Gamma_{m}$ Fermi coordinates:
\begin{align*}
	((Z_{\Gamma_m})^* \overline{\mathbb{H}}_{\eps,m})(y,z) & \triangleq \overline{\mathbb{H}}((-1)^{m-1}\eps^{-1}(z-h_{m}(y))), \\
	((Z_{\Gamma_{m}})^* \overline{\xi}_{\eps,m})(y,z) & \triangleq \overline{\xi}((-1)^{m-1}\eps^{-1}(z-h_{m}(y))).
\end{align*}
Recall that $Z_{\Gamma_{m}}(y,z)$ is the point $(y,z)$ in Fermi coordinates over $\Gamma_{m}$ (see the definition after \eqref{eq:sheets.constants}), and that $g = dz^2 + g_z$ in Fermi coordinates. Recall also \eqref{eq:app.xi.defn.error.barH}, \eqref{eq:approximate.heteroclinic.behavior}.

Choose functions $\rho^{\pm}_{\ell}(y) = \frac 12 f_{\ell,\ell\pm1}(y)$, where we recall that $\Gamma_{\ell\pm1}$ is the normal graph of $f_{\ell,\ell\pm1}$ over $\Gamma_{\ell}$. Note that $\rho^{\pm}_{\ell}$ is thus uniformly bounded in $C^1(\Gamma_{\ell}(\tfrac{8}{10}))$ by \eqref{eq:sheets.enhanced.sff.bound}-\eqref{eq:sheets.enhanced.sff.grad.bound}. We consider a vertical cutoff function $\chi(y,z)$ defined by
\[ \chi(y,z) \triangleq \tilde \chi \left( \eps^{-1}L^{-1}(z-\rho_{\ell}^{+}(y)) \right)\tilde \chi \left( \eps^{-1}L^{-1}(\rho_{\ell}^{-}(y)-z) \right) \]
where $\tilde\chi$ is a smooth function with $\tilde\chi(t) = 1$ for  $t \in (-\infty,-1)$ and $\supp\tilde\chi \subset (-\infty,0)$. We will fix $L>0$ sufficiently large (independent of $\eps>0$ small) below. Note that 
\begin{equation} \label{eq:app.stab.chi.bds}
	\eps L |\nabla \chi| \leq c
\end{equation}
and for fixed $y$,  
\begin{equation}\label{eq:chi-z-der-length}
\supp |\partial_{z}\chi(y,\cdot)| \subset 
[\rho^{-}_{\ell}(y),\rho^{-}_{\ell}(y) + \eps L] \cup [\rho^{+}_{\ell}(y)-\eps L,\rho^{+}_{\ell}(y)].
\end{equation}

We will frequently use the observation that on $\supp |\nabla \chi| \cap \{\pm z > 0\}$, 
\begin{equation} \label{eq:stab.calc.spt.chi.H.der.est}
\eps^k |\partial_z^k \overline{\mathbb{H}}_{\eps,\ell}(y,z)| \leq c' \exp(-\sqrt{2}\eps^{-1} \rho^{\pm}_{\ell}(y)) \leq c' \eps^2 + c' \sup \left\{ \exp(-\tfrac12 \sqrt{2} \eps^{-1} D_\ell(y')) : y' \in B^{n-1}_{\eps |\log \eps|}(y) \right\},
\end{equation}
for integers $k\geq1$ (we used Lemma \ref{lemm:WW.tilting.comparison.d} in the last step), as well as the fact that on $\support \chi$ we have
\begin{equation} \label{eq:app.improved.stability.ddt.vol}
	|\partial_z d\mu_{g_z}| + |\partial_z^2 d\mu_{g_z}| = O(1) d\mu_{g_z},
\end{equation}
which follows from \eqref{eq:sheets.enhanced.sff.bound}, \eqref{eq:mean.curv.ddt.metric}, \eqref{eq:mean.curv.ddt.sff}, \eqref{eq:mean.curv.ddt.h}.

Moreover, we note for future reference that the following expression holds on $\supp \chi$:
\begin{equation*}
	u = \overline{\mathbb{H}}_{\eps,\ell} + \phi + \sum_{m<\ell} \left(\overline{\mathbb{H}}_{\eps,m} - (-1)^{m-1}\right) + \sum_{m>\ell} \left(\overline{\mathbb{H}}_{\eps,m} + (-1)^{m-1}\right) = \overline{\mathbb{H}}_{\eps,\ell} + \phi + \sum_{m \neq \ell} O(\eps (\partial_z \overline{\mathbb{H}}_{\eps,m})),
\end{equation*}
so 
\begin{equation}\label{eq:app.stab.Wpp.u}
W''(u) = W''(\overline{\mathbb{H}}_{\eps,\ell}) + O(\phi) + \sum_{m\not = \ell}  O(\eps (\partial_z \overline{\mathbb{H}}_{\eps,m})).
\end{equation}

Let us set 
$\varphi(y,z) \triangleq \zeta(y) \chi(y,z) (\partial_z \overline{\mathbb{H}}_{\eps,\ell}(y,z)).$ Because $u$ is stable, 
\[ \int_{C_{8/10}(0)} (\eps |\nabla \varphi|^{2} + \eps^{-1} W''(u) \varphi^{2}) \, d\mu_g \geq 0. \]
We will write this integral in Fermi coordinates over $\Gamma_{\ell}$ and expand using the choice of $\varphi$. Note that 
\[ |\nabla \varphi|^{2} = (\partial_{z}\varphi)^{2} + |\nabla_{\Gamma_{\ell,z}} \varphi|^{2}. \]
We begin with the contribution of the vertical derivative, $\partial_{z}\varphi = \zeta (\partial_{z}\chi) (\partial_z \overline{\mathbb{H}}_{\eps,\ell}) + \zeta \chi (\partial_z^2 \overline{\mathbb{H}}_{\eps,\ell})$, to stability:
\begin{align*}
	& \int_{-\eta}^{\eta} \int_{\Gamma_{\ell,z}} \eps (\partial_{z} \varphi)^{2} \, d\mu_{g_{z}} \, dz \\
	& = \eps \int_{-\eta}^{\eta} \int_{\Gamma_{\ell,z}} \zeta^{2} \chi^{2} (\partial_z^2 \overline{\mathbb{H}}_{\eps,\ell})^{2} \, d\mu_{g_{z}} \, dz  + \eps \int_{-\eta}^{\eta} \int_{\Gamma_{\ell,z}} \zeta^{2} (\partial_{z}\chi)^{2} (\partial_z \overline{\mathbb{H}}_{\eps,\ell})^{2} \, d\mu_{g_{z}} \, dz \\
	& \qquad + 2 \eps \int_{-\eta}^{\eta} \int_{\Gamma_{\ell,z}} \zeta^{2} (\partial_{z}\chi)\chi (\partial_z \overline{\mathbb{H}}_{\eps,\ell}) (\partial_z^2 \overline{\mathbb{H}}_{\eps,\ell}) \, d\mu_{g_{z}} \, dz \\
	& = - \eps^{-1} \int_{-\eta}^{\eta} \int_{\Gamma_{\ell,z}} {\zeta}^2 \chi^2 W''(\overline{\mathbb{H}}_{\eps,\ell}) (\partial_z \overline{\mathbb{H}}_{\eps,\ell})^2 \, d\mu_{g_z} \, dz + \eps \int_{-\eta}^{\eta} \int_{\Gamma_{\ell,z}} {\zeta}^{2} (\partial_{z}\chi)^{2} (\partial_z \overline{\mathbb{H}}_{\eps,\ell})^{2} \, d\mu_{g_{z}} \, dz \\
	& \qquad - \eps^{-1} \int_{-\eta}^{\eta} \int_{\Gamma_{\ell,z}} {\zeta}^2 \chi^2 {(\partial_{z} \overline{\xi}_{\eps,\ell})} (\partial_z \overline{\mathbb{H}}_{\eps,\ell}) \, d\mu_{g_z} \, dz  -  \eps \int_{-\eta}^{\eta} \int_{\Gamma_{\ell,z}} {\zeta}^2 \chi^2 (\partial_z \overline{\mathbb{H}}_{\eps,\ell})(\partial_z^2 \overline{\mathbb{H}}_{\eps,\ell}) \, (\partial_z d\mu_{g_z}) \, dz \\
	& = - \eps^{-1} \int_{-\eta}^{\eta} \int_{\Gamma_{\ell,z}} {\zeta}^2 \chi^2 W''(\overline{\mathbb{H}}_{\eps,\ell}) (\partial_z \overline{\mathbb{H}}_{\eps,\ell})^2 \, d\mu_{g_z} \, dz  + \eps \int_{-\eta}^{\eta} \int_{\Gamma_{\ell,z}} {\zeta}^{2} (\partial_{z}\chi)^{2} (\partial_z \overline{\mathbb{H}}_{\eps,\ell})^{2} \, d\mu_{g_{z}} \, dz \\
	& \qquad - \eps^{-1} \int_{-\eta}^{\eta} \int_{\Gamma_{\ell,z}} {\zeta}^2 \chi^2 {(\partial_{z} \overline{\xi}_{\eps,\ell})} (\partial_z \overline{\mathbb{H}}_{\eps,\ell}) \, d\mu_{g_z} \, dz + \eps \int_{-\eta}^{\eta} \int_{\Gamma_{\ell,z}} {\zeta}^2 \chi (\partial_z \chi) (\partial_z \overline{\mathbb{H}}_{\eps,\ell})^2 \, (\partial_z d\mu_{g_z}) \, dz \\
	& \qquad + \tfrac12 \eps \int_{-\eta}^{\eta} \int_{\Gamma_{\ell,z}} {\zeta}^2 \chi^{2} (\partial_z \overline{\mathbb{H}}_{\eps,\ell})^2 \, (\partial_z^2 d\mu_{g_z}) \, dz,
\end{align*}
where we integrated by parts on the final term after the first equality and the second equality. Using \eqref{eq:approximate.heteroclinic.behavior}, \eqref{eq:app.stab.chi.bds}, \eqref{eq:chi-z-der-length}, \eqref{eq:stab.calc.spt.chi.H.der.est}, \eqref{eq:app.improved.stability.ddt.vol}, we find that:
\begin{align*}
	& \int_{-\eta}^{\eta} \int_{\Gamma_{\ell,z}} \eps (\partial_{z} \varphi)^{2} \, d\mu_{g_{z}} \, dz \\
	& =  - \eps^{-1} \int_{-\eta}^{\eta} \int_{\Gamma_{\ell,z}} {\zeta}^{2} \chi^{2} W''(\overline{\mathbb{H}}_{\eps,\ell}) (\partial_z \overline{\mathbb{H}}_{\eps,\ell})^{2}  \, d\mu_{g_{z}} \, dz \\
	& \qquad + O(\eps^{-2} L^{-1} + \eps^{-1}) \int_{\Gamma_{\ell}} {\zeta}^{2} \left( \exp(-2\sqrt{2}\eps^{-1}\rho_{\ell}^{+}) + \exp(-2\sqrt{2}\eps^{-1}\rho_{\ell}^{-})\right) \, d\mu_{\Gamma_{\ell}}\\
	&   \qquad + O(1)  \int_{\Gamma_{\ell}} {\zeta}^{2} \, d\mu_{\Gamma_\ell}.
\end{align*}

We now turn to the second term. We use Cauchy--Schwartz to estimate the mixed terms with a factor of $L^{-1/2}$ and $L^{1/2}$ respectively, in the first inequality below. 
\begin{align*}
	& \int_{-\eta}^{\eta} \int_{\Gamma_{\ell,z}} \eps |\nabla_{\Gamma_{\ell,z}} \varphi|^{2} \, d\mu_{g_{z}} \, dz \\
	& \leq (1+O(L^{-\tfrac12})) \cdot \eps \int_{-\eta}^{\eta} \int_{\Gamma_{\ell,z}}\Big(  |\nabla_{\Gamma_{\ell,z}} {\zeta}|^{2} \chi^{2} (\partial_z \overline{\mathbb{H}}_{\eps,\ell})^{2} \Big) \, d\mu_{g_{z}} \, dz \\
	& \qquad + O(L^{\tfrac12}) \cdot \eps \int_{-\eta}^{\eta} \int_{\Gamma_{\ell,z}}\Big( {\zeta}^{2} |\nabla \chi|^{2} (\partial_z \overline{\mathbb{H}}_{\eps,\ell})^{2} + {\zeta}^{2} \chi^{2} (\partial_z^2 \overline{\mathbb{H}}_{\eps,\ell})^{2}|\nabla_{\Gamma_{\ell,z}}h_{\ell}|^{2}  \Big) \, d\mu_{g_{z}} \,  dz\\
	& = (1+O(L^{-\tfrac12})) \cdot \energyunit \int_{\Gamma_{\ell}}  |\nabla_{\Gamma_{\ell}} {\zeta}|^{2} \,  d\mu_{\Gamma_{\ell}} \\
	& \qquad + O(\eps^{-2} L^{-\tfrac12})  \int_{\Gamma_{\ell}} {\zeta}^{2} \left( \exp(-2\sqrt{2}\eps^{-1}\rho_{\ell}^{+}) + \exp(-2\sqrt{2}\eps^{-1}\rho_{\ell}^{-}) \right) \, d\mu_{\Gamma_{\ell}} \\
	& \qquad +  O(L^{\frac12}) \int_{\Gamma_{\ell}} {\zeta}^{2} \,  d\mu_{\Gamma_{\ell}}.
\end{align*}
We have used \eqref{eq:mean.curv.ddt.metric}, \eqref{eq:mean.curv.ddt.sff}, \eqref{eq:mean.curv.ddt.h}, \eqref{eq:app.stab.chi.bds}, \eqref{eq:chi-z-der-length},  \eqref{eq:stab.calc.spt.chi.H.der.est}, \eqref{eq:app.improved.stability.ddt.vol}.

Putting these two computations together and multiplying by $\eps^2$, the stability condition becomes 
\begin{align} \label{eq:stab.reduct.prelim}
	& (1+O(L^{-\tfrac12})) \cdot \eps^2 \energyunit \int_{\Gamma_{\ell}}  |\nabla_{\Gamma_{\ell}} \zeta|^{2} \, d\mu_{\Gamma} \\
	& \geq \eps \cdot \int_{-\eta}^{\eta} \int_{\Gamma_{\ell,z}} {\zeta}^{2} \chi^{2} ( W''(\overline{\mathbb{H}}_{\eps,\ell}) - W''(u)) (\partial_z \overline{\mathbb{H}}_{\eps,\ell})^{2} \,  d\mu_{g_{z}} \, dz \nonumber\\
	& \qquad + O(L^{-\tfrac12} + \eps)  \int_{\Gamma_{\ell}} {\zeta}^{2} \left( \exp(-2\sqrt{2}\eps^{-1}\rho_{\ell}^{+}) + \exp(-2\sqrt{2}\eps^{-1}\rho_{\ell}^{-}) \right) \, d\mu_{\Gamma_{\ell}} \nonumber\\
	& \qquad + O(L^{\frac 12}\eps^2) \int_{\Gamma_{\ell}} {\zeta}^2 \, d\mu_{\Gamma_\ell} \nonumber
\end{align}
The first term of the right hand side represents the interaction between the sheets, and requires further consideration. To this end, we rewrite \eqref{eq:phi-eqn-from-WW} slightly, using the definition of $\cR(\phi)$
\begin{align*}
	& \eps^{2}( \Delta_{\Gamma_{\ell,z}} \phi + H_{\Gamma_{\ell,z}}\partial_{z} \phi + \partial^{2}_{z}\phi) \\
	& = W'(u) - \sum_{m=1}^{Q} W'(\overline{\mathbb{H}}_{\eps,m}) + \eps^2  (\Delta_{\Gamma_{\ell,z}} h_{\ell} - H_{\Gamma_{\ell,z}}) (\partial_z \overline{\mathbb{H}}_{\eps,\ell}) - \eps^2 |\nabla_{\Gamma_{\ell,z}} h_{\ell}|^{2} (\partial_z^2 \overline{\mathbb{H}}_{\eps, \ell}) \\
	& \qquad + \sum_{m\not = \ell} \big( \eps \cR_{m,1} ((Z_{\Gamma_m})_* \partial_z) \overline{\mathbb{H}}_{\eps,m}) - \eps^2 \cR_{m,2} ((Z_{\Gamma_m})_* \partial_z)^2 \overline{\mathbb{H}}_{\eps,m}) \big) - \sum_{m=1}^{Q} \overline{\xi}_{\eps,m}
\end{align*}
and then differentiate this with respect to $z$ to obtain
\begin{align} \label{eq:app.diff.phi.eqn.z}
	& \eps^{2}( \partial_{z}\Delta_{\Gamma_{\ell,z}} \phi + \partial_{z}(H_{\Gamma_{\ell,z}}\partial_{z} \phi)+ \partial^{3}_{z}\phi) \\
	& = W''(u) (\partial_{z}\phi) + \left(W''(u)-W''(\overline{\mathbb{H}}_{\eps,\ell})\right) (\partial_z \overline{\mathbb{H}}_{\eps,\ell})  \nonumber \\
	& \qquad + \sum_{m\not = \ell} 
	\left(W''(u) - W''(\overline{\mathbb{H}}_{\eps,m})\right) \partial_{z}\overline{\mathbb{H}}_{\eps,m} \nonumber\\
	& \qquad + \eps^2  \partial_{z}\left((\Delta_{\Gamma_{\ell,z}} h_{\ell} - H_{\Gamma_{\ell,z}}) (\partial_z \overline{\mathbb{H}}_{\eps,\ell}) \right) - \eps^2 \partial_{z}\left(|\nabla_{\Gamma_{\ell,z}}h_{\ell}|^{2} (\partial_z^2 \overline{\mathbb{H}}_{\eps,\ell}) \right) \nonumber \\
	& \qquad + \sum_{m\not = \ell} \left( \eps \partial_{z} \big( \cR_{m,1} ((Z_{\Gamma_m})_* \partial_z) \overline{\mathbb{H}}_{\eps,m} \big) - \eps^2 \partial_{z} \big( \cR_{m,2} ((Z_{\Gamma_m})_* \partial_z)^2 \overline{\mathbb{H}}_{\eps,m} \big) \right) - \sum_{m=1}^{Q} \partial_{z} \overline{\xi}_{\eps,m}. \nonumber
\end{align}
We multiply this by ${\zeta}(y)^{2}\chi(y,z)^{2}(\partial_z \overline{\mathbb{H}}_{\eps,\ell}(y,z))$, integrate in $(y,z)$, and estimate each term. The first term on the left hand side of \eqref{eq:app.diff.phi.eqn.z} yields:
\begin{align} \label{eq:app.improved.stability.diff.phi.1}
& \int_{-\eta}^{\eta} \int_{\Gamma_{\ell,z}} \eps^{2} ( \partial_{z}\Delta_{\Gamma_{\ell,z}} \phi ) {\zeta}^{2}\chi^{2} (\partial_z \overline{\mathbb{H}}_{\eps,\ell}) \, d\mu_{g_{z}} \, dz \\
& = -\int_{-\eta}^{\eta} \int_{\Gamma_{\ell,z}} \eps^{2} (\Delta_{\Gamma_{\ell,z}}\phi) \partial_{z}({\zeta}^{2}\chi^{2} (\partial_z\overline{\mathbb{H}}_{\eps,\ell}) \, d\mu_{g_{z}}) \, dz \nonumber \\
& =  O(\eps^{-1} |\cE_{\zeta}|) \int _{\Gamma_{\ell}} \zeta^{2} \, d\mu_{\Gamma_{\ell}}. \nonumber
\end{align}
Here, we have bounded $\eps^{2}\Delta_{\Gamma_{\ell,z}}\phi$ by \eqref{eq:app.improved.stability.phi.bounds.2} and the remaining terms using \eqref{eq:sheets.enhanced.sff.bound}, \eqref{eq:mean.curv.ddt.metric}, \eqref{eq:mean.curv.ddt.h},  \eqref{eq:app.stab.chi.bds}, \eqref{eq:chi-z-der-length}, \eqref{eq:stab.calc.spt.chi.H.der.est}, \eqref{eq:app.improved.stability.ddt.vol}. Continuing on, the second term on the left hand side of \eqref{eq:app.diff.phi.eqn.z} can be estimated as 
\begin{align} \label{eq:app.improved.stability.diff.phi.2}
& \int_{-\eta}^{\eta} \int_{\Gamma_{\ell,z}} \eps^{2} \partial_{z}(H_{\Gamma_{\ell,z}} \partial_{z}\phi) \zeta^{2} \chi^{2} \partial_z \overline{\mathbb{H}}_{\eps,\ell} \, d\mu_{g_{z}} \, dz \\
& = - \int_{-\eta}^{\eta} \int_{\Gamma_{\ell,z}}  \eps^{2} H_{\Gamma_{\ell,z}} \partial_{z}\phi \partial_{z}( \zeta^{2} \chi^{2} \partial_z \overline{\mathbb{H}}_{\eps,\ell} \,  d\mu_{g_{z}}) \, dz \nonumber \\
& = O(\eps^{-1} |\cE_{\zeta}|) \int_{\Gamma_{\ell}} \zeta^{2} d\mu_{\Gamma_{\ell}}  \nonumber
\end{align}
similarly. We now consider the third term on the left hand side of \eqref{eq:app.diff.phi.eqn.z}. It is \emph{not} an error term, but instead will cancel (up to error terms) with the first term on the right hand side:
\begin{align} \label{eq:app.improved.stability.diff.phi.3}
	& \int_{-\eta}^{\eta} \int_{\Gamma_{\ell,z}} \eps^{2} (\partial^{3}_{z}\phi) {\zeta}^{2} \chi^{2} (\partial_z \overline{\mathbb{H}}_{\eps,\ell}) \, d\mu_{g_{z}} \, dz \\
	& = - \int_{-\eta}^{\eta} \int_{\Gamma_{\ell,z}} \eps^2 (\partial^{2}_{z}\phi) {\zeta}^{2} \chi^{2} (\partial_z^2 \overline{\mathbb{H}}_{\eps,\ell}) \, d\mu_{g_{z}} \, dz  - \int_{-\eta}^{\eta} \int_{\Gamma_{\ell,z}} \eps^{2} (\partial^{2}_{z}\phi) {\zeta}^{2} (\partial_z \overline{\mathbb{H}}_{\eps,\ell}) \partial_{z}(\chi^{2} \, d\mu_{g_{z}}) \, dz \nonumber \\
	& = \int_{-\eta}^{\eta} \int_{\Gamma_{\ell,z}} \eps^2 (\partial_{z}\phi) {\zeta}^{2} \chi^{2} (\partial_z^3 \overline{\mathbb{H}}_{\eps,\ell}) \, d\mu_{g_{z}} \, dz  + \int_{-\eta}^{\eta} \int_{\Gamma_{\ell,z}} \eps^2 (\partial_{z}\phi) {\zeta}^{2}  (\partial_z^2 \overline{\mathbb{H}}_{\eps,\ell}) \partial_{z}(\chi^{2} \, d\mu_{g_{z}}) \, dz \nonumber \\
	& \qquad - \int_{-\eta}^{\eta} \int_{\Gamma_{\ell,z}} \eps^{2} (\partial^{2}_{z}\phi) {\zeta}^{2} (\partial_z \overline{\mathbb{H}}_{\eps,\ell}) \partial_{z}(\chi^{2} \, d\mu_{g_{z}}) \, dz \nonumber \\
	& = \int_{-\eta}^{\eta} \int_{\Gamma_{\ell,z}} (\partial_{z}\phi) {\zeta}^{2} \chi^{2} W''(\overline{\mathbb{H}}_{\eps,\ell}) (\partial_z \overline{\mathbb{H}}_{\eps,\ell}) \, d\mu_{g_{z}} \, dz  + O(\eps^{-1} |\cE_{\zeta}|) \int_{\Gamma_{\ell}} {\zeta}^{2} \, d\mu_{\Gamma_{\ell}} \nonumber \\
	& = \int_{-\eta}^{\eta} \int_{\Gamma_{\ell,z}}  (\partial_{z}\phi) {\zeta}^{2} \chi^{2} W''(u) (\partial_z \overline{\mathbb{H}}_{\eps,\ell}) \,  d\mu_{g_{z}} \, dz  + \int_{-\eta}^{\eta} \int_{\Gamma_{\ell}}  (\partial_{z}\phi) {\zeta}^{2} \chi^{2} (W''(\overline{\mathbb{H}}_{\eps,\ell}) - W''(u) ) (\partial_z \overline{\mathbb{H}}_{\eps,\ell}) \, d\mu_{g_{z}} \, dz \nonumber \\
	& \qquad + O(\eps^{-1} |\cE_{\zeta}|) \int_{\Gamma_{\ell}} {\zeta}^{2} \, d\mu_{\Gamma_{\ell}}; \nonumber 
\end{align}
we have used \eqref{eq:app.improved.stability.phi.bounds},  \eqref{eq:app.stab.chi.bds}, \eqref{eq:chi-z-der-length}, \eqref{eq:stab.calc.spt.chi.H.der.est}, \eqref{eq:app.improved.stability.ddt.vol}, \eqref{eq:mean.curv.ddt.metric}, \eqref{eq:mean.curv.ddt.h}. The first term of the final expression above cancels with the first term on the right hand side of \eqref{eq:app.diff.phi.eqn.z}. We now study the second term of \eqref{eq:app.improved.stability.diff.phi.3}. Using \eqref{eq:app.stab.Wpp.u}, the second term of the right hand side of \eqref{eq:app.improved.stability.diff.phi.3} can be rewritten as
\begin{align} \label{eq:app.improved.stability.mixed.estimate}
& \int_{-\eta}^{\eta} \int_{\Gamma_{\ell_z}}  (\partial_{z}\phi) {\zeta}^{2} \chi^{2} (W''(\overline{\mathbb{H}}_{\eps,\ell}) - W''(u)) (\partial_z \overline{\mathbb{H}}_{\eps,\ell}) \,  d\mu_{g_{z}} \, dz \\
& = O(\eps) \int_{\Gamma_{\ell}} {\zeta}^{2} \, d\mu_{\Gamma_{\ell}}  + O(1) \sum_{m \neq \ell}  \int_{-\eta}^{\eta}  \int_{\Gamma_{\ell,z}} \chi^{2} {\zeta}^{2} \cdot \eps |\partial_z \overline{\mathbb{H}}_{\eps,m}| (\partial_z \overline{\mathbb{H}}_{\eps,\ell}) \, d\mu_{g_{z}} \, dz \nonumber \\
& = O(\eps) \int_{\Gamma_{\ell}} {\zeta}^{2} \, d\mu_{\Gamma_{\ell}}  + O\left( \eps^{-1} \exp(-\sqrt{2} \eps^{-1} D_\ell(\cdot) \right)  \sum_{m \neq \ell}  \int_{-\eta}^{\eta}  \int_{\Gamma_{\ell,z}} \chi^{2} {\zeta}^{2} \, d\mu_{g_{z}} \, dz \nonumber \\
& = O\left(\eps + \sup_{\support {\zeta}} \Big[ \eps^{-1} D_{\ell}(\cdot)  \exp(-\sqrt{2}\eps^{-1}D_{\ell}(\cdot))\Big] \right) \int_{\Gamma_{\ell}} {\zeta}^{2} d\mu_{\Gamma_{\ell}} \nonumber \\
& = O(\eps^{-1} |\cE_{\zeta}|) \int_{\Gamma_{\ell}} {\zeta}^{2} \nonumber d\mu_{\Gamma_{\ell}}.
\end{align}
Note that we estimated $|(\partial_z \overline{\mathbb{H}}_{\eps,m}) (\partial_z \overline{\mathbb{H}}_{\eps,\ell})|$ using \eqref{eq:heteroclinic.expansion.ii} and Lemma \ref{lemm:WW.tilting.comparison.d}:
\begin{equation} \label{eq:app.improved.stability.interference.product}
\eps^2 |(\partial_z \overline{\mathbb{H}}_{\eps,m}) (\partial_z \overline{\mathbb{H}}_{\eps,\ell})|(y,z) \leq c' \eps^2 + c' \exp\left( -\sqrt{2} \eps^{-1}D_{\ell}(y) \right), \; m \neq \ell. 
\end{equation}

We continue estimating terms on the right hand side in \eqref{eq:app.diff.phi.eqn.z}. 

We have just seen that the first term on the right hand side will cancel with a term of \eqref{eq:app.improved.stability.diff.phi.3}. The second term of \eqref{eq:app.diff.phi.eqn.z} is the term we are interested in estimating. We now consider the third term of \eqref{eq:app.diff.phi.eqn.z}. For $m\not = \ell$, we note that on $\supp \chi$,
\[
W''(\overline{\mathbb{H}}_{\eps,m}) = W''(\pm 1) + O(\eps (\partial_z \overline{\mathbb{H}}_{\eps,m}))
\]
Thus, combined with \eqref{eq:app.stab.Wpp.u}, we find that on $\supp\chi$,
\[ W''(u) - W''(\overline{\mathbb{H}}_{\eps,m}) = W''(\overline{\mathbb{H}}_{\eps,\ell}) - W''(\pm 1) + O(\eps) + \sum_{m'\neq \ell} O(\eps (\partial_z \overline{\mathbb{H}}_{\eps,m'})). \]
Hence, using Lemma \ref{lemm:WW.tilting.comparison.d}, \eqref{eq:phi.est.stab.proof}, and bounding $|(\partial_z \overline{\mathbb{H}}_{\eps,\ell}) (\partial_z \overline{\mathbb{H}}_{\eps,m}) (\partial_z \overline{\mathbb{H}}_{\eps,m'})|$ like in \eqref{eq:app.improved.stability.interference.product},
\begin{align} \label{eq:app.improved.stability.diff.phi.4}
& \int_{-\eta}^{\eta} \int_{\Gamma_{\ell,z}} \left(W''(u)-W''(\overline{\mathbb{H}}_{\eps,m}) \right) 
(\partial_z\overline{\mathbb{H}}_{\eps,m}) (\partial_{z}d_{m} + O(|\nabla_{\Gamma_{m}} h_{m}|)) 
{\zeta}^{2}\chi^{2} (\partial_z \overline{\mathbb{H}}_{\eps,\ell}) \, d\mu_{g_{z}} \, dz \\
& = \int_{-\eta}^{\eta} \int_{\Gamma_{\ell,z}} {\zeta}^{2} \chi^{2} \left(W''(\overline{\mathbb{H}}_{\eps,\ell}) - W''(\pm1)  \right) (\partial_z\overline{\mathbb{H}}_{\eps,m}) (\partial_z\overline{\mathbb{H}}_{\eps,\ell}) \, d\mu_{g_z} \, dz \nonumber \\
& \qquad + O\left( \sup_{\support {\zeta}} \Big[ \eps^{-1} D_{\ell}(\cdot) \exp( -\sqrt{2}\eps^{-1}D_{\ell}(\cdot)) + \eps^{-2} D_{\ell}(\cdot) \exp(- \tfrac 32 \sqrt{2}\eps^{-1}D_{\ell}(\cdot))] \right) \int_{\Gamma_{\ell}} {\zeta}^{2} \, d\mu_{\Gamma_{\ell}} \nonumber \\
& = \int_{-\eta}^{\eta} \int_{\Gamma_{\ell,z}} {\zeta}^{2} \chi^{2} \left( W''(\overline{\mathbb{H}}_{\eps,\ell}) - W''(\pm1) ) \right) (\partial_z\overline{\mathbb{H}}_{\eps,m}) (\partial_z \overline{\mathbb{H}}_{\eps,\ell}) \, d\mu_{g_z} \, dz \nonumber \\
& \qquad + O(\eps^{-1} |\cE_{\zeta}|) \int_{\Gamma_{\ell}} {\zeta}^{2} \, d\mu_{\Gamma_{\ell}}. \nonumber
\end{align}

We now turn to the next term of  \eqref{eq:app.diff.phi.eqn.z}. Note that, on $\supp\chi$,
\[
|\Delta_{\Gamma_{\ell,z}} h_{\ell}| + |\partial_{z} \Delta_{\Gamma_{\ell,z}} h_{\ell}| + |H_{\Gamma_{\ell,z}}| + |\partial_{z}H_{\Gamma_{\ell,z}}| \leq c'
\]
by \eqref{eq:phi.est.stab.proof}, \eqref{eq:sheets.enhanced.sff.bound},  \eqref{eq:mean.curv.ddt.h}, \eqref{eq:mean.curv.ddt.grad},  \eqref{eq:mean.curv.ddt.laplace}. Thus, 
\begin{equation} \label{eq:app.improved.stability.diff.phi.5}
	\int_{-\eta}^{\eta}\int_{\Gamma_{\ell,z}} \eps^2 \partial_{z}\left( (\Delta_{\Gamma_{\ell,z}} h_{\ell} - H_{\Gamma_{\ell,z}}) (\partial_z\overline{\mathbb{H}}_{\eps,\ell}) \right) {\zeta}^{2}\chi^{2} (\partial_z\overline{\mathbb{H}}_{\eps,\ell}) \, d\mu_{g_{z}} \, dz = O(\eps) \int_{\Gamma_{\ell}} {\zeta}(y)^{2} d\mu_{\Gamma}. 
\end{equation}
The next term of \eqref{eq:app.diff.phi.eqn.z} is estimated similarly. 

The term of \eqref{eq:app.diff.phi.eqn.z} involving $\cR_{m,1}$ is estimated by an integration by parts as follows. First, recall the definition of $\cR_{m,1}$ from Appendix \ref{app:WW-results} and note that \eqref{eq:app.improved.stability.phi.bounds} implies that 
\[
|\cR_{m,1}| \leq c'\eps^{2} + c' \sum_{m=1}^{Q} \sup\left\{\exp(-\sqrt{2}\eps^{-1}D_{m}(\cdot))\right\}.
\]
with the $\sup$ taken as in \eqref{eq:app.improved.stability.phi.bounds}. This bound thus implies
\begin{align} \label{eq:app.improved.stability.diff.phi.6}
& \int_{-\eta}^{\eta} \int_{\Gamma_{\ell,z}} \eps  \partial_{z} \left( \cR_{m,1} ((Z_{\Gamma_m})_* \partial_z) \overline{\mathbb{H}}_{\eps,m} \right) {\zeta}^{2} \chi^{2} (\partial_z \overline{\mathbb{H}}_{\eps,\ell}) \, d\mu_{g_{z}} \, dz \\
& = - \int_{-\eta}^{\eta} \int_{\Gamma_{\ell,z}} \eps \cR_{m,1} \big( ((Z_{\Gamma_m})_* \partial_z) \overline{\mathbb{H}}_{\eps,m} \big) {\zeta}^{2} \partial_{z}\left( \chi^{2} (\partial_z \overline{\mathbb{H}}_{\eps,\ell}) \, d\mu_{g_{z}}\right) dz \nonumber \\
& = O(\eps^{-1} |\cE_{\zeta}|) \int_{\Gamma_{\ell}} {\zeta}^{2} \, d\mu_{\Gamma} \nonumber
\end{align}
where we additionally used Lemma \ref{lemm:WW.tilting.comparison.d}, \eqref{eq:app.stab.chi.bds}, \eqref{eq:chi-z-der-length}, \eqref{eq:app.improved.stability.ddt.vol}, \eqref{eq:app.improved.stability.interference.product}. The terms in \eqref{eq:app.diff.phi.eqn.z} involving $\cR_{m,2}$, $\overline{\xi}_{\eps,m}$, are estimated similarly.

Plugging \eqref{eq:app.improved.stability.diff.phi.1}, \eqref{eq:app.improved.stability.diff.phi.2}, \eqref{eq:app.improved.stability.diff.phi.3}, \eqref{eq:app.improved.stability.diff.phi.4}, \eqref{eq:app.improved.stability.diff.phi.5}, \eqref{eq:app.improved.stability.diff.phi.6} into \eqref{eq:app.diff.phi.eqn.z}, we find that 
\begin{align*}
& \int_{-\eta}^{\eta} \int_{\Gamma_{\ell,z}} {\zeta}^2 \chi^{2} (W''(u) - W''(\overline{\mathbb{H}}_{\eps,\ell}))) (\partial_z\overline{\mathbb{H}}_{\eps,\ell})^{2} \, d\mu_{g_{z}} \, dz \\
& = \sum_{m\not = \ell} \int_{-\eta}^{\eta} \int_{\Gamma_{\ell,z}} {\zeta}^{2} \chi^{2} \left(W''(\pm1) - W''(\overline{\mathbb{H}}_{\eps,\ell}) \right) (\partial_z\overline{\mathbb{H}}_{\eps,m}) (\partial_z\overline{\mathbb{H}}_{\eps,\ell}) \, d\mu_{g_{z}} \, dz + O(\eps^{-1} |\cE_{\zeta}|) \int_{\Gamma_{\ell}} {\zeta}^{2} \, d\mu_{\Gamma_{\ell}}.
\end{align*}
Observe that for $m\not \in \{\ell-1,\ell+1\}$, 
\[ \eps^2 |(\partial_z \overline{\mathbb{H}}_{\eps,m}) (\partial_z \overline{\mathbb{H}}_{\eps,\ell})|(y,z) = O(|\cE_{\zeta}|^{\frac{2}{1+\kappa}}) \]
on $\supp\chi$. Thus, we can write (because $(-1)^{\ell\pm1} = -1$ by \eqref{eq:app.improved.stability.parity.assumption}):
\begin{align*}
& \int_{-\eta}^{\eta} \int_{\Gamma_{\ell,z}}  {\zeta}^2 \chi^{2} (W''(u) - W''(\overline{\mathbb{H}}_{\eps,\ell}))) (\partial_z\overline{\mathbb{H}}_{\eps,\ell})^{2} \, d\mu_{g_{z}} \, dz \\
& = \sum_{m \in \{\ell\pm1\}} \int_{-\eta}^{\eta} \int_{\Gamma_{\ell,z}} {\zeta}^{2} \chi^{2} \left( W''(\overline{\mathbb{H}}_{\eps,\ell}) - W''(\pm1) ) \right) (\partial_z\overline{\mathbb{H}}_{\eps,m}) (\partial_z\overline{\mathbb{H}}_{\eps,\ell}) \, d\mu_{g_z} \, dz \\
& \qquad + O(\eps^{-1} |\cE_{\zeta}|) \int_{\Gamma_{\ell}} {\zeta}^{2} \, d\mu_{\Gamma_{\ell}} \\
& = \eps^{-2} \int_{\Gamma_{\ell}} {\zeta}^2 \left( \int_{-\eta}^{\eta} \left(W''(\overline{\mathbb{H}}(\eps^{-1}t)) - W''(\pm1)  \right) \overline{\mathbb{H}}'(-\eps^{-1}(d_{\ell+1}(y)-t)) \overline{\mathbb{H}}'(\eps^{-1}t) \, dt \right) d\mu_{\Gamma_{\ell}} \\
& \qquad + \eps^{-2} \int_{\Gamma_{\ell}} {\zeta}^2 \left( \int_{-\eta}^{\eta} \left(W''(\overline{\mathbb{H}}(\eps^{-1}t)) - W''(\pm1)  \right) \overline{\mathbb{H}}'(\eps^{-1}(t+d_{\ell-1}(y))) \overline{\mathbb{H}}'(\eps^{-1}t) \, dt \right) d\mu_{\Gamma_{\ell}} \\
& \qquad + O(\eps^{-1} |\cE_{\zeta}|) \int_{\Gamma_{\ell}} {\zeta}^{2} \, d\mu_{\Gamma_{\ell}} \\
& = \eps^{-1} \int_{\Gamma_{\ell}} {\zeta}^2 \left( \int_{-\infty}^{\infty} \left(W''(\mathbb{H}(t)) - 2 \right) \mathbb{H}'(t-\eps^{-1}|d_{\ell+1}(y)|) \mathbb{H}'(t) dt \right) d\mu_{\Gamma_{\ell}} \\
& \qquad + \eps^{-1} \int_{\Gamma_{\ell}} {\zeta}^2 \left( \int_{-\infty}^{\infty} \left(W''(\mathbb{H}(t)) - 2 \right) \mathbb{H}'(t-\eps^{-1}|d_{\ell-1}(y))|) \mathbb{H}'(t) dt \right) d\mu_{\Gamma_{\ell}} \\
& \qquad + O(\eps^{-1} |\cE_{\zeta}|) \int_{\Gamma_{\ell}} {\zeta}^{2} \, d\mu_{\Gamma_{\ell}} \\
& = - 4\sqrt{2}(\expansioncoeff)^{2} \cdot \eps^{-1} \int_{\Gamma_{\ell}} {\zeta}^2 \left( \exp(-\sqrt{2}\eps^{-1}|d_{\ell+1}|) + \exp(-\sqrt{2}\eps^{-1}|d_{\ell-1}|) \right) d\mu_{\Gamma_{\ell}} \\
& \qquad + O(\eps^{-1}) \int_{\Gamma_{\ell}} {\zeta}^2 \left( \exp(-2(1-\kappa)\sqrt{2}\eps^{-1}|d_{\ell+1}|) + \exp(-2(1-\kappa)\sqrt{2}\eps^{-1}|d_{\ell-1}|) \right) d\mu_{\Gamma_{\ell}} \\
& \qquad + O(\eps^{-1} |\cE_{\zeta}|) \int_{\Gamma_{\ell}} {\zeta}^2 \, d\mu_{\Gamma_{\ell}} \\
& = -(4\sqrt{2}(\expansioncoeff)^{2} + o(1)) \cdot \eps^{-1} \int_{\Gamma_{\ell}} \left( \exp(-\sqrt{2}\eps^{-1}|d_{\ell+1}|) + \exp(-\sqrt{2}\eps^{-1}|d_{\ell-1}|)  \right) {\zeta}^{2} \, d\mu_{\Gamma_{\ell}} \\
& \qquad + O(\eps^{-1} |\cE_{\zeta}|) \int_{\Gamma_{\ell}} {\zeta}^2 \, d\mu_{\Gamma_{\ell}}.
\end{align*}
In the last two equalities we used Lemma  \ref{lemm:app.sheets.interaction.1d}. Together with \eqref{eq:stab.reduct.prelim},  Lemma \ref{lemm:WW.tilting.comparison.d}, we 
get \eqref{eq:toda.stability.estimate.sharper}.

\section{An interpolation lemma}


\label{app:interpolation.lemma}

We record a proof of the following interpolation inequality:

\begin{lemm} \label{lemm:holder.space.interpolation}
	For $0 < \alpha < \theta < 1$ and $f : \RR^n \to \RR$, we have
	\[ \Vert \nabla f \Vert_{C^{0,\alpha}(\RR^n)} \leq C \Vert f \Vert_{C^{0,\theta}(\RR^n)}^{\theta-\alpha} \Vert \nabla f \Vert_{C^{0,\theta}(\RR^n)}^{1+\alpha-\theta}, \]
	with $C = C(n)$.
\end{lemm}
\begin{proof}
	We assume $\nabla f \not \equiv 0$. Fix $\mathbf{x} \in \RR^n$ with $\nabla f(\mathbf{x}) \neq 0$, and set $\mathbf{e}:= \nabla f(\mathbf{x})/|\nabla f(\mathbf{x})|$. For $t > 0$:
	\begin{align*}
		f(\mathbf{x} + t \mathbf{e}) - f(\mathbf{x}) 
			& = \int_0^1 \nabla f(\mathbf{x} + st\mathbf{e}) \cdot t\mathbf{e} \, ds \\
			& = \int_0^1 (\nabla f(\mathbf{x} + st\mathbf{e}) - \nabla f(\mathbf{x})) \cdot t \mathbf{e} \, ds + \nabla f(\mathbf{x}) \cdot t\mathbf{e} \\
			& = \int_0^1 (\nabla f(\mathbf{x} + st\mathbf{e}) - \nabla f(\mathbf{x})) \cdot t \mathbf{e} \, ds + t |\nabla f(\mathbf{x})|.
	\end{align*} 
	Rearranging, and using the H\"older estimate on $f(\mathbf{x}+t\mathbf{e})-f(\mathbf{x})$ and $\nabla f(\mathbf{x} + st \mathbf{e}) - \nabla f(\mathbf{x})$ we deduce
	\[ t |\nabla f(\mathbf{x})| \leq [f]_{\theta} t^\theta + [\nabla f]_{\theta} t^{1+\theta}. \]
	Dividing through by $t$ and optimizing in $t$ (using calculus) and using the fact that $x$ was arbitrary:
	\begin{equation} \label{eq:holder.space.interpolation.i}
		\Vert \nabla f \Vert_{C^0(\RR^n)} \leq 2 [f]_\theta^\theta [\nabla f]_\theta^{1-\theta}.
	\end{equation}
	By the trivial $C^{0,\theta} \hookrightarrow C^{0,\alpha} \hookrightarrow C^0$ interpolation on $\nabla f$ and the previous estimate we conclude:
	\begin{equation} \label{eq:holder.space.interpolation.ii}
		[\nabla f]_\alpha \leq C \Vert \nabla f \Vert_{C^0(\RR^n)}^{\frac{\theta-\alpha}{\theta}} [\nabla f]_\theta^{\frac{\alpha}{\theta}} \leq 2 C [f]_\theta^{\theta-\alpha} [\nabla f]_\theta^{\frac{(1-\theta)(\theta-\alpha)}{\theta} + \frac{\alpha}{\theta}} = 2 C [f]_\theta^{\theta-\alpha} [\nabla f]_\theta^{1+\alpha-\theta}.
	\end{equation}
	Together, \eqref{eq:holder.space.interpolation.i}, \eqref{eq:holder.space.interpolation.ii} give the required estimate when we replace the seminorms by norms.
\end{proof}

\bibliographystyle{alpha}
\bibliography{main}

\begin{thebibliography}{dPKWY10}

\bibitem[AC00]{AmbrosioCabre00}
Luigi Ambrosio and Xavier Cabr\'e.
\newblock Entire solutions of semilinear elliptic equations in {$\bold R^3$}
  and a conjecture of {D}e {G}iorgi.
\newblock {\em J. Amer. Math. Soc.}, 13(4):725--739, 2000.

\bibitem[ACS18]{AmbrozioCarlottoSharp:index.genus}
Lucas Ambrozio, Alessandro Carlotto, and Ben Sharp.
\newblock Comparing the {M}orse index and the first {B}etti number of minimal
  hypersurfaces.
\newblock {\em J. Differential Geom.}, 108(3):379--410, 2018.

\bibitem[Aie16]{Aiex:ellipsoids}
Nicolau~Sarquis Aiex.
\newblock The width of ellipsoids.
\newblock {\em To appear in Comm.\ Anal.\ Geom.\,
  \url{https://arxiv.org/abs/1601.01032}}, 2016.

\bibitem[Car17]{Carlotto:arb-large}
Alessandro Carlotto.
\newblock Generic finiteness of minimal surfaces with bounded {M}orse index.
\newblock {\em Ann. Sc. Norm. Super. Pisa Cl. Sci. (5)}, 17(3):1153--1171,
  2017.

\bibitem[CKM17]{CKM}
Otis Chodosh, Daniel Ketover, and Davi Maximo.
\newblock Minimal hypersurfaces with bounded index.
\newblock {\em Invent. Math.}, 209(3):617--664, 2017.

\bibitem[CM19]{ChodoshMantoulidis:unbounded-area}
Otis Chodosh and Christos Mantoulidis.
\newblock {Minimal Hypersurfaces with Arbitrarily Large Area}.
\newblock {\em International Mathematics Research Notices}, 07 2019.

\bibitem[CS85]{ChoiSchoen}
Hyeong~In Choi and Richard Schoen.
\newblock The space of minimal embeddings of a surface into a three-dimensional
  manifold of positive {R}icci curvature.
\newblock {\em Invent. Math.}, 81(3):387--394, 1985.

\bibitem[dCP79]{doCarmoPeng}
Manfredo do~Carmo and Chiakuei Peng.
\newblock Stable complete minimal surfaces in {${\bf R}^{3}$} are planes.
\newblock {\em Bull. Amer. Math. Soc. (N.S.)}, 1(6):903--906, 1979.

\bibitem[dPKW11]{delPinoKowalczykWei:DG-counterexample}
Manuel del Pino, Michal Kowalczyk, and Juncheng Wei.
\newblock On {D}e {G}iorgi's conjecture in dimension {$N\geq 9$}.
\newblock {\em Ann. of Math. (2)}, 174(3):1485--1569, 2011.

\bibitem[dPKW13]{delPinoKowalczykWei}
Manuel del Pino, Michal Kowalczyk, and Juncheng Wei.
\newblock Entire solutions of the {A}llen-{C}ahn equation and complete embedded
  minimal surfaces of finite total curvature in {$\Bbb R^3$}.
\newblock {\em J. Differential Geom.}, 93(1):67--131, 2013.

\bibitem[dPKWY10]{delPinoKowalczykWeiYang:interface}
Manuel del Pino, Michal Kowalczyk, Juncheng Wei, and Jun Yang.
\newblock Interface foliation near minimal submanifolds in {R}iemannian
  manifolds with positive {R}icci curvature.
\newblock {\em Geom. Funct. Anal.}, 20(4):918--957, 2010.

\bibitem[EM08]{EjiriMicallef}
Norio Ejiri and Mario Micallef.
\newblock Comparison between second variation of area and second variation of
  energy of a minimal surface.
\newblock {\em Adv. Calc. Var.}, 1(3):223--239, 2008.

\bibitem[FCS80]{Fischer-Colbrie-Schoen}
Doris Fischer-Colbrie and Richard Schoen.
\newblock The structure of complete stable minimal surfaces in {$3$}-manifolds
  of nonnegative scalar curvature.
\newblock {\em Comm. Pure Appl. Math.}, 33(2):199--211, 1980.

\bibitem[FMV13]{FarinaMariValdinoci13}
Alberto Farina, Luciano Mari, and Enrico Valdinoci.
\newblock Splitting theorems, symmetry results and overdetermined problems for
  {R}iemannian manifolds.
\newblock {\em Comm. Partial Differential Equations}, 38(10):1818--1862, 2013.

\bibitem[Gas17]{Gaspar}
Pedro Gaspar.
\newblock The second inner variation of energy and the {M}orse index of limit
  interfaces.
\newblock To appear in J.\ Geom.\ Anal.,
  \url{https://arxiv.org/abs/1710.04719}, 2017.

\bibitem[GG98]{GhoussoubGui98}
N.~Ghoussoub and C.~Gui.
\newblock On a conjecture of {D}e {G}iorgi and some related problems.
\newblock {\em Math. Ann.}, 311(3):481--491, 1998.

\bibitem[GG18]{GasparGuaraco}
Pedro Gaspar and Marco A.~M. Guaraco.
\newblock The {A}llen-{C}ahn equation on closed manifolds.
\newblock {\em Calc. Var. Partial Differential Equations}, 57(4):Art. 101, 42,
  2018.

\bibitem[GG19]{GasparGuaraco:weyl}
Pedro Gaspar and Marco A.~M. Guaraco.
\newblock The {W}eyl law for the phase transition spectrum and density of limit
  interfaces.
\newblock {\em Geom. Funct. Anal.}, 29(2):382--410, 2019.

\bibitem[Gro03]{Gromov:waist}
M.~Gromov.
\newblock Isoperimetry of waists and concentration of maps.
\newblock {\em Geom. Funct. Anal.}, 13(1):178--215, 2003.

\bibitem[GT01]{GilbargTrudinger01}
David Gilbarg and Neil~S. Trudinger.
\newblock {\em Elliptic partial differential equations of second order}.
\newblock Classics in Mathematics. Springer-Verlag, Berlin, 2001.
\newblock Reprint of the 1998 edition.

\bibitem[Gua18]{Guaraco}
Marco A.~M. Guaraco.
\newblock Min--max for phase transitions and the existence of embedded minimal
  hypersurfaces.
\newblock {\em J. Differential Geom.}, 108(1):91--133, 2018.

\bibitem[Gut09]{Guth:minimax}
Larry Guth.
\newblock Minimax problems related to cup powers and {S}teenrod squares.
\newblock {\em Geom. Funct. Anal.}, 18(6):1917--1987, 2009.

\bibitem[Hie18]{Hiesmayr}
Fritz Hiesmayr.
\newblock Spectrum and index of two-sided {A}llen-{C}ahn minimal hypersurfaces.
\newblock {\em Comm. Partial Differential Equations}, 43(11):1541--1565, 2018.

\bibitem[HL97]{HanLin}
Qing Han and Fanghua Lin.
\newblock {\em Elliptic partial differential equations}, volume~1 of {\em
  Courant Lecture Notes in Mathematics}.
\newblock New York University, Courant Institute of Mathematical Sciences, New
  York; American Mathematical Society, Providence, RI, 1997.

\bibitem[HT00]{HutchinsonTonegawa00}
John~E. Hutchinson and Yoshihiro Tonegawa.
\newblock Convergence of phase interfaces in the van der
  {W}aals-{C}ahn-{H}illiard theory.
\newblock {\em Calc. Var. Partial Differential Equations}, 10(1):49--84, 2000.

\bibitem[IMN18]{IrieMarquesNeves}
Kei Irie, Fernando Marques, and Andr\'e Neves.
\newblock Density of minimal hypersurfaces for generic metrics.
\newblock {\em Ann. of Math. (2)}, 187(3):963--972, 2018.

\bibitem[KL17]{KetoverLiokumovich}
Daniel Ketover and Yevgeny Liokumovich.
\newblock On the existence of unstable minimal {H}eegaard surfaces.
\newblock {\em \url{https://arxiv.org/abs/1709.09744}}, 2017.

\bibitem[KLP12]{KowalczykLiuPacard12}
Michal Kowalczyk, Yong Liu, and Frank Pacard.
\newblock The space of 4-ended solutions to the {A}llen-{C}ahn equation in the
  plane.
\newblock {\em Ann. Inst. H. Poincar\'e Anal. Non Lin\'eaire}, 29(5):761--781,
  2012.

\bibitem[KLS]{KetoverLiokumovichSong}
Daniel Ketover, Yevgeny Liokumovich, and Antoine Song.
\newblock in preparation.

\bibitem[KMN16]{KMN:catenoid}
Daniel Ketover, Fernando~C. Marques, and Andr\'e Neves.
\newblock The catenoid estimate and its geometric applications.
\newblock {\em To appear in J.\ Differential Geom.\,
  \url{https://arxiv.org/abs/1601.04514}}, 2016.

\bibitem[Le11]{Le:2ndvar}
Nam~Q. Le.
\newblock On the second inner variation of the {A}llen-{C}ahn functional and
  its applications.
\newblock {\em Indiana Univ. Math. J.}, 60(6):1843--1856, 2011.

\bibitem[LMN18]{LMN:Weyl}
Yevgeny Liokumovich, Fernando Marques, and Andr\'e Neves.
\newblock Weyl law for the volume spectrum.
\newblock {\em Ann. of Math. (2)}, 187(3):933--961, 2018.

\bibitem[LWW17]{LiuWangWei}
Yong Liu, Kelei Wang, and Juncheng Wei.
\newblock Global minimizers of the {A}llen-{C}ahn equation in dimension {$n\geq
  8$}.
\newblock {\em J. Math. Pures Appl. (9)}, 108(6):818--840, 2017.

\bibitem[LZ16]{LiZhou}
Haozhao Li and Xin Zhou.
\newblock Existence of minimal surfaces of arbitrarily large {M}orse index.
\newblock {\em Calc. Var. Partial Differential Equations}, 55(3):Art. 64, 12,
  2016.

\bibitem[Man17]{Mantoulidis}
Christos Mantoulidis.
\newblock {A}llen-{C}ahn min-max on surfaces.
\newblock {\em To appear in J. Differential Geom,
  \url{https://arxiv.org/abs/1706.05946}}, 2017.

\bibitem[Mar14]{marques:ICM}
Fernando~C. Marques.
\newblock Minimal surfaces - variational theory and applications.
\newblock {\em Proceedings of the International Congress of Mathematicians,
  \url{https://arxiv.org/abs/1409.7648}}, 2014.

\bibitem[MN12]{MarquesNeves:rigidity.min.max}
Fernando~C. Marques and Andr\'e Neves.
\newblock Rigidity of min-max minimal spheres in three-manifolds.
\newblock {\em Duke Math. J.}, 161(14):2725--2752, 2012.

\bibitem[MN14]{MarquesNeves:Willmore}
Fernando~C. Marques and Andr\'e Neves.
\newblock Min-max theory and the {W}illmore conjecture.
\newblock {\em Ann. of Math. (2)}, 179(2):683--782, 2014.

\bibitem[MN16a]{MarquesNeves:multiplicity}
Fernando~C. Marques and Andr\'e Neves.
\newblock Morse index and multiplicity of min-max minimal hypersurfaces.
\newblock {\em Camb. J. Math.}, 4(4):463--511, 2016.

\bibitem[MN16b]{MarquesNeves:spaceOfCycles}
Fernando~C. Marques and Andr\'e Neves.
\newblock Topology of the space of cycles and existence of minimal varieties.
\newblock In {\em Surveys in differential geometry 2016. {A}dvances in geometry
  and mathematical physics}, volume~21 of {\em Surv. Differ. Geom.}, pages
  165--177. Int. Press, Somerville, MA, 2016.

\bibitem[MN17]{MarquesNeves:posRic}
Fernando~C. Marques and Andr\'e Neves.
\newblock Existence of infinitely many minimal hypersurfaces in positive
  {R}icci curvature.
\newblock {\em Invent. Math.}, 209(2):577--616, 2017.

\bibitem[MN18]{MarquesNeves:uper-semi-index}
Fernando~C. Marques and Andr\'e Neves.
\newblock Morse index of multiplicity one min-max minimal hypersurfaces.
\newblock {\em \url{https://arxiv.org/abs/1803.04273}}, 2018.

\bibitem[MNS19]{MarquesNevesSong}
Fernando~C. Marques, Andr\'{e} Neves, and Antoine Song.
\newblock Equidistribution of minimal hypersurfaces for generic metrics.
\newblock {\em Invent. Math.}, 216(2):421--443, 2019.

\bibitem[Mod87]{Modica}
Luciano Modica.
\newblock The gradient theory of phase transitions and the minimal interface
  criterion.
\newblock {\em Arch. Rational Mech. Anal.}, 98(2):123--142, 1987.

\bibitem[Nev14]{neves:ICM}
Andr\'e Neves.
\newblock New applications of min-max theory.
\newblock {\em Proceedings of International Congress of Mathematics,
  \url{https://arxiv.org/abs/1409.7537}}, pages 939--957, 2014.

\bibitem[Pac12]{Pacard12}
Frank Pacard.
\newblock The role of minimal surfaces in the study of the {A}llen-{C}ahn
  equation.
\newblock In {\em Geometric analysis: partial differential equations and
  surfaces}, volume 570 of {\em Contemp. Math.}, pages 137--163. Amer. Math.
  Soc., Providence, RI, 2012.

\bibitem[Pit81]{Pitts}
Jon~T. Pitts.
\newblock {\em Existence and regularity of minimal surfaces on {R}iemannian
  manifolds}, volume~27 of {\em Mathematical Notes}.
\newblock Princeton University Press, Princeton, N.J.; University of Tokyo
  Press, Tokyo, 1981.

\bibitem[Pog81]{Pogorelov}
Aleksei~V. Pogorelov.
\newblock On the stability of minimal surfaces.
\newblock {\em Dokl. Akad. Nauk SSSR}, 260(2):293--295, 1981.

\bibitem[PR03]{PacardRitore03}
Frank Pacard and Manuel Ritor\'e.
\newblock From constant mean curvature hypersurfaces to the gradient theory of
  phase transitions.
\newblock {\em J. Differential Geom.}, 64(3):359--423, 2003.

\bibitem[PS]{PacardSun}
Frank Pacard and Taoniu Sun.
\newblock Doubling construction for {CMC} hypersurfaces in {R}iemannian
  manifolds.
\newblock
  \url{http://www.cmls.polytechnique.fr/perso/pacard/Publications/PR-01.pdf}.

\bibitem[PW13]{PacardWei:stable}
Frank Pacard and Juncheng Wei.
\newblock Stable solutions of the {A}llen-{C}ahn equation in dimension 8 and
  minimal cones.
\newblock {\em J. Funct. Anal.}, 264(5):1131--1167, 2013.

\bibitem[Sav09]{Savin:DGconj}
Ovidiu Savin.
\newblock Regularity of flat level sets in phase transitions.
\newblock {\em Ann. of Math. (2)}, 169(1):41--78, 2009.

\bibitem[Sav10]{Savo}
Alessandro Savo.
\newblock Index bounds for minimal hypersurfaces of the sphere.
\newblock {\em Indiana Univ. Math. J.}, 59(3):823--837, 2010.

\bibitem[Sch83]{Sch83}
Richard Schoen.
\newblock Estimates for stable minimal surfaces in three-dimensional manifolds.
\newblock In {\em Seminar on minimal submanifolds}, volume 103 of {\em Ann. of
  Math. Stud.}, pages 111--126. Princeton Univ. Press, Princeton, NJ, 1983.

\bibitem[Sha17]{Sharp}
Ben Sharp.
\newblock Compactness of minimal hypersurfaces with bounded index.
\newblock {\em J. Differential Geom.}, 106(2):317--339, 2017.

\bibitem[Sim83]{Simon83}
Leon Simon.
\newblock {\em Lectures on geometric measure theory}, volume~3 of {\em
  Proceedings of the Centre for Mathematical Analysis, Australian National
  University}.
\newblock Australian National University, Centre for Mathematical Analysis,
  Canberra, 1983.

\bibitem[Sim97]{Simon97}
Leon Simon.
\newblock Schauder estimates by scaling.
\newblock {\em Calc. Var. Partial Differential Equations}, 5(5):391--407, 1997.

\bibitem[Smi82]{SimonSmith}
F.~Smith.
\newblock {\em On the existence of embedded minimal two spheres in the three
  sphere, endowed with an arbitrary {R}iemannian metric.}
\newblock PhD thesis, University of Melbourne, Supervisor: Leon Simon, 1982.

\bibitem[Son17]{Song}
Antoine Song.
\newblock Local min-max surfaces and strongly irreducible minimal {H}eegaard
  splittings.
\newblock {\em \url{https://arxiv.org/abs/1706.01037}}, 2017.

\bibitem[{Son}18]{Song:full-yau}
Antoine {Song}.
\newblock {Existence of infinitely many minimal hypersurfaces in closed
  manifolds}.
\newblock {\em \url{https://arxiv.org/abs/1806.08816}}, 2018.

\bibitem[SS81]{SchoenSimon}
Richard Schoen and Leon Simon.
\newblock Regularity of stable minimal hypersurfaces.
\newblock {\em Comm. Pure Appl. Math.}, 34(6):741--797, 1981.

\bibitem[Ste88]{Sternberg}
Peter Sternberg.
\newblock The effect of a singular perturbation on nonconvex variational
  problems.
\newblock {\em Arch. Rational Mech. Anal.}, 101(3):209--260, 1988.

\bibitem[Ton05]{Tonegawa05}
Yoshihiro Tonegawa.
\newblock On stable critical points for a singular perturbation problem.
\newblock {\em Comm. Anal. Geom.}, 13(2):439--459, 2005.

\bibitem[TW12]{TonegawaWickramasekera12}
Yoshihiro Tonegawa and Neshan Wickramasekera.
\newblock Stable phase interfaces in the van der {W}aals--{C}ahn--{H}illiard
  theory.
\newblock {\em J. Reine Angew. Math.}, 668:191--210, 2012.

\bibitem[Wan17]{Wang:Allard}
Kelei Wang.
\newblock A new proof of {S}avin's theorem on {A}llen-{C}ahn equations.
\newblock {\em J. Eur. Math. Soc. (JEMS)}, 19(10):2997--3051, 2017.

\bibitem[Whi87]{White:curvature}
Brian White.
\newblock Curvature estimates and compactness theorems in {$3$}-manifolds for
  surfaces that are stationary for parametric elliptic functionals.
\newblock {\em Invent. Math.}, 88(2):243--256, 1987.

\bibitem[Whi91]{White:bumpy.old}
Brian White.
\newblock The space of minimal submanifolds for varying {R}iemannian metrics.
\newblock {\em Indiana Univ. Math. J.}, 40(1):161--200, 1991.

\bibitem[Whi16]{White:PCMI}
Brian White.
\newblock Introduction to minimal surface theory.
\newblock In {\em Geometric analysis}, volume~22 of {\em IAS/Park City Math.
  Ser.}, pages 387--438. Amer. Math. Soc., Providence, RI, 2016.

\bibitem[Whi17]{White:bumpy.new}
Brian White.
\newblock On the bumpy metrics theorem for minimal submanifolds.
\newblock {\em Amer. J. Math.}, 139(4):1149--1155, 2017.

\bibitem[Whi18]{White:compactness.new}
Brian White.
\newblock On the compactness theorem for embedded minimal surfaces in
  3-manifolds with locally bounded area and genus.
\newblock {\em Comm. Anal. Geom.}, 26(3):659--678, 2018.

\bibitem[Wic14]{Wickramasekera14}
Neshan Wickramasekera.
\newblock A general regularity theory for stable codimension 1 integral
  varifolds.
\newblock {\em Ann. of Math. (2)}, 179(3):843--1007, 2014.

\bibitem[WW18]{WangWei2}
Kelei Wang and Juncheng Wei.
\newblock Second order estimates on transition layers.
\newblock {\em \url{https://arxiv.org/abs/1810.09599}}, 2018.

\bibitem[WW19]{WangWei}
Kelei Wang and Juncheng Wei.
\newblock Finite {M}orse index implies finite ends.
\newblock {\em Comm. Pure Appl. Math.}, 72(5):1044--1119, 2019.

\bibitem[Yau82]{Yau:problems}
Shing~Tung Yau.
\newblock Problem section.
\newblock In {\em Seminar on {D}ifferential {G}eometry}, volume 102 of {\em
  Ann. of Math. Stud.}, pages 669--706. Princeton Univ. Press, Princeton, N.J.,
  1982.

\bibitem[Zho15]{Zhou:posRic}
Xin Zhou.
\newblock Min-max minimal hypersurface in {$(M^{n+1},g)$} with {$Ric>0$} and
  {$2 \leq n\leq 6$}.
\newblock {\em J. Differential Geom.}, 100(1):129--160, 2015.

\bibitem[{Zho}19]{Zhou:multiplicity-one}
Xin {Zhou}.
\newblock {On the Multiplicity One Conjecture in Min-max theory}.
\newblock {\em \url{https://arxiv.org/abs/1901.01173}}, 2019.

\end{thebibliography}

\end{document}